\documentclass{amsart}

\usepackage[english]{babel}

\usepackage[letterpaper,top=2cm,bottom=2cm,left=3cm,right=3cm,marginparwidth=1.75cm]{geometry}

\usepackage{amssymb}
\usepackage{amsmath}
\usepackage{mathabx}

\usepackage{scalerel,stackengine}
\stackMath
\newcommand\reallywidehat[1]{%
\savestack{\tmpbox}{\stretchto{%
  \scaleto{%
    \scalerel*[\widthof{\ensuremath{#1}}]{\kern-.6pt\bigwedge\kern-.6pt}%
    {\rule[-\textheight/2]{1ex}{\textheight}}
  }{\textheight}%
}{0.5ex}}%
\stackon[1pt]{#1}{\tmpbox}%
}
\usepackage{amsthm}
\usepackage{xfrac}
\usepackage{mathrsfs}
\usepackage{indentfirst}
\usepackage[colorlinks=true, linkcolor=blue, urlcolor=blue]{hyperref}
\usepackage{cases}
\usepackage{graphicx}

\title{Global Solutions of Multispeed Semilinear Klein-Gordon Systems in Space Dimension Two}
\author{Xilu Zhu}
\date{}

\address{Dept. of Mathematics, University of Southern California,
 Los Angeles, CA 90089}
\email{xiluzhu@usc.edu}

\begin{document}
\maketitle
\newtheorem{lemma}{Lemma}[section]
\newtheorem{cor}[lemma]{Corollary}
\newtheorem{prop}[lemma]{Proposition}
\newtheorem{defn}[lemma]{Definition}
\newtheorem{theorem}[lemma]{Theorem}

\theoremstyle{definition}
\newtheorem{remark}[lemma]{Remark}
\begin{abstract}
We consider general semilinear, multispeed Klein-Gordon systems in space dimension two with some non-degeneracy conditions. We prove that with small initial data such solutions are always global and scatter to a linear solution. This result partly extends the previous result obtained by Deng \cite{y1}, who completely proved the 3D quasilinear case. To prove our result, we mainly work on Fourier side and explore the contribution from the vicinity of space-time resonaonce.
\end{abstract}

\section{Introduction}
\subsection{Introduction of the Problem}
\ \par
In this article, we will consider a system of semilinear, multispeed Klein-Gordon equations in space dimension two, namely
\begin{align}
    \left(\partial^2_t-c^2_\alpha \Delta+b^2_\alpha\right) u_\alpha=\sum_{\alpha,\beta,\gamma=1}^d A_{\alpha \beta \gamma} u_\beta u_\gamma,\ \ \ \ 1\le\alpha\le d, \tag{1.1} \label{1.1}
\end{align} 
where $A_{\alpha\beta\gamma},\ \alpha,\beta,\gamma\in\left\{1,2,\cdots,d\right\}$ are constants and the speeds $c_\alpha$ and the masses $b_\alpha$ are arbitrary positive parameters. In this article, we mainly consider the long-time evolution of small initial valued solution to (\ref{1.1}), which has been studied in many previous works.\par
Klein-Gordon equations of form (\ref{1.1}) naturally arise from many physical systems. For example, both Euler-Poisson system and Euler-Maxwell system for one-fluid can be reduced to Klein-Gordon equations after some suitable transforms. Euler-Possion system reads:
\begin{align*}
    \begin{cases}
    \partial_t n_e+\nabla\cdot\left(n_e v_e\right)=0, \\
    n_e m_e\left(\partial_t v_e+v_e\cdot\nabla v_e\right)+\nabla p(n_e)=e n_e \nabla\phi, \\
    \Delta\phi=4\pi e(n_e-n_0),
\end{cases}
\end{align*}
and Euler-Maxwell system for one-fluid reads:
\begin{align*}
\begin{cases}
    \partial_t n_e+\nabla\cdot\left(n_e v_e\right)=0, \\
    n_e m_e\left(\partial_t v_e+v_e\cdot\nabla v_e\right)+\nabla p(n_e)=-n_e e\left(E-\frac{b v_e^\bot}{c}\right), \\
    \partial_t b+c\cdot \mbox{curl}(E)=0,\\
    \partial_t E+c \nabla^\bot b=4\pi e n_e v_e.
\end{cases}
\end{align*}
It is pointed out that in \cite{a2} and \cite{y2} after some transforms the Euler-Poisson system can be reduced to
$$(\partial_t+i\Lambda)U=\mathcal{N},$$
and Euler-Maxwell system for one-fluid can be reduced to
$$\begin{cases}
    (\partial_t+i\Lambda_e)u_e=\mathcal{N}_e, \\
    (\partial_t+i\Lambda_b)u_b=\mathcal{N}_b,
\end{cases}$$
where $\Lambda(\xi)=\sqrt{a\left|\xi\right|^2+b}$, $\Lambda_e(\xi)=\sqrt{-\lambda\left|\xi\right|^2+1}$, $\Lambda_b(\xi)=\sqrt{-\left|\xi\right|^2+1}$ ($a,b>0$ are two positive constants and $0<\lambda<1$ is a constant) and $\mathcal{N}$, $\mathcal{N}_e, \mathcal{N}_b$ are three nonlinearties. This implies that on a linear level Euler-Possion system is actually a single Klein-Gordon equation and Euler-Maxwell system is actually a Klein-Gordon system with two equations and $c_1=\sqrt{d}, c_2=1, b_1=b_2=1$. In addition, the Euler-Maxwell system for two-fluid studied in \cite{ya1} also has strong connection to Klein-Gordon equations, since after suitable transforms some dispersive relations are also in the form of $\sqrt{-c_\alpha^2\Delta+b_\alpha^2}$. \par
Now, let's go back to the long-time behavior problem of the nonlinear Klein-Gordon system as in (\ref{1.1}). First of all, fix $\varphi(\xi)$ to be a real-valued radially symmetric bump function adapted to the ball $\left\{\xi\in\mathbb{R}^d:\left|\xi\right|\le 2\right\}$ which equals $1$ on the ball $\left\{\xi\in\mathbb{R}^d:\left|\xi\right|\le 1\right\}$ and notice that the basic dispersive estimate gives us 
$$\left\|P_k e^{-it\Lambda}f\right\|_{L^\infty}\lesssim (1+t)^{-\frac{n}{2}}(1+2^{2k})\left\|f\right\|_{L^1},$$
where $\widehat{P_k f}(\xi)\triangleq\left(\varphi(\xi/2^k)-\varphi(\xi/2^{k-1})\right)\widehat{f}(\xi)$ and $n$ refers to the dimension of the domain. This means that the decay of the linear solution of the Klein-Gordon system is $t^{-\frac{3}{2}}$ (respectively $t^{-1}$) as $t\rightarrow +\infty$ in 3D (respectively 2D). Therefore, for the purpose of proving the decay of global wellposedness and scattering, it is easier when the dimension is high. In this sense, the easiest case is a single Klein-Gordon equation in 3D case. It was studied by Klainerman \cite{s1} and Shatah \cite{ja1} independently. Later on, Hayashi, Naumkin and Wibowo \cite{n1} considered the case of arbitrarily many equations in 3D with the same speed but different masses and Delort, Fang and Xue \cite{j1} considered the case of two equations in 2D with the same speed but different masses (they also imposed a non-resonance condition). Note that the above works are more or less based on physical space analysis. \par
After that, a useful method called Space-time Resonance Method was developed by Germain, Masmoudi and Shatah. The idea is to combine the classical concept of resonances with the feature of dispersive relations and use the corresponding analytical methods explore them. It is used to understand the global existence for nonlinear dispersive equations, set in the whole space $\mathbb{R}^d$, and with small data. Early works involving this space-time resonance method include \cite{p1}, \cite{p2}, \cite{p3} and \cite{p4}. In the context of multispeed Klein-Gordon equations, Germain \cite{p1} considered the case of two equations in 3D with the same mass but different speeds. More generally, we are also interested in the multispeed, multimass Klein-Gordon system with arbitrarily many equations. In recent years, more and more people are using this new, Fourier-based methods to deal with space-time resonances to attack these problems. Ionescu and Pausader \cite{a1} obtained the global existence in 3D with two nondegeneracy assumptions (See (\ref{1.2}) below). Later on, Deng \cite{y1} removed these two nondegeneracy conditions in 3D by modifing the definition of the Z-norm and utilizing the rotation vector fields method. \par
Now, in this paper, we will prove the same result in 2D but with two nondegeneracy assumptions and semilinear nonlinearity. Here is the statment of our main theorem.
\begin{theorem}
    Consider the system (\ref{1.1}) in $\left[0,+\infty\right)\times\mathbb{R}_x^2$. Assume the following \underline{nondegeneracy} \underline{conditions}
    \begin{align*}
        \begin{cases}
        b_\sigma-b_\mu-b_\nu\neq 0 \\
        (c_\mu-c_\nu)(c_\mu^2b_\nu-c_\nu^2b_\mu)\ge 0    \end{cases}\ \ \ \ \ \ \ \ \mbox{for any }\sigma,\mu,\nu\in\left\{1,\dots,d\right\}\tag{1.2}.\label{1.2}
    \end{align*}
    Define $\left\|\cdot\right\|_X$, $N_0$, $N_1$, $H^{N_1}_\Omega$ as in Section 2.2 below and let $\varepsilon_0>0$ be small enough depending on $b_\alpha$ and $c_\alpha$. If the initial data $\textbf{u}(0)=\textbf{g}$, $\partial_t \textbf{u}(0)=\textbf{h}$ satisfy the bound
    $$\left\|\left(\textbf{g},\partial_x \textbf{g},\textbf{h}\right)\right\|_X\le \varepsilon \le \varepsilon_0 \ll 1,$$
    then there exists a unique global solution $\textbf{u}$, with prescribed initial data, such that
    $$(\textbf{u},\partial_x \textbf{u}, \partial_t \textbf{u})\in C\left(\left[0,+\infty\right):H^{N_0}(\mathbb{R}^2)\cap H^{N_1}_\Omega(\mathbb{R}^2)\right).$$
    Moreover, we have the following decay estimates
    $$\left\|(\textbf{u},\partial_x \textbf{u},\partial_t \textbf{u})\right\|_{H^{N_0}\cap H^{N_1}_\Omega}+\sup_{a\le N_1/2} (1+t)^{0.999}\left\|\Omega^a \textbf{u}(t)\right\|_{L^\infty}\lesssim \varepsilon,$$
    and there exist $\mathbb{R}^d$ valued functions $\textbf{w}$ verifying the linear equation
    $$\left(\partial_t^2-c_\alpha^2\Delta+b_\alpha^2\right)w_\alpha=0$$
    such that we have scattering in slightly weaker spaces
    $$\lim_{t\rightarrow\pm\infty}\left(\left\|\textbf{u}(t)-\textbf{w}(t)\right\|_\mathcal{H}+\left\|(\partial_x,\partial_t)(\textbf{u}(t)-\textbf{w}(t))\right\|_\mathcal{H}\right)=0,$$
    where $\left\|\textbf{f}\,\right\|_\mathcal{H}\triangleq\left\|\textbf{f}\,\right\|_{H^{N_0-1}}+\sup\limits_{\beta\le N_1-1}\left\|\Omega^\beta \textbf{f}\,\right\|_{L^2}.$
\end{theorem}
\vspace{3em}
\subsection{Description of the Methods}
\ \par
\subsubsection{Outline of the Proof} 
\ \par
Our proof will partly follow the proof in \cite{y2}.\par
The local result follows from \cite{s1}, so we only need to extend the local result to the global result.\par
The idea is to introduce the \underline{profile}:
$$\boldsymbol{V}_\sigma (t)\triangleq e^{it\Lambda_\sigma}\boldsymbol{v}_\sigma(t),$$
which can be viewed as taking back the linear solution. We expect that this profile $\boldsymbol{V}_\sigma$ will stay bounded and settle down as time goes to infinity if the initial data is small enough characterized by the Z-norm. To prove it, we just need to prove a bootstrap type result in a well designed space X: 
$$\mbox{if the initial data is small, and }\sup\limits_{\left[0,T\right]}\left\|\boldsymbol{V}_\sigma\right\|_X\lesssim \varepsilon\mbox{, then we will have }\sup\limits_{\left[0,T\right]}\left\|\boldsymbol{V}_\sigma\right\|_X\lesssim \varepsilon^{3/2}.$$
The exact statement is given in Proposition 2.3 below. Note that the X-norm includes two parts: the $H^s$ Sobolev norms and the Z-norm. Therefore, we mainly have two steps in our proof. The first part is to prove the energy estimate, namely to control the $H^s$ Sobolev norms;
the second part is to prove the dispersive estimate, namely to control the Z-norm. Note that in order to have the energy estimate, we intuitively require the estimate of form
$$\left\|P_k e^{-it\Lambda}f\right\|_{L^\infty}\lesssim (1+t)^{-1}(1+2^{2k})\left\|f\right\|_Z.$$
For the construction of the Z-norm, see Section 1.2.2 below, and for the exact definition of the Z-norm, see Section 2.2 below. \par
As for the energy estimate, the proof is basically same as the one in \cite{y2}. Due to less pointwise decay $t^{-1+\varepsilon}$ in 2D, we cannot use Gronwall argument that was often used before to prove this. Instead, authors in \cite{y2} used the $TT^*$ method and did the tangential integration by parts to prove the energy estimate. In our paper, since we assume our nonlinearity to be semilinear, we will not have the lose of derivatives. Therefore, our proof of energy estimate here is easier than that in the earlier paper \cite{y2}. For the quasilinear case, please see Section 1.2.5 below.\par
As for the dispersive estimate, we first use the Duhamel formula to write down the expression of the profile
\begin{align*}
    \widehat{V_\sigma}(t,\xi)=\widehat{V_\sigma}(0,\xi)+\int_0^t \int_{\mathbb{R}^2} e^{is\Phi_{\sigma\mu\nu}(\xi,\eta)}\widehat{V_\mu}(s,\xi-\eta)\widehat{V_\nu}(s,\eta)\,d\eta ds. \tag{1.3}\label{1.3}
\end{align*}
Now, we observe that if the weight function $\widehat{V_\mu}(s,\xi-\eta)\widehat{V_\nu}(s,\eta)$ in the double integral in (\ref{1.3}) is Schwartz and independent of $s$, then the main contribution of the double integral comes from the vicinity of the set of \underline{space-time resonance} $\mathcal{R}$, where
\begin{align*}
    \mathcal{T}&\triangleq\left\{(\xi,\eta):\Phi_{\sigma\mu\nu}(\xi,\eta)=0\right\},\\
    \mathcal{S}&\triangleq\left\{(\xi,\eta):\nabla_\eta\Phi_{\sigma\mu\nu}(\xi,\eta)=0\right\},\\
    \mathcal{R}&\triangleq\mathcal{T}\cap\mathcal{S}.
\end{align*}
(We also call $\mathcal{T}$ the \underline{time-resonant set} and call $\mathcal{S}$ the \underline{space-resonant set}.) This is because it's easy to see that away from the set $\mathcal{R}$, one can integrate by parts in (\ref{1.3}) to either the space variable $\eta$ or the time variable $s$ to get an enough decay. To control the integral around the set $\mathcal{R}$, we need two ingredients. The first one is to explore the elementary properties of the phase function 
$$\Phi_{\sigma\mu\nu}(\xi,\eta)=\sqrt{c_\sigma^2\left|\xi\right|^2+b_\sigma^2}-\sqrt{c_\mu^2\left|\xi-\eta\right|^2+b_\mu^2}-\sqrt{c_\nu^2\left|\eta\right|^2+b_\nu^2},$$
and then use these elementary properties of $\Phi_{\sigma\mu\nu}(\xi,\eta)$ to control the size and possible structures of the set $\mathcal{R}$. The second one is to quantitatively decompose the integral into two regions. One is near the resonant points, where we may have to use the volume estimate to control the integral; the other one is away from the resonant points, where we can integrate by parts to get a very fast decay. \par
To quantitatively find the boundary of these two regions, we will do the dyadic decomposition and analyze a bunch of cases depending on the relative sizes of the main parameters $m,l,j,j_1,j_2,k,k_1,k_2,n$, $n_1,n_2$ and so on, where we have the following:\par
$m$ corresponds to the time: $t\sim 2^m$;\par
$l$ corresponds to the size of $\left|\Phi_{\sigma\mu\nu}(\xi,\eta)\right|$: $\left|\Phi_{\sigma\mu\nu}(\xi,\eta)\right|\sim 2^l$;\par
$j$ corresponds to the support of the output in the x-space: $\mbox{supp} f^\sigma \approx \left\{x:\left|x\right|\sim 2^j\right\}$ ($j_1$, $j_2$ means the support of the inputs $f^\mu, f^\nu$ in the x-space);\par 
$k$ corresponds to the support of the output in the Fourier side: $\mbox{supp} \widehat{f^\sigma} \approx \left\{\xi:\left|\xi\right|\sim 2^k\right\}$ ($k_1$, $k_2$ means the support of the inputs $f^\mu, f^\nu$ in the Fourier side).\par
Using variables $j$ and $k$, we can decompose any function $f^\sigma=\sum_{j,k}f^\sigma_{j,k}$, where $f_{j,k}$ has coordinates localized in $2^j$ and frequency localized in $2^k$. Analogusly, we decompose $f^\mu=\sum_{j_1,k_1}f^\mu_{j_1,k_1}$ and $f^\nu=\sum_{j_2,k_2}f^\nu_{j_2,k_2}$. For later purposes, we need to introduce some extra parameters $n,n_1,n_2$, where $2^{-n}$ ($2^{-n_1}, 2^{-n_2}$ respectively) means the distance between $\xi$ ($\xi-\eta$, $\eta$ respectively) and resonant spheres. Therefore, we can further decompose $f^\sigma=\sum_{j,k,n}f^\sigma_{j,k,n}$ ($f^\mu=\sum_{j_1,k_1,n_1}f^\mu_{j_1,k_1,n_1}$,  $f^\nu=\sum_{j_2,k_2,n_2}f^\nu_{j_2,k_2,n_2}$ respectively). Please see Section 2.1, formula (\ref{3.1}) and proof of Lemma 4.2 for the detailed meaning of these main parameters. \par
Also, we will do the angle decomposition in our proof. This trick was introduced by Deng, Ionescu and Pausader in \cite{y1} and \cite{y2}. If $\angle\xi,\eta$ is not very small, then we can integrate by parts to the angle to get a very fast decay, otherwise we may again have to use the volume estimate to control the integral. This angle restriction provides another way to further control the volume of the resonance region. This trick of angle decomposition is realized by inserting the cutoff function $\varphi\left(\kappa_r^{-1} \Omega_\eta\Phi(\xi,\eta)\right)$, where $\Omega$ is the rotation field $x_1 \partial_2-x_2 \partial_1$ and $\kappa_r$ is a number that is usually picked as $2^{-\frac{m}{2}+\varepsilon}$. By some computation, it turns out that 
\begin{align*}
    \Omega_\eta\Phi_{\sigma\mu\nu}(\xi,\eta)&=\frac{c_\mu^2}{\sqrt{c_\mu^2\left|\xi-\eta\right|^2+b_\mu^2}}\cdot \left\langle\xi,\eta^\bot\right\rangle=\frac{c_\mu^2}{\sqrt{c_\mu^2\left|\xi-\eta\right|^2+b_\mu^2}}\cdot\left|\xi\right|\cdot\left|\eta\right|\cdot\pm\sin\angle\xi,\eta, 
\end{align*}
so, if $\left|\xi\right|,\left|\xi-\eta\right|,\left|\eta\right|\sim 1$ and $\angle\xi,\eta\ll 1$, then we get $\left|\Omega_\eta\Phi_{\sigma\mu\nu}(\xi,\eta)\right|\approx \left|\angle\xi,\eta\right|$. We remark here that 
this use of rotation vector field is consistent with the general vector field method used by Klainerman \cite{s1}.

\vspace{3em}
\subsubsection{Construction of the Z-norm}
\ \par
In this subsection, we will talk about the idea to consturct the Z-norm (See Definition 2.1).\par
In fact, it turns out that the Z-norm used in \cite{y2} also works here. The main factors in the Z-norm are $2^j$ and $2^{-\frac{n}{2}}$ here, so let's discuss why this is the case. The main idea here is to (at least formally) estimate $\left\|\widehat{f^\sigma}\right\|_{L^2}$ by using the Duhamel's formula and iterating one time
\begin{align*}
    \widehat{f^\sigma}(t,\xi)=\int_0^t \int_{\mathbb{R}^2} e^{is\Phi_{\sigma\mu\nu}(\xi,\eta)}\widehat{f^\mu}(\xi-\eta)\widehat{f^\nu}(\eta)\,d\eta ds, \tag{1.4}\label{1.4}
\end{align*}
where $\widehat{f^*}\triangleq e^{it\Lambda_*}u_*,\ \ \ *\in\left\{\sigma,\mu,\nu\right\}$ are so-called profiles. Note that since this is the first iteration, so $\widehat{f^\mu}$ and $\widehat{f^\nu}$ in the integrand are independent of time $s$. Also, for simplicity, we may just take $\widehat{f^\sigma}(0,\xi)=0$ and $\widehat{f^\mu}, \widehat{f^\nu}$ to be Schwarz functions. We know that the main contribution of the integral (\ref{1.4}) comes from the vicinity of space-time resonant points, namely
$$\mathcal{R}=\left\{(\xi,\eta):\Phi_{\sigma\mu\nu}(\xi,\eta)=\nabla_\eta\Phi_{\sigma\mu\nu}(\xi,\eta)=0\right\}.$$
In view of Proposition 6.5 (a), we denote $p:\mathbb{R}^2\rightarrow\mathbb{R}^2$ such that
$$(\nabla_\eta\Phi)(\xi,\eta)=0\ \ \ \iff\ \ \ \eta=p(\xi).$$
It turns out that we can describe the structure of $\mathcal{R}$ as 
$$\mathcal{R}=\left\{(\xi,\eta):\xi=\alpha e,\ \eta=\beta e,\ e\in\mathbb{S}^2,\ (\alpha,\beta)\in F\right\},$$
where $F$ is a finite set. (For the proof, see Proposition 8.2 in \cite{y1}.) According to (\ref{2.3}), we may also just approximate $\Phi(\xi,\eta(\xi))\approx\big|\left|\xi\right|-\gamma\big|$ around the resonant sphere, where $\gamma$ is a constant. WLOG, we can fix $\gamma=1$ in the following. Also, let's do the localization in time, and after localization consider the part of RHS of (\ref{1.4}) to be
\begin{align*}
    \int_{2^m}^{2^{m+1}} \int_{\mathbb{R}^2} e^{is\Phi_{\sigma\mu\nu}(\xi,\eta)}\widehat{f^\mu}(\xi-\eta)\widehat{f^\nu}(\eta)\,d\eta ds. \tag{1.4a}\label{1.4a}
\end{align*}
Now, we only need to consider the case when $\left|\nabla_\eta\Phi\right|\ll 1$, and by oscillatory integral theory we can calculate
\begin{align*}
    \mbox{(\ref{1.4a})}&\approx\int_{2^m}^{2^{m+1}} e^{is\Phi(\xi,p(\xi))}\frac{1}{s}\,ds \approx\int_{2^m\big|\left|\xi\right|-1\big|}^{2^{m+1}\big|\left|\xi\right|-1\big|} e^{i\Tilde{s}}\frac{1}{\Tilde{s}}\,d\Tilde{s}\\
    &\lesssim\frac{1}{2^m\big|\left|\xi\right|-1\big|}\int_{2^m\big|\left|\xi\right|-1\big|}^{2^{m+1}\big|\left|\xi\right|-1\big|} e^{i\Tilde{s}}\,d\Tilde{s} \sim \frac{1}{2^m\big|\left|\xi\right|-1\big|}\left(e^{i\left(2^{m+1}\big|\left|\xi\right|-1\big|\right)}-e^{i\left(2^m\big|\left|\xi\right|-1\big|\right)}\right) \\
    &\lesssim \frac{1}{2^m\big|\left|\xi\right|-1\big|}.
\end{align*}
Therefore, we can conclude  $\left\|\widehat{f^\sigma}\right\|_{L^2}\lesssim 2^{-m}\cdot 2^{\frac{n}{2}}$. By the finite propagation of Klein-Gordon equations, we may assume $j\le m$. Thus, we get $\left\|\widehat{f^\sigma}\right\|_{L^2}\lesssim 2^{-j}\cdot 2^{\frac{n}{2}}$, which accounts for the reason why we have such two factors in the definition of our Z-norm. However, we remark that this is not enough. We will see the reasons in the next subsection.\par

\vspace{3em}
\subsubsection{Challanges in our Problem}
\ \par
(i) In all previous works in 2D (\cite{a2}, \cite{j1}, \cite{y2}), people have not encountered the iterated resonance yet. Namely, in previous works in 2D, the resonance inputs and outputs are separated. For example, in \cite{y2}, the authors were studying Klein-Gordon systems with only two equations, where the dispersive relations are explicit and easy to deal with; it turns out that 
\begin{align*}
    \mbox{if }\left|\xi-\eta\right|,\left|\eta\right|\in\left\{\gamma_1,\gamma_2\right\}\mbox{ and }\xi\parallel\eta,\mbox{ then }\left|\Phi\right|\gtrsim 1 \tag{1.5}\label{1.5} 
\end{align*}
(See Proposition 8.5 (ii), (iii) in \cite{y2}). In this case, we can avoid the smallness of $\Phi$ or $\nabla_\eta\Phi$ that will potentially give terrible estimates.\par 
In 3D, the existence of iterated resonance points will not be a big deal, since we have a faster pointwise decay $t^{-\frac{3}{2}}$ of the linear solutions. Using spherical symmetry and rotation vector field methods with more precise analysis of the phase function, Deng \cite{y1} overcame this issue. However, in 2D, due to less pointwise decay $t^{-1}$ of the linear solution, the iterated resonant points case becomes more difficult to deal with. In this paper, although two nondegeneracy conditions are imposed to avoid the existence of very degenerate resonant points, we still need to deal with the iterated resonant points. This is because in our multispeed and multimass Klein-Gordon system, we will not have such properties as in (\ref{1.5}) anymore.\par
To get over this issue, we first need to do a correction of the "pre-Z-norm" described in the Section 1.2.2. This is because in the case of iterated resonance, we could have some terrible cases. For example, let's consider a case when $j_1<j_2=\frac{m}{2}$, $j=m$, $k_1,k_2\sim 1$, $\widehat{f^\mu_{j_1,k_1}}(\xi-\eta)=\varphi_{j_1}(\left|\xi-\eta\right|-\gamma_1), \ \ \widehat{f^\nu_{j_2,k_2}}(\eta)=\varphi_{j_2}(\left|\eta\right|-\gamma_2)$ and $\left|\Phi(\xi,\eta)\right|\sim 2^l=2^{-m}$, where $\gamma_i$ are the roots of $\Psi(\xi)\triangleq\Phi(\xi,p(\xi))$ (Again see Section 2.1, formula (\ref{3.1}) for the meaning of these main parameters). Then, we need to consider
$$\widehat{f^\sigma}(\xi)\approx 2^{-\frac{m}{2}}\int_{2^m}^{2^{m+1}}ds \int_{\mathbb{R}} e^{is\Phi(\xi,\eta)}\,\varphi_l(\Phi(\xi,\eta))\,\varphi_{j_1}(\left|\xi-\eta\right|-\gamma_1)\,\varphi_{j_2}(\left|\eta\right|-\gamma_2)\,d\eta,$$
where the coefficient $2^{-\frac{m}{2}}$ comes from the angle cutoff function $\varphi\left(\kappa_\theta^{-1}\Phi(\xi,\eta)\right)$ and we slightly abuse the notation to view $\eta$ as a 1D variable. Now, observe that 
$$\left|\Psi(\xi)\right|=\left|\Phi(\xi,p(\xi))\right|\lesssim \left|\Phi(\xi,\eta)\right|+\left|\nabla_\eta\Phi\right|\cdot \left|\eta-p(\xi)\right|\lesssim 2^{-m},$$
which implies that $n=j=m$ and $\Phi(\xi,\eta)=\Psi(\xi)+O(2^{-m})$. Then, we just need to consider
\begin{align*}
    \widehat{f^\sigma}(\xi)&\approx 2^{-\frac{m}{2}}\int_{2^m}^{2^{m+1}} e^{is\Psi(\xi)}ds \int_{\mathbb{R}} \varphi_l(\Phi(\xi,\eta))\,\varphi_{j_1}(\left|\xi-\eta\right|-\gamma_1)\,\varphi_{j_2}(\left|\eta\right|-\gamma_2)\,d\eta \\
    &\lesssim 2^{-\frac{m}{2}}\cdot 2^m \cdot 2^{-\frac{m}{2}}\sim 1.
\end{align*}
Since $\left|\xi-\gamma_3\right|\lesssim 2^{-m}$, we have $\left\|\widehat{f^\sigma}\right\|_{L^2}\lesssim 2^{-\frac{m}{2}}$. In this case, $2^{-\frac{m}{2}}$ is exactly equal to $2^{-j+\frac{n}{2}}$, so we don't have a convergence factor which will cause issues. (Note that we need to sum up over $j_1,j_2,k_1,k_2$ etc., so we must have a convergence factor in the beginning.) To resolve \vspace{0.4em} this issue, we have to slightly weaken the definition of Z-norm, namely requiring $\left\|\widehat{f^\sigma_{j,k,n}}\right\|_{L^2}\lesssim 2^{-j+\frac{n}{2}+\varepsilon j}$. Now, in this case, we can change our assumption to be $\widehat{f^\mu_{j_1,k_1}}(\xi-\eta)=\varphi_{j_1}(\left|\xi-\eta\right|-\gamma_1)\cdot 2^{\varepsilon j_1} \ \ \widehat{f^\nu_{j_2,k_2}}(\eta)=\varphi_{j_2}(\left|\eta\right|-\gamma_2)\cdot 2^{\varepsilon j_2}$ and redo \vspace{0.4em} the above calculation to get that $\left\|\widehat{f^\sigma}\right\|_{L^2}\lesssim 2^{-\frac{m}{2}+\varepsilon(j_1+j_2)}$, which is strictly less than $2^{-\frac{m}{2}+\varepsilon m}$. So, we can gain a convergence factor here to make sure that the sum over $j_1,j_2,k_1,k_2$ is summable.\par
Next, we have to implement more precise estimates. One example is that due to iterated resonance in our paper we need to additionally consider the case when $\left|\Phi_{\sigma\mu\nu}\right|,\left|\nabla_\eta\Phi_{\sigma\mu\nu}\right|\ll 1$ and $n_1,n_2>0$. In view of Definition 2.1 and Lemma 3.6, if $n_i=0$ ($i\in\left\{1,2\right\}$), then both $L^2$ and $L^\infty$ norm of $\widehat{f_{j_i,k_i,n_i}}$ ($i\in\left\{1,2\right\}$) would be better than the case when $n_i>0$. So, if $n_1,n_2>0$ which is possible in our paper, then the previous arguments in \cite{y2} do not work any more and we have to analyze the integral more precisely (E.g. See Section 5.3). Another important example is that due to the lack of (\ref{1.5}), we will lose some factors in the coefficients in Proposition 6.10 and Corollary 6.11 (compare our Proposition 6.10 and Lemma 8.10 in \cite{y2}). We are still able to prove our desired result, but this will significantly increase our workload of the proof, since Proposition 6.10 and Corollary 6.11 are often used to deal with the case when $j_1,j_2\ge m$, which frequently occurs in our paper. Here, Proposition 6.10 and Corollary 6.11 basically follow from Schur's test. Instead of using them, we may apply a change of variable 
\begin{align*}
    \begin{cases}
        x=\left|\eta\right|^2 \\ y=\left|\xi-\eta\right|^2
    \end{cases}
\end{align*}
to estimate the integral (See Lemma 3.12).\par 
\vspace{2em}
(ii) Compare to the case of two equations in \cite{y2}, our paper, the case of multiple equations, will lose some other elementary properties of the phase function.\par
One example is that we have to consider more complicated second order interactions of space-time resonances and time resonances occurred in Section 5.4.1. For example, in \cite{y2}, we have the following relationship
\begin{align*}
    &\mbox{if }\left|\Phi_{\sigma\mu\nu}(\xi,\eta)\right|,\left|\Psi_{\mu\theta\kappa}(\eta)\right|\ll 1, \\
    &\mbox{then }\left|\nabla_\eta\left[\Phi_{\sigma\mu\nu}(\xi,\eta)+\Psi_{\nu\theta\kappa}(\eta)\right]\right|\le \varepsilon\Rightarrow\left|\nabla_\xi\Phi_{\sigma\mu\nu}(\xi,\eta)\right|\lesssim \varepsilon \\
    &\ \ \ \ \ \ \ \ \mbox{and }\left|\nabla_\xi\Phi_{\sigma\mu\nu}(\xi,\eta)\right|\le \varepsilon\Rightarrow\left|\nabla_\eta\left[\Phi_{\sigma\mu\nu}(\xi,\eta)+\Psi_{\nu\theta\kappa}(\eta)\right]\right|\lesssim \varepsilon,\tag{1.6}\label{1.6} 
\end{align*}
(See Proposition 8.6 in \cite{y2}), which will simplify the discussion a lot. However, in our multispeed and multimass Klein-Gordon system case, we don't have this property (\ref{1.6}) anymore. So, we have to consider all the cases in terms of the size of $\left|\nabla_\xi\Phi_{\sigma\mu\nu}\right|$ and $\left|\nabla_\eta\left[\Phi_{\sigma\mu\nu}(\xi,\eta)+\Psi_{\nu\theta\kappa}(\eta)\right]\right|$ in our proof, which significantly complicate our proof.\par
Another example is that in our paper in the low frequency and small phase function case (i.e. $\min\left\{k,k_1,k_2\right\}<-D$ and $\left|\Phi\right|\ll 1$), right now it's possible that two inputs are near resonant points, which is not possible in \cite{y2}. This will result in an additional weaker term $f_{NCw}$ in our time derivative decomposition in Lemma 3.12. This will occur when $j_1\le m$, $j_2\ge m$ and $n_2>0$, since in this case we are not able to use the $L^\infty\times L^2$ estimate directly in Lemma 3.5 as before in \cite{y2}. To deal with this weaker term $f_{NCw}$ in the dispersive control (See Section 5.4.1), we once again need to estimate the integral more precisely. For example, one way is to plug in the time derivative expression and consider the trilinear integral directly. 
\vspace{3em}
\subsubsection{Why do we need two nondegeneracy conditions?}
\ \par
In this subsection, we point out that two nondegeneracy conditions in (\ref{1.2}) are actually reasonable here. Define the phase function $\Phi(\xi,\eta)=\Phi_{\sigma\mu\nu}(\xi,\eta)$ as in (\ref{2.1}). Note that $\Phi(0,0)=b_\alpha-b_\mu-b_\nu\neq 0$, so the first condition in (\ref{1.2}) guarantees that $(0,0)$ cannot be space-time resonant. This implies that we can avoid "the all-low frequency case" $\left|\xi\right|, \left|\xi-\eta\right|, \left|\eta\right|\ll 1$, which could cause some issues. The main reason here is that if the frequency is low, then we might not always be able to integrate by parts in the angle $\angle\xi,\eta$. Therefore, we cannot always insert the angle cutoff function $\varphi\left(\kappa_\theta^{-1} \Phi(\xi,\eta)\right)$ (See Remark 3.3 below) to the integral need to be controlled, which means we will lose a lot during the volume control. The other reason is that in view of the common used estimate $\left\|\hat{f}\right\|_{L^\infty}$ in (\ref{3.4}), we can find that if $k\sim 1$, then we will have $\left\|\hat{f}\right\|_{L^\infty}\lesssim 2^{2\delta n}$ (In the main case $j\le m+D$, this will give $\left\|\hat{f}\right\|_{L^\infty}\lesssim 2^{2\delta m}$). However, if $k\ll 1$, then we only have $\left\|\hat{f}\right\|_{L^\infty}\lesssim 2^{-21\delta k}$, which is clearly worse than the previous one. \par
Moreover, Deng pointed out in \cite{y1} that the first condition in (\ref{1.2}) is not only reasonable but also necessary. One example here is that
\begin{align*}
    \Phi(\xi,\eta)=\sqrt{\left|\xi\right|^2+1}-\sqrt{2\left|\xi-\eta\right|^2+4}+\sqrt{\left|\eta\right|^2+1}=\frac{1}{2}\left\langle\xi,\eta\right\rangle+O\left(\left|\xi\right|^4+\left|\eta\right|^4\right).\tag{1.9}\label{1.9}
\end{align*}
Intuitively, if we set two localized inputs as
$$\widehat{f_\mu}(s,\xi-\eta)=2^{cj}\varphi_j(\xi-\eta),\ \ \ \ \ \widehat{f_\nu}(s,\eta)=2^{cj}\varphi_j(\eta)$$
with $j=m/4$ and some constant $c$, where $\varphi_j$ is defined in Section 2.1. Now, if we also localize the output at $\left|\xi\right|\sim 2^{-3j}$, then we observe that $\left|\Phi\right|\lesssim 2^{-m}$ due to (\ref{1.9}), so in view of the Duhamel's formula in (\ref{1.3}), the oscillatory factor $e^{is\Phi}$ is irrelevant. Then, the localized output $\widehat{f_\sigma}$ will be approximately $2^m\,2^{-2j}\,2^{2cj}=2^{(3j)(2c+2)/3}$. Thus, if we start wich $c=0$ and keep doing the iteration, we will get
\begin{align*}
    &\widehat{f_\sigma^{(1)}}(\xi)=2^{\frac{2}{3}\cdot 3j}\varphi_{-3j}(\xi)\\
    &\widehat{f_\sigma^{(2)}}(\xi)=2^{\frac{10}{9}\cdot 9j}\varphi_{-9j}(\xi)\\
    &\dots
\end{align*}
This gives that $\left\|\widehat{f_\sigma^{(1)}}\right\|_{L^2}\sim 2^{-j}$,  $\left\|\widehat{f_\sigma^{(2)}}\right\|_{L^2}\sim 2^{j}$ and so on. So we can see that after our second iteration, the $L^2$ norm of the localized output has been already very large, provided that $j>0$ is taken very large. This implies that the first condition in (\ref{1.2}) is actually necessary. This construction will be made precise in the paper []. \par
The second condition in (\ref{1.2}) guarantees that $(\xi,\eta)$ cannot be a degenerate space-time resonant point if $\eta\neq 0$. Namely, denote $q:\mathbb{R}^2\rightarrow\mathbb{R}^2$ be a mapping such that 
$$(\nabla_\eta\Phi)(\xi,\eta)=0\ \ \ \iff\ \ \ \xi=q(\eta)$$
and we have that
\begin{align*}
    \det\left[(\nabla_{\eta,\eta}^2\Phi)(q(\eta),\eta)\right]\neq 0\ \ \ \ \ \mbox{if }\eta\neq 0 \mbox{ and }\Phi(q(\eta),\eta)=0. \tag{1.7}\label{1.7}
\end{align*}
This result was proved by some elementary calculations in \cite{a1} (See Section 1.2.5 of \cite{a1}). We will frequently use Lemma 6.7 to control the volume during our proof. Consider the integral
\begin{align*}
    \int_{\mathbb{R}^2} e^{is\Phi(\xi,\eta)}A(\xi,\eta)\,d\eta, \tag{1.8}\label{1.8}
\end{align*}
where $A(\xi,\eta)$ is a weight function independent of time $s$. Then, we can see that if we have the condition (\ref{1.7}), then (\ref{1.8}) would be approximately $O(s^{-\frac{1}{2}})\cdot O(s^{-\frac{1}{2}})=O(s^{-1})$, otherwise the best control would be only $O(s^{-\frac{1}{2}})\cdot O(s^{-\frac{1}{3}})=O(s^{-\frac{5}{6}})$. Therefore, the condition (\ref{1.7}) imposed here will help to simplify the problem. \par

\vspace{3em}
\subsubsection{Open Problems}
\ \par
First, it remains to study whether the second condition in (\ref{1.2}) is also necessary or not. Can we remove the second condition in (\ref{1.2}) and still prove the same global result for the 2D Multispeed Klein-Gordon system? Otherwise, can we find a counterexample to demonstrate that the second condition in (\ref{1.2}) is also necessary like in \cite{x1} \par
Second, in Section 1.2.1 we talked about that when proving the energy estimate in Section 4, for simplicity we just assume our nonlinear term on the RHS of (\ref{1.1}) is semilinear. Now it's still an open question to generalize our nonlinearity from semilinear to quasilinear (with some symmetry assumption), namely
\begin{align*}
    &\mbox{The RHS of (\ref{1.1})} \\
    =\ &\sum_{\beta,\gamma=1}^d \left(\sum_{j,k=1}^2 A^{jk}_{\alpha\beta\gamma}\, u_\gamma\cdot\partial_j \partial_k u_\beta+\sum_{j,k,l=1}^3 B^{jkl}_{\alpha\beta\gamma}\,\partial_l u_\gamma\cdot\partial_j \partial_k u_\beta \right)+\mathcal{B}_\alpha\left[u,\partial u\right]\ \ \ (1\le \alpha\le d), \tag{1.9}\label{1.9}
\end{align*}
where $A$ and $B$ are tensors symmetric (The symmetry assumption is needed to obtain local well-posedness; see \cite{t1}) and $\mathcal{B}_\alpha$ is an arbitrary quadratic form (of constant coefficient) of $u$ and $\partial u$. In the previous paper \cite{y2}, an explicit nonlinearity with actual physical meaning was given, and it turned out that a certain gain of derivative exists. Thus, the authors in \cite{y2} were able to prove the energy estimate. However, in the general quasilinear case as in (\ref{1.9}), we don't have such an advantage anymore. Therefore, we ask the following questions. Can we just generalize the result to quasilinearity without additional assumption (except the symmetry assumption)? If not, what assumptions do we need?\par
\vspace{3em}
\subsection{Plan of this Article}
\ \par
As said before, our proof will partly follow the proof in \cite{y2}. \par
In Section 2, we will define the relevant notations and particularly the Z-norm. Then, we will give the statements of two main ingredients, namely the local wellpossedness and the main bootstrap proposition.\par
In Section 3, we will collect all main lemmas. This section mainly includes integration by parts-style lemmas, linear dispersive estimates and in particular two time derivative decomposition lemmas, which will be our main tool or results used in later proof.\par 
In Section 4, we will prove the energy estimate. The proof is basically same as the one in \cite{y2}. \par
In Section 5, we will finish the control of the Z norm. In short, we will first integrate by parts in time and divide into several parts to do. Note that here is the place where we need to use our improved time derivative decomposition in Section 3. This section contains our main work.\par
In Section 6, we collect all elementary lemmas. These contains volume estimates, the properties related to the phase function $\Phi$ and several $L^2\times L^2$ estimate lemmas related to the resonant points. \par
\vspace{3em}

\vspace{1em}
\section{Function Spaces and the Main Result}
\subsection{Littlewood-Paley Projections and Notations}
\ \par
To do the Littlewood-Paley projections later on, we define $\varphi:\mathbb{R}\rightarrow\left[0,1\right]$ to be an even smooth function that is supported in $\left[-1.2,1.2\right]$ and equals 1 in $\left[-1.1,1.1\right]$. Abusing the notation, we also write $\varphi:\mathbb{R}^2\rightarrow\left[0,1\right]$ to be the corresponding radial function on $\mathbb{R}^2$. Let
\begin{align*}
    \varphi_k(x)&\triangleq \varphi\left(\left|x\right|/2^k\right)-\varphi\left(\left|x\right|/2^{k-1}\right)\ \ \ \ \ \mbox{for any }k\in\mathbb{Z},\\
    \varphi_I&\triangleq \sum_{m\in I\cap\mathbb{Z}} \varphi_m \ \ \ \ \ \mbox{for any }I \subseteq\mathbb{R}.
\end{align*}
For any $c\in\mathbb{R}$ let
$$\varphi_{\le c}\triangleq\varphi_{\left(-\infty,c\right]},\ \ \varphi_{\ge c}\triangleq\varphi_{\left[c,+\infty\right)},\ \ \varphi_{< c}\triangleq\varphi_{\left(-\infty,c\right)},\ \ \varphi_{> c}\triangleq\varphi_{\left(c,+\infty\right)}.$$
For any $a<b\in\mathbb{Z}$ and $j\in\left[a,b\right]\cap\mathbb{Z}$ let
\begin{align*}
    \varphi_j^{\left[a,b\right]}\triangleq\begin{cases}
        \varphi_j&\mbox{if }a<j<b,\\
        \varphi_{\le a}&\mbox{if }j=a,\\
        \varphi_{\ge b}&\mbox{if }j=b.
    \end{cases}
\end{align*}\par
For any $k\in\mathbb{Z}$ let $k_+\triangleq \max(k,0)$ and $k_-\triangleq \min(k,0)$. Let
$$\mathcal{J}\triangleq\left\{(k,j)\in\mathbb{Z}\times\mathbb{Z}_+: k+j\ge 0\right\}.$$
For any $(k,j)\in\mathcal{J}$ let
\begin{align*}
    \Tilde{\varphi}_j^{(k)}(x)\triangleq\begin{cases}
        \varphi_{\le -k}(x)&\mbox{if }k+j=0\mbox{ and }k\le 0,\\
        \varphi_{\le 0}(x)&\mbox{if }j=0\mbox{ and }k\ge 0,\\
        \varphi_j(x)&\mbox{if }k+j\ge 1\mbox{ and }j\ge 1,
    \end{cases}
\end{align*}
and notice that, for any $k\in\mathbb{Z}$ fixed, we have $\sum_{j\ge -\min(k,0)}\Tilde{\varphi}_j^{(k)}=1$.\par
Let $P_k$, $k\in\mathbb{Z}$, denote the operator on $\mathbb{R}^2$ defined by the Fourier multiplier $\xi\rightarrow\varphi_k(\xi)$. Let $P_{\le c}$ (respectively $P_{>c}$) denote the operators on $\mathbb{R}^2$ defined by the Fourier multipliers $\xi\rightarrow\varphi_{\le c}(\xi)$ (respectively $\xi\rightarrow\varphi_{>B}(\xi)$). For $(k,j)\in\mathcal{J}$ let $Q_{jk}$ denote the operator
$$(Q_{jk}f)(x)\triangleq\Tilde{\varphi}_j^{(k)}(x)\cdot P_k f(x).$$
In view of the uncertainty principle the operators $Q_{jk}$ are relevant only when $2^j\,2^k\gtrsim 1$, which explains the definitions above.\par
Let $\Lambda_\alpha=\sqrt{-c_\alpha^2\Delta+b_\alpha^2}$ be the linear phase and define
$$\Lambda_\alpha(\xi)=\Lambda_\alpha(\left|\xi\right|)\triangleq\sqrt{c_\alpha^2\left|\xi\right|^2+b_\alpha^2};\ \ \ b_{-\alpha}=-b_\alpha,\ \ \ c_{-\alpha}=c_\alpha,\ \ \ \Lambda_{-\alpha}=-\Lambda_\alpha$$
for $1\le \alpha\le d$. For $\sigma,\mu,\nu\in\left\{-d,\dots,-2,-1,1,2,\dots d\right\}$, we define the associated nonlinear phase
\begin{align*}
    \Phi_{\sigma\mu\nu}(\xi,\eta)\triangleq \Lambda_\sigma(\xi)-\Lambda_\mu(\xi-\eta)-\Lambda_\nu(\eta) \tag{2.1}\label{2.1}
\end{align*}

and we often omit the subscripts if no misunderstanding will occur. 
\vspace{3em}
\subsection{the Z-norm}
\ \par
Note that in Lemma 5.6 in \cite{a1}, it is proved that 
$$\nabla_\eta\Phi(\xi,\eta)=0\ \ \ \Leftrightarrow\ \ \ \eta=p(\xi),$$
where $p_+:\mathbb{R}\longrightarrow\mathbb{R}$ is an odd smooth function such that $p(\xi)=p_+(\left|\xi\right|)\frac{\xi}{\left|\xi\right|}$. As in \cite{y2}, we define 
\begin{align}
   \Psi_{\sigma\mu\nu}\triangleq\Phi_{\sigma\mu\nu}(\xi,p(\xi)),\ \ \Psi^*_\sigma(\xi)\triangleq 2^{D_0}\,(1+\left|\xi\right|)\inf\limits_{\mu,\nu;b_\mu+b_\nu\neq0}{\left|\Psi_{\sigma\mu\nu}(\xi)\right|}. \tag{2.2}\label{2.2}
\end{align}
By taking $D_0=D_0(c_\sigma,c_\mu,c_\nu,b_\sigma,b_\mu,b_\nu)$ large enough and applying Lemma 5.8 in \cite{a1}, we know that all roots of $\Psi_{\sigma\mu\nu}$ are simple. Moreover, if $\left|\xi\right|$ is very large, then Corollary 6.2 tells us that $\left|\Psi\right|\gtrsim \frac{1}{2^{\Bar{k}}}$, where $\Bar{k}=\max\left\{k,k_1,k_2,0\right\}$, so we know that $\left|\Psi^*_\sigma\right|\gtrsim 1$ whenever $\left|\xi\right|$ is very large. To sum up, we get that
\begin{align*}
    \Psi^*_{\sigma}\approx
  \begin{cases}
   \left(1+\left|\xi\right|\right)^{-1}\cdot\min\limits_{\mbox{\footnotesize all }\Psi_{\sigma\mu\nu}\mbox{\footnotesize 's roots } \gamma_i}\left\{\big|\left|\xi\right|-\gamma_i\big|\right\}, &\mbox{if }\Psi_{\sigma\mu\nu}\mbox{ has at least one root;} \\
   1,&\mbox{if }\Psi_{\sigma\mu\nu}\mbox{ has no root;}
\end{cases}.\tag{2.3}\label{2.3}
\end{align*}
(Note that we could have at most four roots depending on $c_\sigma,c_\mu,c_\nu,b_\sigma,b_\mu,b_\nu$.) 
In addition, if $\left|\Psi^*_\sigma\right|\ll 1$ , then $\left|\xi\right|$ must be near some $\gamma_i$, which means that $\left|\xi\right|\sim 1$. Therefore, by Proposition 6.4, we know that $\left|\nabla_\xi\Phi_{\sigma\mu\nu}(\xi,p(\xi))\right|\gtrsim 1$, since $\left|\Phi_{\sigma\mu\nu}(\xi,\eta)\right|\ll 1$ and $\left|\nabla_\eta\Phi_{\sigma\mu\nu}(\xi,p(\xi))\right|=0$. Note that we have
$$\nabla\Psi_{\sigma\mu\nu}(\xi)=(\nabla_\xi\Phi_{\sigma\mu\nu})(\xi,p(\xi))+(\nabla_\eta\Phi_{\sigma\mu\nu})(\xi,p(\xi))\nabla p(\xi)=(\nabla_\xi\Phi_{\sigma\mu\nu})(\xi,p(\xi)),$$
which implies that $\left|\nabla\Psi_{\sigma\mu\nu}\right|\gtrsim 1$.
For $n\in\mathbb{Z}$ we define the operators $A^\sigma_n$ by
$$\widehat{A^\sigma_n f}(\xi)\triangleq\varphi_{n}\left(\Psi^*_\sigma(\xi)\right)\cdot\hat{f}(\xi),$$
for $\sigma\in\left\{-d,\cdots,-2,-1,1,2,\cdots,d\right\}$. Given an integer $j\ge 0$ we define the operators $A^{\sigma}_{n,(j)},n\in{0,\cdots,j+1}$, by
$$A^{\sigma}_{0,(j)}\triangleq\sum_{n^\prime\le 0}{A^\sigma_{n\prime}}\,,\ \ A^{\sigma}_{j+1,(j)}\triangleq\sum_{n^\prime\ge j+1}{A^\sigma_{n\prime}}\,,\ \ A^\sigma_{n,(j)}\triangleq A^\sigma_n\ \ \ \ \ \mbox{ if }0<n<j+1.$$
We fix a constant $D=D(c_\sigma,c_\mu,c_\nu,b_\sigma,b_\mu,b_\nu)>0$ which is large enough. Now, we are ready to define the Z norms. First, we pick $N_0$,$N_1$ large enough, and $\delta$ small enough. For example, $\delta\triangleq 4\cdot 10^{-7}$,$N_1\triangleq 8/\delta^2$,$N_0\triangleq 400/\delta^2$ would work. Let $\Omega\triangleq x_1 \partial_2-x_2 \partial_1$ denote the rotation vector-field, and define
$$H^{N_1}_\Omega\triangleq \left\{f\in L^2(\mathbb{R}^2):\left\|f\right\|_{H^{N_1}_\Omega}\triangleq\sup_{m\le N_1}\left\|\Omega^m f\right\|_{L^2}<\infty\right\}.$$
\vspace{4em}
\begin{defn}[Z-norm]
For $\sigma\in\left\{1,2,\dots,d\right\}$ we define
$$Z^\sigma_1\triangleq\left\{f\in L^2(\mathbb{R}^2):\left\|f\right\|_{Z^\sigma_1}\triangleq\sup_{(k,j)\in\mathcal{J}} 2^{6k_+}2^{(1-20\delta)j} \sup\limits_{0\le n\le j+1} 2^{-\left(\frac{1}{2}-19\delta\right)n} \left\|A^\sigma_{n,(j)}Q_{jk}f\right\|_{L^2}<\infty\right\}.$$
Then, we define
$$Z\triangleq\left\{(f_1,f_2,\dots,f_d)\in L^2\times L^2\times\dots\times L^2:\left\|(f_1,f_2,\dots,f_d)\right\|_Z\triangleq\sup_{m\le N_1/2}\sum_{i=1}^d \left\|\Omega^m f_i\right\|_{Z_1^i}<\infty\right\}.$$
Finally, we denote 
$$\left\|\textbf{f}\,\right\|_X\triangleq\left\|\textbf{f}\,\right\|_{H^{N_0}}+\sup\limits_{\beta\le N_1}\left\|\Omega^\beta \textbf{f}\,\right\|_{L^2}+\left\|\textbf{f}\,\right\|_Z,$$
where $\mathcal{J}\triangleq\left\{(k,j)\in\mathbb{Z}\times\mathbb{Z}^+:k+j\ge 0\right\}.$
\end{defn}
\vspace{3em}
\subsection{Proof of the Main Result}
\ \par
The proof of the main result Theorem 1.1 is based on the following two important propositions.
\begin{prop}
    Suppose $\textbf{g,h}:\mathbb{R}^2\rightarrow\mathbb{R}^d$ are such that $\left\|\left(\textbf{g},\partial_x \textbf{g},\textbf{h}\right)\right\|_{H^{N_0}\cap H^{N_1}_\Omega}\le \varepsilon_0$, then there exists a unique solution $\textbf{u}$ to (\ref{1.1}) such that
    $$(\textbf{u},\partial_x \textbf{u}, \partial_t \textbf{u})\in C\left(\left[0,+\infty\right):H^{N_0}(\mathbb{R}^2)\cap H^{N_1}_\Omega(\mathbb{R}^2)\right);\,\,\,\,\,\,\textbf{u}(0)=\textbf{g},\partial_t \textbf{u}(0)=\textbf{h}.$$
    Moreover, if $\left\|\left(\textbf{g},\partial_x \textbf{g},\textbf{h}\right)\right\|_Z\le \varepsilon_0$, then we have $\textbf{V}(t)\in C(\left[0,1\right]\rightarrow Z)$, where $\boldsymbol{V}_\sigma (t)\triangleq e^{it\Lambda_\sigma}\boldsymbol{v}_\sigma(t)$ is the profile.
\end{prop}
\begin{proof}
    This is proved, with slightly different spaces, in \cite{a1}, Proposition 2.1 and 2.4; the proof in our case is basically the same.
\end{proof}
\vspace{4em}
\begin{prop}
    Assume two nondegeneracy conditions as in  (\ref{1.2}). Suppose $\textbf{u}$ is a solution to (\ref{1.1}) on a time interval $\left[0,T\right]$ with initial data $\textbf{u}(0)=\textbf{g}$ and $\textbf{u}_t (0)=\textbf{h}$ such that
    $$(\textbf{u},\partial_x \textbf{u}, \partial_t \textbf{u})\in C\left(\left[0,+\infty\right):H^{N_0}(\mathbb{R}^2)\cap H^{N_1}_\Omega(\mathbb{R}^2)\right),$$
    let $\textbf{V}(t)$ be defined accordingly. Assume
    $$\left\|\left(\textbf{g},\partial_x \textbf{g},\textbf{h}\right)\right\|_X\le\varepsilon \le \varepsilon_0,\,\,\,\,\,\,\sup_{0\le t \le T}\left\|\textbf{V}(t)\right\|_X\le \varepsilon_1 \ll 1,$$
    then we have
    $$\sup_{0\le t\le T} \left\|\textbf{V}(t)\right\|_X \lesssim \varepsilon_1^{3/2}+\varepsilon.$$
\end{prop}
\vspace{4em}
Proposition 2.3 will be proved in section 5.\par
By proposition 2.2 and 2.3, one can easily prove the main result Theorem 1.1 by a standard bootstrap argument.

\vspace{1em}
\section{Main Lemmas}
From now on, we will assume two nondegenarcy conditions as in (\ref{1.2}).\par
We first list two lemmas about integration by parts that will be used often later on. The first result is a standard result about integration by parts.
\begin{lemma}
(i) Assume that $0<\varepsilon\le 1/\varepsilon\le K, N\ge 1$ is an integer, and $f,g\in C^N(\mathbb{R}^2)$. Then
$$\left|\int_{\mathbb{R}^2} e^{iKf} g\,dx\right|\lesssim_{ N }(K\varepsilon)^{-N} \left[\sum_{ \left|\alpha\right|\leq N} \varepsilon ^ {\left|\alpha\right|}\left\| D _ {x} ^ {\alpha}g\right\|_{L^1} \right] ,$$
provided that $f$ is real-valued,
$$\left|\nabla_x f\right|\ge \mathbf{1}_{\operatorname{supp}{g}},\ \ \mbox{and }\left\|D_x^{\alpha}f\cdot\mathbf{1}_{\operatorname{supp}g}\right\|_{L^\infty}\lesssim_N \varepsilon^{1-\left|\alpha\right|},\ 2\le\left|\alpha\right|\le N+1.$$\par
(ii) Similarly, if $0<\rho\le 1/\rho \le K$, then
$$\left|\int_{\mathbb{R}^2} e^{iKf} g\,dx\right|\lesssim_{ N }(K\rho)^{-N} \left[\sum_{m\leq N} \rho^m \left\| \Omega^m 
 g\right\|_{L^1} \right] ,$$
provided that $f$ is real-valued,
$$\left|\Omega f\right|\ge \mathbf{1}_{\operatorname{supp}{g}},\ \ \mbox{and }\left\|\Omega^m f\cdot\mathbf{1}_{\operatorname{supp}g}\right\|_{L^\infty}\lesssim_N \rho^{1-m},\ 2\le m\le N+1.$$
\end{lemma}
The next result is about integration by parts using the vector field $\Omega$. 
\begin{lemma}
Assume that $t\approx 2^m, m\ge 0, k,k_1,k_2\in\mathbb{Z}, L\le 1 \le U$, and 
$$2^{-m/2}\le L\le 2^{k_1},\ \ \ 2^k+2^{k_1}+2^{k_2}\le U \le U^2\le 2^{m/2}L.$$
Assume that $A\ge 1+2^{-k_1}$ and
$$\sup_{0\le a\le 100} \left[\left\|\Omega^a g\right\|_{L^2}+\left\|\Omega^a f\right\|_{L^2}\right]+\sup_{0\le\left|\alpha\right|\le N} A^{-\left|\alpha\right|}\left\|D^\alpha\hat{f}\right\|_{L^2}\le 1,$$
Fix $\xi\in\mathbb{R}^2$ and $\Phi=\Phi_{\sigma\mu\nu}$ as in (\ref{2.1}), and let
$$I_p=I_p(f,g)\triangleq\int_{\mathbb{R}^2} e^{it\Phi(\xi,\eta)} \varphi_p(\Omega_\eta \Phi(\xi,\eta)) \varphi_k(\xi) \varphi_{k_1}(\xi-\eta) \varphi_{k_2}(\eta) \hat{f}(\xi-\eta) \hat{g}(\eta)\,d\eta.$$
If $U^4 2^{-m}\le 2^{2p}\le 1$ and $AL^{-1}U^2\le 2^m$ then
$$\left|I_p\right|\lesssim_N \left(2^{p+m}\right)^{-N} \left[2^{m/2}+U^4 2^{-p}+U^2 L^{-1} A\,2^p\right]^N+2^{-10m}.$$
\end{lemma}
\vspace{0.8em}
\begin{remark} 
In this paper, the main case would be $\left|\xi\right|,\left|\xi-\eta\right|,\left|\eta\right|\sim 1$. Since $$\Omega_\eta \Phi(\xi,\eta)=\frac{\lambda_\mu^\prime(\left|\xi-\eta\right|)}{\left|\xi-\eta\right|}\left(\xi\cdot\eta^{\bot}\right),$$ one can understand $\Omega_\eta \Phi(\xi,\eta)\sim\sin\angle(\xi,\eta)$. Furthermore, if $\left|\Omega_\eta \Phi(\xi,\eta)\right|\ll 1$, then $\Omega_\eta \Phi(\xi,\eta)\sim\angle(\xi,\eta)$.
\end{remark}

\vspace{2em}
We will also use the following lemmas to estimate the integrals later on. The first one is the classical Schur's test.
\vspace{0.8em}
\begin{lemma} [Schur's test]
Consider the operator T given by
$$T(\xi)=\int_{\mathbb{R}^2} K(\xi,\eta)f(\eta)\,d\eta.$$
Assume that
$$\sup_\xi\int_{\mathbb{R}^2}\left|K(\xi,\eta)\right|\,d\eta \le K_1, \ \ \ \sup_\eta \int_{\mathbb{R}^2}\left|K(\xi,\eta)\right|\,d\xi \le K_2.$$
Then
$$\left\|Tf\right\|_{L^2}\lesssim \sqrt{K_1 K_2} \left\|f\right\|_{L^2}.$$
\end{lemma}
\vspace{0.8em}
The next one gives an estimation of the phase function-localized oscillatory integrals.
\vspace{0.8em}
\begin{lemma}
Let $s\approx 2^m, m\ge 0$, and $(1+\varepsilon)\nu\le m$ for some $\varepsilon>0$. Let $\Phi=\Phi_{\sigma\mu\nu}(\xi,\eta)\triangleq\Lambda_\sigma(\xi)-\Lambda_\mu(\xi-\eta)-\Lambda_\mu(\eta)$ and assume that $1/2=1/q+1/r$ and $\chi$ is a Schwartz function. Then
\begin{align*}
 &\left\|\varphi_{\le 10m}(\xi)\int_{\mathbb{R}^2} e^{is\Phi(\xi,\eta)} \chi\left(2^\nu \Phi(\xi,\eta)\right) \hat{f}(\xi-\eta) \hat{g}(\eta)\,d\eta\right\|_{L_\xi^2}   \\
 \lesssim & \sup_{t\in\left[s/10,10s\right]} \left\|e^{it\Lambda_\mu} f\right\|_{L^q} \left\|e^{it\Lambda_\nu} g\right\|_{L^r}+2^{-10m}\left\|f\right\|_{L^2}\left|g\right\|_{L^2},
\end{align*}
where the constant in the inequality only depends on $\varepsilon$ and the function $\chi$.
\end{lemma}
\vspace{0.8em}
Lemma 3.1, Lemma 3.2, Lemma 3.4 and Lemma 3.5 are proved in \cite{y2}.\par
\vspace{2em}
Next, we will give some useful linear estimates.\par
\begin{lemma}
    Assume that $\sigma\in\left\{1,2,\dots,d\right\}$ and
    $$\left\|f\right\|_{Z_1^\sigma\cap H_\Omega^{N_1/8}}\le 1.$$
    For any $(k,j)\in\mathcal{J}$ and $n\in\left\{0,\dots,j+1\right\}$ let
    \begin{align}
    f_{j,k}\triangleq P_{\left[k-2,k+2\right]}Q_{jk}f,\,\,\,\,\,\widehat{f_{j,k,n}}(\xi)\triangleq\varphi_{-n}^{\left[-j-1,0\right]}\left(\Psi^*_\sigma(\xi)\right)\widehat{f_{j,k}}(\xi). \tag{3.1}\label{3.1}  
    \end{align}
    For any $\xi_0\in\mathbb{R}^2\backslash\left\{0\right\}$ and $\kappa,\rho\in[0,\infty)$ let $\mathscr{R}(\xi_0;\kappa,\rho)$ denote the rectangle
    $$\mathscr{R}(\xi_0;\kappa,\rho)\triangleq\left\{\xi\in\mathbb{R}^2:\big|(\xi-\xi_0)\cdot\xi_0/\left|\xi_0\right|\big|\le\rho, \left|(\xi-\xi_0)\cdot\xi_0^\bot/\left|\xi_0\right|\right|\le\kappa\right\}.$$
    Then, for any $(k,j)\in\mathcal{J}$ and $n\in\left\{0,\dots,j+1\right\}$,
    \begin{align*}
    &\left\|\sup\limits_{\theta\in\mathbb{S}^1} \left|\widehat{f_{j,k,n}}(r\theta)\right|\right\|_{L^2(rdr)}+\left\|\sup\limits_{\theta\in\mathbb{S}^1} \left|f_{j,k,n}(r\theta)\right|\right\|_{L^2(rdr)} \\
    \lesssim&\,2^{-5k_+} 2^{(1/2-19\delta)n-(1-20\delta)j+2\delta^2 j}, \tag{3.2}\label{3.2} \\
    &\sup\limits_{\kappa+\rho\le 2^{k-10}}\int_{\mathbb{R}^2}\left|\widehat{f_{j,k,n}}(\xi)\right|\bold{1}_{\mathscr{R}(\xi_0;\kappa,\rho)}(\xi)\,d\xi \\
    \lesssim&\,2^{-5k_+} 2^{-j+21\delta j} 2^{-19\delta n} \kappa 2^{-k/2} \min\left(1,2^n \rho\right)^{1/2}, \tag{3.3}\label{3.3} \\
&\left\|\widehat{f_{j,k,n}}\right\|_{L^\infty}\lesssim\begin{cases}
        2^{2\delta n} 2^{-(1/2-21\delta)(j-n)} &\mbox{if }2^k\approx 1, \\
        2^{-5k_+} 2^{-21\delta k} 2^{-(1/2-21\delta)(j+k)} &\mbox{if } 2^k\gg 1 \mbox{ or } 2^k\ll 1,
    \end{cases} \tag{3.4}\label{3.4} \\
&\left\|D^\alpha\widehat{f_{j,k,n}}\right\|_{L^\infty}\lesssim_{\left|\alpha\right|}\begin{cases}
        2^{\left|\alpha\right| j}2^{2\delta n} 2^{-(1/2-21\delta)(j-n)} &\mbox{if }2^k\approx 1, \\
        2^{\left|\alpha\right| j}2^{-5k_+} 2^{-21\delta k} 2^{-(1/2-21\delta)(j+k)} &\mbox{if } 2^k\gg 1 \mbox{ or } 2^k\ll 1.
    \end{cases} \tag{3.5}\label{3.5} \\
    \end{align*}
    Moreover, if $m\ge 0$ and $\left|t\right|\in\left[2^m-1,2^{m+1}\right]$ then
    \begin{align*}
        2^{\delta^2 n}\left\|e^{-it\Lambda_\sigma} f_{j,k,n}\right\|_{L^\infty}\lesssim\begin{cases}
        2^{-m+20\delta j} &\mbox{if } 2^k \ll 1, \\
        2^{-j+20\delta j} &\mbox{for all }j,k,m, \\
        2^{-m+2\delta m} 2^{-3/4k_-} &\mbox{if } j\le (1-\delta ^2)m +k_-.\end{cases} \tag{3.6}\label{3.6}
    \end{align*}
    In particular,
    \begin{align}
        \left\|e^{-it\Lambda_\sigma} P_k f\right\|_{L^\infty} \lesssim \left(1+\left|t\right|\right)^{-1+21\delta}. \tag{3.7}\label{3.7}
    \end{align}
\end{lemma}
\begin{proof}
    It's proved in \cite{y2} as well.
\end{proof}
\vspace{0.8em}
\begin{remark}
 By rechecking the proof, we could improve the bound in (\ref{3.4}) to 
 \begin{align}
     &\left\|\widehat{f_{j,k,n}}\right\|_{L^\infty}\lesssim\,2^{1.01\delta n} 2^{-(1/2-21\delta)(j-n)} \tag{3.7a}\label{3.7a} \\
     \mbox{or }&\left\|\widehat{f_{j,k,n}}\right\|_{L^\infty}\lesssim\,2^{\delta n+2\delta^2 n} 2^{-(1/2-21\delta)(j-n)}\tag{3.7b}\label{3.7b}
 \end{align}
 when $2^k\approx 1$.   
\end{remark}
\vspace{3em}
Recall (\ref{1.1}). We now suppose $u_\alpha$ ($\alpha\in\left\{1,2,\cdots,d\right\}$) is the solution to 
\begin{align}
    \left(\partial^2_t-c^2_\alpha \Delta+b^2_\alpha\right) u_\alpha=\sum_{\alpha,\beta,\gamma=1}^d A_{\alpha\beta \gamma} u_\beta u_\gamma,\ \ \ \ 1\le\alpha\le d, \tag{3.8} \label{3.8}
\end{align} 
where $A_{\alpha\beta\gamma},\ \alpha,\beta,\gamma\in\left\{1,2,\cdots,d\right\}$ are constants. Let $v_\sigma\triangleq\left(\partial_t-i\Lambda_\sigma\right)u_\sigma$ for $\sigma\in\left\{1,2,\cdots,d\right\}$, then\vspace{0.4em} $\displaystyle{u_\sigma=\Lambda_\sigma^{-1}(\mbox{Im}v_\sigma)=\Lambda_\sigma^{-1}\frac{v_\sigma-\Bar{v_\sigma}}{2i}}$. Therefore, we can get that
\begin{align}
   \left(\partial_t+i\Lambda_\sigma\right)v_\sigma=\sum_{\alpha,\beta,\gamma=1}^d A_{\alpha \beta \gamma} \left(\Lambda_\beta^{-1}\frac{v_\beta-\Bar{v_\beta}}{2i}\right)\left(\Lambda_\gamma^{-1}\frac{v_\gamma-\Bar{v_\gamma}}{2i}\right). \tag{3.9}\label{3.9}
\end{align}\par
Moreover, we define the profile $\boldsymbol{V}_\sigma (t)\triangleq e^{it\Lambda_\sigma}\boldsymbol{v}_\sigma(t)$.\par
\begin{prop}
   Suppose $\textbf{u}$ is the solution to (\ref{3.8}) on a time interval $[0,T]$ with initial data $\textbf{u}(0)=\textbf{g},\partial_t\textbf{u}(0)=\textbf{h}$. Assume that
   $$\left\|(\textbf{g},\partial_x\textbf{g},\textbf{h})\right\|_X\le\varepsilon_0,\ \ \ \ \ \sup\limits_{0\le t\le T}\left\|\textbf{V}(t)\right\|_X\le\varepsilon_1$$
   then for $\sigma\in\left\{1,2,\cdots,d\right\},k\in\mathbb{Z},t\in[0,T]$, we have
   \begin{gather}
      \left\|v_\sigma(t)\right\|_{H^{N_0}\cap H^{N_1}_\Omega}\lesssim\varepsilon_1, \tag{3.10}\label{3.10}\\
      \sup\limits_{\left|\mu\right|\le t}\sum_{a\le N_1/2} {\left\|e^{-i\mu\Lambda_\sigma}P_k\Omega^a v_\sigma(t)\right\|_{L^\infty}\lesssim\varepsilon_1 \left(1+t\right)^{-1+21\delta}}, \tag{3.11}\label{3.11} \\
      \left\|P_{\le k}(\partial_t+i\Lambda_\sigma)v_\sigma(t)\right\|_{H^{N_0}}+\sum_{a\le N_1}\left\|P_{\le k}\Omega^a(\partial_t+i\Lambda_\sigma)v_\sigma(t)\right\|_{L^2}\lesssim\varepsilon_1^2 2^{k_+} (1+t)^{-1+22\delta}. \tag{3.12}\label{3.12}
   \end{gather}
   Moreover, for $0\le a\le N_1/2$ we can decompose $\Omega^a(\partial_t+i\Lambda_\sigma)v_\sigma(t)=G_2(t)+G_\infty(t)$ such that
   \begin{equation}
       \begin{aligned}
       \sup\limits_{\left|\mu\right|\le(1+t)^{1-\delta/4}} \left\|e^{-i\mu\Lambda_\sigma}P_k G_\infty(t)\right\|_{L^\infty}\lesssim\varepsilon_1^2 (1+t)^{-2+50\delta},\\
       \left\|P_k G_2(t)\right\|_{L^2}\lesssim\varepsilon_1^2(1+t)^{-5/4+60\delta}.
       \end{aligned} 
       \tag{3.13}\label{3.13}
   \end{equation}
\end{prop}
\begin{proof}
(\ref{3.10}) is from our assumption. (\ref{3.11}) is proved in Lemma 3.6. (\ref{3.12}) is proved by considering the profile of $v_\sigma$, using Duhamel's formula and using (\ref{3.10}), (\ref{3.11}) and 
$$\sup\limits_{\left|\mu\right|\le t}\sum_{a\le N_1/2} {\left\|e^{-i\mu\Lambda_\sigma}P_{\le k}\Omega^a v_\sigma(t)\right\|_{L^\infty}\lesssim\varepsilon_1 \left(1+t\right)^{-1+21\delta}},\ \ \ \ \ \ \ \mbox{if }k\le 0.$$
Note that the above inequality follows from Lemma 5.2 (i) in \cite{a1} and Lemma 3.12 in \cite{y2}. (\ref{3.13}) follow from Corollary 3.11 below.
\end{proof}
\vspace{2em}
Finally, we will give two time derivatives estimates. The first one is a little bit rough. \par
Recall that we have already defined $\boldsymbol{V}_\sigma$ as $\boldsymbol{V}_\sigma (t)\triangleq e^{it\Lambda_\sigma}\boldsymbol{v}_\sigma(t)$. Then, Duhamel formula gives us that
$$\partial_t \widehat{V_\sigma}(\xi,s)=\sum_{\mu,\nu\in\mathcal{P}} A_{\mu\nu}\int_{\mathbb{R}^2}e^{is\Phi_{\sigma\mu\nu}(\xi,\eta)}\widehat{V_\mu}(\xi-\eta,s)\widehat{V_\nu}(\eta,s)\,d\eta.$$
\vspace{0.8em}
\begin{lemma}[time derivatives 1]
Under the assumption of Proposition 3.8, let $m\ge 0$, $s\approx 2^m$, $k\in\mathbb{Z}$. Then
\begin{align*}
    &\left\|P_{\le k} (\partial_t V_\sigma)(s)\right\|_{H^{N_0}}+\sum_{a\le N_1} \left\|P_{\le k} \Omega^a (\partial_t V_\sigma)(s)\right\|_{L^2}\lesssim \varepsilon_1^2 2^{k_+} 2^{-m+22\delta m},\tag{3.14}\label{3.14} \\
    &\sum_{a\le {N_1}/2} \left\|e^{-is\Lambda_\sigma} P_k \Omega^a (\partial_t V_\sigma)(s)\right\|_{L^\infty}\lesssim \varepsilon_1^2 \min\left\{2^{-2m+43\delta m}, 2^{2k}\right\}. \tag{3.15}\label{3.15}
\end{align*}
In addition, for any $a\in\left[0,N_1/2\right]\cap\mathbb{Z}$, we have the following decomposition of time derivative
\begin{align}
   \Omega^a (\partial_t V_\sigma)(s)=\varepsilon_1^2\left(f_C(s)+f_{NC}(s)\right), \label{3.16}\tag{3.16}
\end{align}
where, with $\Phi_{\sigma\mu\nu}$ as (\ref{2.2}) and for any $k\in\mathbb{Z}$,
\begin{equation}
\begin{aligned}
   &\widehat{f_C}(\xi,s)=\sum_{\mu+\nu\neq0} e^{is\Psi_{\sigma\mu\nu}(\xi)} \sum_{0\le q\le m/2-10\delta m} g^q_{\sigma\mu\nu}(\xi,s),\\
   &\varphi_k(\xi) g^q_{\sigma\mu\nu}(\xi,s)=0\ \ \mbox{if }q\ge 1\ \mbox{and }k\le -D\ \mbox{or }\\[4pt]
   &\mbox{if }k\notin[-m/2+\delta^2 m/5,\delta^2 m/5],\\[4pt]
   &\left|\varphi_k(\xi) D^\alpha_\xi g^q_{\sigma\mu\nu}(\xi,s)\right|\lesssim 2^{-42\delta k_-}\,2^{-m+3\delta m}\, 2^{-q+42\delta q}\,2^{(m/2+q+2\delta^2 m)\left|\alpha\right|},\\[4pt]
   &\left\|\varphi_k \cdot \partial_s g^q_{\sigma\mu\nu}(s)\right\|_{L^\infty}\lesssim 2^{(6\delta-2)m}\,2^{q+42\delta q}\ \ \mbox{if }k\ge -D,
\end{aligned}  \label{3.17}\tag{3.17}  
\end{equation}
and
\begin{equation}
\begin{aligned}
&\left\|P_k f_{NC}(s)\right\|_{L^2}\lesssim\begin{cases}
2^{-3m/2+60\delta m}&\mbox{, if }k\ge 0 \\
2^{-3m/2-k/2+60\delta m}&\mbox{, if }-m\le k\le 0 \\
2^{-m+60\delta m}&\mbox{, if }k\le -m \\
\end{cases},\\[4pt]
&\left\|\widehat{P_k f_{NC}}(s)\right\|_{L^\infty}\lesssim 2^{60\delta m}\left(1+2^{m+k_-}\right)^{-1},
\end{aligned} \label{3.18}\tag{3.18}
\end{equation}
where if $-(1-10\delta)m+D^2\le k_1,k_2\le \delta^2 m/5-D$ and $0\le j_1\le j_2\le (1-10\delta)m-D^2$, then denote $\kappa_r\triangleq 2^{\delta^2 m/2}\left(2^{-m/2}+2^{j_2-m}\right)$, $\kappa_\theta\triangleq 2^{2\delta^2 m-m/2}$ and the term $f_{NC}$ occurs in the following cases \par
(a.1). $\left|\nabla_\eta\Phi_{\sigma\mu\nu}\right|\lesssim \kappa_r$, \ \ \ $\mu_2+\mu_3\neq 0$, \ \ \ $2^{k}\ge 2^{\delta m}\left(2^{-m/2}+2^{j_2-m}\right)$, \ \ \ $\frac{m}{2}\le j_2\le (1-10\delta)m$\vspace{0.4em} \\
$k\le-D$;\par
(a.2). $\left|\nabla_\eta\Phi_{\sigma\mu\nu}\right|\lesssim \kappa_r$, \ \ \ $\mu_2+\mu_3\neq 0$, \ \ \ $2^{k}\ge 2^{\delta m}\left(2^{-m/2}+2^{j_2-m}\right)$, \ \ \ $\frac{m}{2}\le j_2\le (1-10\delta)m$\vspace{0.4em} \\
$k\ge-D$,\ \ \ $\left|\Omega_\theta\Phi_{\mu\mu_2\mu_3}\right|\gtrsim \kappa_\theta$;\par
(b). $\left|\nabla_\eta\Phi_{\sigma\mu\nu}\right|\lesssim \kappa_r$, \ \ \ $\mu_2+\mu_3\neq 0$, \ \ \ $2^{k}\le 2^{\delta m}\left(2^{-m/2}+2^{j_2-m}\right)$; \par
(c). $\left|\nabla_\eta\Phi_{\sigma\mu\nu}\right|\lesssim \kappa_r$, \ \ \ $\mu_2+\mu_3=0$, \ \ \ $2^{k}\le 2^{\delta m}\left(2^{-m/2}+2^{j_2-m}\right)$, \ \ \ $m+k_2\le j_2+3\delta^2 m$,\par
\noindent and in the case (a.2) we can get strong enough control $$\left\|\widehat{P_k f_{NC}}\right\|_{L^\infty}+\left\|P_k f_{NC}\right\|_{L^2}\lesssim 2^{-4m}.$$
Moreover,
\begin{align}
\sup_{o\le b\le N_1/4}\left\{\left\|\Omega^b g^q_{\sigma\mu\nu}(s)\right\|_{L^2}+\left\|\Omega^b f_{NC}(s)\right\|_{L^2}\right\}\lesssim 1. \label{3.19}\tag{3.19}  
\end{align}
In particular, for any $a\in[0,N_1/2]\cap\mathbb{Z}$ and $k\in\mathbb{Z}$,
\begin{align}
\left\|P_k\Omega^a (\partial_t V_\sigma)(s)\right\|_{L^2}\lesssim \varepsilon^2_1 2^{-m+3\delta m+\delta^2 m}. \label{3.20}\tag{3.20}   
\end{align}
\end{lemma}
\begin{proof}
This was basically proved in Lemma 6.2 in \cite{y2}. The only difference occurs when $\mu+\nu=0$, $m+k\le j_2+3\delta^2 m$, and $-m\le k_-\le -9\delta m$, see (6.36) in \cite{y2}. In our case, we don't have the null structure provided by the nonlinear terms, so here we get a slightly weak estimate compared to (6.16) in \cite{y2}.
\end{proof}
\vspace{0.8em}
\begin{remark}
The results above are not sharp. In some cases, it's still possible to improve the results. For example, if $k\ge -\frac{j}{3}+12\delta j$ and $m+D\le j\le m+\delta m$, then we can show that $\left\|P_k f_{NC}\right\|_{L^2}\lesssim 2^{-\frac{4}{3}m+10.8\delta m}$.\par
To prove this, we need to recheck the proof of Lemma 6.2 in \cite{y2}, and it turns out that we may assume that $j_2\ge m+k-3\delta^2 m\ge \frac{2}{3}m+12\delta m-\frac{10}{3}\delta^2 m$, since in other cases we have already got enough control. Denote $f_{j_1,k_1},f_{j_1,k_1,n_1},f_{j_2,k_2},f_{j_2,k_2,n_2}$ as in (\ref{3.1}). If $n_2<j_2-10\delta m$, then we can proceed as in \cite{y2}
\begin{align*}
    \left\|P_k f_{NC}\right\|_{L^2}&\lesssim \left\|e^{is\Lambda_\mu} f^\mu_{j_1,k_1}\right\|_{L^\infty}\cdot\left\|f^\nu_{j_2,k_2,n_2}\right\|_{L^2} \\
    &\lesssim 2^{-m+21\delta m}\cdot 2^{-(1-20\delta)j_2+(\frac{1}{2}-19\delta)n_2} \\
    &\lesssim 2^{-m+21\delta m}\cdot 2^{-(\frac{1}{2}-\delta)j_2-4.9\delta m} \\
    &\lesssim 2^{-m+21\delta m}\cdot 2^{-(\frac{1}{2}-\delta)(m+k_2-3\delta^2 m)-4.9\delta m} \\
    &\lesssim 2^{-m+21\delta m}\cdot 2^{-\frac{1}{2}(\frac{2}{3}m+12\delta m)+0.67\delta m-4.9\delta m}\lesssim 2^{-\frac{4}{3}m+10.8\delta m}.
\end{align*}
On the other hand, if $j_2-10\delta m\le n_2\le j_2$, then we may further assume that $-\frac{j}{3}+12\delta j<k_2<-\frac{j}{4}$. (Otherwise, we can use (\ref{3.18}) to get $\left\|P_k f_{NC}\right\|_{L^2}\lesssim 2^{-\frac{11}{8}m+61\delta m}$ directly, which is enough.) In this case, we observe that $k_1=0$. This is because denote that $\gamma_i$ is the one such that $\big|\left|\eta\right|-\gamma_i\big|\sim 2^{-n_2}$ and we have
\begin{align*}
   \big|\left|\xi-\eta\right|-\gamma_i\big|&\le\big|\left|\xi-\eta\right|-\left|\eta\right|\big|+\big|\left|\eta\right|-\gamma_i\big|\le\left|\xi\right|+\big|\left|\eta\right|-\gamma_i\big| \\
   &\le 2^{-\frac{1}{4}m}+2^{-n_2}\le 2^{-\frac{1}{4}m}+2^{-j_2+10\delta m}\le 2^{-\frac{1}{4}m}+2^{-\frac{2}{3}m+22\delta m}\lesssim 2^{-\frac{1}{4}m}.
\end{align*}
This implies that $\left\|e^{-is\Lambda_\mu}f^\mu_{j_1,k_1}\right\|_{L^\infty}\lesssim 2^{-m+3\delta m}$ by the last line of (\ref{3.6}). Thus, like before, we get 
\begin{align*}
   \left\|P_k f_{NC}\right\|_{L^2}&\lesssim \left\|e^{is\Lambda_\mu} f^\mu_{j_1,k_1}\right\|_{L^\infty}\cdot\left\|f^\nu_{j_2,k_2,n_2}\right\|_{L^2}  \\
   &\lesssim 2^{-m+3\delta m}\cdot 2^{-\frac{1}{2}(\frac{2}{3}m+12\delta m)+0.67\delta m}\lesssim 2^{-\frac{4}{3}m-2\delta m}.
\end{align*}
To sum up, we've shown that $\left\|P_k f_{NC}\right\|_{L^2}\lesssim 2^{-\frac{4}{3}m+10.8\delta m}$. This result will be used in Section 5 later on.
\end{remark}
\vspace{0.8em}
\begin{cor}
Assume $s\approx 2^m$ and $a\in[0,N_1/2]$, we can decompose $\Omega^a(\partial_t V_\sigma)(s)=G_2(s)+G_\infty(s)$ such that, for any $k\in\mathbb{Z}$,
\begin{align*}
   \sup\limits_{\left|\lambda\right|\le 2^{(1-\delta /10)m}} \left\|e^{-i(s+\lambda)\Lambda_\sigma} P_k G_\infty(s)\right\|_{L^\infty}\lesssim\varepsilon_1^2\,2^{-2m+50\delta m},\\
   \left\|P_k G_2(s)\right\|_{L^2}\lesssim\begin{cases}
\varepsilon_1^2\,2^{-3m/2+60\delta m}&\mbox{, if }k\ge 0 \\
\varepsilon_1^2\,2^{-3m/2-k/2+60\delta m}&\mbox{, if }-m/2\le k\le 0 \\
\varepsilon_1^2\,2^{-5m/4+60\delta m}&\mbox{, if }k\le -m/2 \\
\end{cases}.
\end{align*}
\end{cor}
\begin{proof}
This was proved in Corollary 6.3 in \cite{y2}.
\end{proof}
\vspace{0.8em}
Plus, we need a more precise time derivative control which will be used in the dispersive contral later.
\vspace{0.8em}
\begin{lemma}[time derivatives 2]
For any $(k,j)\in\mathcal{J}$ and $n\in\left\{0,\dots j+1\right\}$, denote
$$f_{j,k}\triangleq P_{[k-2,k+2]}\,Q_{jk}f.$$
Under the assumption of Proposition 3.1, let $s\approx 2^m$, $a\in\left[0,N_1/4\right]\cap\mathbb{Z}$, we have the following decomposition of time derivative
\begin{align}
\Omega^a (\partial_t V_\sigma)(s)=\varepsilon_1^2\left(f_C(s)+f_{SR}(s)+f_{NC}(s)+f_{NCw}(s)+\partial_s F_c(s)+\partial_s F_{NC}(s)+\partial_s F_{LO}(s)\right),  \label{3.21}\tag{3.21}  
\end{align}
where we have coherent inputs
\begin{equation}
\begin{aligned}
   &\widehat{f_C}(\xi,s)=\sum_{\mu+\nu\neq0} e^{is\Psi_{\sigma\mu\nu}(\xi)} \sum_{0\le q\le (1/2-40\delta) m} g^q_{\sigma\mu\nu}(\xi,s)\, \varphi_{-3\delta m}(\Psi_{\sigma\mu\nu}(\xi)),\\
   &\left\|D^\alpha_\xi g^q_{\sigma\mu\nu}(\xi,s)\right\|_{L^\infty} \lesssim 2^{-m-q+4.01\delta m}\,2^{(m/2+q+3\delta^2 m)\left|\alpha\right|},\\[4pt]
   &\left\|\partial_s g^q_{\sigma\mu\nu}(s)\right\|_{L^\infty}\lesssim 2^{(6.01\delta-2)m+q},
\end{aligned} \label{3.22}\tag{3.22}
\end{equation}
secondary resonances
\begin{align}
\left\|D^{\alpha}_{\xi}\reallywidehat{f_{SR}}(s)\right\|_{L^\infty}\lesssim 2^{-3m/2+76\delta m}\, 2^{(1-300\delta )m\left|\alpha\right|}, \label{3.23}\tag{3.23}   
\end{align}
a stronger nonresonant contributions
\begin{align}
\left\|f_{NC}(s)\right\|_{L^2}\lesssim 2^{-19m/10},  \label{3.24}\tag{3.24}
\end{align}
and a weaker nonresonant contributions
\begin{align*}
\left\|f_{NCw}(s)\right\|_{L^2}\lesssim 2^{-1.6m+11.4\delta m},\label{3.25}\tag{3.25}
\end{align*}
where if we write
\begin{align*}
    \widehat{f_{NCw}}(s)=\int_{\mathbb{R}^2} e^{is\Phi(\xi,\eta)} \varphi_{\le -3\delta m-4\delta^2 m}\left(\Phi(\xi,\eta)\right) \widehat{f}_{j_1,k_1}(\xi-\eta) \widehat{f}_{j_2,k_2}(\eta)\,d\eta, \tag{3.26}\label{3.26}
\end{align*}
then we have $j_1\le (1-\delta)m$, $j_2\ge (1-\delta)m$ and $\left|\nabla_\xi\Phi(\xi,\eta)\right|\gtrsim 1$. \\[4pt]
In addition, we have
\begin{equation}
\begin{aligned}
   \left\|\widehat{F_C}(s)\right\|_{L^\infty}\lesssim 2^{-m+3.2\delta m+10\delta^2 m},\,\,\,\,\,\left\|F_{NC}(s)\right\|_{L^2}\lesssim 2^{-41m/40}, \\
   \left\|(1+2^m\left|\xi\right|)\widehat{F_{LO}}(\xi,s)\right\|_{L^{\infty}_{\xi}}\lesssim 2^{5\delta m},\,\,\,\,\,P_{\ge {-13m/15}}\,F_{LO}\equiv 0,
\end{aligned} \label{3.27}\tag{3.27}
\end{equation}
and
\begin{align}
   \sup_{o\le b\le N_1/4}\left\{\left\|\Omega^b g^q_{\sigma\mu\nu}\right\|_{L^2}+\left\|\Omega^b f_{SR}\right\|_{L^2}+\left\|\Omega^b f_{NC}\right\|_{L^2}\right\}\lesssim 1. \label{3.28}\tag{3.28}
\end{align}
\end{lemma}
\begin{proof}
This was basically proved in Lemma 6.4 in \cite{y2}. First of all, due to the lack of Proposition 8.5 (i) and (iii) in \cite{y2}, we could lose $2^{2\delta n_1}$ or $2^{2\delta n_2}$ in the estimate of $\left|\widehat{f_{j_1,k_1}}\right|$ or $\left|\widehat{f_{j_2,k_2}}\right|$ which leads to two slightly weaker estimates (\ref{3.23}) and (\ref{3.27}) than those in \cite{y2}. Next, due to the lack of Proposition 8.5 in \cite{y2}, we have to modify the proof in \cite{y2} in a couple of places, and it turns out that we can still get the mostly same estimates as in Lemma 6.4 in \cite{y2}. For the simplicity of notations, we may use $f$ to replace $f_{j_1,k_1}$, and use $g$ to replace $g_{j_2,k_2}$ below. \par
To start with, we need to consider
$$\partial_t \widehat{V_\sigma}(\xi,s)=\sum_{\mu,\nu\in\mathcal{P}} \int_{\mathbb{R}^2}e^{is\Phi_{\sigma\mu\nu}(\xi,\eta)}\widehat{V_\mu}(\xi-\eta,s)\widehat{V_\nu}(\eta,s)\,d\eta.$$
Thus, it suffices to consider
$$I^{\sigma\mu\nu}\left[f^\mu,g^\nu\right]=\mathcal{F}^{-1}\int_{\mathbb{R}^2} e^{is\Phi_{\sigma\mu\nu}(\xi,\eta)}\widehat{f^\mu}(\xi-\eta,s)\widehat{g^\nu}(\eta,s)\,d\eta,$$
where
$$\left\|f^\mu\right\|_{H^{N_0/2}\cap Z_1^\mu\cap H^{N_1/2}_\Omega}+\left\|g^\nu\right\|_{H^{N_0/2}\cap Z_1^\nu\cap H^{N_1/2}_\Omega}\le 1.$$
Write $I^{\sigma\nu\mu}\left[f,g\right]=\sum\limits_{k,k_1,k_2\in\mathbb{Z}} {P_k I^{\sigma\nu\mu}\left[P_{k_1}f,P_{k_2}g\right]}$, then up to acceptable $f_{NC}$ errors we may assume $m\ge D^2$ and restrict the sum to the range
$$-3m\le k,k_1,k_2 \le \delta^2 m/10-D^2.$$
Define $l_0\triangleq\left[-3\delta m-4\delta^2 m\right]$, $\varphi_{hi}(x)\triangleq\varphi_{>lo}(x)$, $\varphi_{lo}(x)\triangleq\varphi_{\le lo}(x)$, and we then decompose
$$I^{\sigma\mu\nu}\left[f,g\right]=I^{hi}\left[f,g\right]+I^{lo}\left[f,g\right],$$
where
$$I^*\left[f,g\right]\triangleq\mathcal{F}^{-1}\int_{\mathbb{R}^2} e^{is\Phi_{\sigma\mu\nu}(\xi,\eta)}\varphi_*\left(\Phi(\xi,\eta)\right)\widehat{f}(\xi-\eta,s)\widehat{g}(\eta,s)\,d\eta.$$\par
First, as for \textbf{contribution of} $\boldsymbol{I^{hi}}$, we could rewrite
$$I^{hi}\left[P_{k_1}f,P_{k_2}g\right]=\partial_s\left\{J\left[P_{k_1}f,P_{k_2}g\right]\right\}-J\left[P_{k_1}\partial_s f,P_{k_2}g\right]-J\left[P_{k_1}f, P_{k_2}\partial_s g \right],$$
where
$$J\left[f,g\right]\triangleq\mathcal{F}^{-1}\int_{\mathbb{R}^2} e^{is\Phi_{\sigma\mu\nu}(\xi,\eta)}\frac{\varphi_{hi}\left(\Phi(\xi,\eta)\right)}{i\Phi(\xi,\eta)}\widehat{f}(\xi-\eta,s)\widehat{g}(\eta,s)\,d\eta.$$
Using (\ref{3.7}), (\ref{3.20}) and Lemma 3.5, we easily see that
$$\left\|J\left[P_{k_1}\partial_s f,P_{k_2}g\right]\right\|_{L^2}+\left\|J\left[P_{k_1} f,P_{k_2} \partial_s g\right]\right\|_{L^2}\lesssim 2^{(50\delta-2)m},$$
which gives an acceptable $f_{NC}$-type contribution as in (\ref{3.24}).\par
Next, we define $f_{j_1,k_1}\triangleq P_{\left[k_1-2,k_1+2\right]} Q_{j_1 k_1}f$, $g_{j_2,k_2}\triangleq P_{\left[k_2-2,k_2+2\right]} Q_{j_2 k_2}g$. If $j_2=\max (j_1,j_2)\ge m/18$, then using Lemma 3.5, we have
$$\left\|J\left[f_{j_1,k_1},g_{j_2,k_2}\right]\right\|_{L^2}\lesssim 2^{-(1/2-\delta)j_2} 2^{-m+50\delta m}.$$
This gives an acceptable $F_{NC}$-type contribution as in (\ref{3.27}). On the other hand, if $j_2=\max (j_1,j_2)\le m/18$ and $\mu+\nu\neq 0$, then integration by parts, using Lemma 3.1(i), shows that
$$\left|\mathcal{F}\left\{J\left[f_{j_1,k_1},g_{j_2,k_2}\right]-J^S\right\}(\xi)\right|\lesssim 2^{-2m},$$
where 
$$\widehat{J^S}(\xi)\triangleq\int_{\mathbb{R}^2} e^{is\Phi(\xi,\eta)}\frac{\varphi_{hi}\left(\Phi(\xi,\eta)\right)}{i\Phi(\xi,\eta)}\varphi\left(\kappa_r^{-1}\nabla_\eta \Phi(\xi,\eta)\right)\widehat{f_{j_1,k_1}}(\xi-\eta)\widehat{g_{j_2,k_2}}(\eta)\,d\eta$$
and $\kappa_r\triangleq 2^{\delta^2 m-m/2}$.
In view of (8.8) in \cite{y2}, the $\eta$-integral in the definition of $\widehat{J^S}$ takes place over a ball of radius $\lesssim 2^{\delta^2 m} \kappa_r$ centered at $p(\xi)$. Note that in view of (\ref{3.7a}) and the assumption that $j_2=\max\left\{j_1,j_2\right\}\le m/18$, we have $\left|\widehat{f_{j_1,k_1}}(\xi-\eta)\right|+\left|\widehat{g_{j_2,k_2}}(\eta)\right|\lesssim 2^{1.01\delta m/18}\lesssim 2^{0.1\delta m}$. Since $\left|\Phi(\xi,\eta)\right|\gtrsim 2^{-3\delta m-4\delta^2 m}$, we can conclude that $\left|\widehat{J^S}(\xi)\right|\lesssim 2^{-m+3.2\delta m+10\delta^2 m}$, which is an acceptable $F_C$-type contribution.\par
Finally, assume that 
$$j_2=\max\left\{j_1,j_2\right\}\le m/18,\,\,\,\,\mu+\nu=0.$$
Notice that
\begin{align*}
   &\left|\nabla_\eta\Phi(\xi,\eta)\right|=\left|\nabla\Lambda_\mu(\eta-\xi)-\nabla\Lambda_\mu(\eta)\right|\gtrsim \left(1+2^{3k_1}+2^{3k_2}\right)^{-1}\left|\xi\right|, \\
   &\left|D_\eta^\alpha\Phi(\xi,\eta)\right|\lesssim_\alpha\,\left|\xi\right|.
\end{align*}
Lemma 3.1 (i) shows that $\left\|\mathcal{F}P_k J\left[f_{j_1,k_1},g_{j_2,k_2}\right]\right\|_{L^\infty}\lesssim 2^{-2m}$, if $2^m\left|\xi\right|\ge 2^{j_2} 2^{2\delta^2 m}$, namely $m+k\ge j_2+\delta^2 m$. On the other hand, if $m+k\le j_2+\delta^2 m$, then we get an $F_{LO}$-type contribution. Indeed we may assume $2^{k_1}\approx2^{k_2}$ and estimate
\begin{align*}
\left|\varphi_k(\xi)\mathcal{F}J\left[f_{j_1,k_1},g_{j_2,k_2}\right](\xi)\right|&\lesssim\,2^{3.1\delta m}\left\|\widehat{f_{j_1,k_1}}\right\|_{L^\infty}\left\|\widehat{g_{j_2,k_2}}\right\|_{L^1}\lesssim 2^{3.1\delta m} 2^{2\delta j_1} 2^{-j_2+21\delta j_2} \\
    &\lesssim\,2^{4.8\delta m} 2^{-j_2},
\end{align*}
using (\ref{3.3}) and (\ref{3.4}). This suffices since $2^{-j_2}\lesssim 2^{\delta^2 m}\left(1+2^{m+k}\right)^{-1}$.\par
Second, as for \textbf{contribution of } $\boldsymbol{I^{lo}}$ \textbf{, } $\boldsymbol{\underline{k}}$ \textbf{ small}, we consider now $P_k I^{lo}\left[P_{k_1}f,P_{k_2}g\right]$ in the case
\begin{align*}
    \underline{k}=\min\left\{k,k_1,k_2\right\}\le-D. \tag{3.29}\label{3.29}
\end{align*}
We may assume that
\begin{align*}
    -D\le \max\left\{k,k_1,k_2\right\}\le D.\tag{3.30}\label{3.30}
\end{align*}
In fact, $\max\left\{k,k_1,k_2\right\}\ge -D$ is due to the assumption that $\left|\Phi\right|$ is very small and the asscumption of $b_\sigma-b_\mu-b_\nu\neq 0$. On the other hand, we also conclude $\max\left\{k,k_1,k_2\right\}\le D$ thanks to Proposition 6.6 (b). However, in general we do not necessarily have $f_{j_1,k_1}=f_{j_1,k_1,0},\,\,\,g_{j_2,k_2}=g_{j_2,k_2,0}$ as in \cite{y2}. Namely, it's possible that $n_1>0$ or $n_2>0$. If $\max\left\{j_1, j_2\right\} \le (1-\delta)m$ then we can integrate by parts using Lemma 3.1 to obtain
again an acceptable $f_{NC}$-type contribution. If $\min\left\{j_1,j_2\right\}\ge(1-\delta) m$, then we can prove that $\left\|f_{NC}\right\|_{L^2}\lesssim 2^{-2m+41\delta m}$. In fact, if at most one of $n_1, n_2>0$, say $n_1>0, n_2=0$, then we can use
\begin{align*}
    \left\|P_k I^{lo}\left[f_{j_1,k_1},g_{j_2,k_2}\right]\right\|_{L^2}\lesssim \left\|e^{it\Lambda_{mu}}f\right\|_{L^\infty}\cdot\left\|g\right\|_{L^2} \tag{3.31}\label{3.31}
\end{align*}
to estimate. Otherwise, we may WLOG assume that $n_1>n_2>0$. In this case, since $\underline{k}\triangleq\min\left\{k,k_1,k_2\right\}<-D$, we know that $\xi-\eta$ and $\eta$ must near a same $\gamma_i$, which implies that $\xi$ and $\eta^{\bot}$ are nearly parallel and $\left|\xi\right|\sim 2^{-n_2}$. In this case, consider
$$I^{lo}(\xi)\triangleq\int e^{is\Phi(\xi,\eta)}\varphi_{lo}(\Phi(\xi,\eta))\widehat{f}(\xi-\eta)\widehat{g}(\eta)\,d\eta.$$
Denote $f_r\triangleq\sup\limits_{\theta} {\left|\widehat{f}(r\theta)\right|}$, $g_r\triangleq\sup\limits_{\theta} {\left|\widehat{g}(r\theta)\right|}$.
Then, we have
$$\left|I^{lo}(\xi)\right|\le\int f_r(\xi-\eta) g_r(\eta)\,d\eta.$$
Let's do the following change of variables
$$\begin{cases}
   x=\left|\eta\right|^2\\
   y=\left|\xi-\eta\right|^2
\end{cases}.$$
Then, we note that 
\begin{align*}
    \left|\frac{dx\,dy}{d\eta_1\,d\eta_2}\right|=&\left|\det\left[\begin{matrix}
    \frac{\partial x}{\partial \eta_1} & \frac{\partial x}{\partial \eta_2} \vspace{0.5em} \\ 
    \frac{\partial y}{\partial \eta_1} &
    \frac{\partial y}{\partial \eta_2}
    \end{matrix}\right]\right|    =4\left|\xi_1 \eta_2-\xi_2 \eta_1\right|\sim \left|\xi\cdot\eta^{\bot}\right|\sim \left|\xi\right|\cdot\left|\eta\right|\sim 2^{-n_2}.
\end{align*}
Note that $f_r$ and $g_r$ are two radial functions, and we have
\begin{align*}
   \left|I^{lo}(\xi)\right|\le\int f_r(\xi-\eta) g_r(\eta)\,d\eta \lesssim\int f_r\left(\sqrt{y}\right) g_r\left(\sqrt{x}\right) \cdot 2^{n_2}\,dxdy.
\end{align*}
 Let $X=\sqrt{x}$, and $Y=\sqrt{y}$, then using (\ref{3.2}) we have
\begin{align*}
    \left|I^{lo}(\xi)\right|\lesssim \, 2^{n_2} \int f_r\left(Y\right) g_r\left(X\right) dXdY    \lesssim  \,2^{n_2} \left\|f_r\right\|_{L^1} \left\|g_r\right\|_{L^1}\lesssim \,2^{-2m+41\delta m+n_2}. 
\end{align*}
This leads to that
\begin{align*}
    \left\|I^{lo}\right\|_{L^2}\lesssim &\left\|I^{lo}\right\|_{L^\infty}\cdot\left|E_\xi\right|^{1/2} \\
    \lesssim&\,2^{-2m+41\delta m+n_2}\cdot \left(2^{-2n_2}\right)^{1/2}=2^{-2m+41\delta m}.
\end{align*}
Finally, if $j_1\le (1-\delta)m$,  $j_2\ge (1-\delta)m$ and $n_2=0$, then like (\ref{3.31}) we can also prove that $\left\|P_k I^{lo}\right\|_{L^2}\lesssim 2^{-2m+41\delta m}$. On the other hand, if $j_1\le (1-\delta)m$,  $j_2\ge (1-\delta)m$ and $n_2>0$, then we can only get a $f_{NCw}$-type contribution, which is weaker than the previous one. In fact, if $n_2\le 0.8m-29\delta m$, then we use Lemma 3.5 to deduce
\begin{align*}
    \left\|P_k I^{lo}\left[f_{j_1,k_1},g_{j_2,k_2,n_2}\right]\right\|_{L^2}\lesssim &\left\|e^{it\Lambda_{\mu}}f\right\|_{L^\infty}\cdot\left\|g\right\|_{L^2}\lesssim 2^{-1.6m+11.4\delta m}.
\end{align*}
Otherwise, we may assume $n_2>0.8m-29\delta m$. Now, if $j_1\le \frac{2}{5}m-\delta^2 m$, then we denote $\kappa_\theta=-m/2+\delta^2 m$, and decompose
$$P_k I^{lo}\left[f_{j_1,k_1},g_{j_2,k_2,n_2}\right]=I^\parallel+I^\bot,$$
where
\begin{align*}
    I^\parallel=&\int_{\mathbb{R}^2} e^{is\Phi(\xi,\eta)} \varphi_{ lo}\left(\Phi(\xi,\eta)\right) \varphi\left(\kappa_\theta^{-1}\Omega_\eta \Phi(\xi,\eta)\right) \widehat{f_{j_1,k_1}}(\xi-\eta) \reallywidehat{g_{j_2,k_2,n_2}}(\eta)\,d\eta, \\
    I^\bot=&\int_{\mathbb{R}^2} e^{is\Phi(\xi,\eta)} \varphi_{ lo}\left(\Phi(\xi,\eta)\right) \left(1-\varphi\left(\kappa_\theta^{-1}\Omega_\eta \Phi(\xi,\eta)\right)\right) \widehat{f_{j_1,k_1}}(\xi-\eta) \reallywidehat{g_{j_2,k_2,n_2}}(\eta)\,d\eta.
\end{align*}
Integration by parts using Lemma 3.2 show that $\left\|I^\bot\right\|_{L^2}\lesssim 2^{-10m}$. On the other hand, fix $\xi$ and the volume of $\eta=$ $\left|E_\eta\right|\lesssim 2^{-2n_2}$;  fix $\eta$ and the volume of $\xi=$ $\left|E_\xi\right|\lesssim 2^{k_1}\cdot\kappa_\theta$. Use Schur's test and we get that
\begin{align*}
\left\|I^\parallel\right\|_{L^2}\lesssim&\left(E_\eta\right)^{1/2} \left(E_\xi\right)^{1/2} \left\|f\right\|_{L^\infty} \left\|g\right\|_{L^2} \\
\lesssim&\, 2^{-n_2+\frac{k_1}{2}-\frac{m}{4}+1/2\delta^2 m}\cdot 2^{-\left(\frac{1}{2}-21\delta\right)j_1-\frac{1}{2}k_1}\cdot 2^{-(1-20\delta)(1-\delta)m+\left(\frac{1}{2}-19\delta\right)n_2} \\
\lesssim&\, 2^{-1.65m+20.5\delta m},
\end{align*}
which gives an acceptable $f_{NCw}$-type contribution as in (\ref{3.25}). If $j_1\ge \frac{2}{5}m-\delta^2 m$, then fix $\xi$, $\left|E_\eta\right|\lesssim 2^{-2n_2}$; fix $\eta$, $\left|E_\xi\right|\lesssim 2^{k_1}$. Use Schur's test again and we get that
\begin{align*}
\left\|I^\parallel\right\|_{L^2}\lesssim&\left(E_\eta\right)^{1/2} \left(E_\xi\right)^{1/2} \left\|f\right\|_{L^\infty} \left\|g\right\|_{L^2} \\
\lesssim&\, 2^{-n_2+\frac{k_1}{2}}\cdot 2^{-\left(\frac{1}{2}-21\delta\right)(\frac{2}{5}m-\delta^2 m)-\frac{1}{2}k_1}\cdot 2^{-(1-20\delta)(1-\delta)m+\left(\frac{1}{2}-19\delta\right)n_2} \\
\lesssim&\, 2^{-1.6m+9\delta m},
\end{align*}
which satisfies $f_{NCw}$-type estimation in (\ref{3.25}) as well. In the end, we make a remark here. Since $n_2>0$, we must have $\left|\eta\right|\sim 1$. Then, according the assumption (\ref{3.29}) and (\ref{3.30}), we must have either $\left|\xi\right|\ll 1$, $\left|\xi-\eta\right|\sim 1$ or $\left|\xi\right|\sim 1$, $\left|\xi-\eta\right|\ll 1$. In both cases, in view of the argument at (\ref{5.47}), we can conclude that $\left|\nabla_\xi\Phi(\xi,\eta)\right|\gtrsim 1$.
\par
Third, as for \textbf{contribution of } $\boldsymbol{I^{lo}}$ \textbf{, } $\boldsymbol{\underline{k}}$ \textbf{ not small}, we may assume that
\begin{align*}
    \min\left\{k,k_1,k_2\right\}\ge -D,\,\,\,j_1\le j_2. \tag{3.32}\label{3.32}
\end{align*}
We first assume that $\min \left\{k,k_1,k_2\right\}\ge -D$ and $\min\left\{j_1,j_2\right\}\ge (1-300\delta) m$, and we will have four cases here. When $n_1=n_2=0$, this is trivial. When $n_1>0$ and $n_2=0$, we only need to use Proposition 6.10 above instead, which is a slightly different version of Lemma 8.10 in \cite{y2}. When $n_1>0$ and $n_2>0$, then we need to estimate this integral with a different method. Consider
$$I^{lo}(\xi)\triangleq\int e^{is\Phi(\xi,\eta)}\varphi_{lo}(\Phi(\xi,\eta))\widehat{f}(\xi-\eta)\widehat{g}(\eta)\,d\eta,$$
and denote $f_r\triangleq\sup\limits_{\theta} {\left|\widehat{f}(r\theta)\right|}$, $g_r\triangleq\sup\limits_{\theta} {\left|\widehat{g}(r\theta)\right|}$.
Then, we have
$$\left|I^{lo}(\xi)\right|\le\int f_r(\xi-\eta) g_r(\eta)\,d\eta.$$
We may also assume that $0\le\angle\xi,\eta\le\pi/2$ in the above integral. Otherwise, one just need to exchange the role of $\xi-\eta$ and $\eta$.
We will do the following change of variables as before
$$\begin{cases}
   x=\left|\eta\right|^2\\
   y=\left|\xi-\eta\right|^2
\end{cases}.$$
For simplicity, we first assume that $-D\le k,k_1,k_2\le D$. We note that 
\begin{align*}
    \frac{dx\,dy}{d\eta_1\,d\eta_2}
    \sim \xi\cdot\eta^{\bot}    \sim \sin\angle \xi,\eta \sim \angle \xi,\eta\sim \frac{1}{2} \angle \xi,\eta \sim \sin\frac{1}{2} \angle \xi,\eta.
\end{align*}
On the other hand, we have 
$$\left(\sqrt{y}\right)^2-\left(\left|\xi\right|-\sqrt{x}\right)^2=2\left|\xi\right|\sqrt{x}\left(1-\cos\angle\xi,\eta\right)\sim \frac{1}{2} \left(1-\cos\angle\xi,\eta\right)\sim\left(\sin\frac{1}{2} \angle \xi,\eta\right)^2,$$
which implies that
$$\frac{d\eta_1\,d\eta_2}{dx\,dy}\sim\frac{1}{\left(\left(\sqrt{y}\right)^2-\left(\left|\xi\right|-\sqrt{x}\right)^2\right)^{1/2}}.$$
Thus, note that in general the assumption is that $-D_0\le k,k_1,k_2\le \delta^2 m/10-D^2$. Then we can similarly get that
$$\frac{2^{-\delta^2 m/5}}{\left(\left(\sqrt{y}\right)^2-\left(\left|\xi\right|-\sqrt{x}\right)^2\right)^{1/2}}\lesssim\frac{d\eta_1\,d\eta_2}{dx\,dy}\lesssim\frac{2^{\delta^2 m/10}}{\left(\left(\sqrt{y}\right)^2-\left(\left|\xi\right|-\sqrt{x}\right)^2\right)^{1/2}}.$$
Note that $f_r$ and $g_r$ are two radial functions, and we have
\begin{align*}
   \left|I^{lo}(\xi)\right|\le&\int f_r(\xi-\eta) g_r(\eta)\,d\eta \\
   \lesssim&\,2^{\delta^2 m/10}\left[\int_{\sqrt{x}\le\left|\xi\right|} f_r\left(\sqrt{y}\right) g_r\left(\sqrt{x}\right) \frac{1}{\left(\sqrt{y}-\left|\xi\right|+\sqrt{x}\right)^{1/2}}\,dxdy\right. \\
   +&\left.\int_{\sqrt{x}\ge\left|\xi\right|} f_r\left(\sqrt{y}\right) g_r\left(\sqrt{x}\right) \frac{1}{\left(\sqrt{y}+\left|\xi\right|-\sqrt{x}\right)^{1/2}}\,dxdy\right] \\
   \triangleq&\,2^{\delta^2 m/10}\left[I_1^{lo}(\xi)+I_2^{lo}(\xi)\right].
\end{align*}
We only need to focus on $I_1^{lo}$, since $I_2^{lo}$ can be done similarly. Let $X=\sqrt{x}$, and $Y=\sqrt{y}$, then by Holder, we have
\begin{align*}
    \left|I_1^{lo}(\xi)\right|\lesssim &\, 2^{\delta^2 m/5} \int f_r\left(Y\right) g_r\left(X\right) \frac{1}{\left(X+Y-\left|\xi\right|\right)^{1/2}}\,dXdY \\[2pt]
    \lesssim & \,2^{\delta^2 m/5} \left\|f_r\right\|_{L^{\frac{248}{123}}} \left\|g_r\right\|_{L^{\frac{248}{123}}} \left(\int \frac{dXdY}{\left(X+Y-\left|\xi\right|\right)^{0.992}}\right)^{\frac{125}{248}}. 
\end{align*}
Note that we have 
$$\int \frac{dXdY}{\left(X+Y-\left|\xi\right|\right)^{0.992}}\lesssim\min\left(2^{-n_1-0.008n_2},2^{-0.008n_1-n_2}\right),$$
since
\begin{align*}
    \int \frac{dXdY}{\left(X+Y-\left|\xi\right|\right)^{0.992}}=
    \int\left(\int\frac{dY}{\left(Y+X-\left|\xi\right|\right)^{0.992}}\right)\,dX\lesssim\,2^{-n_1-0.008n_2}.
\end{align*}
Now, assume that $\max(k,k_1,k_2)\le D$ (the other slightly different case $D\le\max(k,k_1,k_2)\le\delta^2 m/10-D^2$ can be done similarly), by interpolation inequality and (3.35), the first line of (3.37) in \cite{y2}, we have
\begin{align*}
    \left\|f_r\right\|_{L^{\frac{248}{123}}}\lesssim &\left\|f_r\right\|_{L^2}^{123/124}\cdot\left\|f_r\right\|_{L^{\infty}}^{1/124} \\
    \lesssim & \left[2^{-(1-20.5\delta)j_1+\left(\frac{1}{2}-19\delta\right)n_1}\right]^{123/124}\cdot\left[2^{2\delta n_1-\left(\frac{1}{2}-21\delta\right)(j_1-n_1)}\right]^{1/124} \\
    \lesssim & \,2^{-0.995j_1+\left(\frac{1}{2}-19\delta\right)n_1};
\end{align*}
similarly, we have
$$\left\|g_r\right\|_{L^{\frac{248}{123}}}\lesssim\,2^{-0.995j_2+\left(\frac{1}{2}-19\delta\right)n_2}.$$
WLOG, we may assume that $0<n_1\le n_2$, and we can get
\begin{align*}
   \left|I_1^{lo}(\xi)\right|\lesssim \, 2^{-1.98m+\left(\frac{1}{2}-19\delta\right)n_1+\left(\frac{1}{2}-19\delta\right)n_2}\cdot 2^{-0.5n_2-0.004n_1}\lesssim \,2^{-1.98m+\frac{1}{2}n_1}.
\end{align*}
Denote $\left|E_\xi\right|$ by the volume of $\xi$. Since $n_1,n_2>0$ and $\xi$ is located on an annulus, we have
$$\left|E_\xi\right|\lesssim \max\left(2^{-n_1},2^{-n_2}\right)=2^{-n_1}.$$
Thus, we have
\begin{align*}
   \left\|I_1^{lo}\right\|_{L^2}\lesssim\left\|I_1^{lo}\right\|_{L^{\infty}}\cdot\left(\left|E_\xi\right|\right)^{1/2}\lesssim \, 2^{-1.98m}
\end{align*}
So, we conclude  $\left\|I^{lo}\right\|_{L^2}\lesssim 2^{-1.95m}$ as desired. Finally, when $n_1=0$ and $n_2>0$, we can use the above argument to estimate it similarly. In fact, in this case we will have
\begin{align*}
    \int \frac{dXdY}{\left(X+Y-\left|\xi\right|\right)^{0.992}}=
    \int\left(\int\frac{dX}{\left(Y+X-\left|\xi\right|\right)^{0.992}}\right)\,dY\lesssim\,2^{-n_2+0.0008\delta^2 m},
\end{align*}
and
\begin{align*}
   \left|I_1^{lo}(\xi)\right|\lesssim\, 2^{\delta^2 m/5} \left\|f_r\right\|_{L^{\frac{248}{123}}} \left\|g_r\right\|_{L^{\frac{248}{123}}} \left(\int \frac{dXdY}{\left(X+Y-\left|\xi\right|\right)^{0.992}}\right)^{\frac{125}{248}}\lesssim \,2^{-1.98m}.
\end{align*}
So, we get
\begin{align*}
   \left\|I_1^{lo}\right\|_{L^2}\lesssim&\left\|I_1^{lo}\right\|_{L^{\infty}}\cdot\left(\left|E_\xi\right|\right)^{1/2} \\
   \lesssim &\, 2^{-1.98m}\cdot 2^{\delta^2 m/10}\lesssim 2^{-1.95m},
\end{align*}
and thus $\left\|I^{lo}\right\|_{L^2}\lesssim 2^{-1.95m}$ as desired. This gives an acceptable $f_{NC}$-type contribution as in (\ref{3.24}). To control the remaining contributions of $P_k I^{lo}$ we consider two cases.\par
\textbf{Case 1.} Assume first that 
\begin{align*}
    \min\left\{k,k_1,k_2\right\}\ge -D,\,\,j_1\le (1-300\delta)m,\,\,j_2\ge (1-50\delta)m, 
\end{align*}
and decompose, with $\kappa_\theta\triangleq 2^{-m/2+2\delta^2 m}$,
\begin{align*}
    P_k I^{lo}\left[f_{j_1,k_1},g_{j_2,k_2}\right]&=I^\bot+I^\parallel \left[f_{j_1,k_1},g_{j_2,k_2}\right], \\
    \widehat{I^\bot}(\xi)&\triangleq\int_{\mathbb{R}^2} e^{is\Phi(\xi,\eta)} \varphi_{ lo}\left(\Phi(\xi,\eta)\right) \\
    &\times\left(1-\varphi\left(\kappa_\theta^{-1}\Omega_\eta \Phi(\xi,\eta)\right)\right) \widehat{f_{j_1,k_1}}(\xi-\eta) \reallywidehat{g_{j_2,k_2}}(\eta)\,d\eta, \\
    \mathcal{F}\left\{I^\parallel\left[f,g\right]\right\}(\xi)&\triangleq\int_{\mathbb{R}^2} e^{is\Phi(\xi,\eta)} \varphi_{\le lo}\left(\Phi(\xi,\eta)\right) \varphi\left(\kappa_\theta^{-1}\Omega_\eta \Phi(\xi,\eta)\right) \hat{f}(\xi-\eta) \hat{g}(\eta)\,d\eta.
\end{align*}
Integration by parts using Lemma 3.2 shows that $I^\bot$ yields an acceptable $f_{NC}$-type contribution as in (\ref{3.24}). Moreover, notice that, using Lemma 3.5, Lemma 3.6, we have
$$\left\|P_k I^\parallel\left[f_{j_1,k_1},g_{j_2,k_2,n_2}\right]\right\|_{L^2}\lesssim 2^{-m+21\delta m} 2^{-j_2+20\delta j_2} 2^{n_2/2}+2^{-4m}\lesssim 2^{-2m+100\delta m} 2^{n_2/2}.$$
This gives an acceptable $f_{NC}$-type contributions if $n_2\le m/20$. Thus, in the following, we may assume that
$$g_{j_2,k_2}=\sum_{n_2\ge m/20} g_{j_2,k_2,n_2}.$$
(Notice that we do not necessarily have that at most one of $n_1$ and $n_2$ is positive on the bottom of P.824 of \cite{y2} due to the lack of Proposition 8.5 (ii) in \cite{y2}. It's possible that both $n_1$ and $n_2$ are positive. Thus, we are not able to assume $j_1\le m/2$ anymore as on the top of P.825 of \cite{y2}.)
Next, we further decompose
\begin{align*}
    I^\parallel\left[f_{j_1,k_1},g_{j_2,k_2}\right]&=I^C\left[f_{j_1,k_1},g_{j_2,k_2}\right]+I^{NC}\left[f_{j_1,k_1},g_{j_2,k_2}\right],  \\
    \reallywidehat{I^C\left[f,g\right]}(\xi)&\triangleq\int_{\mathbb{R}^2} e^{is\Phi(\xi,\eta)} \varphi_{ lo}\left(\Phi(\xi,\eta)\right) \varphi\left(\kappa_\theta^{-1}\Omega_\eta \Phi(\xi,\eta)\right)\varphi_{Lo}(\xi,\eta)\hat{f}(\xi-\eta) \hat{g}(\eta)\,d\eta,  \\
    \reallywidehat{I^{NC}\left[f,g\right]}(\xi)&\triangleq\int_{\mathbb{R}^2} e^{is\Phi(\xi,\eta)} \varphi_{ lo}\left(\Phi(\xi,\eta)\right) \varphi\left(\kappa_\theta^{-1}\Omega_\eta \Phi(\xi,\eta)\right)\varphi_{Hi}(\xi,\eta)\hat{f}(\xi-\eta) \hat{g}(\eta)\,d\eta, \\
    \varphi_{Lo}(\xi,\eta)&\triangleq\varphi_{\le -400\delta m}\left(\nabla_\xi\Phi(\xi,\eta)\right), \\
    \varphi_{Hi}(\xi,\eta)&\triangleq\varphi_{> -400\delta m}\left(\nabla_\xi\Phi(\xi,\eta)\right).
\end{align*} \par
We first consider the integral $I^C$, which produce the secondary resonances $f_{SR}$. Indeed, using (\ref{3.3}), (\ref{3.4}) and (\ref{3.32}), we estimate
$$\left\|\widehat{I^C}\right\|_{L^\infty}\lesssim \kappa_\theta \left\|\widehat{f_{j_1,k_1}}\right\|_{L^\infty}\cdot\left\|\sup_{\theta} \left|\widehat{g_{j_2,k_2}}(r\theta)\right|\right\|_{L^1(rdr)}\lesssim 2^{-3m/2+76\delta m}.$$
The derivatives can be estimated in the same way, given that $j_1\le (1-300\delta)m$, the definition of the cutoff $\varphi_{Lo}$, the smallness of both $\Phi$ and $\nabla_\xi\Phi$.\par
Next. we consider the integral $I^{NC}$. We will show that $I^{NC}$ gives acceptable contributions, i.e.
\begin{equation}
\begin{aligned}
    P_k I^{NC}\left[f_{j_1,k_1},g_{j_2,k_2}\right]=\partial_s F_1+F_2, \\ [5pt]
    \left\|F_1\right\|_{L^2}\lesssim 2^{-36m/35},\,\,\,\left\|F_2\right\|_{L^2}\lesssim 2^{-21m/11}
\end{aligned}\tag{3.33}\label{3.33}
\end{equation}
Recall that $l_0=\left[-3\delta m-4\delta^2 m\right]$, let $l_-\triangleq\left[-14/15m\right]$, and decompose further
\begin{align*}
    &I^{NC}\left[f_{j_1,k_1},g_{j_2,k_2}\right]=\sum_{l_-\le l\le l_0} I^{NC}_l,\\
    &\widehat{I^{NC}_l}(\xi)\triangleq \int_{\mathbb{R}^2} e^{is\Phi(\xi,\eta)} \varphi_l^{\left[l_-,l_0+1\right]}\left(\Phi(\xi,\eta)\right) \varphi\left(\kappa_\theta^{-1}\Omega_\eta \Phi(\xi,\eta)\right)\varphi_{Hi}(\xi,\eta)\widehat{f_{j_1,k_1}}(\xi-\eta) \widehat{g_{j_2,k_2}}(\eta)\,d\eta.
\end{align*}
Note that we will not have the cutoff function $\varphi_{\ge -D}\left(\nabla_\eta\Phi(\xi,\eta)\right)$ anymore due to the lack of Proposition 8.5 (iii). This will affect the proof in \cite{y2} in the following part to some extent. First, we cannot use Lemma 8.9 (i) in \cite{y2} anymore due to the lack of $\left|\nabla_\eta\Phi\right|\ge -D$.\vspace{0.3em} Therefore, to estimate $\left\|P_k I^{NC}_{l_\_}\right\|_{L^2}$ , we fix $\eta$ and get $\left|E_\xi\right|\lesssim 2^{l_\_}\cdot\kappa_\theta$; fix $\xi$ and get $\left|E_\eta\right|\lesssim \begin{cases}2^{l_\_/2}\cdot\kappa_\theta&\mbox{, if } n_2\le m/2 \\ 2^{-n_2}\cdot\kappa_\theta &\mbox{, if }n_2 \ge m/2  
\end{cases}$. Then, we can use Schur's test to get $\left\|P_k I^{NC}_{l_\_}\right\|_{L^2}\lesssim 2^{-1.94m}$.\par
On the other hand, for $l_-<l\le l_0$ we write
\begin{align*}
    &iI^{NC}_l=\partial_s \mathcal{J}_l-\mathscr{A}_l-\mathscr{B}_l,\\
    &\widehat{\mathcal{J}_l}(\xi)\triangleq \int_{\mathbb{R}^2} e^{is\Phi(\xi,\eta)} \Tilde{\varphi_l}\left(\Phi(\xi,\eta)\right) \varphi\left(\kappa_\theta^{-1}\Omega_\eta \Phi(\xi,\eta)\right)\varphi_{Hi}(\xi,\eta)\widehat{f_{j_1,k_1}}(\xi-\eta) \widehat{g_{j_2,k_2}}(\eta)\,d\eta, \\
    &\widehat{\mathscr{A}_l}(\xi)\triangleq\int_{\mathbb{R}^2} e^{is\Phi(\xi,\eta)} \Tilde{\varphi_l}\left(\Phi(\xi,\eta)\right) \varphi\left(\kappa_\theta^{-1}\Omega_\eta \Phi(\xi,\eta)\right)\varphi_{Hi}(\xi,\eta)\partial_s \widehat{f_{j_1,k_1}}(\xi-\eta) \widehat{g_{j_2,k_2}}(\eta)\,d\eta,\\
    &\widehat{\mathscr{B}_l}(\xi)\triangleq\int_{\mathbb{R}^2} e^{is\Phi(\xi,\eta)} \Tilde{\varphi_l}\left(\Phi(\xi,\eta)\right) \varphi\left(\kappa_\theta^{-1}\Omega_\eta \Phi(\xi,\eta)\right)\varphi_{Hi}(\xi,\eta) \widehat{f_{j_1,k_1}}(\xi-\eta) \partial_s\widehat{g_{j_2,k_2}}(\eta)\,d\eta,\\
\end{align*}
where $\tilde{\varphi_l}(x)\triangleq x^{-1} \varphi_l(x)$. 
In terms of  $\left\|\mathcal{J}_l\right\|_{L^2}$ and $\left\|\mathscr{A}_l\right\|_{L^2}$, we note that the authors of \cite{y2} only used Lemma 8.9(i) of \cite{y2} to estimate the volume of $\xi$. Since we still have the condition $\left|\nabla_\xi\Phi\right|\ge 2^{-400\delta m}$ here, the estimates of $\left\|\mathcal{J}_l\right\|_{L^2}$ and $\left\|\mathscr{A}_l\right\|_{L^2}$ in \cite{y2} still works here in our paper.
Namely,
$$\left\|\mathcal{J}_l\right\|_{L^2}\lesssim \sum_{n_2\ge1} 2^{2\delta^2 m}\kappa_\theta\cdot 2^{-l}\cdot 2^{\frac{l-n_2}{2}+400\delta m}\left\|\widehat{f_{j_1,k_1}}\right\|_{L^\infty}\left\|g_{j_2,k_2,n_2}\right\|_{L^2}\lesssim 2^{-\frac{m+l}{2}+(1000\delta-1)m},$$
$$\left\|\mathscr{A}_l\right\|_{L^2}\lesssim \sum_{n_2\ge1} 2^{2\delta^2 m}\kappa_\theta\cdot 2^{\frac{l-n_2}{2}+400\delta m}\left\|\partial_s \widehat{f_{j_1,k_1}}\right\|_{L^\infty}\left\|g_{j_2,k_2,n_2}\right\|_{L^2}\lesssim 2^{-2m},$$
using Lemma 3.6, Lemma 3.9 and Lemma 8.9 (i) in \cite{y2}, which give acceptable contributions as in (\ref{3.33}). \par
To control the term $\mathscr{B}_l$, we need more precise estimates. We use Lemma 3.9 to decompose 
$$\partial_s g_{j_2,k_2}=g_C+g_{NC},$$
and let $\widehat{\mathscr{B}_l}=\widehat{\mathscr{B}_l^1}+\widehat{\mathscr{B}_l^2}$ denote the corresponding decomposition of $\widehat{\mathscr{B}_l}$. 
First, we can use the fact that fix $\eta$, $\left|E_\xi\right|\lesssim 2^l\cdot\kappa_\theta$ and fix $\xi$, $\left|E_\eta\right|\lesssim 2^{l/2}\cdot\kappa_\theta$ to avoid using Lemma 8.9 (i) of \cite{y2}. Then, Schur's test will still give us an acceptable estimate of $\left\|\mathscr{B}_l^2\right\|_{L^2}$. In fact, 
\begin{align*}
    \left\|\mathscr{B}_l^2\right\|_{L^2}\lesssim \kappa_\theta\cdot 2^{400\delta m}\cdot\left\|g_{NC}\right\|_{L^2}\lesssim 2^{-2m+1000\delta m},\tag{3.34}\label{3.34}
\end{align*}
which gives an acceptable contribution as in (\ref{3.33}).\par
Second, using (\ref{3.17}), we can write $\widehat{\mathscr{B}_l^1}$ as a sum over $q\in\left[0,m/2-10\delta m\right]$ and over $\theta,\kappa\in\mathcal{P},\theta+\kappa\neq 0$, of integrals of the form
$$\mathcal{C}_l(\xi)=\int_{\mathbb{R}^2} e^{is\left[\Phi_{\sigma\mu\nu}(\xi,\eta)+\Psi_{\nu\theta\kappa}(\eta)\right]}\tilde{\varphi}_l(\Phi(\xi,\eta))\varphi(\kappa^{-1}_\theta\Omega_\eta\Phi(\xi,\eta))\varphi_{Hi}(\nabla_\xi\Phi(\xi,\eta))\widehat{f_{j_1,k_1}}(\xi-\eta)h^q(\eta)\,d\eta.$$
In views of (\ref{3.17}) and that $n_2\ge m/20$, the functions $h^q=h^q_{\mu\theta\kappa}$ satisfy the properties
\begin{align*}
    &h^q(\eta)=h^q(\eta)\varphi_{\le -m/21}(\Psi^\star_b(\eta)), \\
    &\left\|D^\alpha_\eta h^q(s)\right\|_{L^\infty}\lesssim\,2^{-m+3\delta m} 2^{-q+42\delta q} 2^{(m/2+q+2\delta^2 m)\left|\alpha\right|}, \\
    &\left\|\partial_s h^q(s)\right\|_{L^\infty}\lesssim\,2^{(-2+6\delta)m} 2^{q+42\delta q}.
\end{align*}
However, since we don't have $j_1\le m/2$, $\varphi_{\ge -D}\left(\nabla_\eta\Phi(\xi,\eta)\right)$ and iterated resonances relationships anymore as in \cite{y2}, we have to redo the proof of $\left\|\mathcal{C}_l\right\|_{L^2}$. The contributions of exponents $q\ge 19m/40$ can be estimated as in (\ref{3.34}), since the functions $h^q$ in this case have sufficiently small $L^2$ norm. Therefore we may assume that $q\le 19m/40$. \par
Now, we observe that we have either $\left|\Psi_{\nu\theta\kappa}(\eta)\right|\le 2^{-m/22}$ or $\left|\Psi_{\nu\theta\kappa}(\eta)\right|\ge 2^{-D}$ in the support of the integrals $\mathcal{C}_l$, due to the cutoff function $\varphi_{\le -m/21}(\Psi^*_b(\eta))$. Therefore, we decompose $\mathcal{C}_l=\mathcal{C}_l^1+\mathcal{C}_l^2$, where $\mathcal{C}_l^1$ and $\mathcal{C}_l^2$ are defined by inserting the factors $\varphi_{\le -m/25}(\Psi_{\nu\theta\kappa}(\eta))$ and $\varphi_{\ge -D}(\Psi_{\nu\theta\kappa}(\eta))$. To simplify notation, we will write $\Phi=\Phi_{\sigma\mu\nu}$, and $\Psi=\Psi_{\nu\theta\kappa}$.\par
First, we consider the the integral $\mathcal{C}_l^1$ and Suppose that $\left|\nabla_\eta\Phi\right|\ge 2^{-D}$. If $2^l\ge 2^{-m/25}$, then $\left|\frac{1}{\Phi(\xi,\eta)+\Psi(\eta)}\right|\lesssim 2^{-l}\le 2^{m/25}$. Therefore, we do the integration by parts to the time variable $s$ $i\mathcal{C}_l^1(\xi)=\partial_s \mathcal{K}_l^1(\xi)-\mathcal{\varepsilon}_l^1(\xi)$, where
\begin{equation}
\begin{aligned}
   \mathcal{K}_l^1(\xi)=\int_{\mathbb{R}^2} &e^{is\left[\Phi_{\sigma\mu\nu}(\xi,\eta)+\Psi_{\nu\theta\kappa}(\eta)\right]} \frac{\tilde{\varphi}_l(\Phi(\xi,\eta))}{\Phi_{\sigma\mu\nu}(\xi,\eta)+\Psi_{\nu\theta\kappa}(\eta)}\varphi(\kappa^{-1}_\theta\Omega_\eta\Phi(\xi,\eta))\varphi_{\le -m/25}(\Psi_{\nu\theta\kappa}(\eta)) \\
   &\times\varphi_{\ge -D}(\nabla_\eta\Phi(\xi,\eta))\varphi_{Hi}(\nabla_\xi\Phi(\xi,\eta))\widehat{f_{j_1,k_1}}(\xi-\eta)h^q(\eta)\,d\eta, 
\end{aligned}\label{3.35}\tag{3.35}
\end{equation}
and
\begin{equation}
\begin{aligned}
\mathcal{\varepsilon}_l^1(\xi)=\int_{\mathbb{R}^2}& e^{is\left[\Phi_{\sigma\mu\nu}(\xi,\eta)+\Psi_{\nu\theta\kappa}(\eta)\right]} \frac{\tilde{\varphi}_l(\Phi(\xi,\eta))}{\Phi_{\sigma\mu\nu}(\xi,\eta)+\Psi_{\nu\theta\kappa}(\eta)}\varphi(\kappa^{-1}_\theta\Omega_\eta\Phi(\xi,\eta))\varphi_{\le -m/25}(\Psi_{\nu\theta\kappa}(\eta)) \\
   &\times\varphi_{\ge -D}(\nabla_\eta\Phi(\xi,\eta))\varphi_{Hi}(\nabla_\xi\Phi(\xi,\eta))\partial_s\left[\widehat{f_{j_1,k_1}}(\xi-\eta,s)h^q(\eta,s)\right]\,d\eta. 
\end{aligned}\label{3.36}\tag{3.36}
\end{equation}
Note that fix $\xi$, $\left|E_\eta\right|\lesssim 2^l\cdot\kappa_\theta$; fix $\eta$, $\left|E_\xi\right|\lesssim 2^{l+400\delta m}\cdot\kappa_\theta$ and by Schur's test we get that
\begin{align*}
   \left\|\mathcal{K}^1_l(\xi)\right\|_{L^2}&\lesssim \sup\left|\frac{1}{\Phi(\xi,\eta)+\Psi(\eta)}\right|\cdot\sup\left|\frac{1}{\Phi(\xi,\eta)}\right|\cdot\left|E_\eta\right|^{1/2}\cdot\left|E_\xi\right|^{1/2}\cdot\left\|\widehat{f_1}\right\|_{L^\infty}\cdot\left(\left\|h^q(\eta)\right\|_{L^\infty}\left|E_\eta\right|^{1/2}\right) \\
   &\lesssim\,2^{m/25}\cdot 2^{-l}\cdot (2^{\frac{l}{2}}\kappa_\theta^{\frac{1}{2}}2^{\frac{l}{2}}2^{200\delta m}\kappa_\theta^{\frac{1}{2}})\cdot 2^{2\delta n_1}\cdot (2^{-m+3\delta m}\cdot 2^{\frac{l}{2}})\lesssim\,2^{-35m/36}.
\end{align*}
Also note that
\begin{align*}
    \left|\partial_s(\widehat{f_1}(\xi-\eta)h^q(\eta))\right| &\lesssim\left|\partial_s\widehat{f_1}\cdot h^q\right|+\left|\widehat{f_1}\cdot \partial_s h^q\right| \\
    &\lesssim\,2^{-m+60\delta m}\cdot2^{-m+3\delta m}+2^{2\delta m}\cdot2^{-2m+6\delta m+(1-42\delta)q} \\
    &\lesssim\,2^{-2m+63\delta m+(1-42\delta)q},
\end{align*}
and by Schur's test again we obtain that
\begin{align*}
   \left\|\mathcal{\varepsilon}^1_l(\xi)\right\|_{L^2}&\lesssim \sup\left|\frac{1}{\Phi(\xi,\eta)+\Psi(\eta)}\right|\cdot\sup\left|\frac{1}{\Phi(\xi,\eta)}\right|\cdot\left|E_\eta\right|^{1/2}\cdot\left|E_\xi\right|^{1/2}\cdot\left|\partial_s(\widehat{f_1}(\xi-\eta)h^q(\eta))\right|\cdot\left|E_\eta\right|^{1/2} \\
   &\lesssim\,2^{m/25}\cdot 2^{-l}\cdot (2^{\frac{l}{2}}\kappa_\theta^{\frac{1}{2}}2^{\frac{l}{2}}2^{200\delta m}\kappa_\theta^{\frac{1}{2}})\cdot 2^{-2m+63\delta m+(1+42\delta)q}\cdot 2^{\frac{l}{2}}\lesssim\,2^{-21m/11}.
\end{align*}
These give $f_{NC}$ and $F_{NC}$ contributions respectively as in (\ref{3.24}) and (\ref{3.27}). If $2^l\le 2^{-m/25}$ and $\left|\Phi(\xi,\eta)\right|\sim\left|\Psi(\eta)\right|\sim 2^l$, then we split it into two subcases: $\frac{r_0}{2}-q\le -\frac{m}{2}+500\delta m$ and $\frac{r_0}{2}-q\ge -\frac{m}{2}+500\delta m$, where we suppose $\left|\Phi(\xi,\eta)+\Psi(\eta)\right|\sim 2^{r_0}$. (Since $r_0\le l$, we must have $r_0\le -\frac{m}{25}$.) In the first subcase $\frac{r_0}{2}-q\le -\frac{m}{2}+500\delta m$, we use Schur's test to estimate $\left\|\mathcal{C}^1_l\right\|_{L^2}$ directly, where fix $\eta$, we have $\left|E_\xi\right|\lesssim\,2^{r_0+400\delta m}\cdot\kappa_\theta$ and fix $\xi$, we have $\left|E_\eta\right|\lesssim\,2^l\cdot\kappa_\theta$. We have
\begin{align*}
   \left\|\mathcal{C}^1_l(\xi)\right\|_{L^2}&\lesssim \sup\left|\frac{1}{\Phi(\xi,\eta)}\right|\cdot\left|E_\eta\right|^{1/2}\cdot\left|E_\xi\right|^{1/2}\cdot\left\|\widehat{f_1}\right\|_{L^\infty}\cdot\left(\left\|h^q(\eta)\right\|_{L^\infty}\left|E_\eta\right|^{1/2}\right) \\
   &\lesssim\,2^{-l}\cdot (2^{\frac{r_0}{2}+\frac{l}{2}+200\delta m}\kappa_\theta)\cdot 2^{2\delta n_1}\cdot (2^{-m+3\delta m-(1-42\delta)q}\cdot 2^{\frac{l}{2}})\lesssim\,2^{-2m+727\delta m},
\end{align*}
which gives an acceptable $f_{NC}$-type contribution as in (\ref{3.24}). In the second subcase $\frac{r_0}{2}-q\ge -\frac{m}{2}+500\delta m$, we need do the integration by parts to the time variable $s$ like before. Namely, consider $i\mathcal{C}_l^1(\xi)=\partial_s \mathcal{K}_l^1(\xi)-\mathcal{\varepsilon}_l^1(\xi)$, where $\mathcal{K}^1_l$ and $\mathcal{\varepsilon}^1_l$ are defined in (\ref{3.35}) and (\ref{3.36}). Note that fix $\eta$, we have $\left|E_\xi\right|\lesssim\,2^{r_0+400\delta m}\cdot\kappa_\theta$ and fix $\xi$, we have $\left|E_\eta\right|\lesssim\,2^l\cdot\kappa_\theta$. Now, we use Schur's test to get
\begin{align*}
   \left\|\mathcal{K}^1_l(\xi)\right\|_{L^2}&\lesssim \sup\left|\frac{1}{\Phi(\xi,\eta)+\Psi(\eta)}\right|\cdot\sup\left|\frac{1}{\Phi(\xi,\eta)}\right|\cdot\left|E_\eta\right|^{1/2}\cdot\left|E_\xi\right|^{1/2}\cdot\left\|\widehat{f_1}\right\|_{L^\infty}\cdot\left(\left\|h^q(\eta)\right\|_{L^\infty}\left|E_\eta\right|^{1/2}\right) \\
   &\lesssim\,2^{-r_0}\cdot 2^{-l}\cdot (2^{\frac{l}{2}+\frac{r_0}{2}+200\delta m}\kappa_\theta)\cdot 2^{2\delta n_1}\cdot (2^{-m+3\delta m-q+42\delta q}\cdot 2^{\frac{l}{2}})\lesssim\,2^{-m-273.9\delta m},
\end{align*}
and
\begin{align*}
   \left\|\mathcal{\varepsilon}^1_l(\xi)\right\|_{L^2}&\lesssim \sup\left|\frac{1}{\Phi(\xi,\eta)+\Psi(\eta)}\right|\cdot\sup\left|\frac{1}{\Phi(\xi,\eta)}\right|\cdot\left|E_\eta\right|^{1/2}\cdot\left|E_\xi\right|^{1/2}\cdot\left\|\widehat{f_1}\right\|_{L^\infty}\cdot\left|\partial_s(\widehat{f_1}(\xi-\eta)h^q(\eta))\right|\cdot\left|E_\eta\right|^{1/2} \\
   &\lesssim\,2^{-r_0}\cdot 2^{-l}\cdot (2^{\frac{l}{2}+\frac{r_0}{2}+200\delta m}\kappa_\theta)\cdot (2^{-2m+63\delta m+q+42\delta q}\cdot 2^{\frac{l}{2}})\lesssim\,2^{-2m-215\delta m},
\end{align*}
which give $f_{NC}$ and $F_{NC}$ contributions respectively as in (\ref{3.24}) and (\ref{3.27}). If $2^l\le 2^{-m/25}$ and $\left|\Phi\right|\sim 2^l$, $\left|\Psi\right|\nsim 2^l$, then we will have two subcases: $\left|\Phi\right|\sim 2^l, \left|\Psi\right|\sim 2^{r_0}, r_0> l$, and $\left|\Phi\right|\sim 2^l, \left|\Psi\right|\sim 2^{l_1}, l_1< l$. In the first subcase, we still have $\left|\Phi+\Psi\right|\sim 2^{r_0}$. Then we can prove it as in the previous case $2^l\le 2^{-m/25}, \left|\Phi\right|\sim\left|\Psi\right|\sim2^l$, since we actually didn't use the condition $\left|\Psi\right|\sim2^{r_0}$ as before. In the second subcase, we will have $\left|\Phi+\Psi\right|\sim 2^l$. Then the proof is also similar as the previous case $2^l\le 2^{-m/25}, \left|\Phi\right|\sim\left|\Psi\right|\sim2^l$. In fact, by checking the proof, we only need to replace all $r_0$ into $l$ in the proof.\par
Next, we still consider the integral $\mathcal{C}_l^1$ but suppose that $\left|\nabla_\eta\Phi\right|\le 2^{-D}$. Note that in this case we have $\left|\nabla_\eta\right|\gtrsim 1$. This implies that $\left|\nabla_\eta\left[\Phi(\xi,\eta)+\Psi(\eta)\right]\right|\gtrsim 1$. This means that we can use Lemma 3.1 to integrate by parts to $\eta$ to get an acceptable control. (Recall that we have already assumed that $j_1\le(1-200\delta)m$, $q\le \frac{19}{40}m$, and we also have $\left\|D^\alpha_\eta h^q\right\|_{L^\infty}\lesssim\,2^{-m+3\delta m-(1-42\delta)q+(\frac{39.1}{40}m)\left|\alpha\right|}$.)\par
Finally, we consider the integral $\mathcal{C}^2_l$. Note that in this case we have $\left|\Phi\right|\sim 2^l$ and $\left|\Psi\right|\gtrsim 1$, which implies that $\left|\Phi+\Psi\right|\gtrsim 1$ and $\frac{1}{\left|\Phi+\Psi\right|}\lesssim 1$. Thus, we integrate by parts to the time variable $s$: $i\mathcal{C}_l^2(\xi)=\partial_s \mathcal{K}_l^2(\xi)-\mathcal{\varepsilon}_l^2(\xi)$. Note that fix $\xi$, $\left|E_\eta\right|\approx 2^{l/2}\cdot\kappa_\theta$; fix $\eta$, $\left|E_\xi\right|\approx 2^{l+400\delta m}\cdot\kappa_\theta$. Thus, by Schur's test, we get that
\begin{align*}
   \left\|\mathcal{K}^2_l(\xi)\right\|_{L^2}&\lesssim \sup\left|\frac{1}{\Phi(\xi,\eta)+\Psi(\eta)}\right|\cdot\sup\left|\frac{1}{\Phi(\xi,\eta)}\right|\cdot\left|E_\eta\right|^{1/2}\cdot\left|E_\xi\right|^{1/2}\cdot\left\|\widehat{f_1}\right\|_{L^\infty}\cdot\left(\left\|h^q(\eta)\right\|_{L^\infty}\left|E_\eta\right|^{1/2}\right) \\
   &\lesssim\,1\cdot 2^{-l}\cdot (2^{\frac{3l}{4}+200\delta m}\kappa_\theta)\cdot 2^{2\delta n_1}\cdot (2^{-m+3\delta m-q+42\delta q}\cdot 2^{\frac{l}{4}})\lesssim\,2^{-\frac{3}{2}m+205.1\delta m},
\end{align*}
and
\begin{align*}
   \left\|\mathcal{\varepsilon}^2_l(\xi)\right\|_{L^2}&\lesssim \sup\left|\frac{1}{\Phi(\xi,\eta)+\Psi(\eta)}\right|\cdot\sup\left|\frac{1}{\Phi(\xi,\eta)}\right|\cdot\left|E_\eta\right|^{1/2}\cdot\left|E_\xi\right|^{1/2}\cdot\left\|\widehat{f_1}\right\|_{L^\infty}\cdot\left|\partial_s(\widehat{f_1}(\xi-\eta)h^q(\eta))\right|\cdot\left|E_\eta\right|^{1/2} \\
   &\lesssim\,1\cdot 2^{-l}\cdot (2^{\frac{3l}{4}+200\delta m}\kappa_\theta)\cdot (2^{-2m+63\delta m+(1+42\delta) q})\cdot 2^{\frac{l}{4}}\lesssim\,2^{-2.025m+300\delta m}.
\end{align*}
which again give $f_{NC}$ and $F_{NC}$ contributions respectively as in (\ref{3.24}) and (\ref{3.27}).\par
\textbf{Case 2. } We consider now $I^{lo}$ in the case
$$\min\left\{k,k_1,k_2\right\}\ge -D, \,\,\,\,\,\,j_1\le\min(j_2,(1-300\delta)m), \,\,\,\,\,\,j_2\le (1-50\delta)m.$$
Integrations by parts, first in $\eta$ using Lemma 3.1 then in $\Omega_\eta$ using Lemma 3.2 show that
\begin{align*}
&\left\|I^{lo}\left[f_{j_1,k_1},g_{j_2,k_2}\right]-I^S\left[f_{j_1,k_1},g_{j_2,k_2}\right]\right\|_{L^2}\lesssim 2^{-2m},\\
&\mathcal{F}\left\{I^S\left[f,g\right]\right\}(\xi)\triangleq\int_{\mathbb{R}^2} e^{is\Phi(\xi,\eta)} \varphi_{lo}(\Phi(\xi,\eta)) \varphi\left(\kappa^{-1}_r \nabla_\eta\Phi(\xi,\eta)\right)\varphi\left(\kappa^{-1}_\theta \Omega_\eta\Phi(\xi,\eta)\right)\hat{f}(\xi-\eta)\hat{g}(\eta)\,d\eta,  
\end{align*}
where $\kappa_r\triangleq 2^{\delta^2 m}(2^{j_2-m}+2^{-m/2})$, and $\kappa_\theta\triangleq 2^{\delta^2 m-m/2}$. Observe that if $\mu+\nu=0$, then by Lemma 6.3, we know that $\left|\nabla_\eta\Phi\right|\gtrsim 2^{-\delta^2 m/2}$. This means that $I^S\left[f_{j_1,k_1},g_{j_2,k_2}\right]\equiv0$, if $\mu+\nu=0$.

However, in general, we do not necessarily have $I^S\left[f_{j_1,k_1},g_{j_2,k_2}\right]\equiv I^S\left[f_{j_1,k_1,0},g_{j_2,k_2,0}\right]$, which means $n_1, n_2$ could be positive. Write
$$\mathcal{F}\left\{I^S\left[f_{j_1,k_1},g_{j_2,k_2}\right]\right\}(\xi)=e^{is\Psi(\xi)}g(\xi,s),$$
where
\begin{align*}
g(\xi,s)\triangleq\int_{\mathbb{R}^2}&e^{is\left[\Phi(\xi,\eta)-\Psi(\xi)\right]}\varphi\left(\kappa_\theta^{-1}\Omega_\eta\Phi(\xi,\eta)\right)\varphi\left(\kappa_r^{-1}\nabla_\eta\Phi(\xi,\eta)\right)\varphi_{lo}\left(\Phi(\xi,\eta)\right) \\
&\times\reallywidehat{f_{j_1,k_1,n_1}}(\xi-\eta)\reallywidehat{g_{j_2,k_2,n_2}}(\eta)\,d\eta, \tag{3.37}\label{3.37}
\end{align*}
and we have to modify the proof of the pointwise estimate of $g(\xi,s)$. Denote $q\triangleq \max\left\{0,j_2-m/2\right\}$. If $j_2\le m/2$, then by volume counting ($\left|E_\eta\right|\sim\kappa_r^2$), we have that
\begin{align*}
    \left|\varphi_k(\xi)g(\xi,s)\right|&\lesssim2^{\delta^2 m}\cdot\left\|\widehat{f_1}\right\|_{L^\infty}\cdot\left\|\widehat{g_2}\right\|_{L^\infty}\cdot\left|E_\eta\right|\lesssim 2^{-m+4.01\delta m}.
\end{align*}
If $j_2\ge m/2$, then we first observe that by Proposition 6.5 (a), the function $p_+$ defined in Proposition 6.5(a) has an inverse function $p_+^{-1}$ and $\left|(p_+^{-1})^\prime\right| \lesssim 2^{0.3\delta^2 m}$. Moreover, we notice that $\left|\eta-p(\xi)\right|\lesssim 2^{0.4\delta^2 m}\,\kappa_r$ by Proposition 6.5 (a). These tell us that
\begin{align*}
    \big|\left|\xi\right|-p_+^{-1}\left(\left|\eta\right|\right)\big|=\big|p_+^{-1}\left(p_+\left(\left|\xi\right|\right)\right)-p_+^{-1}\left(\left|\eta\right|\right)\big|\le \left|\left(p_+^{-1}\right)^\prime\right|\cdot\big|p_+\left(\left|\xi\right|\right)-\left|\eta\right|\big|\lesssim 2^{0.7\delta^2 m}\,\kappa_r. \tag{3.38}\label{3.38}
\end{align*}
Therefore, we have that fix $\eta$, $\left|E_\xi\right|\lesssim 2^{0.7\delta^2 m}\,\kappa_r\cdot\kappa_\theta$; fix $\xi$, $\left|E_\eta\right|\lesssim 2^{-n_2}\cdot\kappa_\theta$. By Schur's test and (\ref{3.7b}), we get
\begin{align*}
   \left|\varphi_k(\xi)g(\xi,s)\right|\lesssim 2^{0.35\delta^2 m}\,\kappa_r^{\frac{1}{2}}\,2^{-\frac{1}{2}n_2}\,\kappa_\theta\,\left\|\widehat{f_1}\right\|_{L^\infty}\cdot\left\|\widehat{g_2}\right\|_{L^2}\lesssim 2^{-m-q-4\delta m}. \tag{3.39}\label{3.39}
\end{align*}
To sum up, we can always get $\left|\varphi_k(\xi)g(\xi,s)\right|\lesssim 2^{-m-q+4.01\delta m}$. The bound on $\xi$-derivatives follows from the fact that 
\begin{align*}
    \left|\nabla_\xi s\left[\Phi(\xi,\eta)-\Psi(\xi)\right]\right|&\lesssim \left|s\right|\left|\nabla_\xi\Phi(\xi,\eta)-\nabla_\xi\Phi(\xi,p(\xi))\right| \\
    &\lesssim 2^{\delta^2 m}\,2^m\kappa_r\lesssim (2^{m/2}+2^{j_2})2^{2\delta^2 m}.
\end{align*}
Finally, the bound on $\partial_s g$ follows in the same way as in the proof of Lemma 6.2 in \cite{y2}, see (6.38) in \cite{y2}. The bounds (\ref{3.28}) follow by examining the defining formulas above and the identities 
\begin{align*}
    &\left\{\Omega_\xi+\Omega_\eta\right\}\chi(\Phi(\xi,\eta))\equiv 0,\,\,\,\,\left\{\Omega_\xi+\Omega_\eta\right\}\chi\left(\kappa^{-1}\Omega_\eta\Phi(\xi,\eta)\right)\equiv 0, \\
    &\left\{\Omega_\xi+\Omega_\eta\right\}\chi\left(\kappa^{-1} \nabla_\eta\Phi(\xi,\eta)\right)=\kappa^{-1} \nabla_\eta^\bot\Phi(\xi,\eta)\cdot\nabla\chi\left(\kappa^{-1}\nabla_\eta\Phi(\xi,\eta)\right).
\end{align*}
\end{proof}
\vspace{4em}
\begin{remark}
In fact, it turns out that (\ref{3.22}) can be further improved. This will be used in Section 5.4 later on. Recall that we have assumed that 
$$\min\left\{k,k_1,k_2\right\}\ge -D, \,\,\,\,\,\,j_1\le\min(j_2,(1-300\delta)m), \,\,\,\,\,\,j_2\le (1-50\delta)m.$$
Let's consider $\left|\varphi_k(\xi)g(\xi,s)\right|$ when $j_2\le m/2$, where $g(\xi,s)$ is defined as in (\ref{3.37}). First, when $j_2\le 0.4m$, we decompose as
\begin{align*}
    g(\xi,s)=&\int_{\mathbb{R}^2}e^{is\left[\Phi(\xi,\eta)-\Psi(\xi)\right]}\varphi\left(\kappa_\theta^{-1}\Omega_\eta\Phi(\xi,\eta)\right)\varphi\left(\kappa_r^{-1}\nabla_\eta\Phi(\xi,\eta)\right)\left(1-\varphi\left(2^{0.55m}\nabla_\eta\Phi(\xi,\eta)\right)\right)\varphi_{lo}\left(\Phi(\xi,\eta)\right) \\
    &\ \ \ \ \ \times\widehat{f_1}(\xi-\eta)\widehat{g_2}(\eta)\,d\eta \\
    +&\int_{\mathbb{R}^2}e^{is\left[\Phi(\xi,\eta)-\Psi(\xi)\right]}\varphi\left(\kappa_\theta^{-1}\Omega_\eta\Phi(\xi,\eta)\right)\varphi\left(\kappa_r^{-1}\nabla_\eta\Phi(\xi,\eta)\right)\varphi\left(2^{0.55m}\nabla_\eta\Phi(\xi,\eta)\right)\varphi_{lo}\left(\Phi(\xi,\eta)\right) \\
    &\ \ \ \ \ \times\widehat{f_1}(\xi-\eta)\widehat{g_2}(\eta)\,d\eta \\ 
    =&I_1+I_2.
\end{align*}
Integration by parts in $\eta$ gives us that $\left\|I_1\right\|_{L^\infty}\lesssim 2^{-2m}$. On the other hand, by volume counting and (\ref{3.7b}), we get that
\begin{align*}
    \left\|I_2\right\|_{L^\infty}&\lesssim 2^{\delta^2 m}\cdot\left\|\widehat{f_1}\right\|_{L^\infty}\cdot\left\|\widehat{g_2}\right\|_{L^\infty}\cdot\left|E_\eta\right| \lesssim 2^{-1.1m+0.8\delta m+1.7\delta^2 m}.
\end{align*}
Next, when $0.4m\le j_2\le 0.5m$, we have $\kappa_r=\kappa_\theta=2^{-\frac{m}{2}+\delta^2 m}$. Note that fix $\xi$, we have $\left|E_\eta\right|\lesssim \kappa_r\cdot\kappa_\theta$; fix $\eta$, we have $\left|E_\xi\right|\lesssim 2^{0.7\delta^2 m}\,\kappa_r\cdot\kappa_\theta$ thanks to (\ref{3.38}) and by Schur's test and (\ref{3.7b}) we get that
\begin{align*}
    \left|\varphi_k(\xi)g(\xi,s)\right|&\lesssim 2^{0.35\delta^2 m}\,\kappa_r\,\kappa_\theta\,2^{-(1-21\delta)j_2+(\frac{1}{2}-19\delta)n_2}\,2^{2\delta^2 j_2}\lesssim 2^{-1.2m+\delta m+2\delta^2 m}.
\end{align*}
Thus, to sum up, we can conclude that when $j_2\le m/2$, we have $\left|\varphi_k(\xi)g(\xi,s)\right|\lesssim 2^{-1.1m+\delta m}$. In general, we have
\begin{align*}
    \left|\varphi_k(\xi)g(\xi,s)\right|\lesssim \begin{cases}
        2^{-1.1m+\delta m} &\mbox{, if }q=0 \\
        2^{-m-q-4\delta m} &\mbox{, if }q>0
    \end{cases}, \tag{3.40}\label{3.40}
\end{align*}
which is an improvement of (\ref{3.22}).
\end{remark}

\vspace{1em}
\section{Energy Estimate}
We will follow the proof in section 5 in \cite{y2} to finish our energy estimate. First of all, we need to have an important lemma, which is analogues to Lemma 5.5 of \cite{y2}. 
\begin{lemma}
   Assume that $\delta_2=10^{-3}$,$m\ge 1$, and $f,g,h\in C^1([2^{m-2},2^{m+2}]:L^2)$ satisfy
   \begin{align}
       &\max_{1\le \sigma\le d}\left\|f(s)\right\|_{Z^\sigma_1\cap H_\Omega^{N_1/2}}+2^{(1-\delta_2/8)m}\left\|(\partial_s f)(s)\right\|_{L^2}  \notag \\
       &+\left\|g(s)\right\|_{L^2}+\left\|h(s)\right\|_{L^2}+2^{(1-\delta_2/8)m}\left[\left\|(\partial_s g)(s)\right\|_{L^2}+\left\|(\partial_s h)(s)\right\|_{L^2}\right]\lesssim 1,  \tag{4.1}\label{4.1}
   \end{align}
   for any $s\in[2^{m-2},2^{m+2}]$. Moreover, assume that
   \begin{align}
       &\left\|e^{-i(s+\lambda)\Lambda_\mu}f(s)\right\|_{L^\infty}\lesssim 2^{-(1-\delta_2/8)m},\ \ \ \ \partial_s f=F_2+F_\infty, \notag \\
       &\left\|e^{-i(s+\lambda)\Lambda_\mu}P_k F_\infty(s)\right\|_{L^\infty}\lesssim 2^{-2m+\delta_2 m/4},\notag \\[4pt]
       &\left\|P_k F_2(s)\right\|_{L^2}\lesssim\begin{cases}
\varepsilon_1^2\,2^{-3m/2+\delta_2 m/8}&\mbox{, if }k\ge 0 \\
\varepsilon_1^2\,2^{-3m/2-k/2+\delta_2 m/8}&\mbox{, if }-m/2\le k\le 0 \\
\varepsilon_1^2\,2^{-5m/4+60\delta_2 m/8}&\mbox{, if }k\le -m/2 \\
\end{cases},   \tag{4.2}\label{4.2}
   \end{align}
   for any $\lambda\in\mathbb{R}$ with $\left|\lambda\right|\le 2^{m(1-\delta_2/10)}$ and any $k\in\mathbb{Z}$. Let
   $$I[f,g,h]\triangleq\int_\mathbb{R}n(s)\int_{\mathbb{R}^2\times\mathbb{R}^2} a(\xi,\eta)e^{is\Phi(\xi,\eta)}\hat{f}(\xi-\eta,s)\hat{g}(\eta,s)\overline{\hat{h}(\xi,s)}\,d\eta d\xi ds$$
   where n is a $C^1$ function supported in $[2^{m-2},2^{m+2}]$,
   \begin{align}
      \int_\mathbb{R}\left|n^\prime (s)\right|\,ds\lesssim 1,\ \ \ \left\|\mathcal{F}^{-1}a\right\|_{L^1(\mathbb{R}^4)}\lesssim 1. \tag{4.3}\label{4.3}
   \end{align}
   Then, for any $k,k_1,k_2\in\mathbb{Z}$,
   \begin{align}
       \left|I\left[P_{k_1}f,P_{k_2}g,P_k h\right]\right|\lesssim 2^{60\max \left(k,k_1,k_2,0\right)} 2^{-\delta_2 m/10}.  \tag{4.4}\label{4.4}
   \end{align}
\end{lemma}
\begin{proof}
This was basically proved in Section 5 in \cite{y2}, except that there are some slight differences. In the following, we use all the notations in Section 5 in \cite{y2}. We first modify (5.2) of \cite{y2} to be 
$$\left\|Tf\right\|_{L^2}\lesssim 2^{60k}\,2^{-1.004\lambda} \left\|f\right\|_{L^2}.$$ Thus, in the following proof, we may assume $\lambda\ge 60k+D$. Next, we move on to Lemma 5.2 of \cite{y2}. In our case, we have to divide $E=E^\prime_1\cup E^\prime_2$ as
$$E=\left\{(\xi,\eta):\max(\left|\xi\right|,\left|\eta\right|)\le 2^k,\left|\Phi(\xi,\eta)\right|\le 2^{-k}\varepsilon,\left|\Upsilon(\xi,\eta)\right|\le 2^{-3k}\varepsilon^\prime \right\},$$
$$E^\prime_1=\left\{(\xi,\eta):\max(\left|\xi\right|,\left|\eta\right|)\le 2^k,\left|\Phi(\xi,\eta)\right|\le 2^{-k}\varepsilon,\left|\Upsilon(\xi,\eta)\right|\le 2^{-3k}\varepsilon^\prime,\left|\nabla_\eta\Phi(\xi,\eta)\right|\ge 2^{-3\Bar{k}} \right\},$$
$$E^\prime_2=\left\{(\xi,\eta):\max(\left|\xi\right|,\left|\eta\right|)\le 2^k,\left|\Phi(\xi,\eta)\right|\le 2^{-k}\varepsilon,\left|\Upsilon(\xi,\eta)\right|\le 2^{-3k}\varepsilon^\prime,\left|\nabla_\xi\Phi(\xi,\eta)\right|\ge 2^{-3\Bar{k}} \right\}.$$
But in this case, we can still use Proposition 6.6 (b), since it can be proved similarly.\par
Then we go to check Lemma 5.3 of \cite{y2}. In our case, the assumptions (5.8) of \cite{y2} have to be changed to 
$$\left|\nabla_\xi\Phi(\nu^i,\nu^j)\right|\gtrsim 2^{-3k},\ \ \left|\nabla_\eta\Phi(\nu^i,\nu^j)\right|\approx\rho\gtrsim R,\ \ \left|\Upsilon(\nu^i,\nu^j)\right|\gtrsim R.$$
Thus, the second formula of (5.13) of \cite{y2} needs to be changed to be $b_1\triangleq\left|\nabla_\xi\Phi(\xi,\eta^\prime_0)\right|\gtrsim 2^{-3k}$. Now, in (5.14) of \cite{y2}, we consider the case $\left|\omega_1\right|\ge 2^D \left(2^{3k} 2^{-\lambda}+\omega_2^2\right)$ instead, and then using the Taylor expansion as shown in that paper we end up with 
$$\left|\Phi(\xi^\prime,\eta)\right|\ge b_1 \omega_1-C\cdot 2^{-\frac{\lambda}{2}}\left(\left|\omega_1\right|+\left|\omega_2\right|\right)-C\cdot\omega_1^2-\frac{1}{2^D}\, 2^{-3k}\omega_1 \ge C\cdot 2^{-3k}\omega_1,$$
which again implies $K_{\eta_0}(\xi,\xi^\prime)=0$ as what we want. On the other hand, in (5.16) of \cite{y2} we have to assume $\left|\omega_1\right|\le 2^D\left(2^{3k}\cdot 2^{-\lambda}+\omega_2^2\right)$. Previously, we had
$$\int_{\mathbb{R}^2} \left|K_{\eta_0}(\xi,\xi^\prime)\right|\,d\xi^\prime\lesssim R^{-3}\,2^{-3\lambda},$$
but now we should have
$$\int_{\mathbb{R}^2} \left|K_{\eta_0}(\xi,\xi^\prime)\right|\,d\xi^\prime\lesssim R^{-3}\,2^{3k}2^{-3\lambda}.$$
Since $2^{3k}\,R^{-3}=2^{3k+\frac{51}{242}\lambda}\le 2^{\frac{1}{20}\lambda+\frac{51}{242}\lambda}\le R^{-4}$, we can still get that
$$\int_{\mathbb{R}^2} \left|K_{\eta_0}(\xi,\xi^\prime)\right|\,d\xi^\prime\lesssim R^{-4}\,2^{-3\lambda},$$
which is required in (5.10) of \cite{y2}.\par
Finally, we move on to Lemma 5.4 and Lemma 5.5 of \cite{y2}. They still hold if we modify (5.20) of \cite{y2} to be
$$\left\|R\left[P_{\le k}f,P_{\le k}g\right]\right\|_{L^2}\lesssim 2^{60k} 2^{-(1+\delta_2/2)m}\cdot\max_{1\le\sigma\le d}\left\|f\right\|_{Z^\sigma_1\cap H_\Omega^{N_1/2}}\left\|g\right\|_{L^2},$$
and modify (5.27) of \cite{y2} to be
(\ref{4.4}). Note that even though the assumption of $\left\|P_k F_2(s)\right\|_{L^2}$ is weaker than its corresponding part in \cite{y2}, we are still able to get the desired estimate $\left|II\right|$ in Page 812 in \cite{y2}.\vspace{0.3em} In fact, let $K(\xi,\eta)\triangleq a(\xi,\eta) e^{i(s+\lambda)\Phi(\xi,\eta)}\widehat{F_2}(\xi-\eta,s)$ where
$$a(\xi,\eta)=\int_{\mathbb{R}^2\times\mathbb{R}^2} k(x,y)e^{ix\cdot\xi}e^{iy\cdot\eta}\ \ \ \ \left\|k\right\|_{L^1(\mathbb{R}^4)}\lesssim 1$$
as in \cite{y2}, then we have
\begin{align*}
    &\int \left|K(\xi,\eta) \right|\,d\xi\\
    \le &\left(\int_E \left|a(\xi,\eta)\right|^2\,d\xi\right)^{1/2} \left(\int_E\left|\widehat{F_2}(\xi-\eta)\right|^2\,d\xi\right)^{1/2}\\
    \le &\left(\int_E \left|1_E(\xi,\eta)\right|^2\,d\xi\right)^{1/2} \left(\int_E\left|\widehat{F_2}(\xi-\eta)\right|^2\,d\xi\right)^{1/2}\\
    \le\, & 2^{5k+\tau/2} \left(\int_E\left|\widehat{F_2}(\xi-\eta)\right|^2\,d\xi\right)^{1/2},\\
    \le\, & 2^{5k+\tau/2}\,2^{-5m/4+60\delta m}.
\end{align*}
where using Proposition 6.6 (a) for the second last inequality. Note that Proposition 6.6 (a) also holds for $k\le 0$. We may also assume $k<\delta^2 m$, thus whenever $-k<\tau<0$, we can estimate it directly without applying Proposition 6.6 (a). Now, if $\tau<-k$, then we take $\varepsilon\triangleq 2^{10k+\tau}$, and get the second last inequality. By Schur's test, we get that
\begin{align*}
   \left\|\int_{\mathbb{R}^2} a(\xi,\eta) e^{i(s+\lambda)\Phi(\xi,\eta)}\widehat{F_2}(\xi-\eta,s)\hat{g}(\eta,s)\,d\eta\right\|_{L^2_\xi}\lesssim &2^{-5m/4+60\delta m+\tau/2+5k}\left\|g(s)\right\|_{L^\infty}\\
   \lesssim &2^{-5m/4+60\delta m+\tau/2+5k}, 
\end{align*}
so we estimate
\begin{align*}
    \left|II_{F_2}\right|\triangleq&\left|\int_{\mathbb{R}\times\mathbb{R}} n(s)\,P(2^\tau \lambda) \int_{\mathbb{R}^4} a(\xi,\eta) e^{i(s+\lambda)\Phi(\xi,\eta)}\widehat{F_2}(\xi-\eta,s)\hat{g}(\eta,s)\overline{\hat{h}(\xi,s)}\,d\eta d\xi\,dsd\lambda\right| \\
    \lesssim&\int_{\mathbb{R}\times\mathbb{R}} n(s)\,P(2^\tau \lambda) \left\|h(s)\right\|_{L^2}\cdot\left\|\int_{\mathbb{R}^4} a(\xi,\eta) e^{i(s+\lambda)\Phi(\xi,\eta)}\widehat{F_2}(\xi-\eta,s)\hat{g}(\eta,s)\,d\eta\right\|_{L^2_\xi}\,dsd\lambda \\
    \lesssim&\int_{\mathbb{R}\times\mathbb{R}} n(s)\,P(2^\tau \lambda)\cdot 1\cdot 2^{-5m/4+60\delta m+\tau/2+5k}\,dsd\lambda \\
    \lesssim&\,2^{-m/4-\tau/2+5k}.
\end{align*}
If $-(1-\delta_2/2)m/2\le\tau\le 0$, then we get $\left|II_{F_2}\right|\lesssim 2^{-\delta_2 m/10}$. Thus, we only need to consider the case when $-(1-\delta_2/2)m\le \tau\le -(1-\delta_2/2)m/2$. We may also assume $\max(k,k_1,k_2)<\delta^2 m$. If $-D\le k_1 \le\delta^2 m$, then we can deal with it exactly the same as above (in fact we can improve $\left\|P_k F_2(s)\right\|_{L^2}$ in view of (\ref{4.2})) and get $\left|II_{F_2}\right|\lesssim 2^{-m/2-\tau/2+\frac{9}{2}k}\lesssim 2^{-\delta_2 m/10}$. If $k_1\le -D$ and $k,k_2\ge -D-10$, then by Proposition 6.1, we have $\left|\nabla_\eta\Phi\right|\gtrsim 2^{-3\delta^2 m}$. This implies that (WLOG) $\left|\partial_{\xi_1}\Phi\right|\gtrsim 2^{-3\delta^2 m}$. Note that $\left|\Phi\right|\lesssim 2^\tau$, so we have that $\xi_1$ belongs to an interval with length at most $\sim 2^{\tau-3\delta^2 m}$. On the other hand, since $\eta$ is fixed and $\left|\xi-\eta\right|\sim 2^{k_1}$, we know that $\xi_2$ belongs to an interval with length at most $\sim 2^{k_1}$. To sum up, we get
$$\left(\int_E \left|1_E (\xi,\eta)\right|^2\,d\xi\right)^{1/2}\lesssim 2^{k_1/2+\tau/2-3\delta^2 m/2}.$$
Thus, we again proceed as before using Schur's lemma, and get 
$$\left|II_{F_2}\right|\lesssim 2^m 2^{-\tau} 2^{k_1/2+\tau/2-3\delta^2 m/2} 2^{-3m/2-k_1/2+\delta_2 m/8}\lesssim 2^{-m/2-\tau/2+\delta_2 m/8}\lesssim 2^{-\delta^2 m/4}.$$
Finally, if $k,k_1,k_2\le -D_0$, then $\left|\Phi\right|\sim 1$, which contradicts $\left|\Phi\right|\sim 2^\tau\le 2^{-(1-\delta_2/2)m/2}$.
\end{proof}
\vspace{1em}
Next, we define the energy as
\begin{align*}
   \mathcal{E}(t)&\triangleq\frac{1}{2}\left(\sum_{\sigma\in\left\{1,2,\cdots,d\right\}}\left\|\left\langle\nabla\right\rangle^{N_0}v_\sigma(t)\right\|^2_{L^2}+\sum_{\sigma\in\left\{1,2,\cdots,d\right\},\atop p\in\left\{1,2,\cdots,N_1\right\}} \left\|\Omega^p v_\sigma(t)\right\|^2_{L^2}\right), \\
   &\triangleq\mathcal{E}_1(t)+\mathcal{E}_2(t)
\end{align*}
and we are ready to prove the energy estimate.

\begin{lemma}
   Suppose $\textbf{u}$ is the solution to (\ref{1.1}) on a time interval $[0,T]$ with initial data $\textbf{u}(0)=\textbf{g},\partial_t\textbf{u}(0)=\textbf{h}$. Let $v_\sigma\triangleq\left(\partial_t-i\Lambda_\sigma\right)u_\sigma$ for $\sigma\in\left\{1,2,\cdots,d\right\}$ Assume that
   $$\left\|(\textbf{g},\partial_x\textbf{g},\textbf{h})\right\|_X\le\varepsilon_0,\ \ \ \ \ \sup\limits_{0\le t\le T}\left\|\textbf{V}(t)\right\|_X\le\varepsilon_1,$$
   then we have
   $$\left\|\textbf{v}(t)\right\|_{H^{N_0}}+\sup\limits_{\beta\le N_1}\left\|\Omega^\beta \textbf{v}(t)\right\|_{L^2}\le\varepsilon_0+\varepsilon_1^{3/2}.$$
\end{lemma}
\begin{proof}
With the definition of the energy $\mathcal{E}$ above, it suffices to show that $\displaystyle{\int_0^t \left(\partial_s \mathcal{E}\right)(s)\,ds\le\varepsilon_1^{3/2}}$. Now, we do the time localization, and introduce the following cutoff functions. Let $L=[\log_2t]$, and we pick up a bunch of functions: $q_0,q_1,\cdots,q_L,q_{L+1}:\mathbb{R}\longrightarrow[0,1]$ such that (1) supp\,$q_0\subseteq[0,2]$, supp\,$q_{L+1}\subseteq[t-2,t]$ and supp\,$q_m\subseteq[2^{m-1},2^{m+1}]$ for $m\in\left\{1,2,\cdots,L\right\}$; (2) $\sum_{m=0}^{L+1} q_m(s)=\textbf{1}_{[0.t]}(s)$; (3) $q_m (s)\in C^1(\mathbb{R})$ and $\displaystyle{\int_0^t \left|q_m^\prime(s)\right|\,ds\lesssim 1}$ for $m\in\left\{1,2,\cdots,L\right\}$. Thus, it suffices to show that
$$\left|\int_0^t q_m(s) \left(\partial_s \mathcal{E}\right) (s)\,ds\right|\lesssim \varepsilon_1^3 2^{-\delta^2 m}.$$\par
Let's deal with $\mathcal{E}_2$ first. This case is a little bit complicated. First, note that
\begin{align*}
    \partial_t \mathcal{E}_2(t)&=\sum_{\sigma\in\left\{1,2,\cdots,d\right\},\atop p\in\left\{1,2,\cdots,N_1\right\}}\mbox{Re}\int \partial_t \Omega^p v_\sigma\cdot \overline{\Omega^p v_\sigma}\,dx  \\
    &=\sum_{\sigma\in\left\{1,2,\cdots,d\right\},\atop p\in\left\{1,2,\cdots,N_1\right\}}\mbox{Re}\int \Omega^p \sum_{\alpha,\beta,\gamma=1}^d A_{\alpha \beta \gamma} \left(\Lambda_\beta^{-1}\frac{v_\beta-\Bar{v_\beta}}{2i}\right)\left(\Lambda_\gamma^{-1}\frac{v_\gamma-\Bar{v_\gamma}}{2i}\right)\cdot \overline{\Omega^p v_\sigma}\,dx  \\
    &=\sum_{\sigma\in\left\{1,2,\cdots,d\right\},\atop p\in\left\{1,2,\cdots,N_1\right\}}\sum_{\alpha,\beta,\gamma=1}^d \sum_{p_1+p_2=p} A_{\alpha \beta \gamma}\,\mbox{Re}\int  \Omega^{p_2} \left(\Lambda_\beta^{-1}\frac{v_\beta-\Bar{v_\beta}}{2i}\right)\cdot \Omega^{p_1}\left(\Lambda_\gamma^{-1}\frac{v_\gamma-\Bar{v_\gamma}}{2i}\right)\cdot \overline{\Omega^p v_\sigma}\,dx,
\end{align*}
where the last equality is due to the fact that the operator $e^{it\Lambda}$ is preserved under $L^2$ norm.
Thus, it suffices to consider the following forms:
\begin{align*}
   I&\triangleq\mbox{Re}\int  \Omega^{p_2} \left(\Lambda_\beta^{-1} v_\beta\right)\cdot \Omega^{p_1}\left(\Lambda_\gamma^{-1} v_\gamma\right)\cdot \overline{\Omega^p v_\sigma}\,dx \\
   &=\mbox{Re}\iint \reallywidehat{\Lambda_\beta^{-1} \Omega^{p_2} v_\beta}(\eta,s)\, \reallywidehat{\Lambda_\gamma^{-1}\, \Omega^{p_1}v_\gamma}(\xi-\eta,s)\,d\eta\ \widecheck{\overline{\Omega^p v_\sigma}}(\xi,s)\,d\xi \\
   &=\mbox{Re}\int_{\mathbb{R}^2\times\mathbb{R}^2} \widehat{\Omega^{p_2} u_\beta}(\eta,s)\,\widehat{\Omega^{p_1} u_\gamma}(\xi-\eta,s)\,\overline{\widehat{\Omega^p v_\sigma}}(-\xi,s)\,d\xi d\eta
\end{align*}
where $p_1+p_2=p$ and $1\le\beta,\gamma,\sigma\le d$.
Now, we do the Littlewood-Paley projection and consider
$$\left|I_{k,k_1,k_2}\right|\le\int_{\mathbb{R}^2\times\mathbb{R}^2} \left|\reallywidehat{P_{k_2}\Omega^{p_2} u_\beta}(\eta,s)\right|\cdot\left|\reallywidehat{P_{k_1}\Omega^{p_1} u_\gamma}(\xi-\eta,s)\right|\cdot\left|\overline{\reallywidehat{P_k\Omega^p v_\sigma}}(-\xi)\right|\,d\xi d\eta.$$
Then, we have
\begin{align*}
    I&=\sum_{k,k_1,k_2\atop\text{s.t. }2^k,2^{k_1},2^{k_2}\text{ are three sides of a triangle}} I_{k,k_1,k_2} \\
    &=\sum_{\max(k,k_1,k_2)\ge 23\delta m} I_{k,k_1,k_2}+\sum_{\min(k,k_1,k_2)\le -2m} I_{k,k_1,k_2}+
    \sum_{-2m\le k,k_1,k_2\le 23\delta m} I_{k,k_1,k_2} \\
    &\triangleq I^{(1)}+I^{(2)}+I^{(3)}. \tag{4.5}\label{4.5}
\end{align*}
Define
\begin{align*}
   \chi_k^1&\triangleq\left\{(k_1,k_2):\left|\max(k_1,k_2)-k\right|\le 4\right\} \\ \chi_k^2&\triangleq\left\{(k_1,k_2):\max(k_1,k_2)\ge k+4,\left|k_1-k_2\right|\le 4\right\},
\end{align*}
then we consider $I^{(1)}$ first
\begin{align*}
    I^{(1)}=&\sum_{\chi_k^1\atop k\ge 23\delta m-4} I_{k,k_1,k_2}+\sum_{\chi_k^2\atop k\ge 23\delta m-4} I_{k,k_1,k_2}+\sum_{\chi_k^2\atop k\le 23\delta m-4} I_{k,k_1,k_2} \\
    \approx&\left(\sum_{k\ge 23\delta m-4} \sum_{0\le k_2 \le k} I_{k,k_1,k_2}+\sum_{k\ge 23\delta m-4} \sum_{-100m\le k_2\le 0} I_{k,k_1,k_2}+\sum_{k\ge 23\delta m-4} \sum_{k_2\le -100m} I_{k,k_1,k_2} \right.\\
    &\left.+\sum_{k\ge 23\delta m-4} \sum_{0\le k_1 \le k} I_{k,k_1,k_2}+\sum_{k\ge 23\delta m-4} \sum_{-100m\le k_1\le 0} I_{k,k_1,k_2}+\sum_{k\ge 23\delta m-4} \sum_{k_1\le -100m} I_{k,k_1,k_2}\right) \\
    &+\sum_{k\ge 23\delta m-4} \sum_{k_1\ge k+4} I_{k,k_1,k_2} \\
    &+\left(\sum_{k\le -100m} \sum_{k_1\ge 23\delta m} I_{k,k_1,k_2}+\sum_{-100m\le k\le 23\delta m-4} \sum_{k_1\ge 23\delta m} I_{k,k_1,k_2}\right) \\
    \triangleq& I^{(1,1)}+I^{(1,2)}+I^{(1,3)}+I^{(1,4)}+I^{(1,5)}+I^{(1,6)}+I^{(1,7)}+I^{(1,8)}+I^{(1,9)}. \tag{4.6}\label{4.6}
\end{align*}
Note that we have $k_1\approx k$ in the terms $I^{(1,1)},I^{(1,2)},I^{(1,3)}$, $k_2\approx k$ in the terms $I^{(1,4)},I^{(1,5)},I^{(1,6)}$ and $k_1\approx k_2$ in the terms $I^{(1,7)},I^{(1,8)},I^{(1,9)}$.
Then we use Lemma 3.2 in \cite{y2} and Proposition 3.8 to estimate the above terms to get that 
\begin{flalign*}
     &I^{(1,1)}\approx& \\
     &\begin{cases}\sum\limits_{k\ge 23\delta m-4} \sum\limits_{0\le k_2 \le k} \left\| P_{k_2}\Omega^{p_2} u_\beta(\eta,s)\right\|_{L^2}\cdot\left\|P_{k_1}\Omega^{p_1} u_\gamma(\xi-\eta,s)\right\|_{L^\infty}\cdot\left\|\overline{P_k\Omega^p v_\sigma}(-\xi)\right\|_{L^2}, &\mbox{if } p_1\le p_2 \\
     \sum\limits_{k\ge 23\delta m-4} \sum\limits_{0\le k_2 \le k} \left\|P_{k_2}\Omega^{p_2} u_\beta(\eta,s)\right\|_{L^\infty}\cdot\left\|P_{k_1}\Omega^{p_1} u_\gamma(\xi-\eta,s)\right\|_{L^2}\cdot\left\|\overline{P_k\Omega^p v_\sigma}(-\xi)\right\|_{L^2}, &\mbox{if } p_1\ge p_2
    \end{cases}& \\[7pt]
    &\approx\begin{cases}\varepsilon_1^3\cdot\sum\limits_{k\ge 23\delta m-4} \sum\limits_{0\le k_2 \le k} 2^{-k_2}\cdot \left(2^{-k_1}2^{-m+21\delta m}\right)\cdot 1, &\mbox{if } p_1\le p_2 \\
    \varepsilon_1^3\cdot\sum\limits_{k\ge 23\delta m-4} \sum\limits_{0\le k_2 \le k} \left(2^{-k_2}2^{-m+21\delta m}\right)\cdot  2^{-k_1}\cdot 1, &\mbox{if } p_1\ge p_2
    \end{cases}\ \ \lesssim \varepsilon_1^3\cdot 2^{-m-2\delta m};
\end{flalign*}
\begin{flalign*} 
    &I^{(1,2)}\approx& \\
    &\begin{cases}\sum\limits_{k\ge 23\delta m-4} \sum\limits_{-100m\le k_2 \le 0} \left\| P_{k_2}\Omega^{p_2} u_\beta(\eta,s)\right\|_{L^2}\cdot\left\|P_{k_1}\Omega^{p_1} u_\gamma(\xi-\eta,s)\right\|_{L^\infty}\cdot\left\|\overline{P_k\Omega^p v_\sigma}(-\xi)\right\|_{L^2}, &\mbox{if } p_1\le p_2 \\
    \sum\limits_{k\ge 23\delta m-4} \sum\limits_{-100m\le k_2 \le 0} \left\|P_{k_2}\Omega^{p_2} u_\beta(\eta,s)\right\|_{L^\infty}\cdot\left\|P_{k_1}\Omega^{p_1} u_\gamma(\xi-\eta,s)\right\|_{L^2}\cdot\left\|\overline{P_k\Omega^p v_\sigma}(-\xi)\right\|_{L^2}, &\mbox{if } p_1\ge p_2
    \end{cases}&\\[7pt]
    &\approx\begin{cases}\varepsilon_1^3\cdot\sum\limits_{k\ge 23\delta m-4} \sum\limits_{-100m\le k_2 \le 0} 1 \cdot \left(2^{-k_1}2^{-m+21\delta m}\right)\cdot 1, &\mbox{if } p_1\le p_2 \\
    \varepsilon_1^3\cdot\sum\limits_{k\ge 23\delta m-4} \sum\limits_{-100m\le k_2 \le 0} \left(2^{-m+21\delta m}\right)\cdot  2^{-k_1}\cdot 1, &\mbox{if } p_1\ge p_2
    \end{cases}\ \ \ \lesssim \varepsilon_1^3\cdot 2^{-m-1.5\delta m};
\end{flalign*}
\begin{flalign*}
    &I^{(1,7)}\approx& \\
    &\begin{cases}\sum\limits_{k\ge 23\delta m-4} \sum\limits_{k_1\ge k+4} \left\| P_{k_2}\Omega^{p_2} u_\beta(\eta,s)\right\|_{L^2}\cdot\left\|P_{k_1}\Omega^{p_1} u_\gamma(\xi-\eta,s)\right\|_{L^\infty}\cdot\left\|\overline{P_k\Omega^p v_\sigma}(-\xi)\right\|_{L^2}, &\mbox{if } p_1\le p_2 \\
    \sum\limits_{k\ge 23\delta m-4} \sum\limits_{k_1\ge k+4} \left\|P_{k_2}\Omega^{p_2} u_\beta(\eta,s)\right\|_{L^\infty}\cdot\left\|P_{k_1}\Omega^{p_1} u_\gamma(\xi-\eta,s)\right\|_{L^2}\cdot\left\|\overline{P_k\Omega^p v_\sigma}(-\xi)\right\|_{L^2}, &\mbox{if } p_1\ge p_2
    \end{cases}&\\[7pt]
    &\approx\begin{cases}\varepsilon_1^3\cdot\sum\limits_{k\ge 23\delta m-4} \sum\limits_{k_1\ge k+4} 1 \cdot \left(2^{-k_1}2^{-m+21\delta m}\right)\cdot 1, &\mbox{if } p_1\le p_2 \\
    \varepsilon_1^3\cdot\sum\limits_{k\ge 23\delta m-4} \sum\limits_{k_1\ge k+4} \left(2^{-k_2}2^{-m+21\delta m}\right)\cdot  1\cdot 1, &\mbox{if } p_1\ge p_2
    \end{cases}\ \ \ \lesssim \varepsilon_1^3\cdot2^{-m-2\delta m};
\end{flalign*}
\begin{flalign*}
    &I^{(1,9)}\approx& \\
    &\begin{cases}\sum\limits_{-100m\le k\le 23\delta m-4} \sum\limits_{k_1\ge 23\delta m} \left\| P_{k_2}\Omega^{p_2} u_\beta(\eta,s)\right\|_{L^2}\cdot\left\|P_{k_1}\Omega^{p_1} u_\gamma(\xi-\eta,s)\right\|_{L^\infty}\cdot\left\|\overline{P_k\Omega^p v_\sigma}(-\xi)\right\|_{L^2}, \\
    \hfill\mbox{if } p_1\le p_2 \\[7pt]
    \sum\limits_{-100m\le k\le 23\delta m-4} \sum\limits_{k_1\ge 23\delta m} \left\|P_{k_2}\Omega^{p_2} u_\beta(\eta,s)\right\|_{L^\infty}\cdot\left\|P_{k_1}\Omega^{p_1} u_\gamma(\xi-\eta,s)\right\|_{L^2}\cdot\left\|\overline{P_k\Omega^p v_\sigma}(-\xi)\right\|_{L^2}, \\
    \hfill\mbox{if } p_1\ge p_2
    \end{cases}&\\[7pt]
    &\approx\begin{cases}\varepsilon_1^3\cdot\sum\limits_{-100m\le k\le 23\delta m-4} \sum\limits_{k_1\ge 23\delta m} 2^{-k_2} \cdot \left(2^{-k_1}2^{-m+21\delta m}\right)\cdot 1, &\mbox{if } p_1\le p_2 \\
    \varepsilon_1^3\cdot\sum\limits_{-100m\le k\le 23\delta m-4} \sum\limits_{k_1\ge 23\delta m} \left(2^{-k_2}2^{-m+21\delta m}\right)\cdot  2^{-k_1}\cdot 1, &\mbox{if } p_1\ge p_2
    \end{cases}\ \ \ \lesssim \varepsilon_1^3\cdot2^{-m-24.5\delta m}.
\end{flalign*}
By symmetry, we also know that $I^{(1.4)}\approx I^{(1.1)}\lesssim \varepsilon_1^3\cdot2^{-m-2\delta m}$, and $I^{(1.5)}\approx I^{(1.2)}\lesssim \varepsilon_1^3\cdot 2^{-m-1.5\delta m}$. Finally, we use volume estimation, Hölder and Hausdorff-Young inequalities to control the remaining terms. 
\begin{align*}
    I^{(1,3)}&\approx\sum_{k\ge 23\delta m-4} \sum_{k_2\le -100m}\left\|\reallywidehat{\Lambda_\beta^{-1}\Omega^{p_2}v_\beta}(\eta,s)\cdot\reallywidehat{\Lambda_\gamma^{-1}\Omega^{p_1}v_\gamma}(\xi-\eta,s)\cdot\overline{\reallywidehat{\Omega^{p}v_\sigma}}(-\xi,s)\right\|_{L^1_{\xi,\eta}\left(\mathbb{R}^2\times\mathbb{R}^2\right)} \\
    &\lesssim \sum_{k\ge 23\delta m-4} \sum_{k_2\le -100m} \left\|\Omega^p v_\sigma\right\|_{L^2}\cdot\left\|\Lambda_\gamma^{-1}\Omega^{p_1}v_\gamma\right\|_{L^2}\left\|\reallywidehat{\Lambda_\beta^{-1}\Omega^{p_2}v_\beta}(\eta)\right\|_{L^\infty}\cdot 2^{2k_2} \\
    &\lesssim\varepsilon_1^2\cdot \sum_{k\ge 23\delta m-4} \sum_{k_2\le -100m} 1\cdot 2^{-k}\cdot\left\|\Lambda_\beta^{-1}\Omega^{p_2}v_\beta(\eta)\right\|_{L^2}\cdot 2^{k_2} \cdot 2^{2k_2}\lesssim \varepsilon_1^3\cdot2^{-300m}; \\[10pt]
\end{align*}
Similarly, we could achieve $I^{(1,8)}\lesssim\varepsilon_1^3\cdot 2^{-300m}$. Again by symmetry, we also know that $I^{(1,6)}\approx I^{(1,3)}\lesssim \varepsilon_1^3\cdot 2^{-300m}$. Thus, we have already concluded that $I^{(1)}\lesssim \varepsilon_1^3\cdot2^{-m-\delta^2 m}$. Next, we consider $I^{(2)}$.
\begin{align*}
    I^{(2)}=&\sum_{\chi_k^1\atop k\le -2m} I_{k,k_1,k_2}+\sum_{\chi_k^2\atop k\le -2m} I_{k,k_1,k_2}+\sum_{\chi_k^1\atop k\ge -2m} I_{k,k_1,k_2} \\
    \approx&\left(\sum_{k\le -2m} \sum_{k_2 \le k} I_{k,k_1,k_2}+\sum_{k\le -2m} \sum_{k_1 \le k} I_{k,k_1,k_2}\right) \\
    &+\sum_{k\le -2m} \sum_{k_1\ge k+4} I_{k,k_1,k_2} \\
    &+\left(\sum_{k\ge -2m} \sum_{k_2\le -2m} I_{k,k_1,k_2}+\sum_{k\ge -2m} \sum_{k_1\le -2m} I_{k,k_1,k_2}\right) \\
    \triangleq& I^{(2,1)}+I^{(2,2)}+I^{(2,3)}+I^{(2,4)}+I^{(2,5)}. \tag{4.7}\label{4.7}
\end{align*}
Note that we have $k_1\approx k$ in the terms $I^{(2,1)},I^{(2,4)}$, $k_2\approx k$ in the terms $I^{(2,2)},I^{(2,5)}$ and $k_1\approx k_2$ in the terms $I^{(2,3)}$. In the case of $I^{(2)}$, we again use volume estimation, Hölder and Hausdorff-Young inequalities to control the remaining terms.
\begin{flalign*}
    I^{(2,1)}&\lesssim \sum_{k\le -2m} \sum_{k_2\le k} \left\|\Omega^p v_\sigma\right\|_{L^2}\cdot\left\|\Lambda_\gamma^{-1}\Omega^{p_1} v_\gamma\right\|_{L^2}\cdot\left\|\reallywidehat{\Omega^{p_2} v_\beta}\right\|_{L^\infty}\cdot 2^{2k_2}& \\
    &\lesssim \sum_{k\le -2m} \sum_{k_2\le k} \left\|\Omega^p v_\sigma\right\|_{L^2}\cdot\left\|\Lambda_\gamma^{-1}\Omega^{p_1} v_\gamma\right\|_{L^2}\cdot\left\|\Omega^{p_2} v_\beta\right\|_{L^2}\cdot2^{k_2}\cdot2^{2k_2}\lesssim \varepsilon_1^3\cdot2^{-6m};&
\end{flalign*}
\begin{flalign*}
    I^{(2,3)}&\lesssim \sum_{k\le -2m} \sum_{k_1 \ge k+4} \left\|\Lambda_\beta^{-1}\Omega^{p_2} v_\beta\right\|_{L^2}\cdot\left\|\Lambda_\gamma^{-1}\Omega^{p_1}v_\gamma\right\|_{L^2}\cdot\left\|\reallywidehat{\Omega^p v_\sigma}\right\|_{L^\infty}\cdot 2^{2k}& \\
    &\lesssim \sum_{k\le -2m} \sum_{k_1 \ge k+4} \left\|\Lambda_\beta^{-1}\Omega^{p_2} v_\beta\right\|_{L^2}\cdot\left\|\Lambda_\gamma^{-1}\Omega^{p_1}v_\gamma\right\|_{L^2}\cdot\left\|\reallywidehat{\Omega^p v_\sigma}\right\|_{L^2}\cdot 2^k \cdot 2^{2k}\lesssim \varepsilon_1^3\cdot2^{-2m};& \\
\end{flalign*}
\begin{flalign*}
    I^{(2,4)}&\lesssim \sum_{k\ge -2m} \sum_{k_2\le -2m} \left\|\Omega^p v_\sigma\right\|_{L^2}\cdot\left\|\Lambda_\gamma^{-1}\Omega^{p_1} v_\gamma\right\|_{L^2}\cdot\left\|\reallywidehat{\Omega^{p_2} v_\beta}\right\|_{L^\infty}\cdot 2^{2k_2} \\
    &\lesssim \sum_{k\ge -2m} \sum_{k_2\le -2m} \left\|\Omega^p v_\sigma\right\|_{L^2}\cdot\left\|\Lambda_\gamma^{-1}\Omega^{p_1} v_\gamma\right\|_{L^2}\cdot\left\|\Omega^{p_2} v_\beta\right\|_{L^2}\cdot2^{k_2}\cdot2^{2k_2}\lesssim \varepsilon_1^3\cdot2^{-4m}.&
\end{flalign*}
Finally, by symmetry, we get that $I^{(2,2)}\approx I^{(2,1)}\lesssim \varepsilon_1^3\cdot 2^{-6m}$, and $I^{(2,5)}\approx I^{(2,4)}\lesssim \varepsilon_1^3\cdot 2^{-4m}$.  Thus, we conclude that $I^{(2)}\lesssim \varepsilon_1^3\cdot2^{-m-\delta^2 m}$. As for $I^{(3)}$, it is also bounded by $\varepsilon_1^3\cdot2^{-m-\delta^2 m}$ due to Lemma 4.1. Thus, we have proved that
\begin{align}
   \left|\int_0^t q_m(s) \left(\partial_s \mathcal{E}_2\right) (s)\,ds\right|\lesssim \varepsilon_1^3 2^{-\delta^2 m}. \tag{4.8}\label{4.8}
\end{align} \par
Next, let's consider the $\mathcal{E}_1$ part. This case is quite easy. Similar as before, we have that
\begin{align*}
    \partial_t \mathcal{E}_1 (t)=\sum_{\sigma\in\left\{1,2,\cdots,d\right\},\atop p_1+p_2=N_0}\sum_{\alpha,\beta,\gamma=1}^d A_{\alpha \beta \gamma}\,\mbox{Re}\int  \left\langle\nabla\right\rangle^{p_2} \left(\Lambda_\beta^{-1}\frac{v_\beta-\Bar{v_\beta}}{2i}\right)\cdot \left\langle\nabla\right\rangle^{p_1}\left(\Lambda_\gamma^{-1}\frac{v_\gamma-\Bar{v_\gamma}}{2i}\right)\cdot \overline{\left\langle\nabla\right\rangle^{N_0} v_\sigma}\,dx,
\end{align*}
Thus, like before, it suffices to consider the following forms:
$$I\triangleq\mbox{Re}\int_{\mathbb{R}^2\times\mathbb{R}^2} \reallywidehat{\left\langle\nabla\right\rangle^{p_2} u_\beta}(\eta,s)\,\reallywidehat{\left\langle\nabla\right\rangle^{p_1} u_\gamma}(\xi-\eta,s)\,\overline{\reallywidehat{\left\langle\nabla\right\rangle^{N_0} v_\sigma}}(-\xi,s)\,d\xi d\eta,$$
where $p_1+p_2=p$ and $1\le\beta,\gamma,\sigma\le d$. Then, we again do the Littlewood-Paley projection and consider
$$\left|I_{k,k_1,k_2}\right|\le\int_{\mathbb{R}^2\times\mathbb{R}^2} \left|\reallywidehat{P_{k_2}\left\langle\nabla\right\rangle^{p_2} u_\beta}(\eta,s)\right|\cdot\left|\reallywidehat{P_{k_1}\left\langle\nabla\right\rangle^{p_1} u_\gamma}(\xi-\eta,s)\right|\cdot\left|\overline{\reallywidehat{P_k\left\langle\nabla\right\rangle^{N_0} v_\sigma}}(-\xi)\right|\,d\xi d\eta.$$
Then, we decompose $I$ as in (\ref{4.5}), (\ref{4.6}), (\ref{4.7}) before. First, we note that 
\begin{align*}
    \left|I_{k,k_1,k_2}\right|&\le \left\|v_\sigma\right\|_{H^{N_0}}\cdot\left\|\reallywidehat{\left\langle\nabla\right\rangle^{p_2}u_\beta\cdot\left\langle\nabla\right\rangle^{p_1}u_\gamma}\right\|_{L^2} \\
    &\lesssim \left\|v_\sigma\right\|_{H^{N_0}}\cdot\left(\left\|u_\beta\right\|_{H^{N_0}}\cdot\left\|u_\gamma\right\|_{L^\infty}+\left\|u_\beta\right\|_{L^\infty}\cdot\left\|u_\gamma\right\|_{H^{N_0}}\right).  \tag{4.9}\label{4.9}
\end{align*}
Now, we can use (\ref{3.4}) and (\ref{4.9}) to control the following terms
\begin{align*}
    I^{(1,1)}&\lesssim \sum_{k\ge 23\delta m-4} \sum_{0\le k_2 \le k} 2^{-k}\cdot\left(2^{-k_2}+2^{-k_1}\right)\cdot 2^{-m+21\delta m}\lesssim 2^{-m-2\delta m}; \\[5pt]
    I^{(1,2)}&\lesssim \sum_{k\ge 23\delta m-4} \sum_{-100m\le k_2 \le 0} 2^{-k}\cdot\left(1+2^{-k}\right)\cdot 2^{-m+21\delta m}\lesssim 2^{-m-1.5\delta m}; \\[5pt]
    I^{(1,7)}&\lesssim \sum_{k\ge 23\delta m-4} \sum_{k_1\ge k+4} 2^{-k}\cdot\left(2^{-k_2}+2^{-k_1}\right)\cdot 2^{-m+21\delta m}\lesssim 2^{-m-48\delta m}; \\[5pt]
    I^{(1,9)}&\lesssim \sum_{-100m\le k\le 23\delta m} \sum_{k_1\ge 23\delta m} 1\cdot \left(2^{-k_1}+2^{-k_2}\right)\cdot 2^{-m+21\delta m}\lesssim 2^{-m-24.5\delta m}.
\end{align*}
The remaining terms are controlled by volumes or symmetry, so they can be estimated exactly the same as before. Thus, we have proved that
\begin{align}
   \left|\int_0^t q_m(s) \left(\partial_s \mathcal{E}_1\right) (s)\,ds\right|\lesssim \varepsilon_1^3 2^{-\delta^2 m}. \tag{4.10}\label{4.10}
\end{align} \par
Combine (\ref{4.8}), and (\ref{4.10}) , and our proof is completed.
\end{proof}

\vspace{1em}
\section{Dispersive Control - Control of the Z Norm}
In this section, we show how to control the Z component of the norms. Using Duhamel formulas, Proposition 2.3 follows easily from Proposition 5.1 below.
\begin{prop}
    Assume that $t\in\left[0,T\right]$ is fixed
    $$\sup_{0\le s\le t}\left\{\left\|f^\mu(s)\right\|_{H^{N_0/2}\cap Z^\mu_1\cap H^{N_1/2}_\Omega}+\left\|f^\nu(s)\right\|_{H^{N_0/2}\cap Z^\nu_1\cap H^{N_1/2}_\Omega}\right\}\le 1$$
    and that $\partial_s f^\mu,\partial_s f^\nu$ satisfy the conclusions of Lemma 3.9 and Lemma 3.12. For $\sigma,\mu,\nu\in\left\{1,2,\dots,d\right\}$ and $m\in\left\{0,\dots,L+1\right\}$ define
    $$\mathcal{F}\left\{T^{\sigma\mu\nu}_m \left[f,g\right]\right\}(\xi)\triangleq \int_{\mathbb{R}} q_m(s) \int_{\mathbb{R}^2} e^{is\Phi_{\sigma\mu\nu}}(\xi,\eta)\widehat{f}(\xi-\eta,s)\widehat{g}(\eta,s)\,d\eta ds.$$
    Then
    $$\sum_{k_1,k_2\in\mathbb{Z}} \left\|P_k T^{\sigma\mu\nu}_m \left[P_{k_1} f^\mu,P_{k_2} f^\nu\right]\right\|_{Z^\sigma_1}\lesssim 2^{-\delta^5 m}.$$
\end{prop}
\vspace{4em}
The rest of this section is concerned with the proof of Proposition 5.1. We consider first a few simple cases before moving to the main analysis in the next subsections. For simplicity of notation, we ofter omit the subscripts $\sigma\mu\nu$ and write $\Phi_{\sigma\mu\nu}=\Phi$ and $\Psi_{\sigma\mu\nu}=\Psi$. Let $I_m$ denote the support of the function $q_m$. Moreover, for simplicity, we denote
$$\left\|g\right\|_{B^\sigma_j}\triangleq2^{(1-20\delta)j} \sup_{0\le n\le j+1} 2^{-(1/2-19\delta)n}\left\|A^\sigma_{n,(j)} g\right\|_{L^2}.$$
Thus, the norm $Z^\sigma_1$ can be expressed as 
$$Z^\sigma_1\triangleq\left\{f\in L^2(\mathbb{R}^2):\left\|f\right\|_{Z^\sigma_1}\triangleq\sup_{(k,j)\in\mathcal{J}} 2^{6k_+}\left\|Q_{jk} f\right\|_{B^\sigma_j}<\infty\right\}.$$
\vspace{4em}
\begin{lemma}
    Assume that $f^\mu$, $f^\nu$ are as in Proposition 5.1 and let $(k,j)\in\mathcal{J}$. Then
    \begin{align*}
    &2^{6k_+} \sum_{\max\left\{k_1,k_2\right\}\ge 0.05\delta^2 (j+m)-D^2} \left\|Q_{jk}\, T^{\sigma\mu\nu}_m \left[P_{k_1}f^\mu,P_{k_2}f^\nu\right]\right\|_{B^\sigma_j}\lesssim 2^{-\delta^4 m}, \tag{5.1}\label{5.1} \\
    &2^{6k_+} \sum_{\min\left\{k_1,k_2\right\}\le-(j+m)(1+11\delta)/2+D^2} \left\|Q_{jk}\,T^{\sigma\mu\nu}_m \left[P_{k_1}f^\mu,P_{k_2}f^\nu\right]\right\|_{B^\sigma_j}\lesssim 2^{-\delta^4 m}, \tag{5.2}\label{5.2} \\
    &\mbox{if }j+k\le 19\delta j-17\delta m\mbox{ then }\sum_{k_1,k_2\in\mathbb{Z}} \left\|Q_{jk}\,T^{\sigma\mu\nu}_m \left[P_{k_1}f^\mu,P_{k_2}f^\nu\right]\right\|_{B^\sigma_j}\lesssim 2^{-\delta^4 m}, \tag{5.3}\label{5.3} \\
    &\mbox{if }j\ge 3m \mbox{ then }2^{6k_+}\sum_{-j\le k_1,k_2\le 2\delta^2 j} \left\|Q_{jk}\, T^{\sigma\mu\nu}_m \left[P_{k_1}f^\mu,P_{k_2}f^\nu\right]\right\|_{B^\sigma_j}\lesssim 2^{-\delta^4 m}. \tag{5.4}\label{5.4}
    \end{align*}
\end{lemma}
\begin{proof}
    These estimates are basically proved in Lemma 7.2 in \cite{y2}, where (\ref{5.1}), (\ref{5.2}) and (\ref{5.4}) follow from \cite{y2} exactly.\par
    In terms of (\ref{5.3}), we may assume that
    $$j+k\le 19\delta j-17\delta m,\ \ \ \ -2(j+m)/3\le k_1,k_2 \le \delta^2 (j+m)-D^2.$$
    With $l\triangleq-12\delta m-D$ we decompose
    \begin{align*}
    T^{\sigma\mu\nu}_m\left[P_{k_1}f^\mu,P_{k_2} f^\nu\right]&= T^{hi}_m\left[P_{k_1}f^\mu,P_{k_2} f^\nu\right]+T^{lo}_m\left[P_{k_1}f^\mu,P_{k_2} f^\nu\right], \\
    \widehat{T^*\left[f,g\right]}(\xi)&\triangleq\int_{\mathbb{R}} q_m(s) \int_{\mathbb{R}^2} e^{is\Phi(\xi,\eta)}\varphi_*\left(\Phi(\xi,\eta)\right)\widehat{f}(\xi-\eta,s)\widehat{g}(\eta,s)\,d\eta ds, \\
    \varphi_{lo}(x)&\triangleq\varphi_{\le l}(x),\ \ \ \varphi_{hi}(x)\triangleq 1-\varphi_{lo}(x),\ \ \ *\in\left\{hi,lo\right\}.
    \end{align*}
The control of $T^{hi}$ is done in Lemma 7.2 in \cite{y2} by integrating by parts in time. To bound the contribution of $T^{lo}$, use $L^2\times L^\infty$ control and it suffices to show that
\begin{align}
    \sum_{(k_1,j_1),(k_2,j_2)\in\mathcal{J}} \left\|\mathcal{F}P_k\,T^{lo}\left[f^\mu_{j_1,k_1},f^\nu_{j_2,k_2}\right]\right\|_{L^\infty}\lesssim 2^{16.5\delta m}, \tag{5.5}\label{5.5}
\end{align}
where
$f^\mu_{j_1,k_1}=P_{\left[k_1-2,k_1+2\right]}Q_{j_1 k_1}f^\mu$ and $f^\nu_{j_2,k_2}=P_{\left[k_2-2,k_2+2\right]}Q_{j_2 k_2}f^\nu$ as before.\par
If $\max\left\{j_1,j_2\right\}\le (1-\delta^2)m$, then integration by parts in $\eta$, using Proposition 6.6 (a) and Lemma 3.1, gives an acceptable contribution. On the other hand, if $j_2=\max\left\{j_1,j_2\right\}\ge (1-\delta^2)m$, then we have
$$\left\|\widehat{f^\mu_{j_1,k_1}}(s)\right\|_{L^\infty}\lesssim 2^{2\delta n_1},\ \ \ \left\|\widehat{f^\nu_{j_2,k_2}}(s)\right\|_{L^2}\lesssim 2^{-j_2+20\delta j_2}\cdot 2^{\frac{1}{2}n_2-19\delta n_2}$$
as a consequence of Lemma 3.6 (Also, since $k\le -D$, we know that $k_1,k_2\sim 1$ by Proposition 6.6 (b)) Note that $n_1\le j_1 \le j_2$ and $\left|E_\eta\right|\lesssim 2^l 2^{-n_2}$ due to that $\left|\nabla_\eta\Phi(0,\eta)\right\|\gtrsim 1$. Then, we get
\begin{align*}
    \left\|\mathcal{F}P_k\,T^{lo}\left[f^\mu_{j_1,k_1},f^\nu_{j_2,k_2}\right]\right\|_{L^\infty}&\lesssim 2^m \sup_s \left\|\widehat{f^\mu_{j_1,k_1}}(s)\right\|_{L^\infty} \left\|\widehat{f^\nu_{j_2,k_2}}(s)\right\|_{L^2} (\left|E_\eta\right|)^{1/2}\lesssim 2^{16.1\delta m}
\end{align*}
The desired bound (\ref{5.5}) follows.
\end{proof}
\vspace{4em}
\subsection{The Main Decomposition}
\ \par
Now, we may assume that 
\begin{equation}
    \begin{aligned}
        &-(j+m)(1+11\delta)/2\le k_1,k_2 \le 0.05\,\delta^2 (j+m),\ \ \ j+k\ge 19\delta j-17\delta m, \\
        &j\le 3m,\ \ \ m\ge D^2/8.
    \end{aligned} \tag{5.6}\label{5.6}
\end{equation}
We fix $l_-\triangleq\left[-m+\delta^3 m\right]$ and $l_0\triangleq\left[-12\delta m\right]$, and decompose
$$T^{\sigma\mu\nu}_m\left[f,g\right]=\sum_{l_-\le l\le l_0} T_{m,l}\left[f,g\right],$$
where
$$\reallywidehat{T_{m,l}\left[f,g\right]}\triangleq \int_{\mathbb{R}} q_m(s) \int_{\mathbb{R}^2} e^{is\Phi(\xi,\eta)}\varphi_l^{\left[l_-,l_0\right]}(\Phi(\xi,\eta))\hat{f}(\xi-\eta,s)\hat{g}(\eta,s)\,d\eta ds.$$
When $l_-<l\le l_0$, we may integrate by parts in time to rewrite $T_{m,l}\left[P_{k_1}f^\mu,P_{k_2}f^\nu\right]$,
\begin{equation}
    \begin{aligned}
    &T_{m,l}\left[P_{k_1}f^\mu,P_{k_2}f^\nu\right]=i\mathcal{A}_{m,l}\left[P_{k_1}f^\mu,P_{k_2}f^\nu\right]+i\mathcal{B}_{m,l}\left[P_{k_1}\partial_s f^\mu,P_{k_2}f^\nu\right]+i\mathcal{B}_{m,l}\left[P_{k_1} f^\mu,P_{k_2}\partial_sf^\nu\right] \\
    &\mathcal{F}\mathcal{A}_{m,l}\left[P_{k_1}f,P_{k_2}g\right](\xi)\triangleq \int_{\mathbb{R}} q^\prime_m(s) \int_{\mathbb{R}^2} e^{is\Phi(\xi,\eta)} \Tilde{\varphi}_l(\Phi(\xi,\eta))\widehat{P_{k_1}f}(\xi-\eta,s)\widehat{P_{k_2}g}(\eta,s)\,d\eta ds, \\
    &\mathcal{F}\mathcal{B}_{m,l}\left[P_{k_1}f,P_{k_2}g\right](\xi)\triangleq \int_{\mathbb{R}} q_m(s) \int_{\mathbb{R}^2} e^{is\Phi(\xi,\eta)} \Tilde{\varphi}_l(\Phi(\xi,\eta))\widehat{P_{k_1}f}(\xi-\eta,s)\widehat{P_{k_2}g}(\eta,s)\,d\eta ds,
    \end{aligned}\tag{5.7}\label{5.7}
\end{equation}
where $\Tilde{\varphi}_l(x)=x^{-1}\varphi_l(x)$ for $l<l_0$ and $\Tilde{\varphi}_{l_0}(x)=x^{-1}\varphi_{\ge l_0}(x)$. It is easy to see that the main Proposition 5.1 follows from Lemma 5.2 and Lemma 5.3-5.6 below.
\vspace{4em}
\begin{lemma}
   Assume that (\ref{5.6}) holds and, in addition, $m+D\le j$. Then, for $l_-\le l\le l_0$,
   \begin{align*}
    \left\|Q_{jk}\,T_{m,l}\left[P_{k_1}f^\mu,P_{k_2}f^\nu\right]\right\|_{B^\sigma_j}\lesssim 2^{-50\delta^2 m}. \tag{5.8}\label{5.8}  
   \end{align*}
\end{lemma}
\begin{lemma}
   Assume that (\ref{5.6}) holds and, in addition, $j\le m+D$. Then
   \begin{align*}
    2^{(1-20\delta)j} \left\|Q_{jk}\,T_{m,l_0}\left[P_{k_1}f^\mu,P_{k_2}f^\nu\right]\right\|_{L^2}\lesssim 2^{-50\delta^2 m}. \tag{5.9}\label{5.9}  
   \end{align*}
\end{lemma}
\begin{lemma}
   Assume that (\ref{5.6}) holds and, in addition, $j\le m+D$. Then, for $l_-< l< l_0$,
   \begin{align*}
   \left\|Q_{jk}\,T_{m,l_-}\left[P_{k_1}f^\mu,P_{k_2}f^\nu\right]\right\|_{B^\sigma_j}+\left\|Q_{jk}\,\mathcal{A}_{m,l}\left[P_{k_1}f^\mu,P_{k_2}f^\nu\right]\right\|_{B^\sigma_j}\lesssim 2^{-50\delta^2 m}.\tag{5.10}\label{5.10}    
   \end{align*}
\end{lemma}
\begin{lemma}
   Assume that (\ref{5.6}) holds and, in addition, $j\le m+D$. Then, for $l_-< l< l_0$,
   \begin{align*}
    \left\|Q_{jk}\,\mathcal{B}_{m,l}\left[P_{k_1}f^\mu,P_{k_2}\partial_s f^\nu\right]\right\|_{B^\sigma_j}\lesssim 2^{-50\delta^2 m}.\tag{5.11}\label{5.11}   
   \end{align*} 
\end{lemma}
\vspace{0.8em}
Lemma 5.4 is proved in Section 7.3 in \cite{y2}. Thus, we will only prove Lemma 5.3, Lemma 5.5 and Lemma 5.6 in the following sections.
\vspace{4em}
\subsection{Approximate Finite Speed of Propagation}
\ \par
In this subsection, we will prove Lemma 5.3. In fact, in most cases, we are able to prove a stronger estimate
\begin{align*}
    2^{(1-20\delta)j}\left\|Q_{jk}\,T_{m,l}\left[P_{k_1}f^\mu,P_{k_2}f^\nu\right]\right\|_{L^2}\lesssim 2^{-50\delta^2 m}.\tag{5.12}\label{5.12}
\end{align*}
We define the functions $f^\mu_{j_1,k_1}$ and $f^\nu_{j_2,k_2}$ as before. If $\min\left\{j_1,j_2\right\}\le j-\delta^2 m$ then we rewrite
\begin{align*}
    &Q_{jk} T_{m,l}\left[f^\mu_{j_1,k_1},f^\nu_{j_2,k_2}\right](x) \\
    \sim &\Tilde{\varphi_j}^{(k)}(x)\cdot \int_\mathbb{R} q_m(s) \int_{\mathbb{R}^2}\left[\int_{\mathbb{R}^2} e^{i\left[s\Phi(\xi,\eta)+x\cdot \xi\right]}\varphi_l(\Phi(\xi,\eta))\varphi_k(\xi)\widehat{f^\mu_{j_1,k_1}}(\xi-\eta,s)d\xi\right]\widehat{f^\nu_{j_2,k_2}}(\eta,s)\,d\eta ds.
\end{align*}
Note that 
\begin{align*}
   \varphi_l(\Phi(\xi,\eta))=2^l\int e^{i\lambda\Phi(\xi,\eta)}P(2^l\lambda)\,d\lambda, \tag{5.13}\label{5.13}
\end{align*}
where $P$ is a Schwartz function and then we get
\begin{align*}
   Q_{jk}& T_{m,l}\left[f^\mu_{j_1,k_1},f^\nu_{j_2,k_2}\right](x)=  \\
   &\Tilde{\varphi_j}^{(k)}(x)\cdot 2^l \int_\mathbb{R} \int_\mathbb{R} q_m(s) P(2^l \lambda) \int_{\mathbb{R}^2}\left[\int_{\mathbb{R}^2} e^{i\left[(s+\lambda)\Phi(\xi,\eta)+x\cdot \xi\right]}\varphi_k(\xi)\widehat{f^\mu_{j_1,k_1}}(\xi-\eta,s)d\xi\right]\widehat{f^\nu_{j_2,k_2}}(\eta,s)\,d\eta ds d\lambda \\
   =&\Tilde{\varphi_j}^{(k)}(x)\cdot \int_\mathbb{R} \int_\mathbb{R} q_m(s) P(\Tilde{\lambda}) \int_{\mathbb{R}^2}\left[\int_{\mathbb{R}^2} e^{i\left[(s+\frac{\Tilde{\lambda}}{2^l})\Phi(\xi,\eta)+x\cdot \xi\right]}\varphi_k(\xi)\widehat{f^\mu_{j_1,k_1}}(\xi-\eta,s)d\xi\right]\widehat{f^\nu_{j_2,k_2}}(\eta,s)\,d\eta ds d\Tilde{\lambda}
\end{align*}
Since $P$ has rapid decay, we only need to consider the case when $\left|\Tilde{\lambda}\right|$ is small enough, such that $s+\frac{\Tilde{
\lambda}}{2^l}\sim 2^m$. Then, in the support of integration, we have the lower bound $\left|\nabla_\xi\left[(s+\frac{\Tilde{\lambda}}{2^l})\Phi(\xi,\eta)+x\cdot \xi\right]\right|\approx\left|x\right|\approx 2^j$. Integration by parts in $\xi$ using Lemma 3.1 gives an acceptable contribution as in (\ref{5.12}).\par
Next, we may assume that $\min\left\{j_1,j_2\right\}\ge j-\delta^2 m$. For simplicity, we may write \\$T_{m.l}=T_{m,l}\left[f^\mu_{j_1,k_1},f^\nu_{j_2,k_2}\right]$, $f_1=f^\mu_{j_1,k_1}$ and $f_2=f^\nu_{j_2,k_2}$ below. If $\min\left\{n_1,n_2\right\}\ge \frac{2}{3}j-6\delta j$, then we can simply control it by counting the volume
$$\left\|T_{m,l}\right\|_{L^\infty}\lesssim 2^m \left\|f_1\right\|_{L^\infty}\left\|f_2\right\|_{L^\infty}\left|E_\eta\right|\lesssim 2^m\cdot 2^{\delta n_1+2\delta^2 n_1}\cdot 2^{-(2-\delta-2\delta^2)n_2}.$$
Then, we have
\begin{align*}
   2^{(1-20\delta)j}\left\|T_{m,l}\right\|_{L^2}&\lesssim 2^{(1-20\delta)j} \left\|T_{m,l}\right\|_{L^\infty}\cdot\left|E_\xi\right|^{1/2}\lesssim 2^{-\frac{1}{3}\delta j}.
\end{align*}
Therefore, we may assume $\min\left\{n_1,n_2\right\}<\frac{2}{3}j-6\delta j$ below. Now, we will divide into three cases.\par
\textbf{Case 1.} $l<-\frac{2}{3}j+25\delta j$ (WLOG, let $n_2=\min\left\{n_1,n_2\right\}$)\par
In this case, if $n_1>\frac{2}{3}j-40\delta j$, then we can apply Proposition 6.10 (b) to get
\begin{align*}
   2^{(1-20\delta)j}\left\|T_{m,l}\right\|_{L^2}&\lesssim 2^{(1-20\delta)j}\,2^m\,2^{\frac{l}{4}-\frac{n_1}{2}-\frac{n_2}{4}}\,\sup_s\left[\left\|\sup_\theta \left|f_1(r\theta)\right|\right\|_{L^2(rdr)} \cdot\left\|f_2\right\|_{L^2}\right]\lesssim 2^{-0.35\delta j+2\delta^2 m}.
\end{align*}
If $n_1<\frac{2}{3}j-40\delta j$ and $\left|\nabla_\eta\Phi(\xi,\eta)\right|\gtrsim 1$, then we can simply apply Corollary 6.11 (a) to get
\begin{align*}
    2^{(1-20\delta)j}\left\|T_{m,l}\right\|_{L^2}&\lesssim 2^m \cdot 2^{(1-20\delta)j}\cdot 2^{l/2}\cdot 2^{-(1-20\delta)j_1+(\frac{1}{4}-19\delta)n_1}\cdot2^{-(1-20\delta)j_2+(\frac{1}{4}-19\delta)n_2}\lesssim 2^{-4.2\delta j+2\delta^2 m},
\end{align*}
which gives an acceptable contribution as in (\ref{5.12}).
On the other hand, if $n_1<\frac{2}{3}j-40\delta j$ and $\left|\nabla_\eta\Phi(\xi,\eta)\right|\sim 2^q$, where $-j\le q<-D<0$, then we first note that, by Proposition 6.5(a), 
$$\left|\Psi(\xi)\right|\lesssim\left|\Phi(\xi,\eta)\right|+\left|\Phi(\xi,\eta)-\Phi(\xi,p(\xi))\right|\lesssim \left|\Phi(\xi,\eta)\right|+\left|\nabla_\eta\Phi(\xi,\eta)\right|\cdot\left|\eta-p(\xi)\right|\lesssim 2^l+2^{0.4\delta^2 m}\cdot 2^{2q},$$
and consider two cases $0>l\ge 2q$ and $l\le 2q<0$. (Note that we must have $\mu+\nu\neq 0$ here, otherwise by Proposition 6.5(b) $\xi=0$, which is a contradiction.)
If $0>l\ge 2q+0.4\delta^2 m$, then $n\ge -l$ and by Corollary 6.11 (b), we have
\begin{align*}
    2^{(1-20\delta)j} 2^{(\frac{1}{2}-19\delta)l}\left\|T_{m,l}\right\|_{L^2}&\lesssim 2^m\cdot2^{(1-20\delta)j}\cdot2^{(\frac{1}{2}-19\delta)l}\cdot2^{-(1-20\delta)j_1}\cdot2^{-(1-20\delta)j_2}\cdot2^{-19\delta n_1}\cdot2^{-19\delta n_2} \\
    &\lesssim 2^{-0.25j+2\delta^2 m},
\end{align*}
which gives an acceptable contribution as in (\ref{5.8}). If $l\le 2q+0.4\delta^2 m<0$, then $n\ge -2q$ and by Corollary 6.11 (a), we have
\begin{align*}
    2^{(1-20\delta)j} 2^{(1-38\delta)q}\left\|T_{m,l}\right\|_{L^2}&\lesssim 2^m\cdot2^{(1-20\delta)j}\cdot2^{(1-38\delta)q}\cdot2^{\frac{l}{2}}\cdot2^{-\frac{q}{4}}\cdot2^{-(1-20\delta)j_1}\cdot2^{-(1-20\delta)j_2} \\ 
    &\ \ \ \ \ \times2^{(\frac{1}{4}-19\delta) n_1}\cdot2^{(\frac{1}{4}-19\delta) n_2} \\
    &\lesssim 2^{-4.2\delta j+2.2\delta^2 m},
\end{align*}
which again gives an acceptable contribution as in (\ref{5.8}). Finally, if $n_1<\frac{2}{3}j-40\delta j$ and $\left|\nabla_\eta\Phi(\xi,\eta)\right|\lesssim 2^{-j}$, then $n\ge -l$. Apply Corollary 6.11 (a) and we can also get an acceptable contribution as in (\ref{5.8}). This is because comparing the coefficients of Corollary 6.11 (a), we can get $2^{-\frac{j}{2}+\frac{n}{4}+\frac{p}{4}}\lesssim 2^{\frac{l}{2}+\frac{n}{4}+\frac{p}{4}-\frac{q}{4}}$. \par
\textbf{Case 2.} $l\ge -\frac{2}{3}j+25\delta j$, $k_2\le -\frac{j}{3}+12\delta j$ \ \ \ \ (WLOG, let $n_1=\min\left\{n_1,n_2\right\}$)\par
First, we show that 
\begin{align*}
    \sup_{\left|\lambda\right|\le 2^{-l+\delta j}} \left\|e^{-i(s+\lambda)\Lambda_\nu}\left( P_{k_2} \partial_s f^\nu\right) (s)\right\|_{L^\infty}\lesssim 2^{-2j+40.1\delta j}. \tag{5.14}\label{5.14}
\end{align*}
In fact, we first note that $\left\|e^{-is\Lambda_\nu}\left( P_{k_2} \partial_s f^\nu\right) (s)\right\|_{L^\infty}\lesssim 2^{-2j+40.1\delta j}$ due to the second line (\ref{3.6}). Then it suffices to show that $e^{-i\lambda \Lambda_\nu}:L^\infty \rightarrow L^\infty$ is bounded, where $\Lambda_\nu(\xi)=\sqrt{1+\left|\xi\right|^2}\sim 1+\frac{1}{2}\left|\xi\right|^2$ (i.e. $\left|\xi\right|\ll 1$). Note that $\reallywidehat{e^{-i\lambda\Lambda_\sigma}f}(\xi)=e^{-i\lambda\sqrt{1+\left|\xi\right|^2}}\hat{f}(\xi)=e^{i\lambda}\cdot e^{i\frac{1}{2}\lambda\left|\xi\right|^2}\hat{f}(\xi)$, and \vspace{0.5em} it suffices to show that $T$ is bounded,
where $\widehat{Tf}(\xi)\triangleq \varphi\left(\frac{\xi}{2^{-\frac{1}{3}j}}\right)\,e^{i\frac{1}{2}\lambda\left|\xi\right|^2}\hat{f}(\xi)$ and $\varphi$ is a Schwartz cutoff function with \vspace{0.5em} compact $\approx$ the unit ball\vspace{0.5em}. Let $G\left(\Tilde{\xi}\right)\triangleq \varphi\left(\Tilde{\xi}\right)\,e^{i\frac{1}{2}\lambda\,2^{-\frac{2}{3}j}\left|\Tilde{\xi}\right|^2}$ and $H(\xi)\triangleq G\left(\frac{\xi}{2^{-\frac{1}{3}j}}\right)$. Since $\left|\lambda\,2^{-\frac{2}{3}j}\right|\lesssim 1$, we know that $G$ is also a Schwartz function, which implies that $\int\left|G\right|\lesssim 1$. Then, we get
$$\left\|\hat{H}\right\|_{L^1}=\int \left|\hat{H}(x)\right|\,dx=\int 2^{-\frac{2}{3}j}\left|\hat{G}\left(2^{-\frac{1}{3}j}x\right)\right|\,dx=\int\left|\hat{G}(x)\right|\,dx\lesssim 1.$$
Finally, by Young's Inequality, we conclude that $\left\|T\right\|_{L^\infty\rightarrow L^\infty}\lesssim 1$, since $H$ is the Fourier multiplier of $T$.\par
Next, we use the formula (\ref{5.7}) and the contribution of $\mathcal{A}_{m,l}$ can be estimated by using Proposition 6.10 (b) as
\begin{align*}
   2^{(1-20\delta)j}\left\|T_{m,l}\right\|_{L^2}&\lesssim 2^{(1-20\delta)j}\,2^{-l}\,2^{\frac{l}{4}-\frac{n_1}{2}-\frac{n_2}{4}}\,\sup_s\left[\left\|\sup_\theta \left|f_1(r\theta)\right|\right\|_{L^2(rdr)} \cdot\left\|f_2\right\|_{L^2}\right]\lesssim 2^{-\frac{1}{12}j+20.2\delta j},
\end{align*}
which implies that we only need to focus on the term $\mathcal{B}_{m,l}$. To deal with $\mathcal{B}_{m,l}$, we will mainly use the following estimate that is similar to Lemma 3.5
\begin{equation}
    \begin{aligned}
    \left\|\mathcal{B}_{m,l}\left[f^\mu_{j_1,k_1,n_1},P_{k_2} \partial_s f^\nu\right]\right\|_{L^2}&\lesssim \left\|f^\mu_{j_1,k_1,n_1}\right\|_{L^2}\cdot\sup_{\left|\lambda\right|<2^{-l+\delta j}}\left\|e^{-i(s+\lambda)\Lambda_\nu}\left(P_{k_2}\partial_s \widehat{f^\nu}\right)(s)\right\|_{L^\infty} \\
    &+2^{-100j}\left\|f^\mu_{j_1,k_1,n_1}\right\|_{L^2}\cdot\left\|P_{k_2}\partial_s f^\nu\right\|_{L^2}.   
\end{aligned}\tag{5.15}\label{5.15}
\end{equation}
Finally, by (\ref{5.15}), we get
\begin{align*}
    &2^{(1-20\delta j)}\left\|\mathcal{B}_{m,l}\left[f^\mu_{j_1,k_1,n_1},P_{k_2} \partial_s f^\nu\right]\right\|_{L^2}\lesssim 2^{(1-20\delta)j}\cdot 2^{m-l}\cdot 2^{-2j+40.1\delta j}\cdot 2^{-j_1+20\delta j_1+\frac{n_1}{2}-19\delta n_1}\\
    \lesssim&\,2^{-0.5\delta j+\delta^2 m},
\end{align*}
which gives an acceptable contribution as in (\ref{5.12}). \par
\textbf{Case 3.} $l\ge-\frac{2}{3}j+25\delta j$, $k_2\ge -\frac{j}{3}+12\delta j$ \ \ \ \ (WLOG, let $n_1=\min\left\{n_1,n_2\right\}$)\par
In this case, we also integrate by part in time as in (\ref{5.7}), and $\mathcal{A}_{m,l}$ can be done exactly as in case 2 above. So we only need to focus on $\mathcal{B}_{m,l}$ in the following.\par
If $m+D\le j\le m+\delta m+D$, then we decompose, according to Lemma 3.9 and Remark 3.10,
\begin{align*}
   \partial_s f^\nu (s)=&\widetilde{f^\nu_C}(s)+\widetilde{f^\nu_{NC}}(s),\ \ \ \ \left\|P_{k_2}\widetilde{f^\nu_{NC}}\right\|_{L^2}\lesssim 2^{-\frac{4}{3}m+10.8\delta m}, \\
   \widehat{\widetilde{f^\nu_C}}(\xi,s)=&\sum_{\alpha,\beta\in\mathcal{P},\alpha+\beta\neq 0} e^{is\Psi_{\mu\alpha\beta}}(\xi) g_{\nu\alpha\beta}(\xi,s), \\
   &\left\|\varphi_{k_2}(\xi) D^\rho_\xi g_{\nu\alpha\beta}(\xi,s)\right\|_{L^\infty}\lesssim 2^{-m/2+(1-\delta)m\left|\rho\right|}.
\end{align*}
On one hand, we can rewrite
\begin{align*}
   Q_{jk}&\mathcal{B}_{m,l}\left[f^\mu_{j_1,k_1,n_1},P_{k_2}\widetilde{f^\nu_C}\right](x)\sim\sum_{\alpha,\beta\in\mathcal{P},\alpha+\beta\neq 0} \widetilde{\varphi_j}^{(k)}(x)\cdot \int_{\mathbb{R}} q_m(s) \int_{\mathbb{R}^2} \reallywidehat{f^\mu_{j_1,k_1,n_1}}(\eta,s) \\
   &\times\left[\int_{\mathbb{R}^2}e^{i\left[s\Phi_{\sigma\mu\nu}(\xi,\xi-\eta)+s\Psi_{\nu\alpha\beta}(\xi-\eta)+x\cdot\xi\right]}\widetilde{\varphi_l}(\Phi(\xi,\xi-\eta))\varphi_k(\xi)\varphi_{k_2}(\xi-\eta)\widehat{g_{\nu\alpha\beta}}(\xi-\eta,s)\,d\xi\right]\,d\eta ds.
\end{align*}
Then, integration by parts in $\xi$ using Lemma 3.1 leads to an acceptable contribution. On the other hand, we use Proposition 6.10 (a) to get
\begin{align*}
    \left\|\mathcal{B}_{m,l}\right\|_{L^2}&\lesssim 2^{m-l}\cdot 2^{\frac{l-n_1}{2}}\cdot 2^{2\delta^2 m}\cdot \sup_s\left[\left\|\sup_\theta\left|\widehat{f_1}(r\theta)\right|\right\|_{L^2(rdr)}\cdot\left\|P_{k_2}\widetilde{f^\nu_{NC}}\right\|_{L^2}\right]\lesssim 2^{-j+19.7\delta j},
\end{align*}
which gives an acceptable contribution as in (\ref{5.12}).\par
If $j\ge m+\delta m+D$, then we can show that $\left\|P_{k_2}\partial_s f^\nu (s)\right\|_{L^2}\lesssim 2^{-m-\frac{1}{2}j+22\delta m+\delta j}.$
In fact, we have
\begin{align*}
    P_{k_2}\partial_s \widehat{f^\nu}(\xi,s)&=\int e^{is\Phi_{\nu\alpha\beta}(\xi,\eta)}\widehat{f^\alpha_{\Bar{j}_1,\Bar{k}_1}}(\xi-\eta)\widehat{f^\beta_{\Bar{j}_2,\Bar{k}_2}}(\eta)\,d\eta \\
    &=e^{is\Lambda_\nu(\xi)} \int e^{-is\Lambda_\alpha(\xi-\eta)}\widehat{f^\alpha_{\Bar{j}_1,\Bar{k}_1}}(\xi-\eta)\cdot e^{-is\Lambda_\beta(\eta)}\widehat{f^\beta_{\Bar{j}_2,\Bar{k}_2}}(\eta)\,d\eta.
\end{align*}
Denote $\widehat{U}(\cdot)=e^{-is\Lambda_\nu(\cdot)}\widehat{\partial_s f^\nu}(\cdot)$, $\widehat{U_1}(\cdot)=e^{-is\Lambda_\alpha(\cdot)}\widehat{f^\alpha_{\Bar{j}_1,\Bar{k}_1}}(\cdot)$ and $\widehat{U_2}(\cdot)=e^{-is\Lambda_\beta(\cdot)}\widehat{f^\beta_{\Bar{j}_2,\Bar{k}_2}}(\cdot)$, and we get $\widecheck{P_{k_2}\hat{U}}(x)=U_1(x)\cdot U_2(x)$. Since $2^m\ll 2^{j-\delta^2 m}$, the support of $U_i$ is an annulus with radius $\sim 2^{\Bar{j}_i}$ ($i=1,2$). So, we get that $\max\left\{\Bar{j}_1,\Bar{j}_2\right\}\ge j-\delta^2 m$. WLOG, suppose $\Bar{j}_2\ge j-\delta^2 m$. Now, using Lemma 3.5 we finally get 
\begin{align*}
    \left\|P_{k_2}\partial_s f^\nu\right\|_{L^2}&\lesssim\left\|e^{-is\Lambda_\nu}\widehat{f^\alpha_{\Bar{j}_1,\Bar{k}_1}}\right\|_{L^\infty}\cdot\left\|\widehat{f^\beta_{\Bar{j}_2,\Bar{k}_2}}\right\|_{L^2}\lesssim 2^{-m-\frac{1}{2}j+22\delta m+\delta j}.
\end{align*}\par
Now, turn back to our main proof and we again use Proposition 6.10 (a) to get
\begin{align*}
    \left\|\mathcal{B}_{m,l}\right\|_{L^2}&\lesssim 2^{m-l}\cdot 2^{\frac{l-n_1}{2}}\cdot 2^{2\delta^2 m}\cdot \sup_s\left[\left\|\sup_\theta\left|\widehat{f_1}(r\theta)\right|\right\|_{L^2(rdr)}\cdot\left\|P_{k_2}\widetilde{f^\nu_{NC}}\right\|_{L^2}\right]\lesssim 2^{-\frac{7}{6}j+9.5\delta j+22.1\delta m},
\end{align*}
which gives an acceptable contribution as in (\ref{5.12}).\par
Our proof of Lemma 5.3 is now complete.
\vspace{4em}
\subsection{The Case of Strongly Resonant Interactions}
\ \par
In this subsection, we will prove Lemma 5.5. To be more clear, it suffices to show the following:\par
Let $\varphi\in C^\infty_c(\mathbb{R}^2)$ be supported in $\left[-1,1\right]$ and assume that
\begin{align*}
    l_-=-m+\delta^3 m\le l\le -7\delta m,\ \ \ \ s\sim 2^m,\ \ \ \ j\le m+D.\tag{5.16}\label{5.16}
\end{align*}
Assume that
$$\left\|f\right\|_{H^{N_0/2}\cap H^{N_1/2}_\Omega\cap Z^\mu_1}+\left\|g\right\|_{H^{N_0/2}\cap H^{N_1/2}_\Omega\cap Z^\nu_1}\le 1,$$
and define
$$\reallywidehat{I_l\left[f,g\right]}(\xi)\triangleq \int_{\mathbb{R}^2} e^{is\Phi(\xi,\eta)} \varphi_l(\Phi(\xi,\eta)) \hat{f}(\xi-\eta) \hat{g}(\eta)\,d\eta,$$
where $\varphi_l(x)=\varphi(2^{-l}x)$.
Assume also that $k,k_1,k_2,j,m$ satisfy (\ref{5.6}). Then 
\begin{align*}
    \left\|Q_{jk} I_l\left[P_{k_1}f,P_{k_2}g\right]\right\|_{B^\sigma_j} \lesssim 2^l\cdot 2^{-0.01\delta^2 m}  \tag{5.17}\label{5.17}
\end{align*}
In practice, in many cases, we can prove stronger results:
\begin{align*}
    \left\|Q_{jk} I_l\left[P_{k_1}f,P_{k_2}g\right]\right\|_{L^2} \lesssim 2^l\cdot 2^{-(1-20\delta)j}\cdot 2^{-0.01\delta^2 m}  \tag{5.18}\label{5.18}
\end{align*}
In the following subsections, we will prove the bound of $\mathcal{A}_{m,l}$ and $T_{m,l_-}$ (i.e. prove (\ref{5.17}) or (\ref{5.18})) by analyzing several cases depending on the relative sizes of the main parameters $m,l,j,j_1,j_2,k,k_1,k_2$ and so on.
\vspace{3em}
\subsubsection{Medium Frequency}
\ \par
In this subsection, we assume that $\min\left\{k,k_1,k_2\right\}\ge -D$. (Keep in mind that we always have $\max\left\{k,k_1,k_2\right\}\le 2^{0.1\delta^2 m}$.)\par
First, we observe that we may assume $\min\left\{j_1,j_2\right\}\le (1-\delta^2)m$, because the contrary case can be done by using (\ref{5.29a}) and (\ref{5.29b}) below. Now, we denote 
\begin{align*}
    \kappa_r\triangleq2^{\delta^2 m}\left(2^{-m/2}+2^{j_2-m}\right),\ \ \ \ \ \kappa_\theta\triangleq 2^{-m/2+\frac{1}{2}\delta^2 m},\tag{5.19}\label{5.19}
\end{align*}
and decompose
\begin{align*}
    &\reallywidehat{I_l\left[f,g\right]}=\reallywidehat{I^\parallel_l\left[f,g\right]}+\reallywidehat{I^\bot_l\left[f,g\right]}, \\
    &\reallywidehat{I_l^\parallel\left[f,g\right]}(\xi)\triangleq \int_{\mathbb{R}^2} e^{is\Phi(\xi,\eta)} \varphi_l(\Phi(\xi,\eta)) \varphi(\kappa_\theta^{-1} \Omega_\eta\Phi(\xi,\eta))\hat{f}(\xi-\eta) \hat{g}(\eta)\,d\eta,  \\
    &\reallywidehat{I_l^\bot\left[f,g\right]}(\xi)\triangleq \int_{\mathbb{R}^2} e^{is\Phi(\xi,\eta)} \varphi_l(\Phi(\xi,\eta)) \left(1-\varphi(\kappa_\theta^{-1} \Omega_\eta\Phi(\xi,\eta))\right)\hat{f}(\xi-\eta) \hat{g}(\eta)\,d\eta.
\end{align*}
By Lemma 3.2 and (\ref{5.13}), we can get $\left\|I^\bot_l\right\|_{L^2}\lesssim 2^{-4m}$. Thus, we only need to focus on the term $\widehat{I^\parallel_l}$. Next, we further decompose
\begin{align*}
   &\widehat{I_l^\parallel}=\widehat{I_{l,lo}^\parallel}+\widehat{I_{l,hi}^\parallel}=\widehat{I_{l,\widetilde{lo}}^\parallel}+\widehat{I_{l,\widetilde{hi}}^\parallel} \\ 
   &\reallywidehat{I_{l,*}^\parallel\left[f,g\right]}(\xi)\triangleq \int_{\mathbb{R}^2} e^{is\Phi(\xi,\eta)} \varphi_l(\Phi(\xi,\eta)) \varphi_*(\nabla_\eta\Phi(\xi,\eta))\varphi(\kappa_\theta^{-1} \Omega_\eta\Phi(\xi,\eta))\hat{f}(\xi-\eta) \hat{g}(\eta)\,d\eta \\
   &\ \ \ \ \ \ \ \ \ \ \ \ *\in\left\{lo,hi,\widetilde{lo},\widetilde{hi}\right\},   
\end{align*}
where $\varphi_{lo}=\varphi_{\lesssim \kappa_r}, \varphi_{\widetilde{lo}}=\varphi_{\lesssim 1}, \varphi_{hi}=1-\varphi_{lo}, \varphi_{\widetilde{hi}}=1-\varphi_{\widetilde{lo}}$.\par
\textbf{Case 1.} $j_2=\max\left\{j_1,j_2\right\}\le (1-\delta^2) m$ \par
In this case, we use the decomposition $\widehat{I_l^\parallel}=\widehat{I_{l,lo}^\parallel}+\widehat{I_{l,hi}^\parallel}$. Recall the definition of $\kappa_r$, see (\ref{5.19}), we observe that $\widehat{I^\parallel_{l,hi}}$ can be controlled by integration by parts in $\eta$ using Lemma 3.1 and (\ref{5.13}). Next, we will focus on $\widehat{I^\parallel_{l,lo}}$, and we will prove that it satisfies the weaker bound (\hyperref[5.17]{5.17}). \par
In fact, here is one of the few places where we need the localization operator $A^\sigma_{n,(j)}$, namely we, by Proposition 6.5(a), observe that 
$$\left|\Phi(\xi,p(\xi))\right|\le\left|\Phi(\xi,\eta)\right|+\left|\Phi(\xi,\eta)-\Phi(\xi,p(\xi))\right|\lesssim\left|\Phi(\xi,\eta)\right|+\left|\nabla_\eta\Phi\right|\cdot\left|\eta-p(\xi)\right|\lesssim \max\left\{2^l,2^{0.4\delta^2 m}\kappa_r^2\right\}.$$
(Note that Proposition 6.5(b) guarantees that $\mu+\nu=0$ is impossible here.) Next, we will mainly divide into three subcases. \par
\ \ \textbf{Case 1.1.}\  $\frac{1}{2}m\le j_2\le (1-\delta^2)m$, $2^l\le 2^{0.4\delta^2 m}\kappa^2_r$ \par
In this subcase, we have $\kappa_r=2^{j_2-m+\delta^2 m}$ and $2^{-n}\le 2^l$. Moreover, 
\begin{align*}
   \mbox{fix }\eta,\mbox{ we have } \left|E_\xi\right|\lesssim 2^{l+0.4\delta^2 m}\cdot \kappa_\theta;\tag{5.20}\label{5.20}
\end{align*}
\begin{numcases}
   \ \mbox{fix }\xi,\mbox{ we have } \left|E_\eta\right|\lesssim \kappa_r\cdot \kappa_\theta; \tag{5.21}\label{5.21} \\
   \mbox{fix }\xi,\mbox{ we have } \left|E_\eta\right|\lesssim 2^{-n_2}\cdot \kappa_\theta; \tag{5.22}\label{5.22} \\
   \mbox{fix }\xi,\mbox{ we have } \left|E_\eta\right|\lesssim 2^{\frac{l}{2}}\cdot \kappa_\theta; \tag{5.23}\label{5.23} \\
   \mbox{fix }\xi,\mbox{ we have } \left|E_\eta\right|\lesssim 2^{-2n_2}, \tag{5.24}\label{5.24}
\end{numcases}
where (\ref{5.20}) follows from Proposition 6.4 (b) due to $\left|\nabla_\eta\Phi\right|\gtrsim 2^{-0.4\delta^2 m}$.\par
First, if $n_2\le j_2-3\delta m$, then using 
(\ref{5.20}), (\hyperref[5.21]{5.21}), Lemma 3.6, Remark 3.7 and Schur's Lemma, we get that
\begin{align*}
    &2^{(1-20\delta)j}\,2^{(\frac{1}{2}-19\delta)n}\left\|\widehat{I^\parallel_{l,lo}}\right\|_{L^2} \\
    \lesssim\,&2^{(1-20\delta)j}\,2^{(\frac{1}{2}-19\delta)n}\,2^{\frac{l}{2}+0.2\delta^2 m}\,\kappa_r^{\frac{1}{2}}\,\kappa_\theta\,2^{\delta n_1+2\delta^2 n_1}\,2^{-(1-20\delta)j_2+(\frac{1}{2}-19\delta)n_2}\lesssim\,2^{l-0.49\delta m},
\end{align*}
which gives an acceptable contribution as in (\ref{5.17}).\par
Second, if $j_2-3\delta m\le n_2$ and $\frac{1}{2}m+1.01\delta m\le j_2\le (1-\delta^2)m$, then using (\ref{5.20}), (\hyperref[5.22]{5.22}), Lemma3.6, Remark 3.7 and Schur's Lemma, we get that
\begin{align*}
    &2^{(1-20\delta)j}\,2^{(\frac{1}{2}-19\delta)n}\left\|\widehat{I^\parallel_{l,lo}}\right\|_{L^2} \\
    \lesssim\,&2^{(1-20\delta)j}\,2^{(\frac{1}{2}-19\delta)n}\,2^{\frac{l}{2}+0.2\delta^2 m}\,2^{-\frac{n_2}{2}}\,\kappa_\theta\,2^{\delta n_1+2\delta^2 n_1}\,2^{-(1-20\delta)j_2+(\frac{1}{2}-19\delta)n_2}\lesssim\,2^{l-0.005\delta m},
\end{align*}
which gives an acceptable contribution as in (\ref{5.17}).\par
Third, if $j_2-3\delta m\le n_2$, $\frac{1}{2}m\le j_2\le \frac{1}{2}m+1.01\delta m$ and $l\ge -m+\frac{5}{19}\delta m$, then using 
(\ref{5.20}), (\hyperref[5.21]{5.21}), Lemma 3.6, Remark 3.7 and Schur's Lemma, we get that
\begin{align*}
    &2^{(1-20\delta)j}\,2^{(\frac{1}{2}-19\delta)n}\left\|\widehat{I^\parallel_{l,lo}}\right\|_{L^2} \\
    \lesssim\,&2^{(1-20\delta)j}\,2^{(\frac{1}{2}-19\delta)n}\,2^{\frac{l}{2}+0.2\delta^2 m}\,\kappa_r^{\frac{1}{2}}\,\kappa_\theta\,2^{\delta n_1+2\delta^2 n_1}\,2^{-(1-20\delta)j_2+(\frac{1}{2}-19\delta)n_2}\lesssim\,2^{l-0.7\delta^2 m},
\end{align*}
which gives an acceptable contribution as in (\ref{5.17}).\par
Finally, we can assume $j_2-3\delta m\le n_2$, $\frac{1}{2}m\le j_2\le \frac{1}{2}m+1.01\delta m$ and $-m\le l\le -m+\frac{5}{19}\delta m$. We first observe that without fixing $\eta$, we still have 
\begin{align*}
   \left|E_\xi\right|\lesssim 2^{l+0.4\delta^2 m}\cdot 2^{-\min\left\{n_1,n_2\right\}}.\tag{5.25}\label{5.25}  
\end{align*}
Then by volume counting, using (\ref{5.20}), (\ref{5.25}) and Remark 3.7, we have
\begin{align*}
    \left\|\widehat{I^\parallel_{l,lo}}\right\|_{L^2}&\lesssim\left(\int_\xi\left(\int_\eta e^{is\Phi(\xi,\eta)}\varphi_l(\Phi(\xi,\eta))\varphi\left(\kappa_r^{-1}\nabla_\eta\Phi(\xi,\eta)\right)\varphi\left(\kappa_\theta^{-1}\Omega_\eta\Phi(\xi,\eta)\right)\hat{f}(\xi-\eta)\hat{g}(\eta)\,d\eta\right)^2\,d\xi\right)^{1/2} \\
    &\lesssim \left(\int_\xi\left(\int_\eta \left\|\hat{f}\right\|_{L^\infty}\cdot\left\|\hat{g}\right\|_{L^\infty}\cdot \sup\limits_\xi \left\{\left|E_\eta(\xi)\right|^2\right\}\right)\,d\xi\right)^{1/2} \\
    &=\sup\limits_\xi\left\{\left|E_\eta(\xi)\right|\right\}\cdot\left|E_\xi\right|^{1/2}\cdot2^{\delta(n_1+n_2)+2\delta^2(n_1+n_2)} \\
    &\lesssim 2^l\,2^{-\frac{m}{2}+0.7\delta^2 m}\,2^{-\frac{1}{2}\min\left\{n_1,n_2\right\}+\delta(n_1+n_2)+2\delta^2(n_1+n_2)}.
\end{align*}
Thus, we get
\begin{align*}
    2^{(1-20\delta)j-(\frac{1}{2}-19\delta)n}\left\|\widehat{I^\parallel_{l,lo}}\right\|_{L^2}\lesssim 2^l\,2^{-\frac{13}{38}\delta m},
\end{align*}
which gives an acceptable contribution as in (\ref{5.17}).\par
\ \ \textbf{Case 1.2.}\  $\frac{1}{2}m\le j_2\le (1-\delta^2)m$, $2^l\ge 2^{0.4\delta^2 m}\kappa^2_r$ \par
In this subcase, we have $\kappa_r=2^{j_2-m+\delta^2 m}$ and $2^{-n}\le 2^{0.4\delta^2 m}\kappa^2_r$.\par
First, if $\frac{1}{2}m+\frac{1}{4}\delta m\le j_2\le (1-\delta^2)m$ and $n_2\le\frac{1}{2}m$, then using 
(\ref{5.20}), (\hyperref[5.23]{5.23}), Lemma 3.6, Remark 3.7 and Schur's Lemma, we get that
\begin{align*}
    &2^{(1-20\delta)j}\,2^{(\frac{1}{2}-19\delta)n}\left\|\widehat{I^\parallel_{l,lo}}\right\|_{L^2} \\
    \lesssim\,&2^{(1-20\delta)j}\,2^{(1-38\delta)(j_2-m)+1.2\delta^2 m}\,2^{\frac{3l}{4}+0.2\delta^2 m}\,\kappa_\theta\,2^{\delta n_1+2\delta^2 n_1}\,2^{-(1-20\delta)j_2+(\frac{1}{2}-19\delta)n_2}\lesssim\,2^{l-0.35\delta^2 m},
\end{align*}
which gives an acceptable contribution as in (\ref{5.17}).\par
Second, if $\frac{1}{2}m+\frac{1}{4}\delta m\le j_2\le (1-\delta^2)m$ and $n_2\ge\frac{1}{2}m$, then using (\ref{5.20}), (\hyperref[5.24]{5.24}), Lemma 3.6, Remark 3.7 and Schur's Lemma, we get that
\begin{align*}
    &2^{(1-20\delta)j}\,2^{(\frac{1}{2}-19\delta)n}\left\|\widehat{I^\parallel_{l,lo}}\right\|_{L^2} \\
    \lesssim\,&2^{(1-20\delta)j}\,2^{(1-38\delta)(j_2-m)+1.2\delta^2 m}\,2^{\frac{l}{2}-\frac{m}{4}+\frac{1}{4}\delta^2 m-n_2}\,2^{\delta n_1+2\delta^2 n_1}\,2^{-(1-20\delta)j_2+(\frac{1}{2}-19\delta)n_2}\lesssim\,2^{l-0.7\delta^2 m},
\end{align*}
which gives an acceptable contribution as in (\ref{5.17}).\par
Finally, we can assume $\frac{1}{2}m\le j_2\le \frac{1}{2}m+\frac{\delta m}{4}$. 
Then, like the last subcase of Case 1.1 above, by volume counting, using (\hyperref[5.21]{5.21}), (\ref{5.25}) and Remark 3.7, we have
\begin{align*}
    \left\|\widehat{I^\parallel_{l,lo}}\right\|_{L^2}&\lesssim\sup\limits_\xi\left\{\left|E_\eta(\xi)\right|\right\}\cdot\left|E_\xi\right|^{1/2}\cdot2^{\delta(n_1+n_2)+2\delta^2(n_1+n_2)} \\
    &= 2^{-m+1.7\delta^2 m+\frac{\delta m}{4}+\frac{l}{2}-\frac{1}{2}\min\left\{n_1,n_2\right\}+\delta (n_1+n_2)+2\delta^2 (n_1+n_2)}. \tag{5.26}\label{5.26}
\end{align*}
Thus, we get
\begin{align*}
    &2^{(1-20\delta)j-(\frac{1}{2}-19\delta)n}\left\|\widehat{I^\parallel_{l,lo}}\right\|_{L^2}\lesssim\,2^{(1-20\delta)j+l-m+20\delta m-5\delta^2 m}\lesssim 2^{l-5\delta^2 m},
\end{align*}
which gives an acceptable contribution as in (\ref{5.17}).\par
\ \ \textbf{Case 1.3.}\  $0\le j_2\le \frac{1}{2}m$ \par
In this subcase, we have $\kappa_r=\kappa_\theta=2^{-\frac{1}{2}m+\delta^2 m}$. We may assume that $2^l\le 2^{0.4\delta^2 m}\kappa_r^2$, since the proof of the other case is exactly same. By (\ref{5.26}), we have
\begin{align*}
    \left\|\widehat{I^\parallel_{l,lo}}\right\|_{L^2}&\lesssim 2^{-m+1.7\delta^2 m+\frac{\delta m}{4}+\frac{l}{2}-\frac{1}{2}\min\left\{n_1,n_2\right\}+\delta (n_1+n_2)+2\delta^2 (n_1+n_2)}\lesssim 2^{\frac{l}{2}-m+\frac{1}{2}\delta m+3\delta^2 m}.
\end{align*}
Thus, we get
\begin{align*}
    2^{(1-20\delta)j-(\frac{1}{2}-19\delta)n}\left\|\widehat{I^\parallel_{l,lo}}\right\|_{L^2}&\lesssim 2^{(1-20\delta)j+(\frac{1}{2}-19\delta)l}\,2^{\frac{l}{2}-m+\frac{1}{2}\delta m+3\delta^2 m} \lesssim 2^{l-0.49\delta m},
\end{align*}
which gives an acceptable contribution as in (\ref{5.17}).\par
\vspace{3em}
\textbf{Case 2.} $j_2=\max\left\{j_1,j_2\right\}\ge (1-10 \delta^2) m$\par
In this case, we use the decomposition $\widehat{I_l^\parallel}=\widehat{I_{l,\widetilde{lo}}^\parallel}+\widehat{I_{l,\widetilde{hi}}^\parallel}$. Let's deal with $\widehat{I_{l,\widetilde{lo}}^\parallel}$ first. If $n_2\le \frac{m}{2}$, then we use (\ref{5.20}), (\hyperref[5.23]{5.23}), Lemma 3.6, Remark 3.7 and Schur's Lemma to get
\begin{align*}
    2^{(1-20\delta)j}\left\|\widehat{I^\parallel_{l,\widetilde{lo}}}\right\|_{L^2}&\lesssim 2^{(1-20\delta)j}\,2^{\frac{3l}{4}+0.2\delta^2 m}\,2^{-\frac{m}{2}+0.5\delta^2 m}\,2^{\delta n_1+2\delta^2 n_1}\,2^{-(1-20\delta)j_2+(\frac{1}{2}-19\delta)n_2}\lesssim 2^{l-8\delta m},
\end{align*}
which gives an acceptable contribution as in (\ref{5.18}); on the other hand, if $n_2\ge\frac{m}{2}$, then we use (\ref{5.20}), (\hyperref[5.22]{5.22}), Lemma 3.6, Remark 3.7 and Schur's Lemma to get
\begin{align*}
    2^{(1-20\delta)j}\left\|\widehat{I^\parallel_{l,\widetilde{lo}}}\right\|_{L^2}&\lesssim 2^{(1-20\delta)j}\,2^{\frac{l}{2}-\frac{n_2}{2}+0.2\delta^2 m}\,2^{-\frac{m}{2}+0.5\delta^2 m}\,2^{\delta n_1+2\delta^2 n_1}\,2^{-(1-20\delta)j_2+(\frac{1}{2}-19\delta)n_2}\lesssim 2^{l-8\delta m},
\end{align*}
which gives an acceptable contribution as in (\ref{5.18}). \par
Next, we deal with $\widehat{I^\parallel_{l,\widetilde{hi}}}$. Now, we further decompose
\begin{align*}
   \widehat{I^\parallel_{l,\widetilde{hi}}}&\triangleq\sum_{l\le r\le 0} \widehat{I^\parallel_{l,\widetilde{hi},r}} \\
   &\triangleq\sum_{l\le r\le 0}\int_{\mathbb{R}^2} e^{is\Phi(\xi,\eta)} \varphi_l(\Phi(\xi,\eta)) \varphi_{\widetilde{hi}}(\nabla_\eta\Phi(\xi,\eta)) \varphi_r^{\left[l,0\right]}(\nabla_\xi\Phi(\xi,\eta))\varphi(\kappa_\theta^{-1} \Omega_\eta\Phi(\xi,\eta))\hat{f}(\xi-\eta) \hat{g}(\eta)\,d\eta. 
\end{align*}
When $r=0$, we have that fix $\eta$, $\left|E_\xi\right|\lesssim 2^l\cdot\kappa_\theta$ thanks to $\left|\Phi\right|\lesssim 2^l$ and $\left|\nabla_\xi\Phi\right|\gtrsim 1$; fix $\xi$, $\left|E_\eta\right|\lesssim 2^l\cdot\kappa_\theta$ thanks to $\left|\Phi\right|\lesssim 2^l$ and $\left|\nabla_\eta\Phi\right|\gtrsim 1$. Then by Lemma 3.6 and Schur's Lemma, we get that 
\begin{align*}
   \left\|\widehat{I^\parallel_{l,\widetilde{hi},0}}\right\|_{L^2}&\lesssim 2^l\cdot\kappa_\theta\cdot 2^{2\delta m}\cdot 2^{-(1-20\delta)j_2+(\frac{1}{2}-19\delta)n_2}\lesssim 2^{l-(1-20\delta)j -17\delta m},\tag{5.27}\label{5.27}
\end{align*}
which gives an acceptable contribution as in (\ref{5.18}). When $r=l$, we still have that fix $\eta$, $\left|E_\xi\right|\lesssim 2^l\cdot\kappa_\theta$ due to $\left|\nabla_\xi\Phi\right|\lesssim 2^l$ and $\left|\nabla_{\xi\xi}\Phi\right|\gtrsim 1$. Therefore, we can proceed as before in (\ref{5.27}) to get that $\left\|\widehat{I^\parallel_{l,\widetilde{hi},l}}\right\|_{L^2}\lesssim 2^{l-(1-20\delta)j-17\delta m}$ as well. Next, we can assume $l<r<0$ and divide into the following three subcases. \par
\ \ \textbf{Case 2.1.}\ \  $j\le m+r+100$, $2r\le l$ \par
In this subcase, we notice that fix $\eta$, we have $\left|E_\xi\right|\lesssim 2^r \cdot \kappa_\theta$ due to the fact that $\left|\nabla_\xi\Phi\right|\sim 2^r$ and $\left|\nabla_{\xi\xi}\Phi\right|\gtrsim 1$, and also
\begin{align*}
    \mbox{fix }\xi,\mbox{ we have }\left|E_\xi\right|\lesssim 2^l\cdot\kappa_\theta \tag{5.28}\label{5.28}
\end{align*}
due to the fact that $\left|\Phi\right|\lesssim 2^l$ and $\left|\nabla_\eta\Phi\right|\gtrsim 1$.
Thus, use Lemma 3.6 and Schur's Lemma and we get
\begin{align*}
    \left\|\widehat{I^\parallel_{l,\widetilde{hi},r}}\right\|_{L^2}&\lesssim 2^{\frac{l}{2}+\frac{r}{2}}\,2^{-\frac{m}{2}+0.5\delta^2 m}\,2^{2\delta m}\,2^{-(1-20\delta)j_2+(\frac{1}{2}-19\delta)n_2}\lesssim 2^{l-(1-20\delta)(m+r)-13\delta m}\lesssim 2^{l-(1-20\delta)j-13\delta m},
\end{align*}
which gives an acceptable contribution as in (\ref{5.18}).\par
\ \ \textbf{Case 2.2.}\ \  $j\le m+r+100$, $2r\ge l$ \par
In this subcase, we notice that fix $\eta$, we have $\left|E_\xi\right|\lesssim 2^{l-r} \cdot \kappa_\theta$ due to the fact that $\left|\Phi\right|\lesssim 2^l$ and $\left|\nabla_\xi\Phi\right|\sim 2^r$. Thus, again use this, (\ref{5.28}), Lemma 3.6 and Schur's Lemma and we get
\begin{align*}
    \left\|\widehat{I^\parallel_{l,\widetilde{hi},r}}\right\|_{L^2}&\lesssim 2^{l-\frac{r}{2}+\delta^2 m}\,2^{-\frac{m}{2}+0.5\delta^2 m}\,2^{2\delta m}\,2^{-(1-20\delta)j_2+(\frac{1}{2}-19\delta)n_2}\lesssim 2^{l-(1-20\delta)j -18\delta m},
\end{align*}
which gives an acceptable contribution as in (\ref{5.18}).\par
\ \ \textbf{Case 2.3.}\ \  $m+r+100\le j\le m+D$\par
First of all, we notice that $\left|\nabla_\xi\left[s\Phi(\xi,\eta)+x\cdot\xi\right]\right|\sim \left|x\right|\sim 2^j$ in this subcase. Hence, if $j_1\le (1-\delta^2) j$, then by the trick in (\ref{5.13}) and Lemma 3.1, integration by parts in $\xi$ will give us an acceptable contribution of $\widehat{I_l}$. (Note that we don't need any decomposition here.) Thus, we may assume $j_1\ge (1-\delta^2) j$ in the following.\par
Recall that we have $j_1\ge (1-\delta^2) j$ and $j_2\ge (1-\delta^2) m$ now. If $n_1\ge 0.07m$, then by Proposition 6.10 (a) and Lemma 3.6, we can get 
\begin{align*}
    \left\|\widehat{I_l}\right\|_{L^2}&\lesssim 2^{\frac{l}{2}}\,2^{-(1-20\delta-2\delta^2 )j_1-19\delta n_1}\,2^{-(1-20\delta)j_2+(\frac{1}{2}-19\delta)n_2}\lesssim 2^{l-(1-20\delta) j-0.32\delta m},
\end{align*}
which gives an acceptable contribution as in (\ref{5.18}).
If $n_1\le 0.07m$ but $n_2\ge 0.07m$, then we can switch the variables $\xi-\eta$ and $\eta$ and proceed as before exactly. Therefore, we may now assume that $n_1\le 0.07m$ and $n_2\le 0.07m$ at the same time. Then, by Cauchy-Schwartz Inequality, we directly get that
\begin{align*}
    \left\|\widehat{I_l}\right\|_{L^\infty}\lesssim\left\|f\right\|_{L^2}\cdot\left\|g\right\|_{L^2}\lesssim 2^{-(1-20\delta)j}\,2^{-m+21\delta m+(\frac{1}{2}-19\delta)(n_1+n_2)}\lesssim 2^{-(1-20\delta)j}\,2^{-0.93m+21\delta m}.\tag{5.29a}\label{5.29a}
\end{align*}
Thus, we end up with
\begin{align*}
    \left\|\widehat{I_l}\right\|_{L^2}\lesssim \left\|\widehat{I_l}\right\|_{L^\infty}\cdot \left|E_\xi\right|^{1/2}\lesssim 2^{-(1-20\delta)j}\,2^{-0.92m}\,\left(2^{\frac{l}{2}}\right)^{1/2}\lesssim 2^{l-(1-20\delta)j-0.17m},\tag{5.29b}\label{5.29b}
\end{align*}
which gives an acceptable contribution as in (\ref{5.18}).
\vspace{3em}
\subsubsection{Low Frequency}
\ \par
In this subsection, we assume that $\min\left\{k,k_1,k_2\right\}\le -D$. By switching the order of $\xi-\eta$ and $\eta$, we may assume that $j_1\le j_2$. If $k_1\ge -D$, then we can proceed as in section 5.3.1. (If $k_1\ge -D$, then we can insert the angle cutoff function in most cases.) On the other hand, if $k_1\le -D$, then we must have $k,k_2\ge -D$ due to the fact that $b_\sigma-b_\mu-b_\nu\neq 0$. From now on, we may also assume $j_1-k_1+\delta^2 m\ge m$, since otherwise we can insert the angle cutoff function $\varphi(\kappa_\theta^{-1}\Omega_\eta\Phi(\xi,\eta))$ and proceed as in section 5.3.1. Thus, we now have that $j_2\ge j_1\ge (1-\delta^2)m+k_1$.\par
If $-0.06m\le k_1\le -D$, then by Lemma 3.6 and Proposition 6.10 (a), we can get 
\begin{align*}
    \left\|\widehat{I_l}\right\|_{L^2}&\lesssim 2^{\frac{l}{2}-\frac{n_2}{2}}\,2^{-(1-20\delta)j_1}\,2^{-(1-20.01\delta)j_2+(\frac{1}{2}-19\delta)n_2}\lesssim 2^{l-(1-20\delta)j-0.37\delta m},
\end{align*}
which gives an acceptable contribution as in (\ref{5.18}).\par
If $k_1\le -0.06m$, then we may need the localization operator $A^\sigma_{n,(j)}$. Recall that $k_1\le -D$, $k,k_2\ge -D$ and 
\begin{align*}
   \nabla_\eta\Phi=\pm\frac{c_\mu^2(\xi-\eta)}{\sqrt{c_\mu^2\left|\xi-\eta\right|^2+b_\mu^2}}\mp\frac{c_\nu^2\eta}{\sqrt{c_\nu^2\left|\eta\right|^2+b_\nu^2}}, \tag{5.30}\label{5.30}
\end{align*}
and we can get that $\left|\nabla_\eta\Phi\right|\gtrsim 1$, since when $\left|\eta\right|\gtrsim 1$, $\left|\frac{c_\nu^2\eta}{\sqrt{c_\nu^2\left|\eta\right|^2+b_\nu^2}}\right|\sim\left|\eta\right|$ and when $\left|\xi-\eta\right|\ll 1$, $\left|\frac{c_\mu^2(\xi-\eta)}{\sqrt{c_\mu^2\left|\xi-\eta\right|^2+b_\mu^2}}\right|\ll 1$. Similarly, we can also get that $\left|\nabla_\xi\Phi\right|\gtrsim 1$. Therefore, due to $\left|\Phi\right|\lesssim 2^l$ and $\left|\xi-\eta\right|\sim 2^{k_1}$, we have that fix $\eta$, $\left|E_\xi\right|\lesssim 2^l\cdot 2^{k_1}$; fix $\xi$, $\left|E_\eta\right|\lesssim 2^l\cdot 2^{k_1}$. Moreover, if $\eta$ is near $\gamma_i$ for some $i$, then we have that
$$\left|\xi-\gamma_i\right|\le \left|\xi-\eta\right|+\left|\eta-\gamma_i\right|\lesssim 2^{\max\left\{k_1,-n_2\right\}},$$
which implies that $n\ge \max\left\{k_1,-n_2\right\}$.
Thus, by Lemma 3.6, Remark 3.7 and Schur's Lemma, we get 
\begin{align*}
    2^{(1-20\delta)j}\,2^{-(\frac{1}{2}-19\delta)n}\left\|\widehat{I_l}\right\|_{L^2}&\lesssim 2^{(1-20\delta)j-(\frac{1}{2}-19\delta)n}\,2^{l+k_1}\,2^{1.01\delta m}\,2^{-(1-20\delta)j_1+(\frac{1}{2}-19\delta)n_2}\lesssim 2^{l-0.19\delta m},
\end{align*}
which gives an acceptable contribution as in (\ref{5.17}).\par
\vspace{3em}
Combine section 5.3.1 and section 5.3.2 and our proof of Lemma 5.5 is now complete.
\vspace{4em}
\subsection{The Case of Resonant Interactions}
\ \par
In this subsection, we will prove Lemma 5.6. This is the most complicated part in our dispersive control section. Using Lemma 3.12 we may write, with $F^\nu=F^\nu_C+F^\nu_{NC}+F^\nu_{LO}$,
$$\partial_s f^\nu(s)=f^\nu_C (s)+f^\nu_{SR} (s)+f^\nu_{NC} (s)+f^\nu_{NCw}(s)+\partial_s F^\nu (s).$$
The proof of 
\begin{align*}
    2^{(1-20\delta)j}\left\|Q_{jk}\mathcal{B}_{m,l}\left[P_{k_1}f^\mu,P_{k_2}f^\nu_{SR}\right]\right\|_{L^2}\lesssim 2^{-50\delta^2 m}, \\
    2^{(1-20\delta)j}\left\|Q_{jk}\mathcal{B}_{m,l}\left[P_{k_1}f^\mu,P_{k_2}f^\nu_{NC}\right]\right\|_{L^2}\lesssim 2^{-50\delta^2 m}
\end{align*}
has been done in \cite{y2} (See Section 7.5 in \cite{y2}). So in the following, we will prove 
$$2^{(1-20\delta)j}\left\|Q_{jk}\mathcal{B}_{m,l}\left[P_{k_1}f^\mu,P_{k_2}g\right]\right\|_{L^2}\lesssim 2^{-50\delta^2 m},$$
where $g=f^\nu_C$ or $f^\nu_{NCw}$ or $\partial_s F^\nu$.
\vspace{3em}
\subsubsection{Contribution of $f_C^\nu$}
\ \par
Recall that from Lemma 3.12, we can decompose
\begin{align*}
    &f_C^\nu=\sum_{0\le q\le m/2-40\delta m} \sum_{\omega,\theta,\omega+\theta\neq 0} f^{\nu\omega\theta}_{C,q},\ \ \ \ \ \widehat{f^{\nu\omega\theta}_{C,q}}(\xi,s)=e^{is\Psi_{\mu\omega\theta}(\xi)}g_q(\xi,s), \\
    &g_q(\xi,s)=g^q(\xi,s)\,\varphi_{\le 3\delta m}(\Psi_{\nu\omega\theta}(\xi)),\ \ \ \ \left\|D_\xi^\alpha g_q(s)\right\|_{L^\infty}\lesssim 2^{-m-q+4.01\delta m}\,2^{\left|\alpha\right|(m/2+3\delta^2 m+q)}, \\
    &\sup\limits_{b\le N_1/4} \left\|\Omega^b g_q(\xi,s)\right\|_{L^2}\lesssim 1, \ \ \ \ \left\|\partial_s g_q(\xi,s)\right\|_{L^\infty}\lesssim 2^{-2m+q+6.01\delta m}.
\end{align*}
\begin{align*}
    \ \tag{5.31}\label{5.31}
\end{align*}
It suffices to show that for $k,j,m,l,k_1,k_2,q$ and $\sigma,\mu,\nu,\omega,\theta$ as before
\begin{align*}
    2^{(1-20\delta)j}\left\|Q_{jk}\mathcal{B}_{m,l}\left[P_{k_1}f^\mu,P_{k_2}f^{\nu\omega\theta}_{C,q}\right]\right\|_{L^2}\lesssim 2^{-0.001\delta^2 m}.\tag{5.32}\label{5.32}
\end{align*}
In the rest of this proof, for simplicity, we set $\Phi=\Phi_{\sigma\mu\nu}$, $\Psi=\Psi_{\nu\omega\theta}$, $\textbf{p}(\xi,\eta)\triangleq \Phi_{\sigma\mu\nu}(\xi,\eta)+\Psi_{\nu\omega\theta}(\eta)$. We define $f^\mu_{j_1,k_1}$ and $f^\mu_{j_1,k_1,n_1}$ as before.
Assume first that $j_1\ge (1-10\delta^2)m$. Schur's Lemma with (\ref{3.40}) and Proposition 6.8 (a), (c) gives
\begin{align*}
    \left\|P_k \mathcal{B}_{m,l}\left[f^\mu_{j_1,k_1,0},P_{k_2}f^{\nu\omega\theta}_{C,q}\right]\right\|_{L^2}&\lesssim 2^{2\delta^2 m}\,2^{m-l}\,\left(2^l 2^{-3\delta m/4}\right)\,\sup_s \left\|f^\mu_{j_1,k_1,0}(s)\right\|_{L^2}\,\left\|\widehat{f^{\nu\omega\theta}_{C,q}}(s)\right\|_{L^\infty} \\
    &\lesssim 2^{19.25\delta m+12\delta^2 m}\,\left\|\widehat{f^{\nu\omega\theta}_{C,q}}(s)\right\|_{L^\infty}\lesssim 2^{-(1-20\delta)j-4\delta m}.
\end{align*}
Moreover, using (\ref{3.40}) and Proposition 6.10 (a), for $n_1\ge 1$
\begin{align*}
    \left\|P_k \mathcal{B}_{m,l}\left[f^\mu_{j_1,k_1,n_1},P_{k_2}f^{\nu\omega\theta}_{C,q}\right]\right\|_{L^2}&\lesssim 2^{2\delta^2 m}\,2^{m-l}\,2^{l-n_1/2}\,\sup_s \left\|\sup_{\theta\in\mathbb{S}^1}\left|\widehat{f^\mu_{j_1,k_1,n_1}}\right|\right\|_{L^2(rdr)}\,\left\|\widehat{f^{\nu\omega\theta}_{C,q}}(s)\right\|_{L^2} \\
    &\lesssim 2^{2\delta^2 m}\,2^{20\delta m+12\delta^2 m}\,2^{-19\delta n_1}\,\left\|\widehat{f^{\nu\omega\theta}_{C,q}}(s)\right\|_{L^2}\lesssim 2^{-(1-20\delta)j-3\delta m}.
\end{align*}
Thus, in both cases we get acceptable contributions as in (\ref{5.32}).\par
Assume now that $j_1\le (1-10\delta^2)m$. If $\left|\nabla_\eta\Phi_{\sigma\mu\nu}(\xi,\eta)+\nabla_\eta\Phi_{\nu\omega\theta}(\eta,\theta)\right|\gtrsim 2^{-\delta^2 m}$, then by plugging (\ref{3.37}), we have
\begin{align*}
    P_k\mathcal{B}_{m,l}\left[f^\mu_{j_1,k_1},P_{k_2}f^{\nu\omega\theta}_{C,q}\right]&\approx 2^{m-l}\,\int_{\mathbb{R}^2} e^{is\Phi_{\sigma\mu\nu}(\xi,\eta)+is\Psi_{\nu\omega\theta}(\eta)} \widehat{f^\mu_{j_1,k_1}}(\xi-\eta) g_q(\eta,s)\,d\eta \tag{5.33}\label{5.33} \\
    &= 2^{m-l} \int_{\mathbb{R}^2} e^{is\left[\Phi_{\sigma\mu\nu}(\xi,\eta)+\Phi_{\nu\omega\theta}(\eta,\theta)\right]} \widehat{f^\mu_{j_1,k_1}}(\xi-\eta)\widehat{g^\omega_{\Bar{j_1},\Bar{k_1}}}(\eta-\theta)\widehat{g^\theta_{\Bar{j_2},\Bar{k_2}}}(\theta)\,d\eta d\theta,
\end{align*}
where $j_1\le (1-10\delta^2)m$, $\Bar{j_1}\le \min\left\{\Bar{j_2},(1-300\delta)m\right\}$ and $\Bar{j_2}\le (1-50\delta)m$.
Thus, if we assume $\left|\nabla_\eta\Phi_{\sigma\mu\nu}(\xi,\eta)+\nabla_\eta\Phi_{\nu\omega\theta}(\eta,\theta)\right|\gtrsim 2^{-\delta^2 m}$, then integration by parts in $\eta$ gives 
$$\left\|P_k\mathcal{B}_{m,l}\left[f^\mu_{j_1,k_1},P_{k_2}f^{\nu\omega\theta}_{C,q}\right]\right\|_{L^2}\lesssim 2^{-2m},$$
which is an acceptable contribution as in (\ref{5.32}). Therefore, we may assume 
\begin{align*}
    \left|\nabla_\eta\Phi_{\sigma\mu\nu}(\xi,\eta)+\nabla_\eta\Phi_{\nu\omega\theta}(\eta,\theta)\right|\lesssim 2^{-\delta^2 m}. \tag{5.34}\label{5.34}
\end{align*}
From Lemma 3.12 and (\ref{3.37}), we know that $\left|\nabla_\theta\Phi_{\nu\omega\theta}\right|\lesssim \kappa_r\ll 1$ and $\left|\Phi_{\nu\omega\theta}\right|\lesssim 2^{-3\delta m}\ll 2^{-3\max\left\{\Bar{k_1},\Bar{k_2},0\right\}}$, which implies that 
\begin{align*}
    \left|\nabla_\eta\Phi_{\nu\omega\theta}\right|\gtrsim 2^{-0.4\delta^2 m} \tag{5.35}\label{5.35}
\end{align*}
by Proposition 6.4. Thus, we must have 
\begin{align*}
    \left|\nabla_\eta\Phi_{\sigma\mu\nu}\right|\gtrsim 2^{-0.4\delta^2 m}. \tag{5.36}\label{5.36}
\end{align*}\par
\vspace{3em}
\textbf{Contribution of Medium Frequency}\par
In this part, we assume that $\min\left\{k,k_1,k_2\right\}\ge -D$. (Again keep in mind that we always have $\max\left\{k,k_1,k_2\right\}\le 2^{0.1\delta^2 m}$.)\par
We first assume that $\left|\Phi_{\sigma\mu\nu}(\xi,\eta)\right|\sim\left|\Psi_{\nu\omega\theta}(\eta)\right|\sim 2^l$. Now, we denote $\kappa_\theta\triangleq 2^{-m/2+\delta^2 m}$ 
and we, recalling (\ref{5.33}), decompose
\begin{align*}
    P_k\mathcal{B}_{m,l}\left[f^\mu_{j_1,k_1,n_1},P_{k_2}f^{\nu\omega\theta}_{C,q}\right]&=P_k\mathcal{B}_{m,l}^\parallel\left[f^\mu_{j_1,k_1,n_1},P_{k_2}f^{\nu\omega\theta}_{C,q}\right]+P_k\mathcal{B}_{m,l}^\bot\left[f^\mu_{j_1,k_1,n_1},P_{k_2}f^{\nu\omega\theta}_{C,q}\right], \\
    P_k\mathcal{B}_{m,l}^\parallel\left[f^\mu_{j_1,k_1,n_1},P_{k_2}f^{\nu\omega\theta}_{C,q}\right]&\triangleq\int q_m(s) \int_{\mathbb{R}^2} e^{is\left[\Phi_{\sigma\mu\nu}(\xi,\eta)+\Psi_{\nu\omega\theta}(\eta)\right]}\frac{\varphi_l(\Phi_{\sigma\mu\nu}(\xi,\eta))}{\Phi_{\sigma\mu\nu}(\xi,\eta)} \varphi_l^{\left(-\infty,-3\delta m\right]}(\Psi_{\nu\omega\theta}(\eta)) \\
    &\ \ \ \ \ \times\varphi(\kappa_\theta^{-1}\Omega_\eta\Phi(\xi,\eta))\widehat{f^\mu_{j_1,k_1,n_1}}(\xi-\eta) g_q(\eta,s)\,d\eta ds, \tag{5.37}\label{5.37} \\
    P_k\mathcal{B}_{m,l}^\bot\left[f^\mu_{j_1,k_1,n_1},P_{k_2}f^{\nu\omega\theta}_{C,q}\right]&\triangleq\int q_m(s) \int_{\mathbb{R}^2} e^{is\left[\Phi_{\sigma\mu\nu}(\xi,\eta)+\Psi_{\nu\omega\theta}(\eta)\right]}\frac{\varphi_l(\Phi_{\sigma\mu\nu}(\xi,\eta))}{\Phi_{\sigma\mu\nu}(\xi,\eta)}\varphi_l^{\left(-\infty,-3\delta m\right]}(\Psi_{\nu\omega\theta}(\eta))  \\
    &\ \ \ \ \ \times\left(1-\varphi(\kappa_\theta^{-1}\Omega_\eta\Phi(\xi,\eta))\right)\widehat{f^\mu_{j_1,k_1,n_1}}(\xi-\eta) g_q(\eta,s)\,d\eta ds.
\end{align*}
By Lemma 3.2 and (\ref{5.13}), we can get $\left\|P_k\mathcal{B}_{m,l}^\bot\right\|_{L^2}\lesssim 2^{-4m}$. Thus, we only need to focus on the term $P_k\mathcal{B}_{m,l}^\parallel$. Moreover, we can further decompose
\begin{align*}
    P_k\mathcal{B}_{m,l}^\parallel&=\sum_{-m\le r_0\le -3\delta m \atop -m\le r\le 0} P_k\mathcal{B}_{m,l,r_0,r}^\parallel, \\
    P_k\mathcal{B}_{m,l,r_0,r}^\parallel&\triangleq\int q_m(s) \int_{\mathbb{R}^2} e^{is\textbf{p}(\xi,\eta)}\frac{\varphi_l(\Phi_{\sigma\mu\nu}(\xi,\eta))}{\Phi_{\sigma\mu\nu}(\xi,\eta)} \varphi(\kappa_\theta^{-1}\Omega_\eta\Phi(\xi,\eta))\varphi_{r_0}^{\left[-m,-3\delta m\right]}(\textbf{p}(\xi,\eta)) \\
    &\ \ \ \ \ \times\varphi_l^{\left(-\infty,-3\delta m\right]}(\Psi_{\nu\omega\theta}(\eta))\varphi_r^{\left[-m,0\right]}(\nabla_\xi\Phi(\xi,\eta))\widehat{f^\mu_{j_1,k_1,n_1}}(\xi-\eta) g_q(\eta,s)\,d\eta ds, \tag{5.38}\label{5.38}
\end{align*}
where $\textbf{p}(\xi,\eta)\triangleq\Phi_{\sigma\mu\nu}(\xi,\eta)+\Psi_{\nu\omega\theta}(\eta)$. We will divide into two cases to deal with. From now on, in this subsection 5.4.1, we may just write $\hat{f}$ instead of $\widehat{f^\mu_{j_1,k_1}}$ or $\widehat{f^\mu_{j_1,k_1,n_1}}$ for simplicity. \par
\textbf{Case 1.} $\frac{r_0}{2}-q\le -\frac{m}{2}$ \par
First, note that fix $\eta$, we have $\left|E_\xi\right|\lesssim 2^{r_0}\cdot\kappa_\theta$ due to $\left|\textbf{p}(\xi,\eta)\right|\lesssim 2^{r_0}$; fix $\xi$, we have $\left|E_\eta\right|\lesssim 2^l\cdot\kappa_\theta$ due to $\left|\Psi_{\nu\omega\theta}(\eta)\right|\sim 2^l$. Thus, using Lemma 3.6, (\ref{5.31}), Schur's Lemma and our assumption that $\frac{r_0}{2}-q\le -\frac{m}{2}$, we get
\begin{align*}
    \left\|P_k\mathcal{B}_{m,l,r_0,0}^\parallel\right\|_{L^2}&\lesssim 2^{m-l}\,2^{\frac{r_0}{2}+\frac{l}{2}}\,\kappa_\theta\,\left\|\hat{f}\right\|_{L^\infty}\,\left\|g_q\right\|_{L^\infty}\,\left|E_\eta\right|^{1/2}\lesssim 2^{-m+6.03\delta m},
\end{align*}, 
which is acceptable as in (\ref{5.32}). Moreover, the proof of $\left\|P_k\mathcal{B}_{m,l,r_0,-m}^\parallel\right\|_{L^2}\lesssim 2^{-m+6.03\delta m}$ is exactly same as before, except that the reason why fix $\eta$, $\left|E_\xi\right|\lesssim 2^{-m}\cdot\kappa_\theta$ is that $\left|\nabla_\xi\Phi\right|\lesssim 2^{-m}$.\par
Next, when $-m<r<0$, we need to divide into several subcases.\par
\ \ \textbf{Case 1.1.}\  $m+r+100\le j\le m$,\ \ \  $j_1\le (1-\delta^2)j$ \par
In this subcase, we have $\left|\nabla_\xi\left[s\Phi_{\sigma\mu\nu}(\xi,\eta)+s\Psi_{\nu\omega\theta}(\eta)+x\cdot\xi\right]\right|\sim s\left|\nabla_\xi\Phi_{\sigma\mu\nu}(\xi,\eta)\right|+x\sim 2^j$. Hence, using the trick in (\ref{5.13}) and Lemma 3.1, we can integrate by parts in $\xi$ to obtain an acceptable control. (Again note that we don't need any decomposition here.)\par 
\ \ \textbf{Case 1.2.}\  $m+r+100\le j\le m$,\ \ \  $j_1\ge (1-\delta^2)j$,\ \ \  $n_1\ge \frac{5}{19}m$ \par
In this subcase, we use Proposition 6.10 (a), Lemma 3.6 and (\ref{5.31}) to get
\begin{align*}
    \left\|P_k\mathcal{B}_{m,l}\right\|_{L^2}&\lesssim 2^{m-l}\,2^{\frac{l}{2}-\frac{n_1}{2}}\,\left\|\sup_\theta\left|\hat{f}(r\theta)\right|\right\|_{L^2}\,\left\|g_q\right\|_{L^\infty}\,\left|E_\eta\right|^{1/2}\lesssim 2^{-j+19.5\delta m},
\end{align*}
which is acceptable as in (\ref{5.32}).\par 
\ \ \textbf{Case 1.3.}\  $m+r+100\le j\le m$,\ \ \  $j_1\ge (1-\delta^2)j$,\ \ \  $n_1\le \frac{5}{19}m$ \par
In this subcase, we have that fix $\xi$, $\left|E_\eta\right|\lesssim 2^{l}\cdot\kappa_\theta$ due to $\left|\Psi_{\nu\omega\theta}(\eta)\right|\sim 2^l$; fix $\eta$, $\left|E_\xi\right|\lesssim 2^{\frac{r_0}{2}}\cdot\kappa_\theta$ due to $\left|\textbf{p}(\xi,\eta)\right|\lesssim 2^{r_0}, \left|\nabla_{\xi\xi}\Phi_{\sigma\mu\nu}(\xi,\eta)\right|\gtrsim 1$. Then, by Schur's test (switching $\xi-\eta$ and $\eta$), we get
\begin{align*}
    \left\|P_k\mathcal{B}_{m,l,r_0,r}^\parallel\right\|_{L^2}&\lesssim 2^{m-l}\,2^{\frac{l}{2}+\frac{r_0}{4}}\,\kappa_\theta\,\left\|g_q\right\|_{L^\infty}\,\left\|f\right\|_{L^2} \lesssim 2^{-j-0.1m},
\end{align*}
which is acceptable as in (\ref{5.32}).\par
\ \ \textbf{Case 1.4.}\  $j\le m+r+100$,\ \ \  $r\le \frac{r_0}{2}$ \par
In this subcase, we have that fix $\xi$, $\left|E_\eta\right|\lesssim 2^l\cdot\kappa_\theta$ due to $\left|\Psi_{\nu\omega\theta}(\eta)\right|\sim 2^l$; fix $\eta$, $\left|E_\xi\right|\lesssim 2^r\cdot\kappa_\theta$ due to $\left|\nabla_\xi\Phi_{\sigma\mu\nu}\right|\lesssim 2^r$. Then, by Schur's test, we get
\begin{align*}
    \left\|P_k\mathcal{B}_{m,l,r_0,r}^\parallel\right\|_{L^2}&\lesssim 2^{m-l}\,2^{\frac{l}{2}+\frac{r}{2}}\,\kappa_\theta\,\left\|\hat{f}\right\|_{L^\infty}\,\left\|g_q\right\|_{L^\infty}\,\left|E_\eta\right|^{1/2}\lesssim 2^{-j+6.02\delta m},
\end{align*}
which is acceptable as in (\ref{5.32}).\par
\ \ \textbf{Case 1.5.}\  $j\le m+r+100$,\ \ \  $r\ge \frac{r_0}{2}$ \par
We can deal with this subcase similarly as Case 1.4 above. Noticing that fix $\xi$, $\left|E_\eta\right|\lesssim 2^l\cdot\kappa_\theta$ due to $\left|\Psi_{\nu\omega\theta}(\eta)\right|\sim 2^l$; fix $\eta$, $\left|E_\xi\right|\lesssim 2^{r_0-r}\cdot\kappa_\theta$ due to $\left|\textbf{p}(\xi,\eta)\right|\sim 2^{r_0}, \left|\nabla_\xi\Phi\right|\sim 2^r$, by Schur's test, we get
\begin{align*}
    \left\|P_k\mathcal{B}_{m,l,r_0,r}^\parallel\right\|_{L^2}&\lesssim 2^{m-l}\,2^{\frac{l}{2}+\frac{r_0}{2}-\frac{r}{2}}\,\kappa_\theta\,\left\|\hat{f}\right\|_{L^\infty}\,\left\|g_q\right\|_{L^\infty}\,\left|E_\eta\right|^{1/2}\lesssim 2^{-j+6.02\delta m},
\end{align*}
which is acceptable as in (\ref{5.32}).\par
Finally, sum over $r$ and $r_0$ if needed and we get "part of" (\ref{5.32}).\par
\vspace{3em}
\textbf{Case 2.} $\frac{r_0}{2}-q\ge -\frac{m}{2}$ \par
First, we still do the decomposition as in (\ref{5.38}). Let's assume that $-m\le r_0\le -\frac{2}{3}m$ and $m+r+100\le j\le m$. If $j_1\le (1-\delta^2)j$, then we have $\left|\nabla_\xi\left[s\Phi(\xi,\eta)+x\cdot\xi\right]\right|\sim \left|x\right|\sim 2^j$, which means that integration by parts in $\xi$ will give us an acceptable contribution of $P_k\mathcal{B}_{m,l,r_0,r}^\parallel$. Next, if $j_1\ge (1-\delta^2)j$ and $n_1\ge \frac{5}{19}m$, then Proposition 6.10 (a) gives that
\begin{align*}
    \left\|P_k\mathcal{B}_{m,l}^\parallel\right\|_{L^2}&\lesssim 2^{m-l}\,2^{\frac{l}{2}-\frac{n_1}{2}}\,\left\|\sup_\theta\left|\hat{f}(r\theta)\right|\right\|_{L^2}\,\left\|g_q\right\|_{L^\infty}\left|E_\eta\right|^{1/2}\lesssim 2^{-j+19.5\delta m},
\end{align*}
which is acceptable as in (\ref{5.32});
if $j_1\ge (1-\delta^2)j$ and $n_1\le \frac{5}{19}m$, then \vspace{0.4em} in view of (\ref{5.38}), we have that fix $\xi$, $\left|E_\eta\right|\lesssim 2^{l}\cdot\kappa_\theta$ due to $\left|\Psi_{\nu\omega\theta}(\eta)\right|\sim 2^l$; fix $\eta$, $\left|E_\xi\right|\lesssim 2^{\frac{r_0}{2}}\cdot\kappa_\theta$ due to $\left|\textbf{p}(\xi,\eta)\right|\lesssim 2^{r_0}, \left|\nabla_{\xi\xi}\Phi_{\sigma\mu\nu}(\xi,\eta)\right|\gtrsim 1$. Then, by Schur's test (switching $\xi-\eta$ and $\eta$), we get
\begin{align*}
    \left\|P_k\mathcal{B}_{m,l,r_0,r}^\parallel\right\|_{L^2}&\lesssim 2^{m-l}\,2^{\frac{l}{2}+\frac{r_0}{4}}\,\kappa_\theta\,\left\|g_q\right\|_{L^\infty}\,\left\|f\right\|_{L^2} \lesssim 2^{-j-0.01m},
\end{align*}
which is acceptable as in (\ref{5.32}). Therefore, from now on, we can exclude the case
$$\left\{m,l,r_0,r:-m\le r_0\le -\frac{2}{3}m \mbox{ and } m+r+100\le j\le m\right\}$$
 if needed.\par
Next, in view of (\ref{5.38}), we integrate by parts in time and write
\begin{align*}
P_k\mathcal{B}_{m,l,r_0,r}^\parallel&=P_k\mathcal{B}_{m,l,r_0,r}^{\parallel\,(1)}+P_k\mathcal{B}_{m,l,r_0,r}^{\parallel\,(2)}+P_k\mathcal{B}_{m,l,r_0,r}^{\parallel\,(3)}, \\
    P_k\mathcal{B}_{m,l,r_0,r}^{\parallel\,(1)}&\triangleq    
    \int q_m^\prime(s) \int_{\mathbb{R}^2} e^{is\textbf{p}(\xi,\eta)}\frac{\varphi_l(\Phi_{\sigma\mu\nu}(\xi,\eta))}{\textbf{p}(\xi,\eta)\cdot\Phi_{\sigma\mu\nu}(\xi,\eta)}\,\varphi_l^{\left(-\infty,-3\delta m\right]}(\Psi_{\nu\omega\theta}(\eta))\, \varphi(\kappa_\theta^{-1}\Omega_\eta\Phi(\xi,\eta)) \\
    &\ \ \ \ \ \times\varphi_{r_0}^{\left[-m,-3\delta m\right]}(\textbf{p}(\xi,\eta))\,\varphi_r^{\left[-m,0\right]}(\nabla_\xi\Phi(\xi,\eta))\,\widehat{f^\mu_{j_1,k_1,n_1}}(\xi-\eta) \,g_q(\eta,s)\,d\eta ds, \\
    P_k\mathcal{B}_{m,l,r_0,r}^{\parallel\,(2)}&\triangleq    
    \int q_m(s) \int_{\mathbb{R}^2} e^{is\textbf{p}(\xi,\eta)}\frac{\varphi_l(\Phi_{\sigma\mu\nu}(\xi,\eta))}{\textbf{p}(\xi,\eta)\cdot\Phi_{\sigma\mu\nu}(\xi,\eta)}\,\varphi_l^{\left(-\infty,-3\delta m\right]}(\Psi_{\nu\omega\theta}(\eta))\, \varphi(\kappa_\theta^{-1}\Omega_\eta\Phi(\xi,\eta)) \\
    &\ \ \ \ \ \times\varphi_{r_0}^{\left[-m,-3\delta m\right]}(\textbf{p}(\xi,\eta))\,\varphi_r^{\left[-m,0\right]}(\nabla_\xi\Phi(\xi,\eta))\,\partial_s \widehat{f^\mu_{j_1,k_1,n_1}}(\xi-\eta)\,g_q(\eta,s)\,d\eta ds, \\
    P_k\mathcal{B}_{m,l,r_0,r}^{\parallel\,(3)}&\triangleq    
    \int q_m(s) \int_{\mathbb{R}^2} e^{is\textbf{p}(\xi,\eta)}\frac{\varphi_l(\Phi_{\sigma\mu\nu}(\xi,\eta))}{\textbf{p}(\xi,\eta)\cdot\Phi_{\sigma\mu\nu}(\xi,\eta)}\,\varphi_l^{\left(-\infty,-3\delta m\right]}(\Psi_{\nu\omega\theta}(\eta)) \varphi(\kappa_\theta^{-1}\Omega_\eta\Phi(\xi,\eta)) \\
    &\ \ \ \ \ \times\varphi_{r_0}^{\left[-m,-3\delta m\right]}(\textbf{p}(\xi,\eta))\,\varphi_r^{\left[-m,0\right]}(\nabla_\xi\Phi(\xi,\eta))\,\widehat{f^\mu_{j_1,k_1,n_1}}(\xi-\eta)\,\partial_s g_q(\eta,s)\,d\eta ds.
\end{align*}\par
\vspace{2em}
\ \ \textbf{Contribution of $P_k\mathcal{B}_{m,l,r_0,r}^{\parallel\,(1)}$}\par
We consider $P_k\mathcal{B}_{m,l,r_0,r}^{\parallel\,(1)}$ first. In general, in this part, the proof would be very similar as what we did before in Case 1. As before, note that fix $\eta$, we have $\left|E_\xi\right|\lesssim 2^{r_0}\cdot\kappa_\theta$ due to $\left|\textbf{p}(\xi,\eta)\right|\lesssim 2^{r_0}$; fix $\xi$, we have $\left|E_\eta\right|\lesssim 2^l\cdot\kappa_\theta$ due to $\left|\Psi_{\nu\omega\theta}(\eta)\right|\sim 2^l$. Thus, using Lemma 3.6, (\ref{5.31}) and Schur's Lemma, we get
\begin{align*}
    \left\|P_k\mathcal{B}_{m,l,r_0,0}^{\parallel\,(1)}\right\|_{L^2}&\lesssim 2^{-r_0-l}\,2^{\frac{r_0}{2}+\frac{l}{2}}\,\kappa_\theta\,\left\|\hat{f}\right\|_{L^\infty}\,\left\|g_q\right\|_{L^\infty}\,\left|E_\eta\right|^{1/2}\lesssim 2^{-m+6.03\delta m},\tag{5.39}\label{5.39}
\end{align*} 
which is acceptable as in (\ref{5.32}). Moreover, the proof of $\left\|P_k\mathcal{B}_{m,l,r_0,-m}^{\parallel\,(1)}\right\|_{L^2}\lesssim 2^{-m+6.03\delta m}$ is exactly same as before, except that the reason why fix $\eta$, $\left|E_\xi\right|\lesssim 2^{-m}\cdot\kappa_\theta$ is that $\left|\nabla_\xi\Phi\right|\lesssim 2^{-m}$.\par
Next, when $-m<r<0$, then like before we need to divide into a couple of subcases.\par
\ \ \textbf{Case 2.1.1.}\  $m+r+100\le j\le m$,\ \ \  $j_1\le (1-\delta^2)j$ \par
This subcase can be done by integration by parts in $\xi$ like before in Case 1.1.\par
\ \ \textbf{Case 2.1.2.}\  $m+r+100\le j\le m$,\ \ \  $j_1\ge (1-\delta^2)j,$\ \ \ $n_1\ge\frac{5}{19}m$ \par
This subcase can be done by using Proposition 6.10 (a) like Case 1.2. (Note that the "coefficient" in the beginning of the integral becomes $2^{-r_0-l}$ instead of $2^{m-l}$, which is even better.)\par
\ \ \textbf{Case 2.1.3.}\  $m+r+100\le j\le m$,\ \ \  $j_1\ge (1-\delta^2)j,$\ \ \ $n_1\le\frac{5}{19}m$ \par
We still use the fact from Case 1.3 before that fix $\xi$, $\left|E_\eta\right|\lesssim 2^{l}\cdot\kappa_\theta$ due to $\left|\Psi_{\nu\omega\theta}(\eta)\right|\sim 2^l$; fix $\eta$, $\left|E_\xi\right|\lesssim 2^{\frac{r_0}{2}}\cdot\kappa_\theta$ due to $\left|\textbf{p}(\xi,\eta)\right|\lesssim 2^{r_0}, \left|\nabla_{\xi\xi}\Phi_{\sigma\mu\nu}(\xi,\eta)\right|\gtrsim 1$. Then, by Schur's test (switching $\xi-\eta$ and $\eta$), we get
\begin{align*}
    \left\|P_k\mathcal{B}_{m,l,r_0,r}^{\parallel\,(1)}\right\|_{L^2}&\lesssim 2^{-r_0-l}\,2^{\frac{l}{2}+\frac{r_0}{4}}\,\kappa_\theta\,\left\|g_q\right\|_{L^\infty}\,\left\|f\right\|_{L^2} \lesssim 2^{-j-0.1m},
\end{align*}
which is acceptable as in (\ref{5.32}).\par
\ \ \textbf{Case 2.1.4.}\  $j\le m+r+100$,\ \ \  $r\le \frac{r_0}{2}$ \par
We use Schur's test and the same volume estimates as in Case 1.4 to get that
\begin{align*}
    \left\|P_k\mathcal{B}_{m,l,r_0,r}^{\parallel (1)}\right\|_{L^2}&\lesssim 2^{r_0-l}\,2^{\frac{l}{2}+\frac{r}{2}}\,\kappa_\theta\,\left\|\hat{f}\right\|_{L^\infty}\,\left\|g_q\right\|_{L^\infty}\,\left|E_\eta\right|^{1/2}\lesssim 2^{-\frac{3}{2}j+6.02\delta m},
\end{align*}
which is acceptable as in (\ref{5.32}).\par
\ \ \textbf{Case 2.1.5.}\  $j\le m+r+100$,\ \ \  $r\ge \frac{r_0}{2}$ \par
Again, we use Schur's test and the same volume estimates as in Case 1.5 to get that
\begin{align*}
    \left\|P_k\mathcal{B}_{m,l,r_0,r}^{\parallel (1)}\right\|_{L^2}&\lesssim 2^{-r_0-l}\,2^{\frac{l}{2}+\frac{r_0}{2}-\frac{r}{2}}\,\kappa_\theta\,\left\|\hat{f}\right\|_{L^\infty}\,\left\|g_q\right\|_{L^\infty}\,\left|E_\eta\right|^{1/2}\lesssim 2^{-j+6.02\delta m},
\end{align*}
which is acceptable as in (\ref{5.32}).\par
Now, we complete the proof of contribution of $P_k\mathcal{B}_{m,l,r_0,r}^{\parallel\,(1)}$.\par
\vspace{2em}
\ \ \textbf{Contribution of $P_k\mathcal{B}_{m,l,r_0,r}^{\parallel\,(2)}$}\par
This part is a little bit complicated, since we have to decompose $\partial_s \widehat{f^\mu_{j_1,k_1,n_1}}=\partial_s f$. In view of Lemma 3.9, we write 
$$\partial_s f=\partial_s f_C+\partial_s f_{NC},$$
where 
\begin{align*}
   &\partial_s \widehat{f_C}(\xi,s)=\sum_{\mu+\nu\neq0} e^{is\Psi_{\sigma\mu\nu}(\xi)} \sum_{0\le q\le m/2-10\delta m} g^q_{\sigma\mu\nu}(\xi,s), \\
    &\left|\varphi_k(\xi) D^\alpha_\xi g^q_{\sigma\mu\nu}(\xi,s)\right|\lesssim 2^{-42\delta k_-}\,2^{-m+3\delta m}\, 2^{-q+42\delta q}\,2^{(m/2+q+2\delta^2 m)\left|\alpha\right|}.\\
    \tag{5.40}\label{5.40}
\end{align*}
Therefore, we have $\left\|\partial_s \widehat{f_C}\right\|_{L^\infty}\lesssim 2^{-m+3\delta m}\,2^{(-1+42\delta)q}$.
Correspondingly, we can write
\begin{align*}
    P_k\mathcal{B}_{m,l,r_0,r}^{\parallel\,(2)}&=P_k\mathcal{B}_{m,l,r_0,r}^{\parallel\,(2,C)}+P_k\mathcal{B}_{m,l,r_0,r}^{\parallel\,(2,NC)},\\
    P_k\mathcal{B}_{m,l,r_0,r}^{\parallel\,(2,C)}&\triangleq    
    \int q_m(s) \int_{\mathbb{R}^2} e^{is\textbf{p}(\xi,\eta)}\frac{\varphi_l(\Phi_{\sigma\mu\nu}(\xi,\eta))}{\textbf{p}(\xi,\eta)\cdot\Phi_{\sigma\mu\nu}(\xi,\eta)}\,\varphi_l^{\left(-\infty,-3\delta m\right]}(\Psi_{\nu\omega\theta}(\eta)) \varphi(\kappa_\theta^{-1}\Omega_\eta\Phi(\xi,\eta)) \\ 
    &\ \ \ \ \ \times\varphi_{r_0}^{\left[-m,-3\delta m\right]}(\textbf{p}(\xi,\eta))\,\varphi_r^{\left[-m,0\right]}(\nabla_\xi\Phi(\xi,\eta))\,\partial_s \widehat{f_C}(\xi-\eta)\,g_q(\eta,s)\,d\eta ds, \\
    P_k\mathcal{B}_{m,l,r_0,r}^{\parallel\,(2,NC)}&\triangleq    
    \int q_m(s) \int_{\mathbb{R}^2} e^{is\textbf{p}(\xi,\eta)}\frac{\varphi_l(\Phi_{\sigma\mu\nu}(\xi,\eta))}{\textbf{p}(\xi,\eta)\cdot\Phi_{\sigma\mu\nu}(\xi,\eta)}\,\varphi_l^{\left(-\infty,-3\delta m\right]}(\Psi_{\nu\omega\theta}(\eta)) \varphi(\kappa_\theta^{-1}\Omega_\eta\Phi(\xi,\eta)) \\
    &\ \ \ \ \ \times\varphi_{r_0}^{\left[-m,-3\delta m\right]}(\textbf{p}(\xi,\eta))\,\varphi_r^{\left[-m,0\right]}(\nabla_\xi\Phi(\xi,\eta))\,\partial_s \widehat{f_{NC}}(\xi-\eta)\,g_q(\eta,s)\,d\eta ds.\\
    \tag{5.41}\label{5.41}
\end{align*}
We deal with $P_k\mathcal{B}_{m,l,r_0,r}^{\parallel\,(2,C)}$ first. We may assume that $-m<r<0$, since otherwise the proof is exactly same as the one in $P_k\mathcal{B}_{m,l,r_0,r}^{\parallel\,(1)}$ (for example, see (\ref{5.39})).\par
\ \ \textbf{Case 2.2.1.}\  $m+r+100\le j\le m$\par
Recall that we have already excluded the case $-m\le r_0\le -\frac{2}{3}m$. Thus, we can assume $r_0\ge -\frac{2}{3}m$. Note that fix $\xi$, we have $\left|E_\eta\right|\lesssim 2^l\cdot\kappa_\theta$; fix $\eta$, we have $\left|E_\xi\right|\lesssim 2^{\frac{r_0}{2}}\cdot\kappa_\theta$. Applying Schur's test, we get
\begin{align*}
    \left\|P_k\mathcal{B}_{m,l,r_0,r}^{\parallel\,(2,C)}\right\|_{L^2}&\lesssim 2^{m-r_0-l}\,2^{\frac{l}{2}+\frac{r_0}{4}}\,\kappa_\theta\,\left\|g_q\right\|_{L^\infty}\,\left\|\partial_s \widehat{f_C}\right\|_{L^\infty}\,\left|E_\eta\right|^{1/2}\lesssim 2^{-m+7.6\delta m},\tag{5.42}\label{5.42}
\end{align*}
which is acceptable as in (\ref{5.32}).\par
\ \ \textbf{Case 2.2.2.}\  $j\le m+r+100$\par
The proof is exactly same as in Case 2.1.4 and Cases 2.1.5, namely using the same volume estimates and Schur's test.\par
Next, we deal with $P_k\mathcal{B}_{m,l,r_0,r}^{\parallel\,(2,NC)}$. By Duhamel's formula, we can write
\begin{align*}
    \partial_s \widehat{f_{NC}}(\xi-\eta)=\sum_{(\widebar{k_1},\widebar{j_1})\in\mathcal{J}\atop (\widebar{k_2},\widebar{j_2})\in\mathcal{J}}\left(\partial_s \widehat{f_{NC}}(\xi-\eta)\right)_{\widebar{j_1},\widebar{k_1},\widebar{j_2},\widebar{k_2}},
\end{align*}
where 
\begin{align*}
    \left(\partial_s \widehat{f_{NC}}(\xi-\eta)\right)_{\widebar{j_1},\widebar{k_1},\widebar{j_2},\widebar{k_2}}\triangleq    \int_{\mathbb{R}^2} e^{is\Phi_{\mu\mu_2\mu_3}(\xi-\eta,\theta)}\widehat{f_{\widebar{j_1},\widebar{k_1}}}(\xi-\eta-\theta)\widehat{f_{\widebar{j_2},\widebar{k_2}}}(\theta)\,d\theta. \tag{5.43}\label{5.43}
\end{align*}
For simplicity, we slightly abuse the notation and view/write $\left(\partial_s \widehat{f_{NC}}(\xi-\eta)\right)_{\widebar{j_1},\widebar{k_1},\widebar{j_2},\widebar{k_2}}$ as $\partial_s \widehat{f_{NC}}(\xi-\eta)$, since we are always able to sum up these components.\par
\ \ \textbf{Case 2.2.3.}\  $\max\left\{\widebar{k_1},\widebar{k_2}\right\}\ge \frac{\delta^2 m}{5}-D$\ \ \ (Say $\widebar{k_1}=\max\left\{\widebar{k_1},\widebar{k_2}\right\}$)\par
In this subcase, we can use the energy estimate, namely
\begin{align*}
    \left\|\partial_s \widehat{f_{NC}}\right\|_{L^\infty}&\le \left\|\widehat{f_{\widebar{j_1},\widebar{k_1}}}\right\|_{L^2}\cdot\left\|\widehat{f_{\widebar{j_2},\widebar{k_2}}}\right\|_{L^2}\lesssim 2^{-N_0\widebar{k_1}}\cdot 1 \lesssim 2^{-4m}.
\end{align*}
Thus, by (\ref{5.41}), we have 
\begin{align*}
    \left\|P_k\mathcal{B}_{m,l,r_0,r}^{\parallel\,(2,NC)}\right\|_{L^2}&\lesssim 2^{m+r_0-l}\,\left\|\partial_s \widehat{f_{NC}}\right\|_{L^\infty}
    \,\left\|g_q\right\|_{L^\infty}\,\left|E_\eta\right|^{1/2}\lesssim 2^{-2m}, \tag{5.44}\label{5.44}
\end{align*}
which is acceptable as in (\ref{5.32}).\par
\ \ \textbf{Case 2.2.4.}\  $\min\left\{\widebar{k_1},\widebar{k_2}\right\}\le -1.27m$\ \ \ (Say $\widebar{k_2}=\min\left\{\widebar{k_1},\widebar{k_2}\right\}$)\par
In this subcase, we have
\begin{align*}
    \left\|\partial_s \widehat{f_{NC}}\right\|_{L^\infty}&\le \left\|\widehat{f_{\widebar{j_1},\widebar{k_1}}}\right\|_{L^\infty}\cdot\left\|\widehat{f_{\widebar{j_2},\widebar{k_2}}}\right\|_{L^\infty}\cdot\left|E_\theta\right|\lesssim 2^{-21\delta \widebar{k_1}}\cdot 2^{-21\delta \widebar{k_2}}\cdot 2^{2\widebar{k_2}} \lesssim 2^{-2.52m}.
\end{align*} 
Then, we proceed as (\ref{5.44}) and get an acceptable contribution as in (\ref{5.32}).\par
\ \ \textbf{Case 2.2.5.}\  $-1.27m\le \widebar{k_1},\widebar{k_2}\le \frac{\delta^2 m}{5}-D$,\ \ \ $\widebar{j_2}=\max\left\{\widebar{j_1},\widebar{j_2}\right\}\ge m-10\delta m-D^2$\par
In this part, we will prove that \vspace{0.4em} either $\left\|\partial_s \widehat{f_{NC}}\right\|_{L^\infty}\lesssim 2^{-m+7\delta m}$ or $\left\|\partial_s \widehat{f_{NC}}\right\|_{L^2}\lesssim 2^{-\frac{3}{2}m}$. If so, then we can proceed our proof as the one of $\left\|P_k\mathcal{B}_{m,l,r_0,r}^{\parallel\,(2,C)}\right\|_{L^2}$. To prove the $L^\infty$ or $L^2$ norm of $\partial_s \widehat{f_{NC}}$, we need to decompose $\partial_s \widehat{f_{NC}}$ more precisely, namely writing
\begin{align*}
    \partial_s \widehat{f_{NC}}(\xi-\eta)=\sum_{(\widebar{k_i},\widebar{j_i})\in\mathcal{J}\ i=1,2\atop \widebar{n_2}\in\left\{0,\dots,\widebar{j_2}+1\right\}}\left(\partial_s \widehat{f_{NC}}(\xi-\eta)\right)_{\widebar{j_1},\widebar{k_1},\widebar{j_2},\widebar{k_2},\widebar{n_2}},
\end{align*}
where 
\begin{align*}
    \left(\partial_s \widehat{f_{NC}}(\xi-\eta)\right)_{\widebar{j_1},\widebar{k_1},\widebar{j_2},\widebar{k_2},\widebar{n_2}}\triangleq    \int_{\mathbb{R}^2} e^{is\Phi_{\mu\mu_2\mu_3}(\xi-\eta,\theta)}\widehat{f_{\widebar{j_1},\widebar{k_1}}}(\xi-\eta-\theta)\widehat{f_{\widebar{j_2},\widebar{k_2},\widebar{n_2}}}(\theta)\,d\theta. \tag{5.43a}\label{5.43a}
\end{align*}
Once again, we slightly abuse the notation and view/write $\left(\partial_s \widehat{f_{NC}}(\xi-\eta)\right)_{\widebar{j_1},\widebar{k_1},\widebar{j_2},\widebar{k_2},\widebar{n_2}}$ as $\partial_s \widehat{f_{NC}}(\xi-\eta)$.\par
In fact, if $\widebar{n_2}\le \frac{18}{19}m$, then by Lemma 3.5 and Lemma 3.6, we can get
\begin{align*}
    \left\|\partial_s \widehat{f_{NC}}\right\|_{L^2}&\lesssim \left\|e^{is\Lambda_{\mu_2}}f_{\widebar{j_1},\widebar{k_1}}\right\|_{L^\infty}\cdot\left\|f_{\widebar{j_2},\widebar{k_2},\widebar{n_2}}\right\|_{L^2} \lesssim 2^{-\frac{3}{2}m-\frac{1}{38}m+40.01\delta m}.
\end{align*}
If $\widebar{n_2}\ge \frac{18}{19}m$ and $\widebar{k_1}\ge -0.2m$, then in view of (\ref{5.43}), we apply Lemma 3.6 and get
\begin{align*}
    \left\|\partial_s \widehat{f_{NC}}\right\|_{L^\infty}&\lesssim \left\|\widehat{f_{\widebar{j_1},\widebar{k_1}}}\right\|_{L^\infty}\cdot\left\|\widehat{f_{\widebar{j_2},\widebar{k_2}}}\right\|_{L^1}\lesssim 2^{-m+6.21\delta m}.
\end{align*}
If $\widebar{n_2}\ge \frac{18}{19}m$, $\widebar{k_1}\le -0.2m$ and $\widebar{j_1}\le \frac{2}{5}m-\delta^2 m$, then $-\frac{2}{5}m+\delta^2 m\le \widebar{k_1}\le -0.2m$, which \vspace{0.3em} implies that we can insert the cutoff function of $\angle\xi,\eta$, namely doing the following decomposition of (\ref{5.43a})
\begin{align*}
    \partial_s \widehat{f_{NC}}(\xi-\eta)&=\left(\partial_s \widehat{f_{NC}}\right)^\parallel(\xi-\eta)+\left(\partial_s \widehat{f_{NC}}\right)^\bot(\xi-\eta), \\
    \left(\partial_s \widehat{f_{NC}}\right)^\parallel(\xi-\eta)&\triangleq\int_{\mathbb{R}^2} e^{is\Phi_{\mu\mu_2\mu_3}(\xi-\eta,\theta)}\varphi(\kappa_\theta^{-1}\Omega_\theta\Phi_{\mu\mu_2\mu_3}(\xi-\eta,\theta))\widehat{f_{\widebar{j_1},\widebar{k_1}}}(\xi-\eta-\theta)\reallywidehat{f_{\widebar{j_2},\widebar{k_2},\widebar{n_2}}}(\theta)\,d\theta, \\
    \left(\partial_s \widehat{f_{NC}}\right)^\bot(\xi-\eta)&\triangleq\int_{\mathbb{R}^2} e^{is\Phi_{\mu\mu_2\mu_3}(\xi-\eta,\theta)}\left(1-\varphi(\kappa_\theta^{-1}\Omega_\theta\Phi_{\mu\mu_2\mu_3}(\xi-\eta,\theta))\right) \\
    &\ \ \ \ \ \ \ \ \ \ \ \ \ \ \times\widehat{f_{\widebar{j_1},\widebar{k_1}}}(\xi-\eta-\theta)\reallywidehat{f_{\widebar{j_2},\widebar{k_2},\widebar{n_2}}}(\theta)\,d\theta.
\end{align*}
We only need to consider $\left(\partial_s \widehat{f_{NC}}\right)^\bot$, since Lemma 3.2 guarantees enough control of $\left(\partial_s \widehat{f_{NC}}\right)^\parallel$. Note that fix $\xi-\eta$, we have $\left|E_\theta\right|\lesssim 2^{-2\widebar{n_2}}$; fix $\theta$, we have $\left|E_{\xi-\eta}\right|\lesssim 2^{\widebar{k_1}}\cdot\kappa_\theta$. Thus, by Schur's test and Lemma 3.6, we get
\begin{align*}
    \left\|\left(\partial_s \widehat{f_{NC}}\right)^\bot\right\|_{L^2}&\lesssim 2^{-\widebar{n_2}+\frac{\widebar{k_1}}{2}-\frac{m}{4}+\frac{1}{2}\delta^2 m}\left\|\widehat{f_{\widebar{j_1},\widebar{k_1}}}\right\|_{L^\infty}\,\left\|\reallywidehat{f_{\widebar{j_2},\widebar{k_2},\widebar{n_2}}}\right\|_{L^2}\lesssim 2^{-1.7m+31\delta m}.\tag{5.45}\label{5.45}
\end{align*}
Therefore, we get $\left\|\partial_s \widehat{f_{NC}}\right\|_{L^2}\lesssim 2^{-1.7m+31\delta m}$. Finally, if $\widebar{n_2}\ge \frac{18}{19}m$, $\widebar{k_1}\le -0.2m$ and $\widebar{j_1}\ge \frac{2}{5}m-\delta^2 m$, then we cannot integrate by parts to the angle $\angle\xi,\eta$. However, we still observe that fix  $\xi-\eta$, we have $\left|E_\theta\right|\lesssim 2^{-2\widebar{n_2}}$; fix $\theta$, we have $\left|E_{\xi-\eta}\right|\lesssim 2^{\widebar{k_1}}$. Then, by Schur's test and Lemma 3.6, as in (\ref{5.45}), we get
\begin{align*}
    \left\|\partial_s \widehat{f_{NC}}\right\|_{L^2}&\lesssim 2^{-\widebar{n_2}+\frac{\widebar{k_1}}{2}}\left\|\widehat{f_{\widebar{j_1},\widebar{k_1}}}\right\|_{L^\infty}\,\left\|\reallywidehat{f_{\widebar{j_2},\widebar{k_2},\widebar{n_2}}}\right\|_{L^2}\lesssim 2^{-1.67m+41\delta m}.
\end{align*}
To sum up, in all cases, we have proved that either $\left\|\partial_s \widehat{f_{NC}}\right\|_{L^\infty}\lesssim 2^{-m+7\delta m}$ or $\left\|\partial_s \widehat{f_{NC}}\right\|_{L^2}\lesssim 2^{-\frac{3}{2}m}$, which is enough.\par
\ \ \textbf{Case 2.2.6.}\  $-1.27m\le \widebar{k_1},\widebar{k_2}\le \frac{\delta^2 m}{5}-D$,\ \ \ $0\le \widebar{j_1}\le \widebar{j_2}\le m-10\delta m-D^2$\par
Recall that for $i=1,2$, we have $\widebar{k_i}+\widebar{j_i}\ge 0$, so $\widebar{k_i}\ge -(1-10\delta)m+D^2$. In this subcase, in view of Lemma 3.9, we know that $\partial_s \widehat{f_{NC}}$ occurs in the following cases. Once again we denote $\kappa_r\triangleq 2^{\delta^2 m/2}\left(2^{-m/2}+2^{j_2-m}\right)$ and $\kappa_\theta\triangleq 2^{2\delta^2 m-m/2}$.\par
\underline{(a.1). $\left|\nabla_\theta\Phi_{\mu\mu_2\mu_3}\right|\lesssim \kappa_r$, \ \ \ $\mu_2+\mu_3\neq 0$, \ \ \ $2^{k_2}\ge 2^{\delta m}\left(2^{-m/2}+2^{\widebar{j_2}-m}\right)$, \ \ \ $\frac{m}{2}\le \widebar{j_2}\le (1-10\delta)m$}\vspace{0.4em} \\
\underline{$k_2\le-D$.}\par
\underline{(a.2). $\left|\nabla_\theta\Phi_{\mu\mu_2\mu_3}\right|\lesssim \kappa_r$, \ \ \ $\mu_2+\mu_3\neq 0$, \ \ \ $2^{k_2}\ge 2^{\delta m}\left(2^{-m/2}+2^{\widebar{j_2}-m}\right)$, \ \ \ $\frac{m}{2}\le \widebar{j_2}\le (1-10\delta)m$}\vspace{0.4em} \\
\underline{$k_2\ge-D$,\ \ \ $\left|\Omega_\theta\Phi_{\mu\mu_2\mu_3}\right|\gtrsim \kappa_\theta$.}\par
\underline{(b). $\left|\nabla_\theta\Phi_{\mu\mu_2\mu_3}\right|\lesssim \kappa_r$, \ \ \ $\mu_2+\mu_3\neq 0$, \ \ \ $2^{k_2}\le 2^{\delta m}\left(2^{-m/2}+2^{\widebar{j_2}-m}\right)$.}\par
\underline{(c). $\left|\nabla_\theta\Phi_{\mu\mu_2\mu_3}\right|\lesssim \kappa_r$, \ \ \ $\mu_2+\mu_3=0$, \ \ \ $2^{k_2}\le 2^{\delta m}\left(2^{-m/2}+2^{\widebar{j_2}-m}\right)$, \ \ \ $m+k_2\le \widebar{j_2}+3\delta^2 m$}\par
\begin{align*}
   \  \tag{5.46}\label{5.46}
\end{align*}
\ \ \ \ In the case (a.2), we can integrate by parts to the angle $\angle\xi,\eta$ to get a sufficient small control of $\partial_s \widehat{f_{NC}}$ (See P.828 of \cite{y2}). In all other cases, we notice that $k_2\le -D$, which contradicts our assumption $\min\left\{k,k_1,k_2\right\}>-D$ in the beginning.\par
Thus, we have already completed the proof of $P_k\mathcal{B}_{m,l,r_0,r}^{\parallel\,(2)}$.\par
\vspace{2em}
\ \ \textbf{Contribution of $P_k\mathcal{B}_{m,l,r_0,r}^{\parallel\,(3)}$}\par
Once again, in general, this part could be proved quite similar as the proof of $P_k\mathcal{B}_{m,l,r_0,r}^{\parallel\,(1)}$. The only main difference here is that we have to use the \vspace{0.4em} assumption $\frac{r_0}{2}-q\ge -\frac{m}{2}$ now. Recall that we have $\left\|\partial_s g_q(\xi,s)\right\|_{L^\infty}\lesssim 2^{-2m+q+6.01\delta m}$. As before, note \vspace{0.4em} that fix $\eta$, we have $\left|E_\xi\right|\lesssim 2^{r_0}\cdot\kappa_\theta$ due to $\left|\textbf{p}(\xi,\eta)\right|\lesssim 2^{r_0}$; fix $\xi$, we have $\left|E_\eta\right|\lesssim 2^l\cdot\kappa_\theta$ due to $\left|\Psi_{\nu\omega\theta}(\eta)\right|\sim 2^l$. Thus, using Lemma 3.6, (\ref{5.31}) and Schur's Lemma, we get
\begin{align*}
    \left\|P_k\mathcal{B}_{m,l,r_0,0}^{\parallel\,(3)}\right\|_{L^2}&\lesssim 2^{m-r_0-l}\,2^{\frac{r_0}{2}+\frac{l}{2}}\,\kappa_\theta\,\left\|\hat{f}\right\|_{L^\infty}\,\left\|\partial_s g_q\right\|_{L^\infty}\,\left|E_\eta\right|^{1/2} \lesssim 2^{-m+8.03\delta m},
\end{align*} 
which is acceptable as in (\ref{5.32}). Moreover, the proof of $\left\|P_k\mathcal{B}_{m,l,r_0,-m}^{\parallel\,(3)}\right\|_{L^2}\lesssim 2^{-m+8.03\delta m}$ is exactly same as before, except that the reason why fix $\eta$, $\left|E_\xi\right|\lesssim 2^{-m}\cdot\kappa_\theta$ is that $\left|\nabla_\xi\Phi\right|\lesssim 2^{-m}$.\par
Next, when $-m<r<0$, then like before we again need to divide into a couple of subcases.\par
\ \ \textbf{Case 2.3.1.}\  $m+r+100\le j\le m$,\ \ \  $j_1\le (1-\delta^2)j$ \par
This subcase can be done by integration by parts in $\xi$ like before in Case 1.1.\par
\ \ \textbf{Case 2.3.2.}\  $m+r+100\le j\le m$,\ \ \  $j_1\ge (1-\delta^2)j,$\ \ \ $n_1\ge\frac{7}{19}m$ \par
This subcase can be done by using Proposition 6.10 (a) like Case 1.2 or Case 2.1.2. In fact, we use Proposition 6.10 (a), Lemma 3.6 and (\ref{5.31}) to get 
\begin{align*}
    \left\|P_k\mathcal{B}_{m,l,r_0}^{\parallel\,(3)}\right\|_{L^2}&\lesssim 2^{m-r_0-l}\,2^{\frac{l}{2}-\frac{n_1}{2}}\,\left\|\sup_\theta\left|\hat{f}(r\theta)\right|\right\|_{L^2}\,\left\|\partial_s g_q\right\|_{L^\infty}\,\left|E_\eta\right|^{1/2} \lesssim 2^{-j+19.1\delta m},
\end{align*}
which is acceptable as in (\ref{5.32}).\par 
\ \ \textbf{Case 2.3.3.}\  $m+r+100\le j\le m$,\ \ \  $j_1\ge (1-\delta^2)j,$\ \ \ $n_1\le\frac{7}{19}m$ \par
This subcase can be done like Case 1.3 or Case 2.1.3 In fact, we have that fix $\xi$, $\left|E_\eta\right|\lesssim 2^{l}\cdot\kappa_\theta$ due to $\left|\Psi_{\nu\omega\theta}(\eta)\right|\sim 2^l$; fix $\eta$, $\left|E_\xi\right|\lesssim 2^{\frac{r_0}{2}}\cdot\kappa_\theta$ due to $\left|\textbf{p}(\xi,\eta)\right|\lesssim 2^{r_0}, \left|\nabla_{\xi\xi}\Phi_{\sigma\mu\nu}(\xi,\eta)\right|\gtrsim 1$. Then, by Schur's test (switching $\xi-\eta$ and $\eta$), we get
\begin{align*}
    \left\|P_k\mathcal{B}_{m,l,r_0,r}^{\parallel\,(3)}\right\|_{L^2}&\lesssim 2^{m-r_0-l}\,2^{\frac{l}{2}+\frac{r_0}{4}}\,\kappa_\theta\,\left\|\partial_s g_q\right\|_{L^\infty}\,\left\|f\right\|_{L^2}\lesssim 2^{-j-0.05m},
\end{align*}
which is acceptable as in (\ref{5.32}).\par
\ \ \textbf{Case 2.3.4.}\  $j\le m+r+100$,\ \ \  $r\le \frac{r_0}{2}$ \par
We use Schur's test and the same volume estimates as in Case 1.4 or Case 2.1.4 to get that
\begin{align*}
    \left\|P_k\mathcal{B}_{m,l,r_0,r}^{\parallel (3)}\right\|_{L^2}&\lesssim 2^{m-r_0-l}\,2^{\frac{l}{2}+\frac{r}{2}}\,\kappa_\theta\,\left\|\hat{f}\right\|_{L^\infty}\,\left\|\partial_s g_q\right\|_{L^\infty}\,\left|E_\eta\right|^{1/2} \lesssim 2^{-j+\frac{r}{2}+8.02\delta m},
\end{align*}
which is acceptable as in (\ref{5.32}).\par
\ \ \textbf{Case 2.3.5.}\  $j\le m+r+100$,\ \ \  $r\ge \frac{r_0}{2}$ \par
Again, we use Schur's test and the same volume estimates as in Case 1.5 or Case 2.1.5 to get that
\begin{align*}
    \left\|P_k\mathcal{B}_{m,l,r_0,r}^{\parallel (3)}\right\|_{L^2}&\lesssim 2^{m-r_0-l}\,2^{\frac{l}{2}+\frac{r_0}{2}-\frac{r}{2}}\,\kappa_\theta\,\left\|\hat{f}\right\|_{L^\infty}\,\left\|\partial_s g_q\right\|_{L^\infty}\,\left|E_\eta\right|^{1/2} \lesssim 2^{-j+8.02\delta m},
\end{align*}
which is acceptable as in (\ref{5.32}).\par
Now, the proof of $P_k\mathcal{B}_{m,l,r_0,r}^{\parallel (3)}$ is completed.\par
\vspace{2em}
To sum up, by summing over $r$ and $r_0$ if necessary, we have showed that if $\min\left\{k,k_1,k_1\right\}\ge -D$ and $\left|\Psi_{\nu\omega\theta}(\eta)\right|\sim 2^l$, then $\left\|Q_{jk}\mathcal{B}_{m,l}\right\|_{L^2}\lesssim 2^{-(1-20\delta)j-0.001\delta^2 m}$ as in (\ref{5.32}). Finally, in general, we may assume that $\left|\Psi_{\nu\omega\theta}(\eta)\right|\sim 2^{l_1}$, then we have $\left|\textbf{p}\right|\sim 2^{\max\left\{l,l_1\right\}}\triangleq 2^{r_0}$. If $l_1>l$, then $l_1=r_0$. In this case, the above proof still holds. For example, we still have that fix $\xi$, $\left|E_\eta\right|\lesssim 2^{l+0.4\delta^2 m}\cdot\kappa_\theta$, whose reason is that $\left|\Phi_{\sigma\mu\nu}\right|\sim 2^l$ and $\left|\nabla_\eta\Phi_{\sigma\mu\nu}\right|\gtrsim 2^{-0.4\delta^2 m}$. The rest of proof is exactly same. If $l_1<l$, then $\left|\Phi_{\sigma\mu\nu}\right|\sim 2^l=2^{r_0}$, $\left|\Psi_{\nu\omega\theta}\right|\sim 2^{l_1}$ and $\left|\textbf{p}\right|\sim 2^l=2^{r_0}$. Like before, we still have that fix $\xi$, $\left|E_\eta\right|\lesssim 2^{l+0.4\delta^2 m}\cdot\kappa_\theta$, so the rest of the proof is also exactly same. Thus, we have already proved that if $\min\left\{k,k_1,k_1\right\}\ge -D$ , then $\left\|Q_{jk}\mathcal{B}_{m,l}\right\|_{L^2}\lesssim 2^{-(1-20\delta)j-0.001\delta^2 m}$ as in (\ref{5.32}).\par
\vspace{3em}
\textbf{Contribution of Low Frequency}\par
In this part, we assume that $\min\left\{k,k_1,k_2\right\}\le -D$. \par
First, we assume that $k_1\le j_1-m+\delta^2 m$. In this case, we have $k_1\le -D$. Since $b_\sigma-b_\mu-b_\nu\neq 0$, we know that at most one of $k,k_1,k_2$ is less than $-D$, which implies that $k,k_2\ge -D$. Recall that 
\begin{align*}
    \nabla_\xi\Phi=\pm\frac{c_\sigma^2\xi}{\sqrt{c_\sigma^2\left|\xi\right|^2+b_\sigma^2}}\mp\frac{c_\mu^2(\xi-\eta)}{\sqrt{c_\mu^2\left|\xi-\eta\right|^2+b_\mu^2}},\tag{5.47}\label{5.47}
\end{align*}
and we can get that $\left|\nabla_\xi\Phi\right|\gtrsim 1$, since when $\left|\xi\right|\gtrsim 1$, $\left|\frac{c_\sigma^2\xi}{\sqrt{c_\sigma^2\left|\xi\right|^2+b_\sigma^2}}\right|\sim\left|\xi\right|$ and when $\left|\xi-\eta\right|\ll 1$, $\left|\frac{c_\mu^2(\xi-\eta)}{\sqrt{c_\mu^2\left|\xi-\eta\right|^2+b_\mu^2}}\right|\ll 1$. Then, due to the fact that $\left|\Phi_{\sigma\mu\nu}\right|\sim 2^l$, $\left|\nabla_\xi\Phi_{\sigma\mu\nu}\right|\gtrsim 1$ and  $\left|\nabla_\eta\Phi_{\sigma\mu\nu}\right|\gtrsim 2^{-0.4\delta^2 m}$, we observe that fix $\eta$, $\left|E_\xi\right|\lesssim 2^l\cdot 2^{k_1}$; fix $\xi$, $\left|E_\eta\right|\lesssim 2^{l+0.4\delta^2 m}\cdot 2^{k_1}$. Also, \vspace{0.4em} $k_1\le -D$ implies $n_1=0$, so we must have $\left\|f\right\|_{L^2}\lesssim 2^{-(1-20\delta)j_1}$. Then by Schur's test (switching $\xi-\eta$ and $\eta$) and (\ref{3.40}), we get
\begin{align*}
    \left\|P_k \mathcal{B}_{m,l}\right\|_{L^2}&\lesssim 2^{m-l}\,2^{l+0.2\delta^2 m+k_1}\,\left\|\hat{f}\right\|_{L^2}\,\left\|g_q\right\|_{L^\infty}\lesssim 2^{-m+16.01\delta m},
\end{align*}
which is acceptable as in (\ref{5.32}).\par
Next, we assume that $k_1\ge j_1-m+\delta^2 m$. This implies that $k_1\ge -0.5m+\frac{1}{2}\delta^2 m$. (Otherwise $k_1\le -0.5m+\frac{1}{2}\delta^2 m$, so $j_1\ge 0.5m-\frac{1}{2}\delta^2 m$. This gives us $k_1\le j_1-m+\frac{1}{2}\delta^2 m<j_1-m+\delta^2 m$, which contradicts our assumption.) By Lemma 3.6, this tells us that 
\begin{align*}
    \left\|\hat{f}\right\|_{L^\infty}\lesssim 2^{-21\delta k_1}\lesssim 2^{-10.6\delta m}. \tag{5.48}\label{5.48}
\end{align*}
Note that in this case, we are able to integrate by parts to the angle $\angle\xi,\eta$ if needed. Thus, by using (\ref{5.48}) and (\ref{3.40}), we can copy the proof in the medium frequency except the part of (\ref{5.46}), since it's possible that $k_1\le -D$ right now.\par
To deal with the part of (\ref{5.46}), we have to analyze $\partial_s \widehat{f_{NC}}$ more precisely. First, note that in all cases where $k_1\le -D$ in (\ref{5.46}) (i.e. Case (a.1), Case (b) and Case (c)) we all have $k_1\le -D$ and $\left|\nabla_\theta\Phi_{\mu\mu_2\mu_3}\right|\ll 1$. Then, by Proposition 6.5 (a), we know that $\left|\theta\right|\approx\left|p(\xi-\eta)\right|\lesssim \left|\xi-\eta\right|\le 2^{-D}$, namely $k_1,\widebar{k_1},\widebar{k_2}\le -D$. This implies that $\left|\Phi_{\mu\mu_2\mu_3}\right|\gtrsim 1$ due to that $b_\mu-b_{\mu_2}-b_{\mu_3}\neq 0$. Also, since $k_1\le -D$, by the argument at (\ref{5.47}), we can get that $\left|\nabla_\xi\Phi_{\sigma\mu\nu}\right|\gtrsim 1$.\par
Now, according to the proof of 
\begin{align*}
    &I^{hi}\left[f,g\right]\triangleq\mathcal{F}^{-1}\int_{\mathbb{R}^2} e^{is\Phi_{\mu\mu_2\mu_3}(\xi-\eta,\theta)}\varphi_{hi}\left(\Phi(\xi-\eta,\theta)\right)\widehat{f}(\xi-\eta-\theta,s)\widehat{g}(\theta,s)\,d\theta, \\
    &\varphi_{hi}(x)=\varphi_{>\left[-3\delta m-4\delta^2 m\right]}(x)
\end{align*}
in Lemma 3.12, we could decompose
$$\partial_s \widehat{f_{NC}}=f^*_{NC}+\left(\partial_s F_C+\partial_s F_{NC}+\partial_s F_{LO}\right)\triangleq I^*_{NC}+J,$$
where
\begin{equation}
    \begin{aligned}
    &\left\|f^*_{NC}\right\|\lesssim 2^{-2m+50\delta m}, \\
    &\left\|\widehat{F_C}\right\|_{L^\infty}\lesssim 2^{-m+3.2\delta m+10\delta^2 m}, \\
    &\left\|F_{NC}\right\|_{L^2}\lesssim 2^{-1.025m}, \\
    &\left\|\left(1+2^m\left|\xi-\eta\right|\right)\widehat{F_{LO}}\right\|_{L^\infty}\lesssim 2^{5\delta m}, \ \ \ \ \ P_{\ge -\frac{13}{15}m} F_{LO}\equiv 0, \\
    &\ \ \ \ \ \ \ \left(\Rightarrow\left\|\widehat{F_{LO}}\right\|_{L^\infty}\lesssim \begin{cases}
        2^{-m+5\delta m-k_1}&\mbox{, if }\left|\xi-\eta\right|\ge 2^{-m}\\
        2^{5\delta m}&\mbox{, if }\left|\xi-\eta\right|\le 2^{-m}
    \end{cases} \right).
\end{aligned}\tag{5.49}\label{5.49}
\end{equation}
Also,  we denote $P_k\mathcal{B}_{m,l,r_0,r}^{\parallel\,(2,NC,b)}$ corresponding to the  $I^*_{NC}$ part, and denote $P_k\mathcal{B}_{m,l,r_0,r}^{\parallel\,(2,NC,c)}$ corresponding to the $J$ part. Namely, we have
\begin{align*}
    P_k\mathcal{B}_{m,l,r_0,r}^{\parallel\,(2,NC,b)}&\triangleq    
    \int q_m(s) \int_{\mathbb{R}^2} e^{is\textbf{p}(\xi,\eta)}\frac{\varphi_l(\Phi_{\sigma\mu\nu}(\xi,\eta))}{\textbf{p}(\xi,\eta)\cdot\Phi_{\sigma\mu\nu}(\xi,\eta)}\,\varphi_l^{\left(-\infty,-3\delta m\right]}(\Psi_{\nu\omega\theta}(\eta)) \varphi(\kappa_\theta^{-1}\Omega_\eta\Phi(\xi,\eta)) \\
    &\ \ \ \ \ \times\varphi_{r_0}^{\left[-m,-3\delta m\right]}(\textbf{p}(\xi,\eta))\,\varphi_r^{\left[-m,0\right]}(\nabla_\xi\Phi(\xi,\eta))\,I^*_{NC}(\xi-\eta)\,g_q(\eta,s)\,d\eta ds, \\
    P_k\mathcal{B}_{m,l,r_0,r}^{\parallel\,(2,NC,c)}&\triangleq    
    \int q_m(s) \int_{\mathbb{R}^2} e^{is\textbf{p}(\xi,\eta)}\frac{\varphi_l(\Phi_{\sigma\mu\nu}(\xi,\eta))}{\textbf{p}(\xi,\eta)\cdot\Phi_{\sigma\mu\nu}(\xi,\eta)}\,\varphi_l^{\left(-\infty,-3\delta m\right]}(\Psi_{\nu\omega\theta}(\eta)) \varphi(\kappa_\theta^{-1}\Omega_\eta\Phi(\xi,\eta)) \\
    &\ \ \ \ \ \times\varphi_{r_0}^{\left[-m,-3\delta m\right]}(\textbf{p}(\xi,\eta))\,\varphi_r^{\left[-m,0\right]}(\nabla_\xi\Phi(\xi,\eta))\,J(\xi-\eta)\,g_q(\eta,s)\,d\eta ds. \\
\end{align*}\par
As for $P_k\mathcal{B}_{m,l,r_0,r}^{\parallel\,(2,NC,b)}$, we notice that fix $\eta$, we have $\left|E_\xi\right|\lesssim 2^{r_0}\cdot \kappa_\theta$ due to that $\left|\textbf{p}\right|\lesssim 2^{r_0}$ and $\left|\nabla_\xi\Phi_{\sigma\mu\nu}\right|\gtrsim 1$; fix $\xi$, we have $\left|E_\eta\right|\lesssim 2^{l+0.4\delta^2 m}\cdot \kappa_\theta$ due to that $\left|\Phi_{\sigma\mu\nu}\right|\sim 2^l$ and $\left|\nabla_\eta\Phi_{\sigma\mu\nu}\right|\gtrsim 2^{-0.4\delta^2 m}$. Then by Schur's test, (\ref{5.31}) and (\ref{5.49}), we get
\begin{align*}
    \left\|P_k\mathcal{B}_{m,l,r_0,r}^{\parallel\,(2,NC,b)}\right\|_{L^2}&\lesssim 2^{m-r_0-l}\,2^{\frac{l}{2}+\frac{r_0}{2}+0.2\delta^2 m}\,\kappa_\theta\,2^{-m+4.01\delta m}\,2^{-2m+50\delta m}\lesssim 2^{-\frac{3}{2}m+75\delta m},\tag{5.50}\label{5.50}
\end{align*}
which is acceptable as in (\ref{5.32}).\par
As for $P_k\mathcal{B}_{m,l,r_0,r}^{\parallel\,(2,NC,c)}$, we have to do the integration by parts in time again. So we write
\begin{align*}
P_k\mathcal{B}_{m,l,r_0,r}^{\parallel\,(2,NC,c)}&=\sum_{\cdot\,\in\left\{NC,C,LO\right\}}\left[P_k\mathcal{B}_{m,l,r_0,r}^{\parallel\,(2, NC,c_1,\cdot\,)}+P_k\mathcal{B}_{m,l,r_0,r}^{\parallel\,(2,NC,c_2,\cdot\,)}+P_k\mathcal{B}_{m,l,r_0,r}^{\parallel\,(2,NC,c_3,\cdot\,)}\right], \\
    P_k\mathcal{B}_{m,l,r_0,r}^{\parallel\,(2,NC,c_1,\cdot\,)}&\triangleq    
    \int q_m^\prime(s) \int_{\mathbb{R}^2} e^{is\textbf{p}(\xi,\eta)}\frac{\varphi_l(\Phi_{\sigma\mu\nu}(\xi,\eta))}{\textbf{p}(\xi,\eta)\cdot\Phi_{\sigma\mu\nu}(\xi,\eta)}\,\varphi_l^{\left(-\infty,-3\delta m\right]}(\Psi_{\nu\omega\theta}(\eta)) \varphi(\kappa_\theta^{-1}\Omega_\eta\Phi(\xi,\eta)) \\
    &\ \ \ \ \ \times\varphi_{r_0}^{\left[-m,-3\delta m\right]}(\textbf{p}(\xi,\eta))\,\varphi_r^{\left[-m,0\right]}(\nabla_\xi\Phi(\xi,\eta))\,\widehat{F_\cdot}(\xi-\eta) \,g_q(\eta,s)\,d\eta ds, \\
    P_k\mathcal{B}_{m,l,r_0,r}^{\parallel\,(2,NC,c_2,\cdot\,)}&\triangleq    
    \int q_m(s) \int_{\mathbb{R}^2} e^{is\textbf{p}(\xi,\eta)}\frac{\varphi_l(\Phi_{\sigma\mu\nu}(\xi,\eta))}{\Phi_{\sigma\mu\nu}(\xi,\eta)}\,\varphi_l^{\left(-\infty,-3\delta m\right]}(\Psi_{\nu\omega\theta}(\eta)) \varphi(\kappa_\theta^{-1}\Omega_\eta\Phi(\xi,\eta)) \\
    &\ \ \ \ \ \times\varphi_{r_0}^{\left[-m,-3\delta m\right]}(\textbf{p}(\xi,\eta))\,\varphi_r^{\left[-m,0\right]}(\nabla_\xi\Phi(\xi,\eta))\,\partial_s \widehat{F_\cdot}(\xi-\eta)\,g_q(\eta,s)\,d\eta ds, \\
    P_k\mathcal{B}_{m,l,r_0,r}^{\parallel\,(2,NC,c_3,\cdot\,)}&\triangleq    
    \int q_m(s) \int_{\mathbb{R}^2} e^{is\textbf{p}(\xi,\eta)}\frac{\varphi_l(\Phi_{\sigma\mu\nu}(\xi,\eta))}{\textbf{p}(\xi,\eta)\cdot\Phi_{\sigma\mu\nu}(\xi,\eta)}\,\varphi_l^{\left(-\infty,-3\delta m\right]}(\Psi_{\nu\omega\theta}(\eta)) \varphi(\kappa_\theta^{-1}\Omega_\eta\Phi(\xi,\eta)) \\
    &\ \ \ \ \ \times\varphi_{r_0}^{\left[-m,-3\delta m\right]}(\textbf{p}(\xi,\eta))\,\varphi_r^{\left[-m,0\right]}(\nabla_\xi\Phi(\xi,\eta))\,\widehat{F_\cdot}(\xi-\eta)\,\partial_s g_q(\eta,s)\,d\eta ds,
\end{align*}
where $F_\cdot(\xi-\eta)=F_{NC}(\xi-\eta),F_C(\xi-\eta)\mbox{ or }F_{LO}(\xi-\eta)$.
Note that $\left|P_k\mathcal{B}_{m,l,r_0,r}^{\parallel\,(2,NC,c_2)}\right|\ge\left|P_k\mathcal{B}_{m,l,r_0,r}^{\parallel\,(2,NC,c_1)}\right|$, so we only need to consider $P_k\mathcal{B}_{m,l,r_0,r}^{\parallel\,(2,NC,c_2)}$ and $P_k\mathcal{B}_{m,l,r_0,r}^{\parallel\,(2,NC,c_3)}$. If $F_\cdot=F_{NC}$, then as in (\ref{5.50}), we still have that fix $\eta$, we have $\left|E_\xi\right|\lesssim 2^{r_0}\cdot \kappa_\theta$; fix $\xi$, we have $\left|E_\eta\right|\lesssim 2^{l+0.4\delta^2 m}\cdot \kappa_\theta$. Then, by Schur's test, (\ref{5.31}) and (\ref{5.49}), we get
\begin{align*}
    \left\|P_k\mathcal{B}_{m,l,r_0,r}^{\parallel\,(2,NC,c_2,NC)}\right\|_{L^2}&\lesssim 2^{m-l}\,2^{\frac{l}{2}+\frac{r_0}{2}+0.2\delta^2 m}\,\kappa_\theta\,\left\|g_q\right\|_{L^\infty}\,\left\|F_{NC}\right\|_{L^2}\lesssim 2^{-m},
\end{align*}
and
\begin{align*}
    \left\|P_k\mathcal{B}_{m,l,r_0,r}^{\parallel\,(2,NC,c_3,NC)}\right\|_{L^2}&\lesssim 2^{m-r_0-l}\,2^{\frac{l}{2}+\frac{r_0}{2}+0.2\delta^2 m}\,\kappa_\theta\,\left\|g_q\right\|_{L^\infty}\,\left\|F_{NC}\right\|_{L^2}\lesssim 2^{-m},\ \ \ \ \ \left(\mbox{Use }\frac{r_0}{2}-q\ge -\frac{m}{2}\right)
\end{align*}
which is acceptable as in (\ref{5.32}).
If $F_\cdot=F_C$, then $\left\|P_k\mathcal{B}_{m,l,r_0,r}^{\parallel\,(2,NC,c_2,C)}\right\|_{L^2}$ can be controlled as $\left\|P_k\mathcal{B}_{m,l,r_0,r}^{\parallel\,(2,C)}\right\|_{L^2}$, since the bound $\left\|\widehat{F_C}\right\|_{L^\infty}\lesssim 2^{-m+3.3\delta m}$ is similar as the bound $\left\|\partial_s \widehat{f_C}\right\|_{L^\infty}\lesssim 2^{-m+3\delta m}$ in (\ref{5.40}) and $\left\|P_k\mathcal{B}_{m,l,r_0,r}^{\parallel\,(2,NC,c_2,C)}\right\|_{L^2}$ even loses a factor $2^{-r_0}$ than $\left\|P_k\mathcal{B}_{m,l,r_0,r}^{\parallel\,(2,C)}\right\|_{L^2}$. Moreover, as the volume estimates in (\ref{5.50}), we apply Schur's test, (\ref{5.31}) and (\ref{5.49}) to get
\begin{align*}
    \left\|P_k\mathcal{B}_{m,l,r_0,r}^{\parallel\,(2,NC,c_3,C)}\right\|_{L^2}&\lesssim 2^{m-r_0-l}\,2^{\frac{l}{2}+\frac{r_0}{2}+0.2\delta^2 m}\,\kappa_\theta\,\left\|\partial_s g_q\right\|_{L^\infty}\,\left\|\widehat{F_{C}}\right\|_{L^\infty}\,\left|E_\eta\right|^{1/2}\lesssim 2^{-m+10\delta m},
\end{align*}
which is acceptable as in (\ref{5.32}). If $F_\cdot=F_{LO}$, then since $k_1\ll 1$ and $k,k_2\sim 1$ (due to $b_\sigma-b_\mu-b_\nu\neq 0$), we have that fix $\xi$, $\left|E_\eta\right|\lesssim 2^{k_1}\cdot \kappa_\theta$; fix $\eta$, $\left|E_\xi\right|\lesssim 2^{k_1}\cdot \kappa_\theta$. Then, if $k_1\ge 2^{-m}$, by Schur's test, (\ref{5.31}) and (\ref{5.49}), we get
\begin{align*}
    \left\|P_k\mathcal{B}_{m,l,r_0,r}^{\parallel\,(2,NC,c_2,LO)}\right\|_{L^2}&\lesssim 2^{m-l}\,2^{k_1}\,\kappa_\theta\,\left\|g_q\right\|_{L^\infty}\,\left\|\widehat{F_{LO}}\right\|_{L^\infty}\,\left|E_\eta\right|^{1/2}\lesssim 2^{-m+9.1\delta m},
\end{align*}
and
\begin{align*}
    \left\|P_k\mathcal{B}_{m,l,r_0,r}^{\parallel\,(2,NC,c_3,LO)}\right\|_{L^2}&\lesssim 2^{m-r_0-l}\,2^{k_1}\,\kappa_\theta\,\left\|\partial_s g_q\right\|_{L^\infty}\,\left\|\widehat{F_{LO}}\right\|_{L^\infty}\,\left|E_\eta\right|^{1/2}\lesssim 2^{-m+11.1\delta m},
\end{align*}
which is acceptable as in (\ref{5.32}). On the other hand, if $k_1\le 2^{-m}$, then we can still apply Schur's test to get an acceptable contribution as in (\ref{5.32}) exactly like before.\par
To sum up, we have already proved that if $\min\left\{k,k_1,k_1\right\}\le -D$ , then we still have $\left\|Q_{jk}\mathcal{B}_{m,l}\right\|_{L^2}\lesssim 2^{-(1-20\delta)j-0.001\delta^2 m}$ as in (\ref{5.32}).\par
\vspace{2em}
Thus, the proof of subsection 5.4.1 is now complete. \par
\vspace{3em}
\subsubsection{Contribution of $f_{NC\omega}^\nu$}
\ \par
We now show that 
\begin{align*}
    2^{(1-20\delta)j}\left\|Q_{jk}\mathcal{B}_{m,l}\left[P_{k_1}f^\mu,P_{k_2}f^\nu_{NCw}\right]\right\|_{L^2}\lesssim 2^{-0.001\delta^2 m}.\tag{5.51}\label{5.51}
\end{align*}
We define $f^\mu_{j_1,k_1}$ and $f^\mu_{j_1,k_1,n_1}$ as before. We first do the following decomposition
\begin{align*}
    &\mathcal{F}\mathcal{B}_{m,l}\left[f^\mu_{j_1,k_1},f^\nu_{NCw}\right](\xi)=\mathcal{F}\mathcal{B}^{Hi}_{m,l}\left[f^\mu_{j_1,k_1},f^\nu_{NCw}\right](\xi)+\mathcal{F}\mathcal{B}^{Lo}_{m,l}\left[f^\mu_{j_1,k_1},f^\nu_{NCw}\right](\xi) \\
    &\mathcal{F}\mathcal{B}^{Hi}_{m,l}\left[f^\mu_{j_1,k_1},f^\nu_{NCw}\right](\xi)\triangleq \int_{\mathbb{R}} q_m(s) \int_{\mathbb{R}^2} e^{is\Phi(\xi,\eta)} \Tilde{\varphi}_l(\Phi(\xi,\eta))\varphi_{\ge -\delta^2 m}(\nabla_\eta\Phi(\xi,\eta)) \\
    &\ \ \ \ \ \times\widehat{f^\mu_{j_1,k_1}}(\xi-\eta,s)\widehat{f^\nu_{NCw}}(\eta,s)\,d\eta ds, \\
    &\mathcal{F}\mathcal{B}^{Lo}_{m,l}\left[f^\mu_{j_1,k_1},f^\nu_{NCw}\right](\xi)\triangleq \int_{\mathbb{R}} q_m(s) \int_{\mathbb{R}^2} e^{is\Phi(\xi,\eta)} \Tilde{\varphi}_l(\Phi(\xi,\eta))\varphi_{< -\delta^2 m}(\nabla_\eta\Phi(\xi,\eta)) \\
    &\ \ \ \ \ \times\widehat{f^\mu_{j_1,k_1}}(\xi-\eta,s)\widehat{f^\nu_{NCw}}(\eta,s)\,d\eta ds.\tag{5.52}\label{5.52}
\end{align*}\par
We first deal with $\mathcal{F}\mathcal{B}^{Lo}_{m,l}\left[f^\mu_{j_1,k_1},f^\nu_{NCw}\right]$. If $j_1\le (1-\delta^2)m$, then we plug (\ref{3.26}) into (\ref{5.52}) and get
\begin{align*}
    &\mathcal{F}\mathcal{B}^{Lo}_{m,l}\left[f^\mu_{j_1,k_1},f^\nu_{NCw}\right]\triangleq \int_{\mathbb{R}} q_m(s) \int_{\mathbb{R}^2\times\mathbb{R}^2} e^{is\left[\Phi_{\sigma\mu\nu}(\xi,\eta)+\Phi_{\nu\nu_2\nu_3}(\eta,\theta)\right]} \Tilde{\varphi}_l(\Phi(\xi,\eta))\varphi_{< -\delta^2 m}(\nabla_\eta\Phi(\xi,\eta)) \\
    &\ \,\times\varphi_{\le -3\delta m-4\delta^2 m}(\Phi_{\nu\nu_2\nu_3}(\eta,\theta))\,\varphi_{\gtrsim 1}(\nabla_\eta\Phi_{\nu\nu_2\nu_3}(\eta,\theta))\,\widehat{f^\mu_{j_1,k_1}}(\xi-\eta,s)\,\widehat{g_{\widebar{j_1},\widebar{k_1}}}(\eta-\theta,s)\,\widehat{g_{\widebar{j_2},\widebar{k_2}}}(\theta,s)d\theta d\eta ds,
\end{align*}
where $\widebar{j_1}\le (1-\delta)m$. Now note that we have $\left|\nabla_\eta\left(\Phi_{\sigma\mu\nu}(\xi,\eta)+\Phi_{\nu\nu_2\nu_3}(\eta,\theta)\right)\right|\gtrsim 1$. So, we can integrate by parts in $\eta$ to get enough control in this case. From now on, in this subsection 5.4.2, we may just write $\widehat{f_1}$ instead of $\widehat{f^\mu_{j_1,k_1}}$ or $\widehat{f^\mu_{j_1,k_1,n_1}}$ for simplicity as before. On the other hand, if $j_1\ge (1-\delta^2) m$, then we apply Proposition 6.10 (a), (\ref{3.2}) and (\ref{3.25}) to (\ref{5.52}) to get
\begin{align*}
    \left\|\mathcal{F}\mathcal{B}^{Lo}_{m,l}\right\|_{L^2}&\lesssim 2^{m-l}\,2^{\frac{l}{2}-\frac{n_1}{2}}\,\left\|\sup_\theta\left|\widehat{f_1}(r\theta)\right|\right\|_{L^2}\,\left\|f^\nu_{NCw}\right\|_{L^2}\lesssim 2^{-1.1m+32\delta m}, \tag{5.53}\label{5.53}
\end{align*}
which is acceptable as in (\ref{5.51}).\par
Next, let's consider $\mathcal{F}\mathcal{B}^{Hi}_{m,l}\left[f^\mu_{j_1,k_1},f^\nu_{NCw}\right]$. We may assume $j_1\le (1-\delta^2)m$ here, since otherwise we can proceed our proof as in (\ref{5.53}).\par
\textbf{Case 1.} $k_1\ge -D$ \par
In this case, we denote $\kappa_\theta\triangleq 2^{-m/2+\delta^2 m}$ as before. We first need to decompose
\begin{align*}
    &\mathcal{F}\mathcal{B}^{Hi}_{m,l}\left[f^\mu_{j_1,k_1},f^\nu_{NCw}\right]=\mathcal{F}\mathcal{B}^{Hi,\parallel}_{m,l}\left[f^\mu_{j_1,k_1},f^\nu_{NCw}\right]+\mathcal{F}\mathcal{B}^{Hi,\bot}_{m,l}\left[f^\mu_{j_1,k_1},f^\nu_{NCw}\right], \\
    &\mathcal{F}\mathcal{B}^{Hi,\parallel}_{m,l}\left[f^\mu_{j_1,k_1},f^\nu_{NCw}\right]\triangleq \int_{\mathbb{R}} q_m(s) \int_{\mathbb{R}^2} e^{is\Phi(\xi,\eta)} \Tilde{\varphi}_l(\Phi(\xi,\eta))\varphi_{\ge -\delta^2 m}(\nabla_\eta\Phi(\xi,\eta)) \\
    &\ \ \ \ \ \times\varphi(\kappa_\theta^{-1}\Omega_\eta\Phi(\xi,\eta))\,\widehat{f^\mu_{j_1,k_1}}(\xi-\eta,s)\widehat{f^\nu_{NCw}}(\eta,s)\,d\eta ds, \\
    &\mathcal{F}\mathcal{B}^{Hi,\bot}_{m,l}\left[f^\mu_{j_1,k_1},f^\nu_{NCw}\right]\triangleq \int_{\mathbb{R}} q_m(s) \int_{\mathbb{R}^2} e^{is\Phi(\xi,\eta)} \Tilde{\varphi}_l(\Phi(\xi,\eta))\varphi_{\ge -\delta^2 m}(\nabla_\eta\Phi(\xi,\eta)) \\
    &\ \ \ \ \ \times\left(1-\varphi(\kappa_\theta^{-1}\Omega_\eta\Phi(\xi,\eta))\right)\,\widehat{f^\mu_{j_1,k_1}}(\xi-\eta,s)\widehat{f^\nu_{NCw}}(\eta,s)\,d\eta ds. 
\end{align*}
By Lemma 3.2 and (\ref{5.13}), we can get $\left\|\mathcal{F}\mathcal{B}^{Hi,\bot}_{m,l}\right\|_{L^2}\lesssim 2^{-4m}$. Thus, we only need to consider the term $\mathcal{F}\mathcal{B}^{Hi,\parallel}_{m,l}$. In fact, we need to further decompose
\begin{align*}
    &\mathcal{F}\mathcal{B}^{Hi,\parallel}_{m,l}\left[f^\mu_{j_1,k_1},f^\nu_{NCw}\right]=\sum_{-m\le r\le 0} \mathcal{F}\mathcal{B}^{Hi,\parallel}_{m,l,r}\left[f^\mu_{j_1,k_1},f^\nu_{NCw}\right], \\
    &\mathcal{F}\mathcal{B}^{Hi,\parallel}_{m,l,r}\left[f^\mu_{j_1,k_1},f^\nu_{NCw}\right](\xi)\triangleq \int_{\mathbb{R}} q_m(s) \int_{\mathbb{R}^2} e^{is\Phi(\xi,\eta)} \Tilde{\varphi}_l(\Phi(\xi,\eta))\varphi^{\left[-m,0\right]}_r(\nabla_\xi\Phi_{\sigma\mu\nu}(\xi,\eta)) \\
    &\ \ \ \ \ \times\varphi_{\ge -\delta^2 m}(\nabla_\eta\Phi_{\sigma\mu\nu}(\xi,\eta))\widehat{f^\mu_{j_1,k_1}}(\xi-\eta,s)\widehat{f^\nu_{NCw}}(\eta,s)\,d\eta ds. \\
\end{align*}\par
If $r=0$, then $\left|\nabla_\xi\Phi_{\sigma\mu\nu}\right|\gtrsim 1$, which gives that fix $\xi$, $\left|E_\eta\right|\lesssim 2^{l+\delta^2 m}\cdot \kappa_\theta$ and fix $\eta$, $\left|E_\xi\right|\lesssim 2^l\cdot \kappa_\theta$. Then, by Schur's test, (\ref{3.4}) and (\ref{3.25}), we get 
\begin{align*}
    \left\|\mathcal{F}\mathcal{B}^{Hi,\parallel}_{m,l,0}\right\|_{L^2}&\lesssim 2^{m-l}\,2^{l+0.5\delta^2 m}\,\kappa_\theta\,\left\|\widehat{f_1}\right\|_{L^\infty}\,\left\|\widehat{f^\nu_{NCw}}\right\|_{L^2}\lesssim 2^{-1.1m+13.5\delta m},
\end{align*}
which is acceptable as in (\ref{5.51}). Similarly, if $r=-m$, then $\left|\nabla_\xi\Phi_{\sigma\mu\nu}\right|\lesssim 2^{-m}$, which gives that fix $\xi$, $\left|E_\eta\right|\lesssim 2^{l+\delta^2 m}\cdot \kappa_\theta$ and fix $\eta$, $\left|E_\xi\right|\lesssim 2^{-m}\cdot \kappa_\theta$. Then, by Schur's test, (\ref{3.4}) and (\ref{3.25}), we get 
\begin{align*}
    \left\|\mathcal{F}\mathcal{B}^{Hi,\parallel}_{m,l,-m}\right\|_{L^2}&\lesssim 2^{m-l}\,2^{\frac{l}{2}+0.5\delta^2 m-\frac{m}{2}}\,\kappa_\theta\,\left\|\widehat{f_1}\right\|_{L^\infty}\,\left\|\widehat{f^\nu_{NCw}}\right\|_{L^2}\lesssim 2^{-1.1m+13.5\delta m},
\end{align*}
which is also acceptable as in (\ref{5.51}). \par
Next, when $-m<r<0$, we need to divide into several subcases.\par
\ \ \textbf{Case 1.1.}\  $m+r+100\le j\le m$,\ \ \  $j_1\le (1-\delta^2)j$ \par
Once again, this subcase can be done by integration by parts in $\xi$ like Case 1.1 in the subsection 5.4.1 before.\par
\ \ \textbf{Case 1.2.}\  $m+r+100\le j\le m$,\ \ \  $j_1\ge (1-\delta^2)j$ \par
In this subcase, we use Proposition 6.10 (a), (\ref{3.2}) and (\ref{3.25}) to get 
\begin{align*}
    \left\|\mathcal{F}\mathcal{B}^{Hi}_{m,l}\right\|_{L^2}&\lesssim 2^{m-l}\,2^{\frac{l}{2}-\frac{n_1}{2}}\,\left\|\sup_\theta\left|\widehat{f}(r\theta)\right|\right\|_{L^2}\,\left\|\widehat{f^\nu_{NCw}}\right\|_{L^2} \lesssim 2^{-(1-20\delta)j-0.1m+12.4\delta m},
\end{align*}
which is acceptable as in (\ref{5.51}). \par
\ \ \textbf{Case 1.3.}\  $j\le m+r+100$,\ \ \  $2r\le l$ \par
In this subcase, we note that fix $\xi$, $\left|E_\eta\right|\lesssim 2^{l+\delta^2 m}\cdot \kappa_\theta$; fix $\eta$, $\left|E_\xi\right|\lesssim 2^r\cdot \kappa_\theta$. Then, by Schur's test, (\ref{3.4}) and (\ref{3.25}), we get
\begin{align*}
    \left\|\mathcal{F}\mathcal{B}^{Hi,\parallel}_{m,l,r}\right\|_{L^2}&\lesssim 2^{m-l}\,2^{\frac{l}{2}+0.5\delta^2 m+\frac{r}{2}}\,\kappa_\theta\,\left\|\widehat{f_1}\right\|_{L^\infty}\,\left\|\widehat{f^\nu_{NCw}}\right\|_{L^2}\lesssim 2^{-(1-20\delta)j-0.1m-\delta m},
\end{align*}
which is acceptable as in (\ref{5.51}). \par
\ \ \textbf{Case 1.4.}\  $j\le m+r+100$,\ \ \  $2r\ge l$ \par
The proof in this subcase is quite similar as the one in Case 1.3 just above. Note that fix $\xi$, $\left|E_\eta\right|\lesssim 2^{l+\delta^2 m}\cdot \kappa_\theta$; fix $\eta$, $\left|E_\xi\right|\lesssim 2^{l-r}\cdot \kappa_\theta$. Then, by Schur's test, (\ref{3.4}) and (\ref{3.25}), we get
\begin{align*}
    \left\|\mathcal{F}\mathcal{B}^{Hi,\parallel}_{m,l,r}\right\|_{L^2}&\lesssim 2^{m-l}\,2^{\frac{l}{2}+0.5\delta^2 m+\frac{l-r}{2}}\,\kappa_\theta\,2^{2\delta m}\,2^{-1.6m+11.4\delta m}\lesssim 2^{-(1-20\delta)j-0.1m-6\delta m},
\end{align*}
which is acceptable as in (\ref{5.51}). \par
\vspace{3em}
\textbf{Case 2.} $k_1\le -D$ \par
Now, due to $b_\sigma-b_\mu-b_\mu\neq 0$, we must have that $k,k_2\sim 0$, namely $\left|\xi\right|,\left|\eta\right|\sim 1$. By the argument at (\ref{5.47}), we yield that $\left|\nabla_\xi\Phi_{\sigma\mu\nu}(\xi,\eta)\right|\gtrsim 1$ in this case. \par
\ \ \textbf{Case 2.1.}\  $k_1\ge j_1-m+\delta^2 m$ and $k_1\ge -0.4m+\delta^2 m$ \par
In this subcase, we denote $\kappa_\theta\triangleq 2^{-m/2+\delta^2 m}$ as before. We again decompose
\begin{align*}
    &\mathcal{F}\mathcal{B}^{Hi}_{m,l}\left[f^\mu_{j_1,k_1},f^\nu_{NCw}\right]=\mathcal{F}\mathcal{B}^{Hi,\parallel}_{m,l}\left[f^\mu_{j_1,k_1},f^\nu_{NCw}\right]+\mathcal{F}\mathcal{B}^{Hi,\bot}_{m,l}\left[f^\mu_{j_1,k_1},f^\nu_{NCw}\right], \\
    &\mathcal{F}\mathcal{B}^{Hi,\parallel}_{m,l}\left[f^\mu_{j_1,k_1},f^\nu_{NCw}\right]\triangleq \int_{\mathbb{R}} q_m(s) \int_{\mathbb{R}^2} e^{is\Phi(\xi,\eta)} \Tilde{\varphi}_l(\Phi(\xi,\eta))\varphi_{\ge -\delta^2 m}(\nabla_\eta\Phi(\xi,\eta)) \\
    &\ \ \ \ \ \times\varphi(\kappa_\theta^{-1}\Omega_\eta\Phi(\xi,\eta))\,\widehat{f^\mu_{j_1,k_1}}(\xi-\eta,s)\widehat{f^\nu_{NCw}}(\eta,s)\,d\eta ds, \\
    &\mathcal{F}\mathcal{B}^{Hi,\bot}_{m,l}\left[f^\mu_{j_1,k_1},f^\nu_{NCw}\right]\triangleq \int_{\mathbb{R}} q_m(s) \int_{\mathbb{R}^2} e^{is\Phi(\xi,\eta)} \Tilde{\varphi}_l(\Phi(\xi,\eta))\varphi_{\ge -\delta^2 m}(\nabla_\eta\Phi(\xi,\eta)) \\
    &\ \ \ \ \ \times\left(1-\varphi(\kappa_\theta^{-1}\Omega_\eta\Phi(\xi,\eta))\right)\,\widehat{f^\mu_{j_1,k_1}}(\xi-\eta,s)\widehat{f^\nu_{NCw}}(\eta,s)\,d\eta ds. 
\end{align*}
By Lemma 3.2 and (\ref{5.13}), we can get $\left\|\mathcal{F}\mathcal{B}^{Hi,\bot}_{m,l}\right\|_{L^2}\lesssim 2^{-4m}$. Thus, we only need to consider the term $\mathcal{F}\mathcal{B}^{Hi,\parallel}_{m,l}$. Now we have that fix $\xi$, $\left|E_\eta\right|\lesssim 2^{l+\delta^2 m}\cdot \kappa_\theta$ and fix $\eta$, $\left|E_\xi\right|\lesssim 2^l\cdot \kappa_\theta$. Then, by Schur's test, (\ref{3.4}) and (\ref{3.25}), we get 
\begin{align*}
    \left\|\mathcal{F}\mathcal{B}^{Hi,\parallel}_{m,l}\right\|_{L^2}&\lesssim 2^{m-l}\,2^{l+0.5\delta^2 m}\,\kappa_\theta\,\left\|\widehat{f_1}\right\|_{L^\infty}\,\left\|\widehat{f^\nu_{NCw}}\right\|_{L^2}\lesssim 2^{-1.1m+25\delta m},
\end{align*}
which is acceptable as in (\ref{5.51}).\par
\ \ \textbf{Case 2.2.}\  $k_1\le j_1-m+\delta^2 m$ \par
In this subcase, in view of Proposition 6.9 (b), we observe that fix $\xi$, $\left|E_\eta\right|\lesssim 2^{l+\delta^2 m}\cdot 2^{k_1}$ and fix $\eta$, $\left|E_\xi\right|\lesssim 2^l\cdot 2^{k_1}$. Then, by Schur's test, (\ref{3.4}) and (\ref{3.25}), we get
\begin{align*}
    \left\|\mathcal{F}\mathcal{B}^{Hi}_{m,l}\right\|_{L^2}&\lesssim 2^{m-l}\,2^{l+0.5\delta^2 m+k_1}\,\left\|\widehat{f_1}\right\|_{L^\infty}\,\left\|\widehat{f^\nu_{NCw}}\right\|_{L^2}\lesssim 2^{-1.1m+35\delta m},
\end{align*}
which is acceptable as in (\ref{5.51}).\par
\ \ \textbf{Case 2.3.}\  $k_1\le -0.4m+\delta^2 m$ \par
We first note that $k_1\le -0.4m+\delta^2 m$ immediately implies that $j_1\ge 0.4m-\delta^2 m$. Then, the proof in this subcase is exactly same the one in Case 2.2 just above. Using Schur's test, we get 
\begin{align*}
    \left\|\mathcal{F}\mathcal{B}^{Hi}_{m,l}\right\|_{L^2}&\lesssim 2^{m-l}\,2^{l+0.5\delta^2 m+k_1}\,2^{-(\frac{1}{2}-21\delta)j_1-\frac{1}{2}k_1}\,2^{-1.6m+11.4\delta m}\lesssim 2^{-m+19.8\delta m},
\end{align*}
which is acceptable as in (\ref{5.51}).\par
\vspace{2em}
To sum up, we have already proved that $\left\|Q_{jk}\mathcal{B}_{m,l}\right\|_{L^2}\lesssim 2^{-(1-20\delta)j-0.001\delta^2 m}$ as in (\ref{5.32}). The proof of subsection 5.4.2 is now complete.\par
\vspace{3em}
\subsubsection{Contribution of $\partial_s F^\nu$}
\ \par
Now, it remains to show that
\begin{align*}
    2^{(1-20\delta)m}\left\|P_k\mathcal{B}_{m,l}\left[P_{k_1}f^\mu,P_{k_2} \partial_s F^\nu_\alpha\right]\right\|_{L^2}\lesssim 2^{-51\delta^2 m},\ \ \ \alpha\in\left\{C,NC,LO\right\}.\tag{5.54}\label{5.54}
\end{align*}
In fact, this was mainly proved in \cite{y2} as well with some slightly differences. For the sake of completeness, we will rewrite the proof here. \par
We define $f^\mu_{j_1,k_1}$ and $f^\mu_{j_1,k_1,n_1}$ as before and integrate by parts in time to rewrite
\begin{align*}
    &\mathcal{F}\mathcal{B}_{m,l}\left[P_{k_1}f^\mu,P_{k_2}\partial_s F^\nu_\alpha\right]=-B_1\left[P_{k_1}f^\mu,P_{k_2}\partial_s F^\nu_\alpha\right]-iB_2\left[P_{k_1}f^\mu,P_{k_2}\partial_s F^\nu_\alpha\right]-B_3\left[P_{k_1}\partial_s f^\mu,P_{k_2}\partial_s F^\nu_\alpha\right] \\
    &\ \ \ B_1\left[P_{k_1}f^\mu,g\right](\xi)\triangleq\int_{\mathbb{R}} q^\prime_m(s) \int_{\mathbb{R}^2} e^{is\Phi(\xi,\eta)}\widetilde{\varphi_l}(\Phi(\xi,\eta))\widehat{P_{k_1}f^\mu}(\xi-\eta,s)\hat{g}(\eta,s)\,d\eta ds, \\
    &\ \ \ B_2\left[P_{k_1}f^\mu,g\right](\xi)\triangleq\int_{\mathbb{R}} q_m(s) \int_{\mathbb{R}^2} e^{is\Phi(\xi,\eta)}\widetilde{\varphi_l}(\Phi(\xi,\eta))\Phi(\xi,\eta)\widehat{P_{k_1}f^\mu}(\xi-\eta,s)\hat{g}(\eta,s)\,d\eta ds, \\
    &\ \ \ B_3\left[P_{k_1}\partial_s f^\mu,g\right](\xi)\triangleq\int_{\mathbb{R}} q_m(s) \int_{\mathbb{R}^2} e^{is\Phi(\xi,\eta)}\widetilde{\varphi_l}(\Phi(\xi,\eta))\partial_s \widehat{P_{k_1}f^\mu}(\xi-\eta,s)\hat{g}(\eta,s)\,d\eta ds. \\
\end{align*}\par
It has already been shown in \cite{y2} that $\left\|B_\beta\left[P_{k_1}f^\mu,P_{k_2}F^\nu_{NC}\right]\right\|_{L^2}\lesssim 2^{-m}$ for $\beta\in\left\{1,2,3\right\}$, which gives an acceptable contribution as in (\ref{5.54}).\par
Now, we consider the contribution of $F^\nu_{LO}$. In view of (\ref{3.27}), we know that $k_2\le -\frac{13}{15}m$. Due to the assumption that $b_\sigma-b_\mu-b_\nu\neq 0$, we must have $k_1\ge -D$. (However, unlike \cite{y2}, in general it's not necessarily that $n_1=0$.) If $\beta\in\left\{1,2\right\}$ and $j_1\ge m/2$, then Young's inequality gives us 
$$\left\|B_\beta\left[f^\mu_{j_1,k_1},P_{k_2}F^\nu_{LO}\right]\right\|_{L^2}\lesssim 2^m\,\sup_s \left\|\widehat{P_{k_2}F_{LO}}(s)\right\|_{L^1}\,\left\|f^\mu_{j_1,k_1}(s)\right\|_{L^2}\lesssim 2^{-4m/5}\,2^{-(\frac{1}{2}-\delta)j_1}\lesssim 2^{-1.025m},$$
which is an acceptable contribution as in (\ref{5.54}). For all other cases, the proofs are exactly same as in \cite{y2}.\par
Next, we consider the contribution of $F^\nu_C$. It turns out that we only need to consider the case
$$j_1\le m-\delta m,\ \ \ k_1\le -\delta m/2,\ \ \ \beta\in\left\{1,2\right\},$$
since all other cases have already been done in \cite{y2}. First, if $j_1\ge m-8.8\delta m$, then using Schur's test and Proposition 6.9 (b), we estimate
\begin{align*}
    2^{(1-20\delta)m}\left\|B_\beta\left[f^\mu_{j_1,k_1},P_{k_2}F^\nu_C\right]\right\|_{L^2}&\lesssim 2^{(1-20\delta)m}\,2^m\,2^l\,2^{2\delta^2 m}\,\sup_s\left\|\widehat{P_{k_2}F^\nu_C}(s)\right\|_{L^\infty}\,\left\|f^\mu_{j_1,k_1}\right\|_{L^2}\lesssim 2^{-150\delta^2 m}.
\end{align*}
On the other hand, if $j_1\le m-8.8\delta m$ and $k_1\ge -8.8\delta m+10\delta^2 m$, then we use Lemma 3.5 and the last line of (\ref{3.6}) to get
\begin{align*}
    \left\|B_\beta\left[f^\mu_{j_1,k_1},P_{k_2}F^\nu_C\right]\right\|_{L^2}&\lesssim 2^m\left[\sup_{s,\lambda\approx 2^m}\left\|e^{-i\lambda\Lambda_\mu}f^\mu_{j_1,k_1}(s)\right\|_{L^\infty}\left\|P_{k_2}F^\nu_C(s)\right\|_{L^2}\right]\lesssim 2^{-m+13\delta m},
\end{align*}
which gives an acceptable contribution as in (\ref{5.54}). Finally, if $j_1\le m-8.8\delta m$ and $k_1\le -8.8\delta m+10\delta^2 m$, then once again we must have $k_1\ll 1$ and $k,k_2\sim 1$ thanks to the assumption that $b_\sigma-b_\mu-b_\nu\neq 0$. By arguments at (\ref{5.30}) and (\ref{5.47}), we conclude that $\left|\nabla_\xi\Phi\right|\gtrsim 1$ and $\left|\nabla_\eta\Phi\right|\gtrsim 1$, which implies that $\left|E_\eta\right|\lesssim 2^l$. Therefore, we get
$$\left\|P_{k_2}F^\nu_C(s)\right\|_{L^2}\lesssim \left\|P_{k_2}F^\nu_C(s)\right\|_{L^\infty}\,\left|E_\eta\right|^{1/2}\lesssim 2^{-m-2.7\delta m}.$$
Using the bound in the first line of (\ref{3.6}) we can estimate
\begin{align*}
    \left\|B_\beta\left[f^\mu_{j_1,k_1},P_{k_2}F^\nu_C\right]\right\|_{L^2}&\lesssim 2^m\left[\sup_{s,\lambda\approx 2^m}\left\|e^{-i\lambda\Lambda_\mu}f^\mu_{j_1,k_1}(s)\right\|_{L^\infty}\left\|P_{k_2}F^\nu_C(s)\right\|_{L^2}\right]\lesssim 2^{-m+18\delta m},
\end{align*}
which again gives an acceptable contribution as in (\ref{5.54}).\par
\vspace{2em}
The proof of subsection 5.4.3 is now complete.
\vspace{1em}
\section{Elementary Lemmas}
In this section, we collect some important facts about the phase function 
$$\Phi(\xi,\eta)=\Phi_{\sigma\mu\nu}(\xi,\eta)=\Lambda_\sigma(\xi)-\Lambda_\mu(\xi-\eta)-\Lambda_\nu(\eta),$$
where $\sigma,\mu,\nu\in\left\{-d,\dots,-2,-2,1,2,\dots,d\right\}$ and assume that
$$\left|\xi\right|\sim 2^k,\ \ \ \left|\xi-\eta\right|\sim 2^{k_1},\ \ \ \left|\eta\right|\sim 2^{k_2}$$
as before. In the following, we also denote $$\Bar{k}\triangleq\max (k,k_1,k_2,0).$$
\\
\begin{lemma}
There exists a large constant $D_0=D_0(c_\sigma,c_\mu,c_\nu,b_\sigma,b_\mu,b_\nu)>0$ such that if we have $\min (k,k_1,k_2)\le -D_0$, $b_\sigma-b_\mu-b_\nu\neq0$, and $(c_\mu-c_\nu)(c_\mu^2b_\nu-c_\nu^2b_\mu)\ge0$, then there exists a small positive constant $C=C(c_\sigma,c_\mu,c_\nu,b_\sigma,b_\mu,b_\nu,D_0)>0$, such that either $\left|\Phi\right|\ge C$ or $\left|\nabla_\eta\Phi\right|\ge C\cdot 2^{-3\Bar{k}}$.
\end{lemma}
\begin{proof}
\setlength{\parindent}{2em}
We prove by contradiction, so we assume $\left|\Phi\right|<C$ and $\left|\nabla_\eta\Phi\right|<C\cdot 2^{-3\Bar{k}}$, where $C$ will be selected as small as we wish later on. Now, We divide it into three cases:\\
\noindent$1^\circ: k_2\le-D_0$ \par
In this case, we have two subcases: (a) $k\approx k_1\approx k_2\le-D_0$\ (say $k,k_1\le -D_1$); (b) $k, k_1\gg k_2$\ (say $k, k_1>D_1$). Here, $D_1>0$, and $1\ll D_1\ll D_0$, which will be determined later on.\par
When we are in the case (a), we may only consider the case $b_\sigma>0$,$b_\mu<0$,$b_\nu<0$, since other cases are similar. 
Now, we first let 
$$A\triangleq\displaystyle{\frac{1}{2}\min_{\begin{subarray}{c}
\sigma>0 \\ \mu<0 \\ \nu<0    
\end{subarray}}{\left\{\left|b_\sigma-b_\mu-b_\nu\right|\right\}}>0},$$
and 
\begin{align*}
    \lambda\triangleq\min \left\{\lambda_{\sigma\mu\nu}\ge1:
    \mbox{ if } b_\sigma-b_\mu-b_\nu>0\mbox{, then } b_\sigma-\widetilde{\lambda_{\sigma\mu\nu}}b_\nu-\widetilde{\lambda_{\sigma\mu\nu}}b_\nu>A \mbox{ for all } 1\le\widetilde{\lambda_{\sigma\mu\nu}}\le\lambda_{\sigma\mu\nu};\right.\\
    \left.\mbox{if } b_\sigma-b_\mu-b_\nu<0\mbox{, then } b_\mu+b_\nu-\widetilde{\lambda_{\sigma\mu\nu}}b_\sigma>A \mbox{ for all } 1\le\widetilde{\lambda_{\sigma\mu\nu}}\le\lambda_{\sigma\mu\nu}\right\}.
\end{align*}
WLOG, we may also assume that $\lambda\le1.01$. Then, we set $C=\frac{1}{2}A$, and pick \\
$$D_0=D_0(c_\sigma,c_\mu,c_\nu,b_\sigma,b_\mu,b_\nu), \ \ \ \ D_1=D_1(c_\sigma,c_\mu,c_\nu,b_\sigma,b_\mu,b_\nu)$$ 
large enough (fix $D_1$) such that if $b_\sigma-b_\mu-b_\nu>0$, then 
$$\Phi=\sqrt{c_\sigma^2\left|\xi\right|^2+b_\sigma^2}-\sqrt{c_\mu^2\left|\xi-\eta\right|^2+b_\mu^2}-\sqrt{c_\nu^2\left|\eta\right|^2+b_\nu^2}>b_\sigma-\lambda b_\mu-\lambda b_\nu>C,$$
and if $b_\sigma-b_\mu-b_\nu<0$, then 
\begin{align*}
   \left|\Phi\right|&=\left|\sqrt{c_\sigma^2\left|\xi\right|^2+b_\sigma^2}-\sqrt{c_\mu^2\left|\xi-\eta\right|^2+b_\mu^2}-\sqrt{c_\nu^2\left|\eta\right|^2+b_\nu^2}\right| \\
   &=\sqrt{c_\mu^2\left|\xi-\eta\right|^2+b_\mu^2}+\sqrt{c_\nu^2\left|\eta\right|^2+b_\nu^2}-\sqrt{c_\sigma^2\left|\xi\right|^2+b_\sigma^2}>b_\mu+b_\nu-\lambda b_\sigma>C,
\end{align*}
which is a contradiction.\par

When we are in the case (b), we first note that 
$$\nabla_\eta\Phi=\mp\frac{c_\mu^2(\xi-\eta)}{\sqrt{c_\mu^2\left|\xi-\eta\right|^2+b_\mu^2}}\pm\frac{c_\nu^2\eta}{\sqrt{c_\nu^2\left|\eta\right|^2+b_\nu^2}}.$$
By adjusting the value of $D_1$ and $D_0$ (they depend on $c_\mu, b_\mu, c_\nu, b_\nu$ and the small constant $C$ in the result inequality in this lemma), we could have
$$\left|\frac{c_\nu^2\eta}{\sqrt{c_\nu^2\left|\eta\right|^2+b_\nu^2}}\right|\le\frac{1}{2}\left|\frac{c_\mu^2(\xi-\eta)}{\sqrt{c_\mu^2\left|\xi-\eta\right|^2+b_\mu^2}}\right|.$$
Now, we prove by contradiction as said before, so we assume $\left|\nabla_\eta\Phi\right|<C\cdot 2^{-3\Bar{k}}$. Note that
$$\left|\nabla_\eta\Phi\right|\sim
\begin{cases}
   c_\mu\left|\xi-\eta\right|, &\mbox{if}\left|\xi-\eta\right|\mbox{ is small;} \\
   c_\mu, &\mbox{if}\left|\xi-\eta\right|\mbox{ is large;}
\end{cases},$$
which is already a contradiction if we select $C$ small enough such that $C\le 2^{-D_0}\ll 2^{-D_1}$. Note that at this moment, $C$ still only depends on $c_\sigma,c_\mu,c_\nu,b_\sigma,b_\mu,b_\nu$.
\par
\noindent$2^\circ: k_1\le-D_0$ \par
This case can be done exactly as in $1^\circ$. \par
\noindent$3^\circ: k\le-D_0$ \par
Similarly, we will have two subcases: (a) $k\approx k_1\approx k_2\le-D_0$;\ (say $k_1,k_2\le -D_1$) (b) $k_1, k_2\gg k$\ (say $k_1, k_2>D_1$). Here, $D_1>0$, and $1\ll D_1\ll D_0$, which will be determined later on.\par
The subcase (a) can be done exactly as in $1^\circ$.\par
The subcase (b) is a little bit complicated to be dealt with. We first assume that $\mu\cdot \nu>0$. Then, we have four subcases: (b1) $c_\mu=c_\nu$ and $b_\mu=b_\nu$; (b2) $c_\mu\neq c_\nu$ and $c_\mu^2 b_\nu-c_\nu^2 b_\mu=0$; (b3) $c_\mu=c_\nu$ and $c_\mu^2 b_\nu-c_\nu^2 b_\mu\neq 0$; (b4) $(c_\mu-c_\nu)(c_\mu^2b_\nu-c_\nu^2b_\mu)>0$.\par
When we are in the case (b1), we may WLOG assume that $c_\mu=c_\nu=1$. Then, from $\left|\nabla_\eta\Phi\right|\ll2^{-3\Bar{k}}$, we can deduce that $\left|\eta-\rho\,\xi\right|\lesssim\min\left(1,2^{3k_2}\right)$, where $\displaystyle{\rho=\frac{b_\nu}{b_\mu+b_\nu}=\frac{1}{2}}$. For example, let's consider the case $\mu,\nu>0$. In fact, if we let 
$$\boldsymbol{g}(\boldsymbol{x})\triangleq\frac{\boldsymbol{x}}{\sqrt{\left|\boldsymbol{x}\right|^2+b_\mu^2}},\quad \alpha\triangleq\frac{b_\nu}{b_\mu}=1,\quad \Tilde{\eta}\triangleq\frac{\eta}{\alpha}=\eta,$$
then we have $\left|\eta-\rho\,\xi\right|=\left|\xi-\eta-\Tilde{\eta}\right|$, and $\left|\nabla_\eta\Phi\right|=\left|\boldsymbol{g}\left(\xi-\eta\right)-\boldsymbol{g}\left(\Tilde{\eta}\right)\right|$. Therefore, it suffices to find the upper bound of $\left\|J(\boldsymbol{g}^{-1})\right\|_{L^\infty}$. In fact, we have
$$\boldsymbol{x}=\boldsymbol{g}^{-1}(\boldsymbol{y})=\frac{b_\mu\cdot \boldsymbol{y}}{\sqrt{1-\left|\boldsymbol{y}\right|^2}}$$
and thus have 
$$\displaystyle{\left\|J(\boldsymbol{g}^{-1})\right\|_{L^\infty}\le\frac{1}{\left(1-\left|\boldsymbol{y}\right|^2\right)^{\sfrac{3}{2}}}\lesssim_{c_\sigma,c_\mu,c_\nu,b_\sigma,b_\mu,b_\nu} \min\left(2^{3\Bar{k}},2^{3k_2}\right)}.$$
Then, we apply intermediate value theorem to obtain that
\begin{align*}
   \left|\eta-\rho\,\xi\right|&=\left|\xi-\eta-\Tilde{\eta}\right|=\left|\boldsymbol{g}^{-1}(\boldsymbol{g}(\xi-\eta))-\boldsymbol{g}^{-1}(\boldsymbol{g}(\Tilde{\eta}))\right|\le\left\|J(\boldsymbol{g}^{-1}(\boldsymbol{\theta}))\right\|_{L^\infty}\cdot\left|\boldsymbol{g}(\xi-\eta)-\boldsymbol{g}(\Tilde{\eta})\right| \\
   &\le C_1(c_\sigma,c_\mu,c_\nu,b_\sigma,b_\mu,b_\nu) \min\left(2^{3\Bar{k}},2^{3k_2}\right)\cdot \left|\nabla_\eta\Phi\right|   \\
   &\le C_1(c_\sigma,c_\mu,c_\nu,b_\sigma,b_\mu,b_\nu)\cdot C\cdot \min\left(1,2^{3k_2}\right).
\end{align*}
Plug in $\rho=\frac{1}{2}$, and we have $\left|\eta-\frac{1}{2}\xi\right|<C_1\cdot C \cdot \min\left(1,2^{3k_2}\right)$, where $C_1\cdot C>0$ could be large. Now, if $k_1,k_2\le0$, then we have
$$\left|\eta\right|<C_1\cdot C \cdot 2^{3k_2}+\frac{1}{2}\left|\xi\right|.$$
Since $\left|\xi\right|\sim 2^k\ll 2^{3k_2}$, $\left|\eta\right|\sim 2^{k_2}$ and $C$ can be selected very small if needed, the above inequality is a contradiction. Note that at this moment, $C$ still only depends on $c_\sigma,c_\mu,c_\nu,b_\sigma,b_\mu,b_\nu$. On the other hand, if $k_1,k_2\ge 0$, then we have
$$\left|\eta\right|<C_1\cdot C +\frac{1}{2}\left|\xi\right|,$$
and we can similarly get a contradiction by taking $C=C(c_\sigma,c_\mu,c_\nu,b_\sigma,b_\mu,b_\nu)$ small enough.
\par
Now let's consider the case (b4) first. We first assume that $c_\mu-c_\nu>0$ and $c_\mu^2 b_\nu-c_\nu^2 b_\mu>0$. Therefore, there exists $\varepsilon_1>0$ and $\varepsilon_3>0$ very small, such that $c_\mu^2 \lambda_1^2-c_\nu^2 \lambda_3^2>0$ for all $1-\varepsilon_1<\lambda_1<1+\varepsilon_1$ and $1-\varepsilon_3<\lambda_3<1+\varepsilon_3$. Since $D_0\gg D_1$, we can decompose $\xi-\eta=\lambda_1\eta+\lambda_2\eta^\bot$ and $\left|\xi-\eta\right|=\lambda_3\left|\eta\right|$, where $1-\varepsilon_1<\lambda_1<1+\varepsilon_1$ and $1-\varepsilon_3<\lambda_3<1+\varepsilon_3$. Now, from $\left|\nabla_\eta\Phi\right|<C\cdot 2^{-3\Bar{k}}$, we can get that
$$\left|\frac{c_\mu^2 \lambda_1 \eta}{\sqrt{c_\mu^2 \lambda_3^2 \left|\eta\right|^2+b_\mu^2}}+\frac{c_\mu^2 \lambda_2 \eta^\bot}{\sqrt{c_\mu^2 \lambda_3^2 \left|\eta\right|^2+b_\mu^2}}-\frac{c_\nu^2\eta}{\sqrt{c_\nu^2 \left|\eta\right|^2+b_\nu^2}}\right|^2<C^2\cdot 2^{-6\Bar{k}}.$$
By Pythagorean theorem, it is equivalent to 
$$\left|\frac{c_\mu^2 \lambda_1}{\sqrt{c_\mu^2 \lambda_3^2 \left|\eta\right|^2+b_\mu^2}}-\frac{c_\nu^2}{\sqrt{c_\nu^2 \left|\eta\right|^2+b_\nu^2}}\right|^2\left|\eta\right|^2+\frac{c_\mu^4 \lambda_2^2}{\sqrt{c_\mu^2 \lambda_3^2 \left|\eta\right|^2+b_\mu^2}}\left|\eta^\bot\right|^2<C^2\cdot 2^{-6\Bar{k}}.$$
Neglect the second term on LHS, and we get that
\begin{align}
   \displaystyle{\left|\frac{c_\mu^2 \lambda_1}{\sqrt{c_\mu^2 \lambda_3^2 \left|\eta\right|^2+b_\mu^2}}-\frac{c_\nu^2}{\sqrt{c_\nu^2 \left|\eta\right|^2+b_\nu^2}}\right|\left|\eta\right|<C \cdot 2^{-3\Bar{k}}}. \tag{6.1}\label{6.1} 
\end{align} \par
\vspace{0.8em}
\noindent If $k_2\ge 0$, then this implies that
\begin{align}
   \frac{c_\mu^2 \lambda_1}{\sqrt{c_\mu^2 \lambda_3^2 \left|\eta\right|^2+b_\mu^2}}<C\cdot 2^{-4k_2}+\frac{c_\nu^2}{\sqrt{c_\nu^2 \left|\eta\right|^2+b_\nu^2}}. \tag{6.2}\label{6.2}
\end{align}
Square on both hand sides and rearrange the terms, and we get that
$$c_\mu^2 c_\nu^2\left(c_\mu^2 \lambda_1^2-c_\nu^2 \lambda_3^2\right)\left|\eta\right|^2+\left(c_\mu^4 b_\nu^2-c_\nu^4 b_\mu^2\right)<C_2\cdot C \cdot 2^{-k_2}\le C_2\cdot C.$$
Since $c_\mu^2 \lambda_1^2-c_\nu^2 \lambda_3^2>0$, $c_\mu^4 b_\nu^2-c_\nu^4 b_\mu^2>0$, and $C$ can be selected as small as we wish, the above inequality is a contradiction. If $k_2<0$, then from (\ref{6.1}), we will have that
$$\frac{c_\mu^2 \lambda_1}{\sqrt{c_\mu^2 \lambda_3^2 \left|\eta\right|^2+b_\mu^2}}<C\cdot 2^{-k_2}+\frac{c_\nu^2}{\sqrt{c_\nu^2 \left|\eta\right|^2+b_\nu^2}}.$$
Do the same calculation as above, and we get that
$$c_\mu^2 c_\nu^2\left(c_\mu^2 \lambda_1^2-c_\nu^2 \lambda_3^2\right)\left|\eta\right|^2+\left(c_\mu^4 b_\nu^2-c_\nu^4 b_\mu^2\right)<C_2\cdot C \cdot 2^{-2k_2}.$$
Since $c_\mu^2 \lambda_1^2-c_\nu^2 \lambda_3^2>0$, $c_\mu^4 b_\nu^2-c_\nu^4 b_\mu^2>0$, $-D_1\le k_2<0$ and $C$ can be selected as small as we wish (for example, we could pick 
\begin{align*}
    C&<\frac{1}{2\cdot C_2\cdot 2^{2D_0}}\cdot \left(c_\mu^2 c_\nu^2\left(c_\mu^2 \lambda_1^2-c_\nu^2 \lambda_3^2\right)2^{-2D_0}+\left(c_\mu^4 b_\nu^2-c_\nu^4 b_\mu^2\right)\right) \\
    &\ll\frac{1}{2\cdot C_2\cdot 2^{2D_1}}\cdot \left(c_\mu^2 c_\nu^2\left(c_\mu^2 \lambda_1^2-c_\nu^2 \lambda_3^2\right)2^{-2D_1}+\left(c_\mu^4 b_\nu^2-c_\nu^4 b_\mu^2\right)\right) 
\end{align*}
), the above inequality is again a contradiction. Now, let's assume that  $c_\mu-c_\nu<0$ and $c_\mu^2 b_\nu-c_\nu^2 b_\mu<0$. In this case, we just need to use another version of (\ref{6.2}) from (\ref{6.1}), namely
$$\frac{c_\nu^2}{\sqrt{c_\nu^2 \left|\eta\right|^2+b_\nu^2}}<
\begin{cases}
   \displaystyle{C\cdot 2^{-4k_2}+\frac{c_\mu^2 \lambda_1}{\sqrt{c_\mu^2 \lambda_3^2 \left|\eta\right|^2+b_\mu^2}}}, &\mbox{if }k_2\ge 0 \\
   \displaystyle{C\cdot 2^{-k_2}+\frac{c_\mu^2 \lambda_1}{\sqrt{c_\mu^2 \lambda_3^2 \left|\eta\right|^2+b_\mu^2}}}, &\mbox{if }k_2<0
\end{cases}.$$
Then, we will have that
$$c_\mu^2 c_\nu^2\left(c_\nu^2 \lambda_3^2-c_\mu^2 \lambda_1^2\right)\left|\eta\right|^2+\left(c_\nu^4 b_\mu^2-c_\mu^4 b_\nu^2\right)<
\begin{cases}
   C_2\cdot C \cdot 2^{-k_2}, &\mbox{if }k_2\ge 0\\
   C_2\cdot C \cdot 2^{-2k_2}, &\mbox{if }k_2< 0
\end{cases}.$$
So, we can obtain a contradiction similarly as before, since we now have $c_\nu^2 \lambda_3^2-c_\mu^2 \lambda_1^2>0$, $c_\nu^4 b_\mu^2-c_\mu^4 b_\nu^2>0$. \par
The cases (b2) and (b3) can be proved exactly as the proof of (b4).\par
Now, let's assume that $\mu\cdot \nu<0$. Then, like before, we still do the orthogonal decomposition:
$\xi-\eta=\lambda_1\eta+\lambda_2\eta^\bot$ and $\left|\xi-\eta\right|=\lambda_3\left|\eta\right|$, where $1-\varepsilon_1<\lambda_1<1+\varepsilon_1$ and $1-\varepsilon_3<\lambda_3<1+\varepsilon_3$. Recall that we have assumed that $\left|\nabla_\eta\phi\right|<C\cdot 2^{-3\Bar{k}}$, so we get that
$$\left|\frac{c_\mu^2 \lambda_1 \eta}{\sqrt{c_\mu^2 \lambda_3^2 \left|\eta\right|^2+b_\mu^2}}+\frac{c_\mu^2 \lambda_2 \eta^\bot}{\sqrt{c_\mu^2 \lambda_3^2 \left|\eta\right|^2+b_\mu^2}}+\frac{c_\nu^2\eta}{\sqrt{c_\nu^2 \left|\eta\right|^2+b_\nu^2}}\right|^2<C^2\cdot 2^{-6\Bar{k}}.$$
Neglect the small orthogonal term as before, we get that
$$\displaystyle{\left|\frac{c_\mu^2 \lambda_1}{\sqrt{c_\mu^2 \lambda_3^2 \left|\eta\right|^2+b_\mu^2}}+\frac{c_\nu^2}{\sqrt{c_\nu^2 \left|\eta\right|^2+b_\nu^2}}\right|\left|\eta\right|<C \cdot 2^{-3\Bar{k}}}.$$
Since both two terms in the first absolute sign are positive, we can furthermore neglect the first term in the first absolute sign and get that
$$\frac{c_\nu^2}{\sqrt{c_\nu^2 \left|\eta\right|^2+b_\nu^2}}<
\begin{cases}
   C \cdot 2^{-4k_2}, &\mbox{if }k_2\ge 0 \\
   C \cdot 2^{-k_2}, &\mbox{if }k_2< 0
\end{cases}.$$
If $k_2\ge 0$, then LHS $\sim\left|\eta\right|^{-1}\sim 2^{-k_2}$. So, if we select $C$ small enough, then the above inequality is a contradiction. If $k_2< 0$, then LHS $\sim 1$. Since $0<-k_2\le D_1$, we can select $C$ small enough so that the above inequality is also a contradiction.
\end{proof}
\vspace{0.8em}
\begin{cor}
There exists a large constant $D_0=D_0(c_\sigma,c_\mu,c_\nu,b_\sigma,b_\mu,b_\nu)>0$ such that if we have $\max (k,k_1,k_2)\ge D_0$, $b_\sigma-b_\mu-b_\nu\neq0$, and $(c_\mu-c_\nu)(c_\mu^2b_\nu-c_\nu^2b_\mu)\ge0$, then there exists a small positive constant $C=C(c_\sigma,c_\mu,c_\nu,b_\sigma,b_\mu,b_\nu)>0$, such that either $\left|\Phi\right|\ge C\cdot 2^{-\Bar{k}}$ or $\left|\nabla_\eta\Phi\right|\ge C\cdot 2^{-3\Bar{k}}$.
\end{cor}
\begin{proof}
With Lemma 6.1 before, we only need to consider the case when $\max (k,k_1,k_2)\ge D_0$ and $\min (k,k_1,k_2)\ge -D_0$. When $c_\mu\neq c_\nu$, we just need follow the proof of Proposition 8.2 (3) in \cite{y1}. When $c_\mu=c_\nu=1$, we can still follow the proof of Proposition 8.2 (3) in \cite{y1}, but we have some slight difference here. We again prove by contradiction, so we assume $\left|\Phi\right|<C\cdot 2^{-\Bar{k}}$ and $\left|\nabla_\eta\Phi\right|<C\cdot 2^{-3\Bar{k}}$. First, we may assume that $b_\mu,b_\nu>0$ (the other cases can be done similarly; only need to note that $1-\rho\in(-\infty,0)\bigcup\,(1,+\infty)$ in other cases). Now, we first note that
\begin{align*}
   \Phi=&\underbrace{\sqrt{c_\sigma^2\left|\xi\right|^2+b_\sigma^2}-\sqrt{\left|\xi\right|^2+(b_\mu+b_\nu)^2}}_{\displaystyle{I}}  \\
   &+\underbrace{(1-\rho)\sqrt{\left|\xi\right|^2+(b_\mu+b_\nu)^2}-\sqrt{\left|\xi-\eta\right|^2+b_\mu^2}+\rho\sqrt{\left|\xi\right|^2+(b_\mu+b_\nu)^2}-\sqrt{\left|\eta\right|^2+b_\mu^2}}_{\displaystyle{II}},  \\
   \left|\nabla_\eta\Phi\right|=&\left|\frac{\displaystyle{\frac{\xi-\eta}{1-\rho}}}{\sqrt{\left|\displaystyle{\frac{\xi-\eta}{1-\rho}}\right|^2+(b_\mu+b_\nu)^2}}-\frac{\displaystyle{\frac{\eta}{\rho}}}{\displaystyle{\sqrt{\left|\frac{\eta}{\rho}\right|^2+(b_\mu+b_\nu)^2}}}\right|\ll2^{-3\Bar{k}},\tag{6.3}\label{6.3}
\end{align*}
where $\rho=\frac{b_\nu}{b_\mu+b_\nu}$ is as before. Then by Taylor expansion, we have 
\begin{align*}
   II=&(1-\rho)\left(\sqrt{\left|\xi\right|^2+(b_\mu+b_\nu)^2}-\sqrt{\left|\frac{\xi-\eta}{1-\rho}\right|^2+(b_\mu+b_\nu)^2}\right) \\
   &+\rho\left(\sqrt{\left|\xi\right|^2+(b_\mu+b_\nu)^2}-\sqrt{\left|\frac{\eta}{\rho}\right|^2+(b_\mu+b_\nu)^2}\right) \\
   =&\left(\frac{\displaystyle{\frac{\xi-\eta}{1-\rho}}}{\sqrt{\left|\displaystyle{\frac{\xi-\eta}{1-\rho}}\right|^2+(b_\mu+b_\nu)^2}}-\frac{\displaystyle{\frac{\eta}{\rho}}}{{\displaystyle{\sqrt{\left|\frac{\eta}{\rho}\right|^2+(b_\mu+b_\nu)^2}}}}\right)(\eta-\rho\,\xi)+O\left((\eta-\rho\,\xi)^2\right)
\end{align*}
Note that from (\ref{6.1}) and $\left|\nabla_\eta\Phi\right|\ll 2^{-4\Bar{k}}$, we see that $\left|\eta-\rho\xi\right|\lesssim 1$. From (\ref{6.3}), we estimate
$$\left|II\right|\lesssim2^{-3\Bar{k}}\cdot 1+2^{-2\Bar{k}}\lesssim 2^{-2\Bar{k}}$$
Thus, since we have assumed that $\left|\Phi\right|\ll2^{-\Bar{k}}$, we now obtain that 
$$\left|\sqrt{c_\sigma^2\left|\xi\right|^2+b_\sigma^2}-\sqrt{\left|\xi\right|^2+(b_\mu+b_\nu)^2}\right|\ll2^{-\Bar{k}}$$
Do the Taylor expansion again, and the first two terms of LHS would be
$$\begin{cases}
   \left(b_\sigma-\left(b_\mu+b_\nu\right)\right)+\left(\frac{c_\sigma^2}{b_\sigma}-\frac{1}{b_\mu+b_\nu}\right)\left|\xi\right|^2,&\mbox{ if }-D_0<k<0 \vspace{0.7em} \\
   (c_\sigma-1)\left|\xi\right|+\displaystyle{\left(\frac{b_\sigma^2}{2c_\sigma}-\frac{(b_\mu+b_\nu)^2}{2}\right)}\frac{1}{\left|\xi\right|},&\mbox{ if }k\ge 0 
\end{cases}.$$
When $-D_0<k<0$, we must have $\Bar{k}=k_1=k_2>D_0$. Then, the RHS will not be larger than $C\cdot 2^{-D_0}$, which can be very small since we can take $C$ small enough. This implies that $b_\sigma-b_\mu-b_\nu=0$, which is a contradiction. On the other hand, if $k\ge 0$, then we first must have $c_\sigma=1$. Focusing on the $\displaystyle{\frac{1}{\left|\xi\right|}}$ term, we must have $k>k_1,k_2$ and $k>D_0$. So in this case we must have
$$\frac{1}{2}\left|b_\sigma^2-\left(b_\mu+b_\nu\right)^2\right|\cdot 2^{-k}<C\cdot 2^{-k}.$$ However, due to the smallness of $C$, this is again a contradiction,
\end{proof}
\vspace{0.8em}
\begin{lemma}
There exists a large constant $D_0=D_0(c_\sigma,c_\mu,c_\nu,b_\sigma,b_\mu,b_\nu)>0$ such that if $\mu+\nu= 0$, $b_\sigma-b_\mu-b_\nu\neq0$, and $(c_\mu-c_\nu)(c_\mu^2b_\nu-c_\nu^2b_\mu)\ge0$, then there exists a small positive constant $C=C\left(c_\sigma,c_\mu,c_\nu,b_\sigma,b_\mu,b_\nu,D_0\right)>0$, such that $\left|\nabla_\eta\Phi\right|\ge C\cdot 2^{-4\Bar{k}}$.
\end{lemma}
\begin{proof}
The proof is similar to the proof of case (b4) of $3^\circ$ of Lemma 6.1. We prove by contradiction, so we assume that $\left|\nabla_\eta\Phi\right|< C\cdot 2^{-4\Bar{k}}$. Note that we have $c_\mu=c_\nu$ and $b_\mu=-b_\nu$. By Lemma 6.1 and Corollary 6.2, we may assume that $k,k_1,k_2\in \left[{-D_0},{D_0}\right]$. Then, $\left|\nabla_\eta\Phi\right|< C\cdot 2^{-4\Bar{k}}$ implies that 
$$\left|\frac{c_\mu^2(\xi-\eta)}{\sqrt{c_\mu^2\left|\xi-\eta\right|^2+b_\mu^2}}+\frac{c_\mu^2\eta}{\sqrt{c_\mu^2\left|\eta\right|^2+b_\mu^2}}\right|<C\cdot 2^{-4\Bar{k}}.$$
We may assume that $\left|\xi-\eta\right|\ge\left|\eta\right|$, since otherwise we only need to flip the sign in (\ref{6.4}). In this case, we have
$$\frac{c_\mu^2\left|\xi-\eta\right|}{\sqrt{c_\mu^2\left|\xi-\eta\right|^2+b_\mu^2}}<C\cdot 2^{-4\Bar{k}}+\frac{c_\mu^2\left|\eta\right|}{\sqrt{c_\mu^2\left|\eta\right|^2+b_\mu^2}}.$$
Square on both sides, rearrange the terms, and we see that
\begin{align}
   c_\mu^4 b_\mu^2\left(\left|\xi-\eta\right|^2-\left|\eta\right|^2\right)<
\begin{cases}
   C, &\mbox{ if } \left\{k_1,k_2\ge 0\right\} \mbox{ or }\left\{k,k_1,k_2<0\right\}  \\
   C\cdot 2^{-2\Bar{k}}, &\mbox{ if }k_2<0,k,k_1\ge 0
\end{cases}. \tag{6.4}\label{6.4}
\end{align}
Note that $\left|\xi-\eta\right|\sim 2^{2k_1}$ , $\left|\eta\right|\sim 2^{2k_2}$ and $C$ can be selected very small. \par
\noindent$1^\circ$: $k_1,k_2\ge 0$ \par
In this case, we must have $k_1=k_2>0$. Now, we do the decomposition $\xi-\eta=\lambda_1\eta+\lambda_2\eta^\bot$ and $\left|\xi-\eta\right|=\lambda_3\left|\eta\right|$. Since $\left|\xi-\eta\right|^2-\left|\eta\right|^2$ is very small and $k_1,k_2\in \left[-D_0,D_0\right]$, $\lambda_3$ is very close to 1, say $\lambda_3 \in [0.99,1.01]$. Therefore,\vspace{0.3em} $\lambda_1,\lambda_2\in [-1.01, 1.01]$. If $1+\lambda_1\ge\frac{1}{100}$, then from $\left|\nabla_\eta\Phi\right|<C\cdot 2^{-4\Bar{k}}$, we can get that
$$\left|\frac{c_\mu^2 \lambda_1 \eta}{\sqrt{c_\mu^2 \lambda_3^2 \left|\eta\right|^2+b_\mu^2}}+\frac{c_\mu^2 \lambda_2 \eta^\bot}{\sqrt{c_\mu^2 \lambda_3^2 \left|\eta\right|^2+b_\mu^2}}+\frac{c_\mu^2\eta}{\sqrt{c_\mu^2 \left|\eta\right|^2+b_\mu^2}}\right|^2<C^2\cdot 2^{-8\Bar{k}}.$$
Like before, we neglect the perpendicular term and the first term. Note that $k_2>0$, so we get that
$$\frac{c_\mu^2}{\sqrt{c_\mu^2\left|\eta\right|^2+b_\mu^2}}<C\cdot 2^{-5k_2}.$$
Since LHS $\sim \left|\eta\right|^{-1}\sim 2^{-k_2}$ and $C$ can be taken as small as we wish, the above inequality is actually a contradiction. On the other hand, if $-\frac{1}{100}\le 1+\lambda_1<\frac{1}{100}$, then we only plug $\left|\xi-\eta\right|=\lambda_3 \left|\eta\right|$ into $\left|\nabla_\eta\Phi\right|<C\cdot 2^{-4\Bar{k}}$, and get that
\begin{align}
   \left|\frac{c_\mu^2(\xi-\eta)}{\sqrt{c_\mu^2\lambda_3^2\left|\eta\right|^2+b_\mu^2}}+\frac{c_\mu^2\eta}{\sqrt{c_\mu^2\left|\eta\right|^2+b_\mu^2}}\right|<C\cdot 2^{-4\Bar{k}}. \tag{6.5}\label{6.5}
\end{align}
Note that when $\left|\eta\right|\in\left[1,2^{D_0}\right]$, we have that $\sqrt{c_\mu^2\lambda_3^2\left|\eta\right|^2+b_\mu^2}\sim c_\mu \lambda_3 \left|\eta\right|$ and $\sqrt{c_\mu^2\left|\eta\right|^2+b_\mu^2}\sim c_\mu \left|\eta\right|$. Since $\lambda_3 \in [0.99,1.01]$, from (\ref{6.5}), we get that
$$\frac{c_\mu^2}{\sqrt{c_\mu^2\left|\eta\right|^2+b_\mu^2}}\left|\xi-\eta+\eta\right|=\frac{c_\mu^2}{\sqrt{c_\mu^2\left|\eta\right|^2+b_\mu^2}}\left|\xi\right|<C\cdot 2^{-4\Bar{k}}.$$
Note that $k_1=k_2>0$ and $k\le k_1+1$, so we have that LHS$\displaystyle{\sim\frac{\left|\xi\right|}{\left|\eta\right|}\sim 2^{k-k_1}}\gtrsim 2^{-2D_0}$. Since the maximum of RHS is $C$, if we take $C$ small enough (depends on $D_0$), the above inequality will be a contradiction.\par
\noindent$2^\circ$: $k,k_1,k_2<0$ \par
In this case, if we take $C$ small enough ($C$ now depends on $D_0$), then we will have two subcases: (a) $-D_0-3\le k_1,k_2<0$ and $k_1=k_2$; (b) $k_1,k_2<-D_0-3$. However, since we already assume that $k,k_1,k_2\in\left[-D_0,D_0\right]$, we only need to consider the first subcase. Now, we still do the decomposition $\xi-\eta=\lambda_1\eta+\lambda_2\eta^\bot$ and $\left|\xi-\eta\right|=\lambda_3\left|\eta\right|$, where $\lambda_1,\lambda_2\in\left[-1.01,1.01\right]$ and $\lambda_3\in\left[0.99,1.01\right]$. If $1+\lambda_1\ge\frac{1}{100}$, then we just plug these back into $\left|\nabla_\eta\Phi\right|<C$, play the same game as in the case $1^\circ$, and we will get that
$$\frac{c_\mu^2\left|\eta\right|}{\sqrt{c_\mu^2\left|\eta\right|^2+b_\mu^2}}<C.$$
Now, LHS $\sim \left|\eta\right| \sim 2^{k_2}$. Since $-D_0-3\le k_2<0$ and $C$ can be selected very small, the above inequality is again a contradiction. On the other hand, if $-\frac{1}{100}\le 1+\lambda_1<\frac{1}{100}$, then we again only plug $\left|\xi-\eta\right|=\lambda_3 \left|\eta\right|$ into $\left|\nabla_\eta\Phi\right|<C$, and get that
$$\frac{c_\mu^2}{\sqrt{c_\mu^2\left|\eta\right|^2+b_\mu^2}}\left|\xi\right|<C.$$
Then, LHS$\sim\left|\xi\right|\sim 2^k\ge 2^{-D_0}$. So, if we take $C$ small enough, then we will again get a contradiction here.\par
\noindent$3^\circ$: $k_2<0,k,k_1\ge 0$\par
Now, since we know that $\left|\xi-\eta\right|^2-\left|\eta\right|^2\ll 1$, we must have that $\left|\xi\right|$ is very small. This contradicts that $k>0$.
\end{proof}
\vspace{0.8em}
\begin{prop}
Let $D_0=D_0(c_\sigma,c_\mu,c_\nu,b_\sigma,b_\mu,b_\nu)>0$ be a large constant. \par
(a) If $\max (k,k_1,k_2)\le D_0$, $b_\sigma-b_\mu-b_\nu\neq0$, and $(c_\mu-c_\nu)(c_\mu^2b_\nu-c_\nu^2b_\mu)\ge0$, then there exists a small positive constant $C=C(c_\sigma,c_\mu,c_\nu,b_\sigma,b_\mu,b_\nu,D_0)>0$, such that $2^{3\Bar{k}}\left|\Phi\right|+\left|\nabla_\xi\Phi\right|+\left|\nabla_\eta\Phi\right|\ge C$.\par
(b) If releasing the restriction $\max (k,k_1,k_2)\le D_0$ in part (a), then we have $2^{3\Bar{k}}\left|\Phi\right|+\left|\nabla_\xi\Phi\right|+2^{4\Bar{k}}\left|\nabla_\eta\Phi\right|\ge C$.
\end{prop}

\begin{proof}
\ \par
(a) We prove by contradiction, so we assume $\left|\Phi\right|<C\cdot 2^{-3\Bar{k}}$, $\left|\nabla_\xi\Phi\right|<C$ and $\left|\nabla_\eta\Phi\right|<C$, where $C$ will be selected as small as we wish later on. Now, we divide it into two cases:\\
\noindent$1^\circ$: $\min (k,k_1,k_2)\le-D_0$ \par
In this case, since $\left|\Phi\right|<C\cdot 2^{-3\Bar{k}}$, by Lemma 6.1, we get $\left|\nabla_\eta\Phi\right|>C\cdot 2^{-3\Bar{k}}$. Take $C=C(D_0)>0$ small enough, and this contradicts $\left|\nabla_\eta\Phi\right|<C$, since $\max (k,k_1,k_2)\le D_0$. \par
\noindent$2^\circ$: $k,k_1,k_2\in[-D_0,D_0]$ \par
In this case, by take $C=C(D_0)>0$ small enough, we can assume that $\left|\Phi\right|<C\cdot 2^{-3\Bar{k}}<C<\frac{1}{2}2^{-20D_0}$ and $\left|\nabla_\xi\Phi\right|<C<2^{-10D_0}$. First we claim that $\mu\neq -\nu$. This is because otherwise we will have $c_\mu=c_\nu$ and $b_\mu=-b_\nu$. Now, note that
$$\left|\nabla_\eta\Phi\right|=\left|\frac{c_\mu^2(\xi-\eta)}{\sqrt{c_\mu^2\left|\xi-\eta\right|^2+b_\mu^2}}+\frac{c_\nu^2\eta}{\sqrt{c_\nu^2\left|\eta\right|^2+b_\nu^2}}\right|<C.$$
Then, by doing the orthogonal decomposition as before, it's easy to see that $\left|\xi\right|\ll 1$, which contradicts $k\in[-D_0,D_0]$. Since $\mu\neq -\nu$, we are able to use Proposition 6.5 (a) to get that $\left|\eta-\frac{p^{+}(\left|\xi\right|)}{\left|\xi\right|}\xi\right|<2^{8D_0}C$, where $p^+:\mathbb{R}\longrightarrow\mathbb{R}$ is an odd smooth function such that $\nabla_\eta\Phi\left(\xi,p^+(\left|\xi\right|)\frac{\xi}{\left|\xi\right|}\right)=0$. Note that 
$$\left|\nabla_\xi\Phi\right|=\left|\frac{c_\sigma^2\xi}{\sqrt{c_\sigma^2\left|\xi\right|^2+b_\sigma^2}}-\frac{c_\mu^2\left(\xi-\frac{p^+(\left|\xi\right|)}{\left|\xi\right|}\xi\right)}{\sqrt{c_\mu^2\left|\xi-\frac{p^+(\xi)}{\left|\xi\right|}\xi\right|^2+b_\mu^2}}-\frac{c_\mu^2(\xi-\eta)}{\sqrt{c_\mu^2\left|\xi-\eta\right|^2+b_\mu^2}}+\frac{c_\mu^2\left(\xi-\frac{p^+(\left|\xi\right|)}{\left|\xi\right|}\xi\right)}{\sqrt{c_\mu^2\left|\xi-\frac{p^+(\xi)}{\left|\xi\right|}\xi\right|^2+b_\mu^2}}\right|.$$
Therefore, by applying the mean value theorem to the last two terms on the RHS of the above formula, we get that
$$\left|\frac{c_\sigma^2 s}{\sqrt{c_\sigma^2s^2+b_\sigma^2}}-\frac{c_\mu^2\left(s-p^+(s)\right)}{\sqrt{c_\mu^2\left(s-p^+(s)\right)^2+b_\mu^2}}\right|\lesssim 2^{8D_0}\,C,$$
where $s\triangleq\left|\xi\right|$.
However, this contradicts the Lemma 5.8 in \cite{a1}, since $C>0$ can be taken very small.\par
(b) Denote $\Psi(s)=\sqrt{c_\sigma^2 s^2+b_\sigma^2}-\sqrt{c_\mu^2 (p_+(s)-s)^2+b_\mu^2}-\sqrt{c_\nu^2 p_+^2(s)+b_\nu^2}$, where $p_+$ is the function in Proposition 6.5 (a). Then, by elementary calculation, we have the identity
\begin{align*}
    \partial_s \Psi(s)=\frac{c_\sigma^2 s^2}{\sqrt{c_\sigma^2 s^2+b_\sigma^2}}+\frac{c_\mu^2 (p_+(s)-s)}{\sqrt{c_\mu^2 (p_+(s)-s)^2+b_\mu^2}}-\left[\frac{c_\mu^2 (p_+(s)-s)}{\sqrt{c_\mu^2 (p_+(s)-s)^2+b_\mu^2}}-\frac{c_\nu^2 p_+(s)}{\sqrt{c_\nu^2 p_+^2(s)+b_\nu^2}}\right]\cdot p_+^\prime(s).
\end{align*}
Using this identity, $\left|p_+^\prime(s)\right|\lesssim 1$, and Proposition 6.5 (a), we recheck the proof of Lemma 5.8 in \cite{a1}, and find that it also holds if we don't have the assumption $\max (k,k_1,k_2)\le D_0$ but assume that $\left|\nabla_\eta\Phi\right|\le C\cdot 2^{-4\Bar{k}}$. This, together with Proposition 6.5 (a), implies the desired conclusion.
\end{proof}
\vspace{0.8em}
\begin{prop}
Assume $b_\sigma-b_\mu-b_\nu\neq0$ and $(c_\mu-c_\nu)(c_\mu^2b_\nu-c_\nu^2b_\mu)\ge0$. \par
(a) If $\nu+\mu\neq 0$, then there exists a function $p=p_{\mu\nu}:\mathbb{R}^2\rightarrow\mathbb{R}^2$ such that $\left|p(\xi)\right|\lesssim\left|\xi\right|$ and $\left|p(\xi)\right|\approx\left|\xi\right|$ for small $\xi$, and
$$\nabla_\eta\Phi(\xi,\eta)=0\ \ \ \iff\ \ \ \eta=p(\xi).$$
There is an odd smooth function $p_+:\mathbb{R}^+\rightarrow\mathbb{R}$, such that $p(\xi)=p_+(\left|\xi\right|)\xi/\left|\xi\right|$. Moreover, if \vspace{0.5em}$\left|\eta\right|+\left|\xi-\eta\right|\le U\in\left[1,+\infty\right)$ and $\left|\nabla_\eta\Phi(\xi,\eta)\right|\le\varepsilon$, then $\left|\eta-p(\xi)\right|\lesssim\varepsilon U^4$, and, for any $s\in\mathbb{R}^+$,
$$\left|D^\alpha p_+(s)\right|\lesssim_\alpha\,1,\ \ \ \ \left|p^\prime_+(s)\right|\gtrsim (1+\left|s\right|)^{-3},\ \ \ \ \left|1-p^\prime_+(s)\right|\gtrsim (1+\left|s\right|)^{-3}.$$\par
(b) If $\nu+\mu=0$ and $\nabla_\eta\Phi(\xi,\eta)=0$, then $\xi=0$. \par
(c) If $\xi\neq 0$, then $\det\left[\left(\nabla^2_{\eta,\eta}\Phi\right)(\xi,p(\xi))\right]\neq 0$.
\end{prop}
\begin{proof}
Part (a) is proved in Lemma 8.2 (iii) in \cite{y2}. Part (b) can be seen by some easy elementary computations. Part (c) is proved in Section 1.2.5 in \cite{a1}.
\end{proof}
\vspace{0.8em}
\begin{prop}
Assume $b_\sigma-b_\mu-b_\nu\neq0$ and $(c_\mu-c_\nu)(c_\mu^2b_\nu-c_\nu^2b_\mu)\ge0$. Then we have \\ [4pt]
(a) $\left|\Phi(\xi,0)\right|+\left|\nabla_\eta\Phi(\xi,0)\right|\gtrsim 1$, and $\left|\Phi(\xi,0)\right|+\left|\nabla_\xi\Phi(\xi,0)\right|\gtrsim 1$; \\ [4pt]
(b) $\left|\Phi(\xi,0)\right|$ (or $\left|\Phi(0,\eta)\right|$) , as a function of $\xi$ (or $\eta$), only has at most one zero.
\end{prop}
\begin{proof}
(a) We prove by contradiction, so we assume that $\left|\Phi(\xi,0)\right|+\left|\nabla_\eta\Phi(\xi,0)\right|<C$, and $\left|\Phi(\xi,0)\right|+\left|\nabla_\xi\Phi(\xi,0)\right|<C$, where $C>0$ is very small. We first consider the first inequality. From $\left|\nabla_\eta\Phi(\xi,0)\right|=\displaystyle{\frac{c_\mu^2\left|\xi\right|}{\sqrt{c_\mu^2\left|\xi\right|^2+b_\mu^2}}<C}$, we get that
$$\left|\xi\right|<\frac{b_\mu}{\sqrt{c_\mu^4-c_\mu^2C^2}}C<\frac{b_\mu}{\frac{1}{2}\sqrt{c_\mu^4}}C\lesssim C,$$
where the second inequality is due to the smallness of $C>0$. Thus, we have
$$\Phi(\xi,0)=\sqrt{c_\sigma^2\left|\xi\right|^2+b_\sigma^2}-\sqrt{c_\mu^2\left|\xi\right|^2+b_\mu^2}-b_\nu>(1+\varepsilon)b_\sigma-(1-\varepsilon)b_\mu-b_\nu,$$
where $\varepsilon>0$ is very small depending on $C$. Since $b_\sigma-b_\mu-b_\nu\neq 0$, if we take $C>0$ small enough, then we will have $\left|\Phi(\xi,0)\right|>C$, which is a contradiction.\par
Next, we consider the second inequality. From above, we note that $\left|\xi\right|$ cannot be too small. Thus, from $\displaystyle{\left|\nabla_\xi\Phi(\xi,0)\right|=\left|\frac{c_\sigma^2}{\sqrt{c_\sigma^2\left|\xi\right|^2+b_\sigma^2}}-\frac{c_\mu^2}{\sqrt{c_\mu^2\left|\xi\right|^2+b_\mu^2}}\right|\cdot\left|\xi\right|<C}$, we get that
$$\frac{c_\sigma^4}{c_\sigma^2\left|\xi\right|^2+b_\sigma^2}<\frac{c_\mu^4}{c_\mu^2\left|\xi\right|^2+b_\mu^2}+\left(\frac{C}{\left|\xi\right|}\right)^2+2\frac{C}{\left|\xi\right|}\frac{c_\mu^2}{\sqrt{c_\mu^2\left|\xi\right|^2+b_\mu^2}}.$$
Eliminate all denominators, use the fact that $\frac{1}{\left|\xi\right|}<C_1$, where $C_1>0$ does not depend on $C$ and $\sqrt{c_\mu^2\left|\xi\right|^2+b_\mu^2}\sim\left|\xi\right|\le C_2\left|\xi\right|^2$ and we get that
$$\left(c_\sigma^4 c_\mu^2-c_\mu^4 c_\sigma^2-\left(C^2-2C C_2\right)c_\sigma^2 c_\mu^2\right)\left|\xi\right|^2+c_\sigma^4 b_\mu^2-c_\mu^4 b_\sigma^2<C^2\left(b_\mu^2 c_\sigma^2+c_\mu^2 b_\sigma^2+C_1 b_\mu^2 b_\sigma^2\right)+C\cdot C_2\,c_\mu^2 b_\sigma^2,$$
where $C_1,C_2\sim 1$ and $C\ll 1$.
Thus, if $C$ is taken small enough, the above inequality contradicts with our assumption $(c_\mu-c_\sigma)(c_\mu^2b_\sigma-c_\sigma^2b_\mu)\ge0$.\\ [6pt]
(b) We only prove the case of $\left|\Phi(\xi,0)\right|$, since the other case can be dealt with similarly. Note that
$$\left|\Phi(\xi,0)\right|=\pm \sqrt{c_\sigma^2\left|\xi\right|^2+b_\sigma^2}\mp\sqrt{c_\mu^2\left|\xi\right|^2+b_\mu^2}\mp b_\nu,$$
so if the first two terms have the same signs, then no zero will occur. Now, let's assume that the first two terms have the different signs, and WLOG we could assume $b_\sigma>0$, and $b_\mu>0$. Let
$$f(r)=\sqrt{c_\sigma^2 r^2+b_\sigma^2}-\sqrt{c_\mu^2 r^2+b_\mu^2}\mp b_\nu,\,\,\,\,\,\,\,\,(r>0)$$
then we have
$$f^{\prime}(r)=\left(\frac{c_\sigma^2}{\sqrt{c_\sigma^2\left|\xi\right|^2+b_\sigma^2}}-\frac{c_\mu^2}{\sqrt{c_\mu^2\left|\xi\right|^2+b_\mu^2}}\right)\,r.$$
Let $f^{\prime}(r)=0$, and we will get
\begin{align*}
c_\sigma^2 c_\mu^2 \left(c_\sigma+c_\mu\right) \left(c_\sigma-c_\mu\right) r^2+\left(c_\sigma^2 b_\mu+c_\mu^2 b_\sigma\right) \left(c_\sigma^2 b_\mu-c_\mu^2 b_\sigma\right)=0, 
\end{align*}
However, by our assumption, these two coefficients of $r$ above should be in the same signs, which leads to a contradiction. Thus, we have either $f^{\prime}(r)>0$ or $f^{\prime}(r)<0$. This implies that $f(r)$ has at most one zero.
\end{proof}
\vspace{0.8em}
Next, we state a general lemma that bounds the size of level sets of functions with non-vanishing high derivatives.
\begin{lemma}
    Suppose $Y:\mathbb{R}^n\rightarrow\mathbb{R}$ satisfies
    $$\sum_{\left|\alpha\right|\le q}\left|\partial_x^\alpha Y(s)\right|\gtrsim 1$$
    for some $q>1$, for all $x\in K$, where $K$ is a compact set contained in the closed ball $B_R$ centered at the origin. Moreover, assume that $\left\|\nabla Y\right\|_{C^q(B_{R+1})}\lesssim 1$, then for any $\varepsilon$ we have
    $$\left|\left\{x:\left|Y(x)\right|
    \le\varepsilon\right\}\right|\lesssim_R \varepsilon^{1/q}.$$
    Moreover, if $K$ is a compact set such that $\partial(K\cap l)$ has at most $O(1)$ points for any straight line $l$, and $Y$ is such that $\left|\nabla Y\right|\gtrsim \gamma$ on $K$, then we have
    $$\left|\left\{x:\left|Y(x)\right|\le\varepsilon\right\}\right|\lesssim_R \varepsilon\gamma^{-1}.$$
\end{lemma}
\begin{proof}
    This was proved in \cite{y2}.
\end{proof}
\vspace{0.8em}
Then, we collect several volume estimates related to the phase function $\Phi$, whose proofs require quite precise information about the dispersive relations. They will mainly used in the energy estimate in Section 4.
\vspace{0.8em}
\begin{prop}
\ \par
(a) Assume $b_\sigma-b_\mu-b_\nu\neq0$ and $(c_\mu-c_\nu)(c_\mu^2b_\nu-c_\nu^2b_\mu)\ge0$. If $k\ge 0$, $\varepsilon\le\frac{1}{2}$, and
$$E\triangleq\left\{(\xi,\eta):\,\max(\left|\xi\right|,\left|\eta\right|)\le 2^k,\left|\Phi(\xi,\eta)\right|\le 2^{-k}\varepsilon\right\},$$
then
$$\sup\limits_{\xi}{\int_{\mathbb{R}^2}\boldsymbol{1}_E(\xi,\eta)\,d\eta}+\sup\limits_{\eta}{\int_{\mathbb{R}^2}\boldsymbol{1}_E(\xi,\eta)\,d\xi}\lesssim 2^{9k}\varepsilon\log\left(\frac{1}{\varepsilon}\right)$$ \par
(b) Assume $b_\sigma-b_\mu-b_\nu\neq0$ and $(c_\mu-c_\nu)(c_\mu^2b_\nu-c_\nu^2b_\mu)\ge0$. If $\varepsilon\le\varepsilon^\prime\le\frac{1}{2}$, and
\begin{equation}
\begin{aligned}
    E_1^\prime\triangleq\left\{(\xi,\eta):\,\right.&\max(\left|\xi\right|,\left|\eta\right|)\le 2^k,\left|\Phi(\xi,\eta)\right|\le 2^{-k}\varepsilon,\left|\Upsilon(\xi,\eta)\right|  \\
    &\left.\le 2^{-3k}\varepsilon^\prime,\left|\nabla_\eta\Phi(\xi,\eta)\right|\ge 2^{-D_0}\right\},  \\[5pt]
    E_2^\prime\triangleq\left\{(\xi,\eta):\,\right.&\max(\left|\xi\right|,\left|\eta\right|)\le 2^k,\left|\Phi(\xi,\eta)\right|\le 2^{-k}\varepsilon,\left|\Upsilon(\xi,\eta)\right|  \\
    &\left.\le 2^{-3k}\varepsilon^\prime,\left|\nabla_\xi\Phi(\xi,\eta)\right|\ge 2^{-D_0}\right\}, \\
\end{aligned}
\end{equation}
where $\Upsilon$ is defined by
$$\Upsilon(\xi,\eta)=\nabla_{\xi,\eta}^2\Phi(\xi,\eta)\left[\nabla_{\xi}^{\bot}\Phi(\xi,\eta),\nabla_{\eta}^{\bot}\Phi(\xi,\eta)\right],$$
then
$$\sup\limits_{\xi}{\int_{\mathbb{R}^2}\boldsymbol{1}_{E_1^\prime}(\xi,\eta)\,d\eta}+\sup\limits_{\eta}{\int_{\mathbb{R}^2}\boldsymbol{1}_{E_2^\prime}(\xi,\eta)\,d\xi}\lesssim 2^{12k}\varepsilon\log\left(\frac{1}{\varepsilon}\right)\cdot \left(\varepsilon^\prime\right)^{\sfrac{1}{8}}.$$ \par
(c) Assume $b_\sigma-b_\mu-b_\nu\neq0$ and $(c_\mu-c_\nu)(c_\mu^2b_\nu-c_\nu^2b_\mu)\ge0$. If $\varepsilon\le\varepsilon^{\prime\prime}\le{\frac{1}{2}}$, $r_0\in[2^{-D_0},2^{D_0}]$ and consider the set
$$E^{\prime\prime}\triangleq\left\{(\xi,\eta):\max\left(\left|\xi\right|,\left|\eta\right|\right)\le 2^k,\left|\Phi(\xi,\eta)\right|\le 2^{-k}\varepsilon,\big|\left|\xi-\eta\right|-r_0\big|\le\varepsilon^{\prime\prime}\right\},$$
Then we can write $E^{\prime\prime}=E_1^{\prime\prime}\cup E_2^{\prime\prime}$ such that
$$\sup\limits_{\xi}{\int_{\mathbb{R}^2}\boldsymbol{1}_{E_1^{\prime\prime}}(\xi,\eta)\,d\eta}+\sup\limits_{\eta}{\int_{\mathbb{R}^2}\boldsymbol{1}_{E_2^{\prime\prime}}(\xi,\eta)\,d\xi}\lesssim 2^{12k}\varepsilon\log\left(\frac{1}{\varepsilon}\right)\cdot \left(\varepsilon^{\prime\prime}\right)^{\sfrac{1}{2}}.$$

\end{prop}
\begin{proof}
\par
(a) This was basically proved in Proposition 8.8 in  \cite{y2}. The only slight difference here is to show that 
\begin{align}
   \left|Z^\prime_\mp(r)\right|+\left|Z^{\prime\prime}_\mp(r)\right|\gtrsim 1\mbox{ , if }s,r\in\left[2^{-D_0},2^{D_0}\right], \tag{6.6}\label{6.6} 
\end{align}
where 
$$Z_\mp(r)\triangleq\pm\sqrt{c_\sigma^2 s^2+b_\sigma^2}\mp\sqrt{c_\mu^2\left|r-s\right|^2+b_\mu^2}\mp\sqrt{c_\nu^2r^2+b_\nu^2},\ \  \xi\triangleq(s,0),\ \  \eta\triangleq(r\cos\theta,r\sin\theta).$$
In fact, it suffices to show that  $\left|Z^\prime_\mp(r)\right|+\left|Z^{\prime\prime}_\mp(r)\right|\le C$ has empty solution set, where $C>0$ is very small. From (\ref{6.6}), we note that the solution set should be a neighborhood of $r_0$, where $r_0$ is a solution of
\begin{align}
   Z^\prime(r)=Z^{\prime\prime}(r)=0. \tag{6.7}\label{6.7}
\end{align}
Moreover, since $r$ has a positive lower bound, we must have $b_\mu \cdot b_\nu<0$ if $r>s$, and $b_\mu \cdot b_\nu>0$ if $r>s$. Therefore, if we let $\gamma\triangleq \max (r-s,s-r)>0$, then (\ref{6.7}) becomes 
\begin{numcases}
     \displaystyle{\frac{c_\nu^2 r}{\sqrt{c_\nu^2 r^2+b_\nu^2}}=\frac{c_\mu^2 \gamma}{\sqrt{c_\mu^2 \gamma^2+b_\mu^2}}} \tag{6.8}\label{6.8}  \\
     \displaystyle{\frac{c_\nu^2 b_\nu^2}{\left(c_\nu^2 r^2+b_\nu^2\right)^\frac{2}{3}}=\frac{c_\mu^2 b_\mu^2}{\left(c_\mu^2 \gamma^2+b_\mu^2\right)^\frac{2}{3}}} \tag{6.9}\label{6.9} 
\end{numcases}
By some elementary calculation, we obtain that
$$r^2=\frac{c_\nu^2 \left(c_\mu^2-c_\nu^2\right)}{b_\nu^{\sfrac{4}{3}} c_\mu^{\sfrac{2}{3}} \left(b_\mu^{\sfrac{2}{3}} c_\nu^{\sfrac{4}{3}}-b_\nu^{\sfrac{2}{3}} c_\mu^{\sfrac{4}{3}}\right)}.$$
Our assumptions tell us that $r^2<0$. Namely, the solution $r=r_0$ is not a real number. This implies that if $r$ satisfies $\left|Z^\prime_\mp(r)\right|+\left|Z^{\prime\prime}_\mp(r)\right|\le C$, then the solution range of $r$ does not touch the real line. Therefore, it has empty solution set if $r$ is restricted in $\left[2^{-D_0},2^{D_0}\right]$.\par
(b) We still follow the proof of Proposition 8.8 in \cite{y2} and use all the notations of the proof of Proposition 8.8 in \cite{y2}. First, by some calcualtion, we note that (8.35) and the first line in (8.36) in \cite{y2} remain the same. However, the precise expressions of $F$ and $G(r_*)\triangleq F(s_*,r_*,s_*-r_*)$ are different. In fact, in our case, we have that
$$G(r_*)=c_\sigma^4\left(c_\mu^2-c_\nu^2\right)^2 r_*^4.$$
Thus, when $c_\mu\neq c_\nu$, we can follow the proof in \cite{y2}. On the other hand, when $c_\mu=c_\nu$, we need to investigate the function $\Upsilon$ more carefully. We first assume that $c_\sigma\neq c_\mu$. From (8.35) in \cite{y2}, we have that
\begin{align*}
    F(s_*,r_*,\rho_*)=&2c_\sigma^2 c_\mu^4 c_\nu^2(r_*^2-b_\nu^2)(s_*^2-b_\sigma^2)+2c_\sigma^4 c_\mu^2 c_\nu^2 (r_*^2-b_\nu^2)(\rho_*^2-b_\mu^2)+2c_\sigma^2 c_\mu^2 c_\nu^4 (s_*^2-b_\sigma^2)(\rho_*^2-b_\mu^2)  \\
    &-c_\sigma^4 c_\mu^4(r_*^2-b_\nu^2)^2-c_\mu^4 c_\nu^4(s_*^2-b_\sigma^2)^2-c_\sigma^4 c_\nu^4(\rho_*^2-b_\mu^2)^2-4c_\sigma^2c_\mu^4 c_\nu^2 s_* r_* (\rho_*^2-b_\mu^2)  \\
    &-2c_\sigma^4 c_\mu^2 c_\nu^4\left[\frac{1}{c_\nu^2}(r_*^2-b_\nu^2)+\frac{1}{c_\sigma^2}(s_*^2-b_\sigma^2)-\frac{1}{c_\mu^2}(\rho_*^2-b_\mu^2)\right]  \\
    &-2c_\sigma^4 c_\mu^4 c_\nu^2 \left[\frac{1}{c_\nu^2}(r_*^2-b_\nu^2)-\frac{1}{c_\sigma^2}(s_*^2-b_\sigma^2)-\frac{1}{c_\mu^2}(\rho_*^2-b_\mu^2)\right]  \\
    &-2c_\sigma^2 c_\mu^4 c_\nu^4 \left[\frac{1}{c_\nu^2}(r_*^2-b_\nu^2)+\frac{1}{c_\mu^2}(\rho_*^2-b_\mu^2)-\frac{1}{c_\sigma^2}(s_*^2-b_\sigma^2)\right]
\end{align*}
By some elementary computation, it turns out that we have that
$$F_{(4)}(s_*,r_*,s_*-r_*)=c_\mu^4\left(c_\sigma^2-c_\mu^2\right)^2 \left(s_*^4-2s_*^3 r_*\right),$$
$$F_{(2)}(s_*,r_*,s_*-r_*)=\left(4c_\sigma^2 c_\mu^6 b_\mu^2+2c_\sigma^4 c_\mu^4 b_\nu^2-2c_\sigma^2 c_\mu^6 b_\sigma^2-2c_\sigma^4 c_\mu^4 b_\mu^2-2c_\sigma^2 c_\mu^6 b_\nu^2-2c_\sigma^2 c_\mu^6 b_\mu^2+2c_\mu^8 b_\sigma^2\right)s_* r_*.$$
Note that $s_*$ is fixed. Therefore, $G(r_*)$ is a linear function in $r_*$ or a constant function in $r_*$. The coefficient of the term $r_*$ is
\begin{align*}
   &-2c_\mu^4\left(c_\sigma^2-c_\mu^2\right)^2 s_*^3+ c_\mu^4\left[4c_\sigma^2 c_\mu^2 b_\mu^2+2c_\sigma^4 b_\nu^2-2c_\sigma^2 c_\mu^2 b_\sigma^2-2c_\sigma^4 b_\mu^2-2c_\sigma^2 c_\mu^2 b_\nu^2-2c_\sigma^2 c_\mu^2 b_\mu^2+2c_\mu^4 b_\sigma^2\right] \\
   =&-2c_\mu^4\left(c_\sigma^2-c_\mu^2\right)s_*\left[\left(c_\sigma^2-c_\mu^2\right)s_*^2+c_\sigma^2 b_\mu^2+c_\mu^2 b_\sigma^2-c_\sigma^2 b_\nu^2\right],
\end{align*}
and the constant term $G(0)$ is
$$c_\mu^4\left(c_\sigma^2-c_\mu^2\right)^2 s_*^4+2c_\sigma^2 c_\mu^6 b_\sigma^2\left(b_\nu^2+b_\mu^2\right)-c_\sigma^4 c_\mu^4 \left(b_\nu^2-b_\mu^2\right)^2-c_\mu^8 b_\sigma^4.$$
Now, we claim that there exists a very small $C=C(c_\sigma,c_\mu,b_\sigma,b_\mu,b_\nu)>0$ such that 
\begin{equation}
   \left\{
   \begin{aligned}
   &\left|-2c_\mu^4\left(c_\sigma^2-c_\mu^2\right)s_*\left[\left(c_\sigma^2-c_\mu^2\right)s_*^2+c_\sigma^2 b_\mu^2+c_\mu^2 b_\sigma^2-c_\sigma^2 b_\nu^2\right]\right|< C \vspace{0.6em} \\
   &\left|c_\mu^4\left(c_\sigma^2-c_\mu^2\right)^2 s_*^4+2c_\sigma^2 c_\mu^6 b_\sigma^2\left(b_\nu^2+b_\mu^2\right)-c_\sigma^4 c_\mu^4 \left(b_\nu^2-b_\mu^2\right)^2-c_\mu^8 b_\sigma^4\right|< C 
\end{aligned}\tag{6.10}\label{6.10}
\right.
\end{equation}
does not have a solution $s_*$. Since $c_\sigma\neq c_\mu$ and $s_*\ge b_\sigma $, it suffice to show that $\exists\, C>0$ such that
$$\begin{cases}
   \left|\left(c_\sigma^2-c_\mu^2\right)s_*^2+c_\sigma^2 b_\mu^2+c_\mu^2 b_\sigma^2-c_\sigma^2 b_\nu^2\right|< C \vspace{0.6em} \\ 
   \left|c_\mu^4\left(c_\sigma^2-c_\mu^2\right)^2 s_*^4+2c_\sigma^2 c_\mu^6 b_\sigma^2\left(b_\nu^2+b_\mu^2\right)-c_\sigma^4 c_\mu^4 \left(b_\nu^2-b_\mu^2\right)^2-c_\mu^8 b_\sigma^4\right|< C
\end{cases}$$
does not have a solution $s_*$ either. In fact, as for the first inequality, we can set 
$$\left(c_\sigma^2-c_\mu^2\right)s_*^2=c_\sigma^2 b_\nu^2-c_\sigma^2 b_\mu^2-c_\mu^2 b_\sigma^2+\varepsilon,$$
where $\varepsilon<C$. Then, we plug it into the second inequality
$$\left|c_\mu^4\left[c_\sigma^2 b_\nu^2-c_\sigma^2 b_\mu^2-c_\mu^2 b_\sigma^2+\varepsilon\right]^2+2c_\sigma^2 c_\mu^6 b_\sigma^2\left(b_\nu^2+b_\mu^2\right)-c_\sigma^4 c_\mu^4 \left(b_\nu^2-b_\mu^2\right)^2-c_\mu^8 b_\sigma^4\right|< C.$$
Simplify it and get that
\begin{align}
   \left|4c_\mu^6 c_\sigma^2 b_\mu^2 b_\sigma^2+\varepsilon\left(2c_\mu^4 c_\sigma^2 b_\nu^2-2c_\mu^4 c_\sigma^2 b_\mu^2-2c_\mu^6 b_\sigma^2\right)+\varepsilon^2 c_\mu^4\right|<C. \tag{6.11}\label{6.11}
\end{align}
Now, if for any $\sigma,\mu,\nu$, we always have $c_\sigma^2 b_\nu^2-c_\sigma^2 b_\mu^2-c_\mu^2 b_\sigma^2=0$, then we take
$$C\triangleq\min\left\{\frac{1}{10}\min\limits_{\sigma,\mu,\nu}{\frac{4c_\mu^6 c_\sigma^2 b_\mu^2 b_\sigma^2}{c_\mu^4}},\,\min\limits_{\sigma,\mu,\nu}\left\{c_\mu^6 c_\sigma^2 b_\mu^2 b_\sigma^2\right\},\,1\right\}>0;$$
otherwise, we take
$$C\triangleq\min\left\{\frac{1}{10}\min\limits_{\sigma,\mu,\nu}{\frac{4c_\mu^6 c_\sigma^2 b_\mu^2 b_\sigma^2}{\left|2c_\mu^4 c_\sigma^2 b_\nu^2-2c_\mu^4 c_\sigma^2 b_\mu^2-2c_\mu^6 b_\sigma^2\right|}},\,\frac{1}{10}\min\limits_{\sigma,\mu,\nu}{\frac{4c_\mu^6 c_\sigma^2 b_\mu^2 b_\sigma^2}{c_\mu^4}},\,\min\limits_{\sigma,\mu,\nu}\left\{c_\mu^6 c_\sigma^2 b_\mu^2 b_\sigma^2\right\},\,1\right\}>0.$$
Then, we have
$$\mbox{LHS of (\ref{6.11})}>4c_\mu^6 c_\sigma^2 b_\mu^2 b_\sigma^2-\frac{1}{10}\left(4c_\mu^6 c_\sigma^2 b_\mu^2 b_\sigma^2\right)-\frac{1}{10}\left(4c_\mu^6 c_\sigma^2 b_\mu^2 b_\sigma^2\right)=3.2 c_\mu^6 c_\sigma^2 b_\mu^2 b_\sigma^2>C,$$
which is a contradiction in view of (\ref{6.11}). This shows that (\ref{6.10}) does not have a solution $s_*$. This implies that for any $s_*\in[b_\sigma,\sqrt{c_\sigma^2\cdot 2^{2k}+b_\sigma^2}]$, either the coefficient of the first power term $r_*$ has a lower positive bound or the constant term $G(0)$ has a lower positive bound. Thus, in both cases, we can follow the proof in \cite{y2} to get our desired results. Finally, we assume that $c_\sigma=c_\mu(=c_\nu)$. In this case, the coefficient of the first power term $r_*$ is unfortunately zero. However, the constant term $G(0)$ is
\begin{align*}
   &c_\sigma^8\left[2 b_\sigma^2\left(b_\nu^2+b_\mu^2\right)-\left(b_\nu^2-b_\mu^2\right)^2- b_\sigma^4\right] \\
   =&c_\sigma^8\left(2b_\sigma^2 b_\mu^2+2b_\sigma^2 b_\nu^2+2b_\mu^2 b_\nu^2-b_\sigma^4-b_\mu^4-b_\nu^4\right)  \\
   =&c_\sigma^8(b_\sigma+b_\mu+b_\nu)(-b_\sigma+b_\mu+b_\nu)(b_\sigma-b_\mu+b_\nu)(b_\sigma+b_\mu-b_\nu).
\end{align*}
Since $b_\sigma-b_\mu-b_\nu\neq0$ for any $\sigma,\mu,\nu$, we get that $G(0)\gtrsim 1$. Thus, we can still get our desired result.\par
(c) This was proved in Proposition 8.8 in \cite{y2}. Note that Lemma 1, Corollary 2 and Proposition 5 are used for the decomposition of $E^{\prime\prime}=E_1^{\prime\prime}\cup E_2^{\prime\prime}$ showed in \cite{y2}. 
\end{proof}
\vspace{0.8em}
\begin{prop}
\ \par
(a) Assume $b_\sigma-b_\mu-b_\nu\neq0$ and $(c_\mu-c_\nu)(c_\mu^2b_\nu-c_\nu^2b_\mu)\ge0$. If $2^l,2^{k_2},\kappa_\theta\le 2^{-\frac{D_0}{10}}$, and define
$$E^{\prime}=\left\{(\xi,\eta): \left|\Phi(\xi,\eta)\right|\le 2^l,\left|\Omega_\eta\Phi(\xi,\eta)\right|\le\kappa_\theta\right\},$$
then, we have
$$\sup\limits_\xi \int_{\mathbb{R}^2} \boldsymbol{1}_{E^{\prime}}(\xi,\eta)\varphi_k(\xi)\varphi_{k_1}(\xi-\eta)\varphi_{k_2}(\eta)\,d\eta\lesssim \kappa_\theta 2^l\left|l\right|,$$
$$\sup\limits_\eta \int_{\mathbb{R}^2} \boldsymbol{1}_{E^{\prime}}(\xi,\eta)\varphi_k(\xi)\varphi_{k_1}(\xi-\eta)\varphi_{k_2}(\eta)\,d\xi\lesssim \kappa_\theta 2^{-k_2} 2^l\left|l\right|.$$\par
(b) Assume $b_\sigma-b_\mu-b_\nu\neq0$ and $(c_\mu-c_\nu)(c_\mu^2b_\nu-c_\nu^2b_\mu)\ge0$. If $2^l,\kappa_\theta\le 2^{-\frac{D_0}{10}}$, and $2^k+2^{k_1}+2^{k_2}\le U\in[1,\infty)$ then
$$\sup\limits_\xi \int_{\mathbb{R}^2} \varphi(2^{-l}\Phi(\xi,\eta))\varphi_k(\xi)\varphi_{k_1}(\xi-\eta)\varphi_{k_2}(\eta)\,d\eta\lesssim U^8 2^l\left|l\right| 2^{\min(k_1,k_2)},$$
$$\sup\limits_\eta \int_{\mathbb{R}^2} \varphi(2^{-l}\Phi(\xi,\eta))\varphi_k(\xi)\varphi_{k_1}(\xi-\eta)\varphi_{k_2}(\eta)\,d\xi\lesssim U^8 2^l\left|l\right| 2^{\min(k_1,k)}.$$
\end{prop}
\begin{proof}
This was proved in Lemma 8.9 in \cite{y2}.
\end{proof}
\vspace{0.8em}
\begin{prop}
Assume that $l$,$-n$,$p\le -D_0/10$. Then
\begin{flalign*}
&\text{(a)} &   &\left\|\int_{\mathbb{R}^2} \varphi_l (\Phi_{\sigma\mu\nu}(\xi,\eta))\,\varphi_n (\Psi^*_\mu(\xi-\eta))\,\widehat{f}(\xi-\eta)\,\widehat{g}(\eta)\,d\eta \right\|_{L_\xi^2} & &\  \\
    &\ & \lesssim\,&2^{\frac{l+n}{2}}\,\left\|\,\sup_{\theta\in\mathbb{S}^1} {\left|\widehat{f}(r\theta)\right|}\,\right\|_{L^2(rdr)}\,\left\|g\right\|_{L^2}.& &\ \\
&\ & &\left\|\int_{\mathbb{R}^2} \varphi_l (\Phi_{\sigma\mu\nu}(\xi,\eta))\,\varphi_n (\Psi^*_\nu(\eta))\,\widehat{f}(\xi-\eta)\,\widehat{g}(\eta)\,d\eta \right\|_{L_\xi^2} & &\  \\
    &\ & \lesssim\,&2^{\frac{l+n}{2}}\,\left\|\,\sup_{\theta\in\mathbb{S}^1} {\left|\widehat{g}(r\theta)\right|}\,\right\|_{L^2(rdr)}\,\left\|f\right\|_{L^2}.& &\ \\
&\text{(b)} &   &\left\|\int_{\mathbb{R}^2} \varphi_l (\Phi_{\sigma\mu\nu}(\xi,\eta))\,\varphi_n (\Psi^*_\mu(\xi-\eta))\,\varphi_p (\Psi^*_\nu(\eta))\,\widehat{f}(\xi-\eta)\,\widehat{g}(\eta)\,d\eta \right\|_{L_\xi^2} & &\  \\
    &\ & \lesssim\,&\min\left\{{2^{l/2},2^{l/4+p/4}}\right\}\cdot2^{n/2}\,\left\|\,\sup_{\theta\in\mathbb{S}^1} {\left|\widehat{f}(r\theta)\right|}\,\right\|_{L^2(rdr)}\,\left\|g\right\|_{L^2}.& &\ 
\end{flalign*}
\end{prop}
\begin{proof}
These were basically proved in Lemma 8.10 in \cite{y2}. The proof of part (a) is exactly 
 same as in \cite{y2}. Let's discuss the proof of part (b). Denote $\xi=(s,0)$ and $\eta=(r\cos{\alpha},r\sin{\alpha})$. The difference here is that we don't have $\left|\angle\xi,\eta\right|\gtrsim 1$ anymore, so it's possible that 
$$\bigg|\partial_\alpha\left|\xi-\eta\right|\bigg|=\bigg|\frac{sr}{\left|\xi-\eta\right|}\sin \alpha\bigg|\lesssim 1.$$
But, if so, then we notice that
$$\bigg|\partial_{\alpha\alpha}\left|\xi-\eta\right|\bigg|=\left|\frac{sr\left(\cos\alpha\left|\xi-\eta\right|^2-sr\sin^2\alpha\right)}{\left|\xi-\eta\right|^3}\right|\sim\frac{sr}{\left|\xi-\eta\right|}\sim 1.$$
Thus, instead, we here will have
$$\sup_{r\approx 1} \int_{\theta\in\mathbb{S}^1} \varphi_l(\Phi_{\sigma\mu\nu}(\xi,\xi-r\theta))\,\varphi_p(\Psi^*_\nu(\xi-r\theta))\,d\theta\lesssim \min\left\{2^{l/2},2^{p/2}\right\},$$
which is the analogue of (8.57) in \cite{y2}. Now on the one hand, we fix $\xi$, let $\xi-\eta=r\theta$ and get
\begin{align*}
   &\sup_{\xi} {\int_{\mathbb{R}^2} \left|\varphi_l (\Phi_{\sigma\mu\nu}(\xi,\eta))\varphi_n (\Psi^*_\mu(\xi-\eta))\varphi_p(\Psi^*_\nu(\eta))\widehat{f}(\xi-\eta)\right|\,d\eta} \\
   =&\sup_{\xi} {\int^{+\infty}_0 \left| \varphi_n (\Psi^*_\mu(r))\,\sup_{\theta}{\widehat{f}(r\theta)}\right|\int_{\mathcal{S}^1}\left|\varphi_l(\Phi_{\sigma\mu\nu}(\xi,\xi-r\theta))\varphi_p(\Psi^*_\nu(\xi-r\theta))\right|\,d\theta\,rdr} \\
   \le&\min(2^{l/2},2^{p/2})\left(\int^{\infty}_0 \varphi_n(\Psi^*_\mu(r))\,rdr\right)^{1/2}\cdot\left(\int^{\infty}_0\left(\sup_{\theta}\left|\widehat{f}(r\theta)\right|\right)^2\,rdr\right)^{1/2} \\
   \lesssim &\min(2^{l/2},2^{p/2})\cdot 2^{n/2} \cdot\left\|\sup_{\theta}\left|\widehat{f}(r\theta)\right|\,\right\|_{L^2(rdr)};
\end{align*}
on the other hand, we fix $\eta$, again let $\xi-\eta=r\theta$ and get
\begin{align*}
   &\sup_{\eta} {\int_{\mathbb{R}^2} \left|\varphi_l (\Phi_{\sigma\mu\nu}(\xi,\eta))\varphi_n (\Psi^*_\mu(\xi-\eta))\varphi_p(\Psi^*_\nu(\eta))\widehat{f}(\xi-\eta)\right|\,d\xi} \\ 
   \le &\varphi_l (\Psi^*_\nu(\eta)) \sup_\theta \int^\infty_0 \left| \varphi_n (\Psi^*_\mu(r))\,\sup_{\theta}{\widehat{f}(r\theta)}\right|\int_{\mathcal{S}^1}\left|\varphi_l(\Phi_{\sigma\mu\nu}(\xi,\xi-r\theta))\right|\,d\theta\,rdr \\
   \lesssim & 2^{l/2}\cdot 2^{n/2}\cdot\left\|\,\sup_{\theta}\left|\widehat{f}(r\theta)\right|\,\right\|_{L^2(rdr)}
\end{align*}
Therefore, by Schur's lemma, we end up with the desired result above.
\end{proof}
\vspace{0.8em}
Sometimes we need to change the weigh of coefficients of the $l,n,p$ on the power, so we also have the following.
\vspace{0.8em}
\begin{cor}
   \ \par 
    (a) Assume that $l$,$-n$,$p\le -D_0/10$ and $\left|\nabla_\eta\Phi_{\sigma\mu\nu}(\xi,\eta)\right|\gtrsim 1$. Then
\begin{align*}
    &\left\|\int_{\mathbb{R}^2} \varphi_l (\Phi_{\sigma\mu\nu}(\xi,\eta))\,\varphi_n (\Psi^*_\mu(\xi-\eta))\,\varphi_p (\Psi^*_\nu(\eta))\,\widehat{f}(\xi-\eta)\,\widehat{g}(\eta)\,d\eta \right\|_{L_\xi^2} \\
    \lesssim\,& 2^{\frac{l}{2}+\frac{n}{4}+\frac{p}{4}}\,\left\| f\right\|_{L^2}\,\left\|g\right\|_{L^2}.
\end{align*}\par
    Similarly, if $l$,$-n$,$p\le -D_0/10$, $-j\le q\le -D$ and $\left|\nabla_\eta\Phi_{\sigma\mu\nu}(\xi,\eta)\right|\sim 2^q$, then
\begin{align*}
    &\left\|\int_{\mathbb{R}^2} \varphi_l (\Phi_{\sigma\mu\nu}(\xi,\eta))\,\varphi_n (\Psi^*_\mu(\xi-\eta))\,\varphi_p (\Psi^*_\nu(\eta))\,\widehat{f}(\xi-\eta)\,\widehat{g}(\eta)\,d\eta \right\|_{L_\xi^2} \\
    \lesssim\,& 2^{\frac{l}{2}+\frac{n}{4}+\frac{p}{4}-\frac{q}{4}}\,\left\|\,\sup_{\theta}\left|\widehat{f}(r\theta)\right|\,\right\|_{L^2(rdr)}\,\left\|g\right\|_{L^2}; 
\end{align*}\par
    if $l$,$-n$,$p\le -D_0/10$ and $\left|\nabla_\eta\Phi_{\sigma\mu\nu}(\xi,\eta)\right|\lesssim 2^{-j}$, then
\begin{align*}
    &\left\|\int_{\mathbb{R}^2} \varphi_l (\Phi_{\sigma\mu\nu}(\xi,\eta))\,\varphi_n (\Psi^*_\mu(\xi-\eta))\,\varphi_p (\Psi^*_\nu(\eta))\,\widehat{f}(\xi-\eta)\,\widehat{g}(\eta)\,d\eta \right\|_{L_\xi^2} \\
    \lesssim\,& 2^{-\frac{j}{2}+\frac{n}{4}+\frac{p}{4}}\,\left\|f\right\|_{L^2}\,\left\|g\right\|_{L^2}. 
\end{align*}\par
(b)  Assume that $l$,$-n$,$p\le -D_0/10$. Then
\begin{align*}
    &\left\|\int_{\mathbb{R}^2} \varphi_l (\Phi_{\sigma\mu\nu}(\xi,\eta))\,\varphi_n (\Psi^*_\mu(\xi-\eta))\,\varphi_p (\Psi^*_\nu(\eta))\,\widehat{f}(\xi-\eta)\,\widehat{g}(\eta)\,d\eta \right\|_{L_\xi^2} \\
    \lesssim\,& 2^{\frac{n}{2}+\frac{p}{2}}\,\left\| f\right\|_{L^2}\,\left\|g\right\|_{L^2}.
\end{align*}
\end{cor}
\begin{proof}
\ \par
(a) We first prove the second case. The first case can be done similarly. First, we note that by Cauchy-Schwarz, we have
\begin{align*}
    &\sup_{\xi} {\int_{\mathbb{R}^2} \left|\varphi_l (\Phi_{\sigma\mu\nu}(\xi,\eta))\varphi_n (\Psi^*_\mu(\xi-\eta))\varphi_p(\Psi^*_\nu(\eta))\widehat{f}(\xi-\eta)\right|\,d\eta} \\
    \lesssim &\left\|f\right\|_{L^2}\cdot\left(\int \left|\varphi_l(\Phi_{\sigma\mu\nu}(\xi,\eta))\cdot\varphi_n(\Psi^*_\mu(\xi-\eta))\cdot\varphi_p(\Psi^*_\nu(\eta))\right|\,d\eta\right)^{1/2} \\
    \lesssim &\left\|f\right\|_{L^2} \left(2^l\cdot2^{-q}\cdot2^p\right)^\frac{1}{2} \\
    \lesssim &\left\|\,\sup_{\theta}\left|\widehat{f}(r\theta)\right|\,\right\|_{L^2(rdr)} \left(2^l\cdot2^{-q}\cdot2^p\right)^\frac{1}{2}
    ;
\end{align*}
on the other hand, as before, we fix $\eta$, again let $\xi-\eta=r\theta$ and get
\begin{align*}
   &\sup_{\eta} {\int_{\mathbb{R}^2} \left|\varphi_l (\Phi_{\sigma\mu\nu}(\xi,\eta))\varphi_n (\Psi^*_\mu(\xi-\eta))\varphi_p(\Psi^*_\nu(\eta))\widehat{f}(\xi-\eta)\right|\,d\xi} \\ 
   \le &\varphi_l (\Psi^*_\nu(\eta)) \sup_\theta \int^\infty_0 \left| \varphi_n (\Psi^*_\mu(r))\,\sup_{\theta}{\widehat{f}(r\theta)}\right|\int_{\mathcal{S}^1}\left|\varphi_l(\Phi_{\sigma\mu\nu}(\xi,\xi-r\theta))\right|\,d\theta\,rdr \\
   \lesssim & 2^{l/2}\cdot 2^{n/2}\cdot\left\|\,\sup_{\theta}\left|\widehat{f}(r\theta)\right|\,\right\|_{L^2(rdr)}.
\end{align*}
Therefore, Schur's Lemma gives us the second result.\par
Next, we prove the third case. This follows from Schur's Lemma similarly by noting that
\begin{align*}
    &\sup_{\xi} {\int_{\mathbb{R}^2} \left|\varphi_l (\Phi_{\sigma\mu\nu}(\xi,\eta))\varphi_n (\Psi^*_\mu(\xi-\eta))\varphi_p(\Psi^*_\nu(\eta))\widehat{f}(\xi-\eta)\right|\,d\eta} \\
    \lesssim &\left\|f\right\|_{L^2}\cdot\left(\int \left|\varphi_l(\Phi_{\sigma\mu\nu}(\xi,\eta))\cdot\varphi_n(\Psi^*_\mu(\xi-\eta))\cdot\varphi_p(\Psi^*_\nu(\eta))\right|\,d\eta\right)^{1/2} \\
    \lesssim &\left\|f\right\|_{L^2} \left(2^{-j}\cdot2^p\right)^\frac{1}{2}
\end{align*}
and
\begin{align*}
    &\sup_{\eta} {\int_{\mathbb{R}^2} \left|\varphi_l (\Phi_{\sigma\mu\nu}(\xi,\eta))\varphi_n (\Psi^*_\mu(\xi-\eta))\varphi_p(\Psi^*_\nu(\eta))\widehat{f}(\xi-\eta)\right|\,d\xi} \\
    \lesssim &\left\|f\right\|_{L^2}\cdot\left(\int \left|\varphi_l(\Phi_{\sigma\mu\nu}(\xi,\eta))\cdot\varphi_n(\Psi^*_\mu(\xi-\eta))\cdot\varphi_p(\Psi^*_\nu(\eta))\right|\,d\xi\right)^{1/2} \\
    \lesssim &\left\|f\right\|_{L^2} \left(2^{-j}\cdot2^n\right)^\frac{1}{2}.
\end{align*}
(b) This follows from Schur's Lemma,
\begin{align*}
    &\sup_{\xi} {\int_{\mathbb{R}^2} \left|\varphi_l (\Phi_{\sigma\mu\nu}(\xi,\eta))\varphi_n (\Psi^*_\mu(\xi-\eta))\varphi_p(\Psi^*_\nu(\eta))\widehat{f}(\xi-\eta)\right|\,d\eta} \\
    \lesssim &\left\|f\right\|_{L^2}\cdot\left(\int \left|\varphi_l(\Phi_{\sigma\mu\nu}(\xi,\eta))\cdot\varphi_n(\Psi^*_\mu(\xi-\eta))\cdot\varphi_p(\Psi^*_\nu(\eta))\right|\,d\eta\right)^{1/2} \\
    \lesssim &\left\|f\right\|_{L^2} \left(2^{2p}\right)^\frac{1}{2}\lesssim \left\|f\right\|_{L^2}\cdot 2^p,
\end{align*}
and
\begin{align*}
    &\sup_{\eta} {\int_{\mathbb{R}^2} \left|\varphi_l (\Phi_{\sigma\mu\nu}(\xi,\eta))\varphi_n (\Psi^*_\mu(\xi-\eta))\varphi_p(\Psi^*_\nu(\eta))\widehat{f}(\xi-\eta)\right|\,d\xi} \\
    \lesssim &\left\|f\right\|_{L^2}\cdot\left(\int \left|\varphi_l(\Phi_{\sigma\mu\nu}(\xi,\eta))\cdot\varphi_n(\Psi^*_\mu(\xi-\eta))\cdot\varphi_p(\Psi^*_\nu(\eta))\right|\,d\xi\right)^{1/2} \\
    \lesssim &\left\|f\right\|_{L^2} \left(2^{2n}\right)^\frac{1}{2}\lesssim \left\|f\right\|_{L^2}\cdot 2^n.
\end{align*}
\end{proof}

\vspace{1em}

\bibliographystyle{plain}
\bibliography{sample}

\end{document}